\documentclass[10pt]{article}
\usepackage{amsthm}
\usepackage{amsmath}
\usepackage{thmtools}
\usepackage{amsxtra}
\usepackage{amssymb}
\usepackage{enumitem}
\usepackage{mathrsfs}
\usepackage[all]{xy}
\usepackage{graphics}
\usepackage{graphicx}
\usepackage{tikz}
\usepackage{accents}
\usepackage{scalerel}
\usepackage[margin=1.6in]{geometry}
\usepackage{float}
\newfloat{lablist}{H}{lol}
\floatname{lablist}{List}

\usetikzlibrary{matrix,positioning,calc}
\usetikzlibrary{decorations.pathmorphing,shapes}
\usepackage{url}
\usepackage[numbers,square,comma]{natbib}
\usepackage[colorlinks=true,linkcolor=blue,citecolor=blue,
urlcolor=blue]{hyperref}

\newcommand{\somevarsigma}{\varpi}
\newcommand{\somevarpi}{\varsigma}
\newcommand{\solTt}{\Uu}
\newcommand{\solUu}{\Tt}
\newcommand{\Shift}{\mathrm{Shift}}

\newcommand{\Da}{\mathrm{Da}}

\newcommand{\CC}{\mathbb C}

\newcommand{\RR}{\mathbb R}

\newcommand{\PP}{\mathbb P}

\newcommand{\sub}{\subseteq}
\newcommand{\cross}{\times}
\newcommand{\all}{\forall}

\newcommand{\inter}{\cap}
\renewcommand{\int}{\inter}

\newcommand{\om}{\omega}
\newcommand{\pow}{\mathcal{P}}
\newcommand{\OR}{\mathrm{OR}}

\newcommand{\Hull}{\mathrm{Hull}}
\newcommand{\cut}{\backslash}

\newcommand{\Tt}{\mathcal{T}}
\newcommand{\Ss}{\mathcal{S}}
\newcommand{\Uu}{\mathcal{U}}
\newcommand{\Vv}{\mathcal{V}}
\newcommand{\Ww}{\mathcal{W}}
\newcommand{\Ll}{\mathcal{L}}
\newcommand{\treeTt}{\Tt}
\newcommand{\treeUu}{\Uu}
\newcommand{\Mm}{\mathcal{M}}

\newcommand{\Gend}{\mathrm{Gen'd}}

\newcommand{\Ssbar}{{\bar{\Ss}}}
\newcommand{\varthetabar}{\bar{\vartheta}}
\newcommand{\Ttbar}{{\bar{\Tt}}}

\newcommand{\Xxbar}{{\bar{\Xx}}}
\newcommand{\rg}{\mathrm{rg}}
\newcommand{\dom}{\mathrm{dom}}
\newcommand{\zetabar}{\bar{\zeta}}
\newcommand{\ins}{\trianglelefteq}
\newcommand{\nins}{\ntrianglelefteq}
\newcommand{\pins}{\triangleleft}
\newcommand{\npins}{\ntriangleleft}
\newcommand{\crit}{\mathrm{cr}}

\newcommand{\union}{\cup}
\newcommand{\rest}{\!\upharpoonright\!}
\newcommand{\com}{\circ}
\newcommand{\range}{\rg}
\newcommand{\lh}{\mathrm{lh}}
\newcommand{\Ult}{\mathrm{Ult}}

\newcommand{\sats}{\models}

\newcommand{\J}{\mathcal{J}}
\newcommand{\AC}{\mathrm{AC}}

\newcommand{\HOD}{\mathrm{HOD}}
\newcommand{\HC}{\mathrm{HC}}
\newcommand{\ZFC}{\mathrm{ZFC}}
\newcommand{\KP}{\mathrm{KP}}
\newcommand{\ZF}{\mathrm{ZF}}
\newcommand{\pistol}{\P}
\newcommand{\Coll}{\mathrm{Col}}
\newcommand{\es}{\mathbb{E}}
\newcommand{\jbar}{\bar{j}}
\newcommand{\kbar}{\bar{k}}
\newcommand{\tbar}{\bar{t}}
\newcommand{\pbar}{\bar{p}}
\newcommand{\sq}{\mathrm{sq}}
\newcommand{\eps}{\varepsilon}

\newcommand{\Ttvec}{{\vec{\Tt}}}

\newcommand{\thetabar}{{\bar{\theta}}}
\newcommand{\deltabar}{\bar{\delta}}
\newcommand{\gammabar}{\bar{\gamma}}
\newcommand{\etabar}{\bar{\eta}}

\newcommand{\Gbar}{{\bar{G}}}

\newcommand{\exchnu}{\widetilde{\nu}}

\newcommand{\core}{\mathfrak{C}}

\newcommand{\her}{\mathcal{H}}

\newcommand{\pred}{\mathrm{pred}}

\newcommand{\un}{\union}

\newcommand{\id}{\mathrm{id}}

\newcommand{\conc}{\ \widehat{\ }\ }

\newcommand{\bfSigma}{\undertilde{\Sigma}}

\newcommand{\rSigma}{\mathrm{r}\Sigma}

\newcommand{\rPi}{\mathrm{r}\Pi}

\newcommand{\vareps}{\varepsilon}

\newcommand{\supp}{\mathbb{S}}
\DeclareMathOperator{\Th}{Th}

\DeclareMathOperator{\card}{card}
\DeclareMathOperator{\cof}{cof}
\DeclareMathOperator{\wfp}{wfp}

\newcommand{\Two}{\mathrm{II}}

\newcommand{\Dd}{\mathcal{D}}

\newcommand{\bfrSigma}{\undertilde{\rSigma}}

\newcommand{\psub}{\subsetneq}

\newcommand{\Xx}{\mathcal{X}}

\newcommand{\cHull}{\mathrm{cHull}}

\newcommand{\Mbar}{\bar{\M}}

\newcommand{\lpole}{\left\lfloor}
\newcommand{\rpole}{\right\rfloor}

\newcommand{\univ}[1]{\lpole #1\rpole}

\newcommand{\tu}{\textup}

\declaretheoremstyle[bodyfont=\sl]{slanted}

\swapnumbers
\declaretheorem[name=Definition,style=definition,qed=$\dashv$,
numberwithin=section]{dfn}
\declaretheorem[name=Definition,style=definition,numbered=no,qed=$\dashv$]{dfn*}
\declaretheorem[name=Definition,style=definition,numbered=no]{dfnnoqed*}

\declaretheorem[name=Theorem,style=slanted,sibling=dfn]{tm}
\declaretheorem[name=Theorem,style=slanted,numbered=no]{tm*}
\declaretheorem[name=Lemma,style=slanted,sibling=dfn]{lem}
\declaretheorem[name=Corollary,style=slanted,sibling=dfn]{cor}
\declaretheorem[name=Corollary,style=slanted,numbered=no]{cor*}
\declaretheorem[name=Remark,style=definition,sibling=dfn]{rem}
\declaretheorem[name=Question,style=definition,sibling=dfn]{ques}

\declaretheorem[name=Fact,style=definition,sibling=dfn]{fact}

\declaretheorem[name=Plan,style=definition,numbered=no]{plan*}
\declaretheorem[name=Proposition,style=slanted,sibling=dfn]{prop}

\swapnumbers
\declaretheoremstyle[headfont=\scshape]{claimstyle}
\declaretheorem[name=Goal,style=claimstyle]{goal}
\declaretheorem[name=Claim,style=claimstyle]{clm}
\declaretheorem[name=Claim,style=claimstyle]{clmone}
\declaretheorem[name=Claim,style=claimstyle]{clmtwo}
\declaretheorem[name=Claim,style=claimstyle]{clmthree}
\declaretheorem[name=Claim,style=claimstyle]{clmfour}
\declaretheorem[name=Claim,style=claimstyle]{clmfive}
\declaretheorem[name=Claim,style=claimstyle]{clmsix}
\declaretheorem[name=Claim,style=claimstyle]{clmseven}
\declaretheorem[name=Claim,style=claimstyle]{clmeight}

\declaretheorem[name=Claim,style=claimstyle,numbered=no]{clm*}

\declaretheorem[name=Subclaim,style=claimstyle,numberwithin=clmone]{sclmone}
\declaretheorem[name=Subclaim,style=claimstyle,numberwithin=clmtwo]{sclmtwo}
\declaretheorem[name=Subclaim,style=claimstyle,numberwithin=clmthree]{sclmthree}

\declaretheorem[name=Subclaim,style=claimstyle,numberwithin=clmfive]{sclmfive}

\declaretheorem[name=Subclaim,style=claimstyle,numberwithin=clmseven]{sclmseven}

\declaretheorem[name=Subclaim,style=claimstyle,numbered=no]{sclm*}

\declaretheorem[name=Subsubclaim,style=claimstyle,numberwithin=sclmthree]
{ssclmthree}

\declaretheorem[name=Subsubclaim,style=claimstyle,numbered=no]{ssclm*}

\declaretheoremstyle[headfont=\scshape]{casestyle}
\declaretheorem[name=Assumption,style=casestyle]{ass}

\declaretheorem[name=Assumption,style=casestyle]{asstwo}
\declaretheorem[name=Case,style=casestyle]{case}
\declaretheorem[name=Case,style=casestyle]{caseone}
\declaretheorem[name=Case,style=casestyle]{casetwo}
\declaretheorem[name=Case,style=casestyle]{casethree}
\declaretheorem[name=Case,style=casestyle]{casefour}
\declaretheorem[name=Case,style=casestyle]{casefive}
\declaretheorem[name=Case,style=casestyle]{casesix}

\declaretheorem[name=Subcase,style=casestyle,numberwithin=case]{scase}
\declaretheorem[name=Subcase,style=casestyle,numberwithin=caseone]{scaseone}

\declaretheorem[name=Subcase,style=casestyle,numberwithin=casethree]{scasethree}
\declaretheorem[name=Subcase,style=casestyle,numberwithin=casefour]{scasefour}

\declaretheorem[name=Subsubcase,style=casestyle,numberwithin=scasefour]
{sscasefour}

\newcommand{\concopy}{\mathrm{concopy}}
\newcommand{\ttilde}{\widetilde{t}}
\newcommand{\stilde}{\widetilde{s}}
\newcommand{\passive}{\mathrm{pv}}
\newcommand{\xibar}{\bar{\xi}}

\newcommand{\ds}{\mathrm{Ds}}

\newcommand{\para}{\mathrm{p}}
\newcommand{\curlyH}{\mathscr{H}}
\renewcommand{\succ}{\mathrm{succ}}
\newcommand{\extcopy}{\mathrm{copy}}
\newcommand{\wcof}{\mathrm{wcof}}
\newcommand{\acof}{\mathrm{acof}}

\renewcommand{\deg}{\mathrm{deg}}
\newcommand{\FF}{\mathbb{F}}
\newcommand{\rep}{\mathrm{rep}}
\newcommand{\reps}{\mathrm{reps}}
\newcommand{\exitside}{\mathrm{exitside}}
\newcommand{\sides}{\mathrm{sides}}

\newcommand{\dropset}{\mathscr{D}}

\newcommand{\simple}{\mathrm{sim}}
\newcommand{\curlyB}{\mathscr{B}}
\newcommand{\movin}{\mathrm{movin}}

\newcommand{\curlyM}{\mathscr{M}}
\newcommand{\dfnemph}{\emph}

\newcommand{\modelset}{\mathfrak{A}}

\newcommand{\rhotilde}{\widetilde{\rho}}

\renewcommand{\Mbar}{\bar{M}}

\newcommand{\Gg}{\mathcal{G}}

\newcommand{\lgcd}{\mathrm{lgcd}}

\renewcommand{\Mbar}{{\bar{M}}}

\newcommand{\scB}{\mathscr{B}}
\newcommand{\nutilde}{\widetilde{\nu}}
\newcommand{\iter}{{\mathrm{iter}}}
\newcommand{\sigmatilde}{\widetilde{\sigma}}

\newcommand{\eqdef}{=_{\mathrm{def}}}

\newcommand{\trivcom}{\mathrm{trvcom}}

\renewcommand{\cut}{\backslash}

\renewcommand{\Ss}{\mathcal{S}}
\newcommand{\pvec}{\vec{p}}
\renewcommand{\Mbar}{\bar{M}}

\newcommand{\successor}{\mathrm{succ}}
\newcommand{\phP}{\mathfrak{P}}
\renewcommand{\Mbar}{{\bar{M}}}

\renewcommand{\hbar}{\bar{h}}

\newcommand{\Dw}{\mathrm{Dw}}
\newcommand{\phU}{\mathfrak{U}}
\newcommand{\phM}{\mathfrak{M}}
\newcommand{\ph}{\mathfrak{H}}
\newcommand{\phW}{\mathfrak{W}}

\newcommand{\tautilde}{\widetilde{\tau}}

\renewcommand{\SS}{\mathbb{S}}

\newcommand{\Fseg}{{F_\downarrow}}

\newcommand{\exit}{\mathrm{exit}}

\newcommand{\scM}{\mathscr{M}}

\renewcommand{\root}{\mathrm{root}}

\newcommand{\gironaurlone}{https://ivv5hpp.uni-muenster.de}
\newcommand{\gironaurltwo}{/u/rds/girona_meeting_2018.html}
\newcommand{\gironaurl}{\gironaurlone\gironaurltwo}
\newcommand{\jensenurl}{https://www.math.uni-bonn.de/~raesch/jensen/}

\newcommand{\muensterimttwfifurl}{\wwuurlone\wwuurltwo}
\newcommand{\wwuurlone}{https://ivv5hpp.uni-muenster.de}
\newcommand{\wwuurltwo}{/u/rds/muenster_meeting_2015.html}

\def\handgrenade{\scalerel*{\begin{tikzpicture}
\draw[black,fill=black] (0,0) circle (5ex);
      \draw[-]  (0.3,0) -- (0.3,-2.4);
\end{tikzpicture}}{X}}

\begin{document}

\title{Fine structure from normal iterability}

\author{Farmer Schlutzenberg\footnote{afirstname dot alastname at tuwien dot ac dot at, afirstname dot alastname at gmail dot com,
\url{https://sites.google.com/site/schlutzenberg/home-1}}\\
Institute of Discrete Mathematics and Geometry, TU Vienna}

\maketitle

\begin{abstract}We  show that (i) the standard fine structural properties for
premice follow from normal iterability (whereas the classical proof relies on iterability for stacks of normal trees). We also show that (ii) every mouse which is finitely generated above its projectum, is an iterate of its core.

That is, let $m<\om$ and let $M$ be an $m$-sound, $(m,\om_1+1)$-iterable  premouse. Then (i) $M$ is $(m+1)$-solid and $(m+1)$-universal, $(m+1)$-condensation holds for $M$, and if $m\geq 1$ then $M$ is super-Dodd-sound, a slight strengthening of Dodd-soundness. And (ii) if there is $x\in M$ such that $M$ is the $\rSigma_{m+1}$-hull of parameters in $\rho_{m+1}^M\un\{x\}$, then $M$ is a normal iterate of its $(m+1)$th core $C=\core_{m+1}(M)$; in fact, there is an $m$-maximal iteration tree $\Tt$ on $C$, of finite length, such that $M=M^\Tt_\infty$, and $i^\Tt_{0\infty}$ is just the core embedding.

Applying fact (ii), we prove that if $M\sats\ZFC$ is a mouse and $W\subseteq
M$ is a ground of $M$ via a forcing $\PP\in W$ such that $W\models$ ``$\PP$ is
strategically $\sigma$-closed'' and $M|\aleph_1^M\in W$ (that is, the initial
segment of $M$ of height $\aleph_1^M$ is in $W$), then $M\sub W$.

And if there is a measurable cardinal then there is a non-solid premouse.

The results hold for premice with Mitchell-Steel indexing, allowing extenders of superstrong type to appear on the extender sequence.
\end{abstract}

\setcounter{tocdepth}{2}
\tableofcontents

\section{Introduction}\label{sec:intro}
\subsection{Background and goals}
The large cardinal hierarchy constitutes a central focus in set theory. Our understanding of large cardinals has been greatly enriched through the discovery and study of fine structural inner models which exhibit them.
These models are highly organized and admit a precise analysis, enabling a detailed understanding of their combinatorial properties, far beyond what can be achieved just working in ZFC plus large cardinals.
They provide our key tool for establishing consistency strength lower bounds for various kinds of principles, particularly through \emph{core model} arguments.
And they have  been shown to arise naturally in other contexts, particularly in descriptive set theory, where they are intimately connected with models of determinacy.

The models come in two main varieties: \emph{pure extender mice}
and \emph{strategy mice}. From now on, this article will deal exclusively with the pure extender variety, which we will  refer to as \emph{mice}. Before we  describe the aims of the paper, we give a brief outline of the basic features of mice.

Large cardinals at the level of measurability and beyond are typically exhibited by some kind of elementary (or partially elementary) embedding
$j:P\to Q$
between  structures $P,Q$ for the language of set theory (usually $P,Q$ will be transitive, and often $P=V$, the entire set-theoretic universe).
Mice $M$ are  set-theoretic structures
having a transitive universe of the form $\J_\alpha[\es]$ where $\es$ is a sequence of \emph{extenders} with special properties.\footnote{A mouse has the form $M=(\univ{M},{\in}\rest{\univ{M}},\es,F)$
where $\univ{M}=\J_\alpha[\es]$
is the universe of $M$,
$\es$ is the aforementioned sequence of extenders, and $F\sub\univ{M}$ is another extender.} Extenders are essentially fragments of elementary embeddings. If $\beta\in\dom(\es)$
then $\beta<\alpha$ and
 $E=\es_\beta$ is an extender over $P=\J_\beta[\es\rest\beta]$, and $E$ is essentially equivalent to the corresponding
ultrapower embedding
$i_E:P\to Q$
where  $Q=\Ult(P,E)$ is the ultrapower of $P$ by $E$. In some cases,
$E$ will in fact be a \emph{total} extender over  the universe $\J_\alpha[\es]$ of $M$, and then $\Ult(M,E)$ can be formed, and $E$ determines an ultrapower embedding $i^M_E:M\to\Ult(M,E)$.
But in general this can fail,
and then we say that $E$ is only a \emph{partial} extender (in the sense of $M$). We also write $\es^M=\es$. A mouse also comes equipped with a further predicate $F\sub\J_\alpha[\es]$,
and we write $F^M=F$
and  $\es_+^M=\es^M\conc\left<F\right>$.
The predicate $F$
is either $\emptyset$
or an extender over $\J_\alpha[\es]$. (So $F$ would be $\es_\alpha$, if we were to extend $M$ above height $\alpha$.)

The keys to our understanding of mice stem from their \emph{iterability} and their \emph{fine structure}. These keys are central to our analysis of the internal combinatorial properties of mice, the relationship between mice and  the wider set-theoretic universe, and also their canonicity.
 A \emph{premouse}
is transitive and satisfies the same first order properties as does a mouse, but for a premouse there is no iterability requirement.

The iterability of $M$ requires, roughly, that we can transfinitely iterate the process of forming ultrapowers via extenders, always resulting in transitive models. It is defined precisely in  terms of the \emph{iteration game},
which is defined basically as follows, glossing over some details.
The game is played between two players, I and II,
and runs through some ordinal number $\lambda$ of stages.
A run of the game
produces a sequence $\left<M_\alpha\right>_{\alpha<\lambda}$ of premice, amongst other things.
We start with $M_0=M$. Given $M_\alpha$ where $\alpha+1<\lambda$, player I selects some $E_\alpha$ from the extender sequence $\es^{M_\alpha}_+$ of $M_\alpha$, and also some $\beta\leq\alpha$
and some $\gamma\leq\OR\cap M_\beta$
such that $\Ult(M_\beta|\gamma,E_\alpha)$ makes sense,
and we define $M_{\alpha+1}$ to be this ultrapower.
We simultaneously define an associated tree order on $\lambda$,
setting here $\beta$ to be the tree-predecessor of $\alpha+1$.
At limit ordinal stages $\gamma<\lambda$, player II
must select some set $b\sub\gamma$,
such that $b$ is  cofinal in $\gamma$ and linearly ordered by the tree ordering. We declare $b$ to be the set of tree-predecessors of $\gamma$,
and set $M_\gamma$ to be the direct limit of the models $M_\alpha$
for $\alpha\in b$ (under ultrapower maps). Player II must ensure that all models produced are well-defined and transitive, and wins iff these conditions are  maintained throughout.
The entire array of information produced in a run of the game is called an \emph{iteration tree on $M$}. An iteration tree with last model $M_\alpha$ yields a $\Sigma_1$-elementary \emph{iteration map} $i_{0\alpha}:M\to M_\alpha$, given by composing ultrapower maps,
as long as the  extenders used along the branch from $0$ to $\alpha$ are total (enough).
The \emph{$\lambda$-iterability} of $M$ requires the existence of a winning strategy for player II in this game (through $\lambda$ stages, as above), which is called an \emph{iteration strategy} for $M$. For example if $M$ is countable then a winning strategy exhibiting $\om_1$-iterability
is essentially a set of reals.
Iterability is not in general simply a feature of the first order theory of $M$.

The most fundamental consequence of iterability is that any two sufficiently iterable mice $M,N$ can be \emph{compared}, which ensures that there are iteration trees $\Tt$ and $\Uu$ on $M$ and $N$ respectively, with last models $M_\infty$ and $N_\infty$, and either $M_\infty\ins N_\infty$ (that is, $M_\infty$ is an \emph{initial segment} of $N_\infty$)\footnote{For premice $R$ and $S$, $R\ins S$ iff $\alpha\leq\OR\cap S$ and $\es_+^R=\es_+^S\cap(\alpha+1)$ where $\alpha=\OR\cap R$. Here $\OR$ denotes the class of all ordinals.} and the iteration map $i^\Tt:M\to M_\infty$ is defined, or $N_\infty\ins M_\infty$ and  $i^\Uu:N\to N_\infty$ is defined. This implies, for example, that
for any two sufficiently iterable mice $M,N$, either $\RR\cap M\sub\RR\cap N$ or vice versa,
and $M,N$ agree with each other on the order of constructibility of their common reals. Therefore all reals in sufficiently iterable mice are ordinal definable, exhibiting  the canonicity of mice.  Arguments involving comparison in various forms are essential and ubiquitous in  inner model theory.

Canonical iteration strategies $\Sigma$ for a mouse $M$ tend to have some form of the \emph{Dodd-Jensen} property,
and when available, this is very useful. Roughly, it says that if $\Tt$ is an iteration tree on $M$ formed according to $\Sigma$
with last model $M_\infty$
and $j:M\to M_\infty$
has the kind of elementarity
that the iteration map $i^\Tt:M\to M_\infty$ would have if it existed, then  $i^\Tt$ \emph{does} exist, and $i^\Tt(\alpha)\leq j(\alpha)$ for all ordinals $\alpha$ in $M$. Given this property in general, one can use it to define the \emph{mouse order} $\leq^*$,
a natural prewellorder
of mice, with $M\leq^*N$
iff the comparison of $M$ with $N$ results in iterates $M_\infty$ and $N_\infty$
with $M_\infty\ins N_\infty$
and an iteration map $i^\Tt:M\to M_\infty$ defined.

Fine structure comprises a collection of properties, of which G\"odel's condensation lemma for $L$ is a fairly prototypical instance.\footnote{The main properties we refer to here are
 condensation, solidity and universality of the standard
parameter, Dodd-solidity, and the Initial Segment Condition (ISC).
The definitions are recalled
in \S\ref{sec:notation} and elsewhere in the paper.}  These properties are all \emph{first order}: for each such property, there is a formula $\varphi$ (or sometimes a recursive theory $T$) such that given any structure $M$ with the first order signature of a mouse, the property holds of $M$ iff $M\sats\varphi$ (respectively, $M\sats T$).\footnote{For example, for $k<\om$, $k$-soundness is expressible with a single formula $\varphi_k$,
whereas $\om$-soundness is expressed by the recursive theory $T=\{\varphi_k\}_{k<\om}$.} So iterability is the more subtle notion. However, fine structure is very central to our understanding of mice, so central that it is even built heavily into the very definition of \emph{mouse}: a mouse is stratified in an increasing hierarchy of initial segments, each of which are also  mice themselves,
and part of the definition  is that every proper segment must satisfy certain key fine structural requirements. In particular, they must be \emph{sound}; this requirement
 is a strong and local form of the GCH. These properties are essential to the general development of the theory from the outset.

Typical constructions of mice are recursive in nature, building them, roughly, by recursion on their initial segments. Having produced a certain mouse $M$ at a stage of the construction, one must verify that it also has the right fine structural properties before proceeding. The proofs that $M$ has these properties rely heavily on the iterability of $M$.
This paper gives new proofs of these these properties,
using an (at least superficially)
weaker iterability hypothesis than do the classical proofs.

To explain the lesser iterability hypothesis, we need to refine our discussion of iteration trees that we began earlier.
If $E=\es_\alpha$ for some extender sequence $\es$,
then the \emph{length} $\lh(E)$ of $E$ is $\alpha$. The most fundamental kinds of iteration trees are \emph{normal}\footnote{The description of \emph{normal} we give here is a slight simplification of the precise one. By a \emph{normal} iteration tree we will formally mean one which is
$k$-maximal for some $k\leq\om$.} trees, in which
\begin{enumerate}[label=(\roman*)]\item for all $\alpha\leq\beta$, we have $\lh(E_\alpha)\leq\lh(E_\beta)$,
\item  given $E_\alpha$, the $\Tt$-predecessor $\beta$ of $\alpha+1$ is always chosen as small as possible that we can use $E_\alpha$ to form the next ultrapower, and
\item given $E_\alpha$ and $\beta$,
$M_{\alpha+1}=\Ult(P,E_\alpha)$ with $P$ the largest possible initial segment of $M_\beta$ over which  $E_\beta$ is an extender.
\end{enumerate}

The second main kind of iteration
allows us to iterate linearly
the process of forming normal trees. With this kind of iteration, we could, for example, first
form a normal tree $\Tt_0$ on $N_0=M$,
with last model $N_1=M^{\Tt_0}_{\alpha_0}$,
and then form a normal tree $\Tt_1$  on $N_1$, with last model $N_2=M^{\Tt_1}_{\alpha_1}$, etc, taking direct limits of the models $\left<N_\alpha\right>_{\alpha<\eta}$ at limit stages $\eta$ of the process,
and producing overall a (transfinite) sequence $\left<\Tt_\alpha\right>_{\alpha<\theta}$ of normal trees $\Tt_\alpha$.
This is called a \emph{stack} of normal trees.\footnote{This description is also slightly simplified. We will formally use the term \emph{stack} in the sense of a $k$-maximal stack, for some $k\leq\om$.}

In the last few years, there has been significant work regarding and exploiting the relationship between normal trees and stacks of normal trees,
particularly on the reduction of stacks to normal trees.
This work is highly important in Steel's  recent progress in the analysis of HOD in models of determinacy \cite{ACPFMP}.
At essentially the same time as Steel's work,
the author
 showed \cite{iter_for_stacks}
that  normal iteration strategies satisfying inflation condensation
can be extended to strategies for stacks of normal trees.\footnote{Versions of the main results in this paper
were actually worked out prior to much of that for  \cite{iter_for_stacks}.}
A significant part of this result was
independently worked out by Steel
\cite{ACPFMP}.
Steel and the author also worked out the more refined process of full normalization, establishing for example under large cardinals that $V_\Theta\cap\HOD^{L(\RR)}$ is the universe of a \emph{normal} iterate of $M_\om|\delta_0$, where $M_\om$ is the  canonical proper class mouse with infinitely many Woodin cardinals,
and $\delta_0$ its least Woodin; see \cite{fullnorm}.\footnote{It was already known \cite{hod_as_core_model} that this model was the \emph{direct limit} of iterates of $M_\om|\delta_0$ given by \emph{stacks} of trees; the new fact was the normality.}
 Siskind and Steel
have also developed this process in the context of comparison of strategy mice \cite{siskind_steel}. Such techniques have  been useful in work of Sargsyan, Schindler and the author \cite{vm2_v2}, and the author \cite{vmom_v2},
on Varsovian models, which explore self-iterability in mice and connections between inner models of mice and HOD in determinacy models.

The classical proof
\footnote{See \cite{fsit},
\cite{outline}, \cite{deconstructing},
\cite{combin}, \cite{imlc}, \cite{zeman_dodd}, \cite{operator_mice_v3}.}
of the basic fine structural properties in a mouse $M$ starts from the hypothesis that $M$ is iterable for stacks of normal trees. It is first observed that we may assume DC and that $M$ is countable. Given this, an iteration strategy for $M$ with the  \emph{weak} Dodd-Jensen property is constructed. Such a strategy is then used to guide a kind of comparison, and the outcome of the comparison is analysed to prove the desired result.
Weak Dodd-Jensen is used significantly in the analysis.
Comparison only requires normal iterability,
and the only role of non-normal iterability is in the first step, achieving weak Dodd-Jensen.
It is actually essential in this step, by \cite{iter_for_stacks}.\footnote{\label{ftn:wDJ_necessary}Suppose $M$ is a countable $m$-sound premouse. By
%conf
\cite[Theorem 1.2]{iter_for_stacks}, if there  is an $(m,\om_1+1)$-strategy for $M$ with the weak Dodd-Jensen property, then $M$ is $(m,\om_1,\om_1+1)^*$-iterable. Conversely, by \cite{wDJ}, assuming DC,
if $M$ is $(m,\om_1,\om_1+1)^*$-iterable then there is an $(m,\om_1+1)$-strategy for $M$ with weak Dodd-Jensen.}
In this paper we  consider
the problem of  proving the fine structural properties assuming only that $M$ is normally iterable,
and thus, without having weak Dodd-Jensen at our disposal.

The recent work mentioned above says a lot about the relationship between the two types of iterability, but whether normal
iterability implies iterability for stacks in general,
without the assumption of inflation condensation,
seems to be open.
So while the result of \cite{iter_for_stacks} mentioned above heads in the right direction, it does not suffice to establish fine structure from normal iterability.  Doing so should provide a basic advance in our understanding of
two of the fundamental building blocks of inner model theory.

On the other hand, an almost complete version of condensation
(which generalizes G\"odel's theorem on condensation for levels of $L$)
was proven in
%conf
\cite[Theorem 5.2]{premouse_inheriting},
from normal iterability.
The incompleteness of this version
is due to the fact that in some cases it assumes a hypothesis
on the solidity of the mouse $M$ in question
(see clause 1f of the cited theorem). (Solidity is another of the fine structural properties,
and is related to condensation.)

These (partial) results  suggest
that normal iterability might indeed suffice to prove all of the fine structure properties, in full.
The main goal of this paper is to show that it does. We show that normal iterability proves
condensation (without the extra solidity assumption mentioned above),
and proves solidity, universality, Dodd-solidity, and for pseudo-premice, the initial segment condition.
To achieve this, we develop
further the methods of
\cite{premouse_inheriting}.
As well as  providing a formal improvement of the classical results,
the proof we give
uses methods which do not feature in the classical proof.\footnote{
The proof of solidity and universality we give here is modelled on that of
condensation in
%conf
\cite[5.2]{premouse_inheriting},
using structures related to the bicephali of
\cite{premouse_inheriting} in
place of the
phalanxes of the classical proofs,
and certain kinds of calculations from \cite{premouse_inheriting}.
However,
there are new fine structural difficulties
to be handled here.

For Dodd-solidity, there is an easy trick
to reduce to the case where classical style arguments work. (This simplicity is probably an artifact of
the assumption of $1$-soundness; that is, one might formulate a generalization
of Dodd-solidity which holds more generally, and the proof would then probably
not be as direct.) Actually, the proof for Dodd-solidity we give here
still involves significant work, but this is just due to the nature of the
classical proof
itself; the reduction to the (roughly) classical case itself is short.}
Some related methods have also
 been important in
fine structural arguments in
work of Schindler \cite{cmfali},
Woodin (on mice with long extenders), the author \cite{premouse_inheriting}, and
Steel \cite{ACPFMP}. These connections are described in a little more detail in  \S\ref{sec:outline_solidity},
Remark \ref{rem:history_gen_biceph}.

We also prove some other facts of independent interest.
In particular, we establish a simple criterion which ensures that a mouse is a normal iterate of its core. (The core of a mouse $M$ is a  natural hull of $M$, corresponding
to the ``least'' set $A$ of ordinals which is definable from parameters over $M$, but such that $A\notin M$.)
It has long been known that, roughly, every mouse which is below the level of a ``cardinal which is strong to a measurable''  is a normal iterate of its core,
and that there are counterexamples to this phenomenon a little beyond that point (see \S\ref{sec:iterates_of_cores} for discussion). However, the criterion we establish has no restriction on large cardinal complexity.

Using an argument related to that
for the previous one, we
also show that  if $M$ is a mouse,
$\mu\in M$ and $M\sats$ ``$\mu$ is a countably complete ultrafilter''
then $\mu$ is equivalent to a finite normal iteration tree on $M$,
and hence has a simple description in terms of extenders in $\es^M$.
This shows that no such ultrafilters can appear in a mouse except for those which were essentially put directly in it.\footnote{This is essentially (and literally below superstrong) by a proof from the author's dissertation \cite{mim}; see \S\ref{sec:results}.}
The result generalizes Kunen's classical one for $L[U]$,
in which the only countably complete ultrafilters are those which are equivalent to  finite iterates of $U$.

Finally, we establish a restriction
on the possible kinds of set-forcing grounds of mice which model ZFC.

\subsection{Results}\label{sec:results}

We now list the main theorems we will prove.
We work throughout  with pure $L[\es]$
premice with Mitchell-Steel indexing,
but allowing extenders of superstrong type in their extender sequence.
The background theory is ZF (though it is straightforward
to see that most of the results do not require much of ZF).

\begin{tm*}[Solidity and universality, \ref{thm:solidity}]
Let $m<\om$. Then every $m$-sound, $(m,\om_1+1)$-iterable premouse is
$(m+1)$-solid and
$(m+1)$-universal.\footnote{Note that we follow Zeman \cite{imlc}
in our use of the terminology \emph{$(m+1)$-solid},
in that we do not  incorporate $(m+1)$-universality into it;
%checked number
see also \cite[\S1.2(Fine structure)]{extmax}.
This is in contrast to Mitchell-Steel \cite{fsit} and Steel \cite{outline},
where $(m+1)$-solidity incorporates $(m+1)$-universality by
definition.}\end{tm*}

As an immediate corollary of the conjunction of this result
with \cite[Theorem 5.2]{premouse_inheriting},
we will also obtain the following result:
\begin{tm}[Condensation]\label{thm:condensation}
Let $k<\om$. Let $M$ be a $k$-sound, $(k,\om_1+1)$-iterable premouse. Let $H$ be a $k$-sound premouse and $\rho\in[\rho_{k+1}^H,\rho_k^H)$ be an $H$-cardinal such that $H$ is $\rho$-sound. Let $\delta=\card^M(\rho)$. Let $\pi:H\to M$ be $k$-lifting with $\crit(\pi)\geq\rho$. Then:
\begin{enumerate}[label=\arabic*.,ref=\arabic*]
\item If $H\notin M$ then $\rho_{k+1}^H=\rho_{k+1}^M\leq\delta$, $H$ is the $\rho$-core of $M$, $\pi$ is the $\rho$-core map and $\pi(\pvec_{k+1}^H)=\pvec_{k+1}^M$.

\item\label{item:condensation_H_in_M} If $H\in M$ then exactly one of the following holds:

\begin{enumerate}[label=\tu{(}\alph*\tu{)}]
\item $H\pins M$,
\item\label{item:M|rho_active_and_H_pseg_of_Ult} $M|\rho$ is active  and $H\pins\Ult(M|\rho,F^{M|\rho})$,
\item\label{item:M|rho_passive,M|rho^+_active_type_1,H=Ult_k(Q)} $M|\rho$ is passive, $N=M|\rho^{+H}$ is active type 1 and $H=\Ult_k(Q,F^N)$
where $Q\pins M$ is such that $\rho=\delta^{+Q}$ and $\rho_{k+1}^Q=\delta<\rho_k^Q$,
\item\label{item:k=0_and_H,M,M|rho_type_2_and_R=Ult(M|rho,F)_F_type_1} $k=0$ and $H,M$ are active type 2
and $M|\rho$ is active type 2 and letting $R=\Ult(M|\rho,F^{M|\rho}$, then $N=R|\rho^{+H}$ is active type 1
and $H=\Ult_0(M|\rho,F^{N})$.
\end{enumerate}
\end{enumerate}
\end{tm}

Note that in clauses \ref{item:condensation_H_in_M}\ref{item:M|rho_active_and_H_pseg_of_Ult}--\ref{item:condensation_H_in_M}\ref{item:k=0_and_H,M,M|rho_type_2_and_R=Ult(M|rho,F)_F_type_1} above,
$\delta<\rho=\delta^{+H}$, so $\rho$ is not an $M$-cardinal. The following fact
is just an abridged version of \cite[Theorem 5.2]{premouse_inheriting}.
We state it here for ease of comparison with Theorem \ref{thm:condensation}.

\begin{fact}[\cite{premouse_inheriting}]\label{fact:condensation}
Let $k,M,H,\rho,\delta,\pi$ be as in Theorem \ref{thm:condensation}. Then:
\begin{enumerate}[label=\arabic*.,ref=\arabic*]
\item\label{item:fact_condensation_H_notin_M} If $H\notin M$ then:
\begin{enumerate}[label=(\roman*)]

\item\label{item:cond_fact_abbr_i} If $M$ is $(k+1)$-solid then $\rho_{k+1}^H=\rho_{k+1}^M$.
\item\label{item:cond_fact_abbr_ii} If $\rho_{k+1}^H=\rho_{k+1}^M$ then
$H$ is the $\rho$-core of $M$, $\pi$ is the $\rho$-core map and $\pi(\pvec_{k+1}^H)=\pvec_{k+1}^M$.
\item\label{item:cond_1d} If $\rho^H_{k+1}=\rho$ and $\rho^{+H}<\rho^{+M}$ then $M|\rho$ is active.
\end{enumerate}
\item\label{item:condensation_from_normal_H_in_M} Part \ref{item:condensation_H_in_M} of Theorem \ref{thm:condensation} holds.
\end{enumerate}
\end{fact}

Note that Theorem \ref{thm:condensation}
follows immediately from Theorem \ref{thm:solidity}
and Fact \ref{fact:condensation}.
We will also use Fact \ref{fact:condensation} in the proof of Theorem \ref{thm:solidity}.

The next theorem is essentially due to Steel and Zeman,
almost via their
classical proofs. Note, however, that Steel's proof is less general,
as it is below superstrong extenders,
and Zeman's is with $\lambda$-indexing, not Mitchell-Steel. Super-Dodd-soundness is a strengthening of Dodd-soundness;
see Definitions \ref{dfn:super_Dodd_param_proj} and \ref{dfn:super_Dodd-sound}.

\begin{tm*}[Super-Dodd-soundness, \ref{thm:super-Dodd-soundness}]
Let $M$ be an active, $(0,\om_1+1)$-iterable premouse, let $\kappa=\kappa^M$,
and suppose that $M$
is either $1$-sound or $\kappa^{+M}$-sound.
Then $M$ is super-Dodd-sound.
\end{tm*}

And the next  is essentially due to Mitchell-Steel, by reducing to  their
classical proof from \cite[\S10]{fsit}:

\begin{tm*}[Initial Segment Condition, \ref{thm:ISC}]
Every $(0,\om_1+1)$-iterable pseudo-premouse is a premouse.
\end{tm*}

\begin{rem}\label{rem:traditional_bicephali}
One further fact should be  mentioned along with those above,
and this regards bicephali, in their traditional, or almost traditional, form.
%conf
By \cite[Theorem 9.2]{fsit} and
%conf
\cite[Lemma 5.4]{extmax}, there
are no iterable bicephali
$(N_0,N_1)$ for which $N_0\neq N_1$ and neither $N_0$ nor $N_1$ has a superstrong extender on its sequence.\footnote{\label{ftn:bicephalus_types}We refer here to the bicephali considered in
%conf
\cite[Definition 9.1.1]{fsit}, and in
%conf
\cite[Definition 5.2]{extmax}.
In \cite{fsit},
a bicephalus is not allowed to have $N_0$ type 1/3 and $N_1$ type 2, or vice versa. In \cite{extmax},
this restriction on types is removed, so \cite{extmax} deals with these  ``mixed type'' bicephali.}
For the argument in \cite{extmax},
in the case that $N_0$ is type 2 and $N_1$ is type 3, or vice versa,
it is important that we are below superstrong.
%conf
By \cite[Theorem 4.3]{premouse_inheriting}, if we remove the superstrong restriction, then either $N_0$ is superstrong and $N_1$ is type 2, or vice versa (and more is established). However,
this situation is not ruled out,
 and the author expects that it is possible. Each of these results already proceeded from only normal iterability.
\end{rem}

$0^\pistol$ is the least active mouse $M$ such that $M|\crit(F^M)\sats$ ``there is a strong cardinal''.
Every mouse below $0^\pistol$
is a normal iterate of its core
%checked cmip
(by the proof of \cite[Theorem 8.13]{cmip}). But not too far beyond $0^\pistol$,
there are mice for which this fails.
These things are discussed in detail in \S\ref{sec:iterates_of_cores}. The next result establishes a simple criterion which guarantees that a mouse is  a normal iterate of its core, but without any limit on
large cardinal complexity.

\begin{dfn}\label{dfn:almost-above}
 Let $\Tt$ be a successor length $m$-maximal iteration tree
 on an $m$-sound premouse $M$
 such that $b^\Tt$ does not drop in model or degree.
 Let $\rho\in\OR$.
 We say that $\Tt$ is \dfnemph{almost-above $\rho$}
 iff for every $\alpha+1<\lh(\Tt)$,
 if $\crit(E^\Tt_\alpha)<\rho$
 then  $M$ is active type 2,
 $m=0$, $\alpha\in b^\Tt$, and $\crit(i^\Tt_{\alpha\infty})$
 is the largest cardinal of $M^\Tt_\alpha$.
\end{dfn}

Let $\Tt$ be almost-above
$\rho$. Note that if $\crit(E^\Tt_\alpha)<\rho$ then $E^\Tt_\alpha=F(M^\Tt_\alpha)$ and $(0,\alpha]^\Tt$ does not drop, so $E^\Tt_\alpha$ is a non-dropping image of $F^M$. Moreover, by taking $\alpha$ least such, we have $\crit(i^\Tt_{0\alpha})\geq\rho$ and $\crit(F^M)=\crit(E^\Tt_\alpha)<\rho$. Since $\alpha\in b^\Tt$, it also follows that $\rho\leq\crit(i^\Tt_{0\infty})$. The definition of \emph{strongly finite} below
is given in
\ref{dfn:strongly_finite}, but it implies
that $\Tt$ has finite length
and every extender used along
$b^\Tt$ is equivalent to a single measure
(there is a finite set of generators which
generates the extender).

\begin{tm*}[Projectum-finite generation, \ref{thm:finite_gen_hull}]
Let $m<\om$ and let $M$ be an $m$-sound, $(m,\om_1+1)$-iterable premouse.
Suppose that
\[ M=\Hull_{m+1}^M(\rho_{m+1}^M\un \{x\})\]
for some $x\in M$.
Then $M$ is an iterate of its $(m+1)$st core $\core_{m+1}(M)$.
In fact, there is a
successor length
$m$-maximal tree $\Tt$
on $\core_{m+1}(M)$,
which is strongly
finite and
almost-above $\rho_{m+1}^M$,
with $M=M^\Tt_\infty$. Moreover,
$i^\Tt_{0\infty}$ is the core map.\end{tm*}

We will also deduce a related fact in Corollary \ref{cor:ult_is_it}.
For now note that Theorem \ref{thm:finite_gen_hull}
has the following corollary:

\begin{cor}\label{cor:def_over_core}
Let $M$ be any $m$-sound, $(m,\om_1+1)$-iterable premouse.
Let $A\sub\rho_{m+1}^M$ be $\bfrSigma_{m+1}^M$-definable.
Let $C=\core_{m+1}(M)$.
Then $A$ is $\bfrSigma_{m+1}^C$-definable.
\end{cor}
\begin{proof}
We may assume that $M=\Hull_{m+1}^M(\rho_{m+1}^M\cup\{x\})$
 for some $x$, by taking a hull. So the theorem applies. But then standard
calculations
 show that $A$ is $\bfrSigma_{m+1}^C$.
\end{proof}

The weaker version of the corollary
in which we assume
$(m,\om_1,\om_1+1)^*$-iterability\footnote{The distinction between $(k,\alpha,\beta)$-
and
$(k,\alpha,\beta)^*$-iterability is specified in
%confirmed
\cite[p.~1202]{cmwmwc}.}
is already clear via standard methods (take $M$ countable
and a strategy with weak Dodd-Jensen, compare $M$
with the phalanx
$((M,{<\rho_{m+1}^M}),C)$, and analyse the outcome).
One can also use classical methods to prove the conclusion of
Theorem
\ref{thm:finite_gen_hull}
from $(m,\om_1,\om_1+1)^*$-iterability, and
this classical style proof is easier.

We also incorporate
a proof of the following theorem,
which was established below the superstrong level in the author's thesis
%confirmed
\cite[Theorem 4.8]{mim}.
Its proof, modulo the validity of Dodd-soundness at the superstrong level,
is essentially the same as in \cite{mim}, but we will use an embellishment of
the argument
to prove both the super-Dodd-soundness and projectum-finite-generation
theorems. So it serves as a good warm-up to those proofs.

\begin{tm*}[Measures in mice, \ref{thm:measures_in_mice}]
 Let $M$ be a $(0,\om_1+1)$-iterable premouse and $\mu\in M$ be such that
 $M\sats$ ``$\mu$ is a countably complete ultrafilter''.
 Then there is a strongly finite
 $0$-maximal iteration tree on $M$
 such that $b^\Tt$ does not drop, $\Ult_0(M,\mu)=M^\Tt_\infty$
 and the ultrapower map $i^{M,0}_\mu$ is the iteration map $i^\Tt_{0\infty}$.
\end{tm*}

Using projectum-finite generation, we
prove the following theorem, which relates to inner model theoretic geology,
and some questions of Gabriel Goldberg and Stefan Miedzianowski;
see \cite{gironaconfproblems}.
We write $\univ{M}$ for the universe of $M$.

\begin{tm*}[\ref{thm:sigma_grounds}]
 Let $M$ be a $(0,\om_1+1)$-iterable premouse such that $\univ{M}\sats\ZFC$
 \tu{(}here $M$ might be proper class
 \tu{)}.
Let $W$ be a ground of  $\univ{M}$ via  a forcing $\PP\in W$
 such that $W\sats$ ``$\PP$ is $\sigma$-strategically-closed''.
 Suppose
 $M|\aleph_1^M\in W$.\footnote{Of course, as $\PP$ is
$\sigma$-strategically-closed in
$W$,
 automatically $\RR^M\sub W$. But note that $M|\om_1^M$ doesn't just give $\HC^M$, but also
 the restriction of the extender sequence to $\om_1^M$.}
Then $W=\univ{M}$.
\end{tm*}

 If $M$ is also tame
then the assumption that $M|\aleph_1^M\in W$
is superfluous. See Footnote \ref{ftn:sigma-closed_forcing_tame} and Question \ref{ques:sigma-closed_forcing_ground_of_mouse}.

As an aside, we observe that solidity is non-trivial:
\begin{tm*}[\S\ref{sec:non-solid}] If there is a measurable cardinal then
 there is a non-solid premouse.\end{tm*}

\begin{rem}
As mentioned above, the background theory for the paper is $\ZF$.
However, for the theorems listed above, we may
assume $\ZFC$.
This is because the theorems only assert something about the first order theory
of a premouse, assuming that it is $(k,\om_1+1)$-iterable,
for some given $k<\om$. But then fixing the premouse $M$
and a $(k,\om_1+1)$-iteration strategy $\Sigma$ for $M$,
we can pass into $\HOD_{\Sigma,M}$, or $L[\Sigma,M]$, appropriately defined,
where we have $\ZFC$ and the original assumptions.
Moreover, the only way that $\AC$ comes in is when proving that comparisons and related processes
terminate; in those arguments we need that $\om_1$ is regular
and to be able to take an appropriate hull of $V_\alpha$.
The rest of the paper goes through with $\ZF$
(we point out some details
regarding this when proving
 $<^{\para*}_k\rest\Mm^{\iter}_k$ is wellfounded,
Lemma
\ref{lem:param_order_wfd}).
\end{rem}

The author thanks the organizers for the opportunity to present key parts of  an earlier version of the main arguments in this paper
 at the   the \emph{3rd M\"unster conference
on inner model theory, the core model induction, and hod mice},
at the University of M\"unster, Germany, in July 2015; cf.~\cite{ralf_notes_solidity_talk}.
The overall argument is still the same as what was outlined there,
but there are some significant differences in components.
One of these is in the proof of
of Lemma \ref{lem:finite_gen_hull} (projectum-finite generation),
where some methods from \cite{premouse_inheriting}
and from \cite{ACPFMP}
turned out to  give a somewhat simpler approach to prove that lemma;
this is explained in more detail at the beginning of
\S\ref{sec:finite_gen_hull}. Second, there are some simplifications
in the proof of solidity, one adopting
a suggestion made by John Steel following the 2015 talk;
this is explained in \S\ref{sec:solidity}. Also,
the idea for \cite[Theorem
9.6]{iter_for_stacks} was only found during the conference;
%confirmed
the author was only aware of a weaker version of this result prior to the
conference, which was however
enough to prove a weakened
version of Lemma \ref{lem:param_order_wfd}, which sufficed for the main
theorems.

\subsection{Notation and terminology}\label{sec:notation}

We assume general familiarity
with \cite{outline} and/or \cite{fsit}, which
develop the basic theory of Mitchell-Steel (MS) indexed premice.\footnote{The author has not thought through
adapting the arguments presented here to Jensen indexed
premice or other forms of fine structure
(such as in  \cite{imlc} or \cite{jensen_book}).}
In particular, we work in MS-indexing,
but allowing superstrong extenders on the extender sequence
(see Remark \ref{rem:superstrong_diff}). We
work with MS fine
structure, but modified as in
%confirmed
 \cite[\S5]{V=HODX_pub}, meaning that we define the \emph{$n$th standard parameter} $p_n$ to be what is denoted $q_n$ there, and the other fine structural notions are defined as there. We
assume familiarity with  generalized
%both confirmed
solidity witnesses (see \cite[\S1.12 + p.~326]{imlc}), and also with \cite[\S1.1, \S2]{premouse_inheriting}.
Familiarity with other parts of
\cite{premouse_inheriting}, and also
%conf
 \cite[\S2]{extmax}, might help, but  these can be referred to as necessary.

The reader should look through \cite[\S1.1]{premouse_inheriting}
for a summary of the notation and terminology we use.
We include some of the most central notions here,
and some further terminology.

\subsubsection{General}\label{sec:notation_general}

See \cite[\S1.1.1]{premouse_inheriting}.
Given finite sequences $p,q$, we write $p\ins q$
iff $p=q\rest\lh(p)$. We write $q\pins p$ iff $q\ins p$ and $q\neq p$.

 We order $[\OR]^{<\om}$ in the ``top-down'' lexicographic ordering.
 That is, given $p,q\in[\OR]^{<\om}$,
 then
 \[ p<q\iff p\neq q\text{ and } \max(p\Delta q)\in q.\]
 We sometimes identify $[\OR]^{<\om}$ with the strictly
 descending sequences of ordinals, so if
 $p=\{p_0,\ldots,p_{n-1}\}$ where $p_0>\ldots>p_{n-1}$
 then $p$ is identified with $(p_0,\ldots,p_{n-1})$,
 and $p\rest i$ denotes $\{p_0,\ldots,p_{i-1}\}$ for $i\leq n$,
 and $p\rest i$ denotes $p$ for $i\geq n$.

 We denote with $\widetilde{\OR}$\footnote{In
%conf
 \cite[Definition 2.6]{extmax},
the notation ``$\mathcal{D}$'' is used instead of ``$\widetilde{\OR}$''.} the class of pairs
 $(z,\zeta)\in[\OR]^{<\om}\cross\OR$
such that $z=\emptyset$ or $\zeta\leq\min(z)$.\footnote{Note that it is
$\zeta\leq\min(z)$, not $\zeta<\min(z)$.}
We order $\widetilde{\OR}$ as follows.
Let $(z,\zeta),(y,\upsilon)\in\widetilde{\OR}$.
Write
\[ z=\{z_0>\ldots>z_{m-1}\}\text{ and }z_m=\zeta,\]
\[ y=\{y_0>\ldots>y_{n-1}\}\text{ and }y_n=\upsilon.\]
Then $(z,\zeta)<(y,\upsilon)$
iff either:
\begin{enumerate}[label=--]
 \item $m<n$ and $z_i=y_i$ for all $i\leq m$, or
 \item there is $k\leq\min(m,n)$ such that $z\rest k=y\rest k$
and $z_k<y_k$.
\end{enumerate}
Whenever we discuss the ordering of such tuples/pairs,
it is with respect to these orderings. Note that they are wellfounded
and set-like.

\subsubsection{Premice and phalanxes}\label{sec:premice_and_phalanxes}

See \cite[\S1.1.2]{premouse_inheriting}
for our use of the terminology \emph{premouse} (with
superstrongs)
and associated notation.
Here is a brief summary:
For a premouse
$M$,
 $F^M$ denotes its active extender,  $\es^M$ is its internal extender sequence (without $F^M$),
 $\es_+^M=\es^M\conc\left<F^M\right>$, and if $M$ is type 2, then $F_\downarrow^M$ denotes  the largest non-type Z proper segment of $F^M$. We write $M^{\passive}$ for its passivization; that is, just like $M$ except with $F^{M^\passive}=\emptyset$. And $M|\alpha$ is the initial segment of $M$ of ordinal height $\alpha$, including the extender active there, and $M||\alpha=(M|\alpha)^{\passive}$. We abbreviate \emph{premouse}
with \emph{pm}.
Such notation is also employed with analogous meaning for other related structures $M$, such as pre-ISC-premice, which
are like premice,
but without any form of ISC (initial segment condition) required:

\begin{dfn}\label{dfn:pre-ISC-premouse}
A \dfnemph{pre-ISC-premouse}\footnote{\label{ftn:pre-ISC_1}These are similar to the segmented-premice defined in
%conf
\cite[Definition 2.9]{extmax}.
(Segmented-premice are \emph{not} defined in
%conf
\cite[\S5]{extmax},
the implication to the contrary in \cite[\S1.1.2]{premouse_inheriting} notwithstanding.)
But in \cite[Definition 2.9]{extmax}, it was demanded that $\nu(F^M)\leq\lgcd(M)$, which is not demanded for pre-ISC-premice.} is a structure  $M=(\J_\alpha[\es],\in,\es,G)$
such that there are $\delta,F,\widetilde{F},U,i,\kappa$
such that:
\begin{enumerate}[label=--]
\item $M^{\passive}=(\J_\alpha[\es],\in,\es,\emptyset)$ is a passive premouse with  largest cardinal $\delta$,
\item $F$ is a weakly amenable short extender  over $\J_\alpha[\es]$,
\item $U=\Ult(\J_\alpha[\es],F)$ and
$i:\J_\alpha[\es]\to U$
is the ultrapower map,
\item $\kappa=\crit(F)$ and $\J_\alpha[\es]\sats$ ``$\kappa^+$ exists'', so $\kappa^{+M}<\alpha$,
\item $M^{\passive}=U|\delta^{+U}$,
\item either: \begin{enumerate}\item $i(\kappa)>\alpha$ and $F$ is the $(\kappa,\alpha)$-extender derived from $i$
and there is $\nu<\alpha$ such that all generators of $F$ are $<\nu$,
or
\item $i(\kappa)=\delta$ and $F$ is the $(\kappa,\delta)$-extender derived from $i$,
\end{enumerate}
and
\item $G=\widetilde{F}$ is the standard
amenable predicate coding $F$,
%conf
as in \cite[2.9--2.10]{outline},
except that letting $\nu=\nu(F)$ be the natural length of $F$ (so $\nu=\max(\kappa^{+M},\nu^-)$ where $\nu^-$ is the strict sup of generators of $F$), we use $\nu'=\max(\delta,\nu)$ in place of $\nu$ in the definition of $\widetilde{F}$. That is, $\widetilde{F}$
is the set of pairs $(A,E)$
such that $A\pins M|\kappa^{+M}$
and
$E=F\rest (A\cross[\nu']^{<\om})$.\qedhere
\end{enumerate}
\end{dfn}

Clearly every premouse is a pre-ISC-premouse.\footnote{\label{ftn:pre-ISC_2}But a type 2 premouse
is not a  segmented-premouse (\cite[Definition 2.9]{extmax}), because
of the requirement there that $\nu(F^M)\leq\delta$ for segmented-premice.
However, it is stated in \cite[Remark 2.10]{extmax}
that every premouse is a segmented-premouse,
which was an oversight. It appears the author had two distinct notions in mind, which became
%conf
blurred. The items \cite[2.11, 2.12, 2.13]{extmax} are stated for segmented-premice $P$,
but they in fact also work fine for type 2 premice, which is a case of key interest, and
apparently when the author wrote them, he had forgotten that a segmented-premouse $M$ was required to have $\nu(F^M)\leq\lgcd(M)$. The assumption that $\nu(F^M)\leq\lgcd(M)$ is important for the seg-pms in \cite[\S5]{extmax} though. In any case, while the present paper depends on some of the proofs from \cite[\S2]{extmax}, it does not literally depend on those results themselves.}

When we write, for example, \emph{type 1/2}, we mean ``(an active premouse of) type 1 or type 2''.
If $M$ is an active pre-ISC-pm,  $\nu(M)=\nu(F^M)$ denotes
the natural length of $F^M$ (so $\nu(F^M)=\max(\kappa^{+M},\nu^-)$
where $\kappa=\crit(F^M)$ and $\nu^-$ is the strict sup of generators of $F^M$),
and $\lgcd(M)$ the largest cardinal of $M$.

 If $N$ is an $n$-sound premouse,
 we write
$(M,m)\ins(N,n)$ iff $M\ins N$
and if $M=N$ then $m\leq n$. If $N$ is a pre-ISC-premouse,
we write $(M,m)\ins (N,0)$ iff $M\ins N$ and if $M=N$ then $m=0$.

If $M$ is a passive premouse and $\xi<\OR^M$,
then $M\wr\xi$ is the corresponding
segment of $M$ in the $\mathcal{S}$-hierarchy. For example,
if $M=\J(N)$ (one step in the Jensen hierarchy above $N$),
$\lambda=\OR^N$ and $n<\om$ then $M\wr(\lambda+n)=\Ss_n(N)$.
If $M=(M^\passive,F^M)$ is an active type 1/2 premouse,
then $M\wr\xi$ is the corresponding segment of
$(M^\passive,\widetilde{F^M})$,
i.e.~$(M^\passive||\xi,\widetilde{F^M}\cap(M^\passive||\xi))$.
And if $M$ is active type 3 and $\xi<\rho_0^M=\nu(F^M)$,
then $M\wr\xi=(M^\passive||\xi,F^M\rest\xi)$.

Regarding phalanxes, we write for
example $\phP=((P,{<\kappa}),(Q,\kappa),R,\lambda)$,
where $\kappa<\lambda$,
for the phalanx consisting of  3 models $P,Q,R$,
where trees $\Tt$ on $\phP$ have $M^\Tt_0=R$,
first extender $E^\Tt_0\in\es_+^R$
with $\lambda\leq\lh(E^\Tt_0)$, extenders $E^\Tt_\alpha$
with $\crit(E^\Tt_\alpha)<\kappa$ apply to $P$,
those with $\crit(E^\Tt_\alpha)=\kappa$ apply to $Q$,
and others apply to some segment of $M^\Tt_\beta$ for some $\beta\geq 0$.
A \emph{$(p,q,r)$-maximal} tree $\Tt$ on $\phP$
has $\deg^\Tt_0=r$, and degrees $p,q$ associated to $P,Q$,
so $M^\Tt_{\alpha+1}=\Ult_p(P,E^\Tt_\alpha)$ if $\crit(E^\Tt_\alpha)<\kappa$,
etc, and $\Tt$ is otherwise formed according to the rules for $n$-maximal trees. We write $M^\Tt_{-2}=P$ and $M^\Tt_{-1}=Q$, and $\root^\Tt(\alpha)\in\{-2,-1,0\}$ denotes the root of $\alpha$ in the tree order. We also write, for example, $\phP'=(((P,p),{<\kappa}),((Q,q),\kappa),(R,r),\lambda)$ for the
same phalanx, coupled with the specification of degrees $p,q,r$ associated to the models $P,Q,R$ of the phalanx. A \emph{degree-maximal} tree on $\phP'$ is just a $(p,q,r)$-maximal tree on $\phP$.
\subsubsection{Fine structure}\label{sec:notation_fine_structure}

Let $M$ be a $k$-sound premouse.
See \cite[\S1.1.3]{premouse_inheriting}
for notation such as
$\Hull^M_{k+1}(X)$ (the $\rSigma_{k+1}$-hull of $M$ from parameters in $X\sub\core_0(M)$), $\cHull_{k+1}^M(X)$
(the transitive collapse variant) and the corresponding theory $\Th_{\rSigma_{k+1}}^M(X)$, and also $\pvec_{k+1}=(p_{k+1},\ldots,p_1)$.
Let  $q\in[\rho_k^M]^{<\om}$.
The \emph{solidity witnesses} for $(M,q)$
are the hulls
$H_{q_i}=\cHull_{k+1}^M(q_i\cup\{q\rest i,\pvec_k^M\})$,
or essentially equivalently, the theories
$T_{q_i}=\Th_{\rSigma_{k+1}}^M(q_i\cup\{q\rest i,\pvec_k^M\})$,
where $i<\lh(q)$. We say that $(M,q)$ is \emph{$(k+1)$-solid}
iff $H_{q_i}\in M$ (equivalently, $T_{q_i}\in M$)
for each $i<\lh(q)$.
And $M$ is
\emph{$(k+1)$-solid}
iff $(M,p_{k+1}^M)$ is $(k+1)$-solid.
We say that $M$ is \emph{$(k+1)$-universal}
iff
\[
\pow(\kappa)\cap M\sub \cHull_{k+1}^M(\rho_{k+1}^M\cup\{\pvec_{k+1}^M\}).\]
Note here that we follow \cite{imlc}, not \cite{outline},
in our definition of \emph{$(k+1)$-solid},
in that we do not build $(k+1)$-universality into it.

Whenever we refer to terms and their interpretations (in a fine structural context), we mean a pair $(k+1,\varphi)$ such that $k<\om$, $\varphi$ is an $\rSigma_{k+1}$ formula, and we interpret $\varphi$ as in
%conf
\cite[Definition 5.2]{V=HODX_pub},
where $k$ is as there. We say that $k+1$ is the \emph{degree} of this term. If $k,\varphi,R,n,q$ are as in \cite[Definition 5.2]{V=HODX_pub},
we define the partial function $\mathrm{m}\tau^R_{\varphi,q}:_{\mathrm{part}}\core_0(R)^n\to\core_0(R)$ as there.
We will, moreover, always use $\pvec_k^R$ for the parameter $q$.
Given an $k$-sound premouse $R$, $x\in\core_0(R)$ and a term $\tau=(k+1,\varphi)$ of  $n>0$ variables (therefore $k,\varphi,R,n,q=\pvec_k^R$ are as in \cite[Definition 5.2]{V=HODX_pub}), we write $f^M_{\tau,x}$ for the partial function $f^M_{\tau,x}:_{\mathrm{part}}\core_0(R)^{n-1}\to\core_0(R)$ where $f^R_{\tau,x}(\vec{y})=\mathrm{m}\tau^R_{\varphi,\pvec_k^R}(x,\vec{y})$.

The following is just the variant
%conf
of \cite[Theorem 5.1]{outline}
with the iterability hypothesis
being only normal iterability,
and (as throughout the paper)
allowing superstrong extenders in $\es_+^H$ and $\es_+^M$:

\begin{fact}[Condensation for $\om$-sound mice]\label{fact:om_condensation} Let $H,M$ be $(\om,\om_1+1)$-iterable $\om$-sound premice. Let $\pi:H\to M$ be elementary with $\crit(\pi)=\rho$ where $\rho=\rho_\om^H$.
Then either
\begin{enumerate}[label=(\alph*)]\item $H\pins M$, or
\item $M|\rho$ is active and $H\pins \Ult(M|\rho,F^{M|\rho})$.
\end{enumerate}
\end{fact}
\begin{proof}
Note that since $\rho=\rho_\om^H=\crit(\pi)$
and by elementarity, we have $\rho<\rho_\om^M$, so $H\in M$.
So Fact \ref{fact:condensation} part \ref{item:condensation_from_normal_H_in_M} applies,
and gives the desired conclusion,
since clauses \ref{item:M|rho_passive,M|rho^+_active_type_1,H=Ult_k(Q)}
and \ref{item:k=0_and_H,M,M|rho_type_2_and_R=Ult(M|rho,F)_F_type_1} of Theorem \ref{thm:condensation} part \ref{item:condensation_H_in_M}
both require that $H$ is non-sound.
\end{proof}

\subsubsection{Extenders and ultrapowers}\label{sec:notation_extenders_and_ultrapowers}

See \cite[\S1.1.4]{premouse_inheriting}.
Let $E$ be a short extender
over a  pre-ISC-premouse $M$, $\kappa=\crit(E)$ and $U=\Ult(M,E)$.
We say  $E$ is  \emph{weakly amenable} to $M$ if
$\pow(\kappa)\cap U=\pow(\kappa)\cap M$.
Note that if all proper segments of $M||\kappa^{+M}$ satisfy the conclusions
of condensation for $\om$-sound mice
(Fact \ref{fact:om_condensation}), then
weak amenability implies
$U|\kappa^{+U}
= M||\kappa^{+M}$.
We will have such condensation available when we form such
ultrapowers, and we may use
the stronger agreement implicitly at times. Let $x\in U$.
 A set $X$ of ordinals \emph{$E$-generates} $x$
 (or just \emph{generates} $x$, if $E$ is determined by context)
 if there is $f\in M$ and $a\in[X]^{<\om}$ such that
 $x=[a,f]^M_E$. We say $X$ \emph{generates} $E$
 iff every $x\in U$ is generated by $X$.
And $E$ is \emph{finitely generated} or is a \emph{measure}
 if there is a finite set $X$ generating $E$.

We write $\rho_{(-1)}^M=\OR^M$,
and say $M$ is
$(-1)$-sound. We write
$\Ult_{(-1)}(M,E)=\Ult(M,E)$ (formed
using functions in $M$, and without squashing), and $i^{M,(-1)}_E=i^M_E$. Note that if
$M$ is a type 3 premouse, we have
$\rho_0^M<\rho_{(-1)}^M$.

For a sufficiently elementary embedding $j$, $E_j$ is the extender derived from $j$.

\subsubsection{Embeddings}\label{sec:notation_embeddings}

See \cite[\S1.1.5]{premouse_inheriting},
in particular for the terminology \emph{c-preserving} (\emph{cardinal-preserving}),
\emph{$\rho_n$-preserving}, \emph{$p_n$-preserving},
\emph{$\pvec_n$-preserving}, and
the notation $i_{MN}:M\to N$ for a context determined
map $i:M\to N$. One difference with
 \cite[\S1.1.5]{premouse_inheriting}
 is that, given a $\Sigma_0$-elementary embedding $\pi:M\to N$ between  premice $M,N$, we write $\Shift(\pi)$
 for the embedding denoted $\psi_\pi$
there. That is, if $M,N$ are active then
 \[\Shift(\pi):\Ult(M,F^M)\to\Ult(N,F^N) \]  is the embedding induced by $\pi$ via the Shift Lemma, whereas if $M,N$ are passive then $\Shift(\pi)=\pi$.
 We sometimes write $\widehat{\pi}=\Shift(\pi)$ to save space. So if
$M,N$ are type 1/2 premice, then $M\sub\dom(\pi)$ and $\pi=\Shift(\pi)\rest M$, but if $M,N$ are type 3 then,
as is the usual convention, fine structural maps $\pi:M\to N$ literally have domain $M^\sq=\core_0(M)$,
whereas $M\sub\dom(\Shift(\pi))$,
so $\Shift(\pi)$ acts directly on the elements of $M\cut M^\sq$ (such as $\nu^M=\nu(F^M)$), whereas $\pi$ does not.
 In case $M,N$ are type 3, $\pi$ is \emph{$\nu$-low} iff $\Shift(\pi)(\nu^M)<\nu^N$,
\emph{$\nu$-preserving} iff $\Shift(\pi)(\nu^M)=\nu^N$, and \emph{$\nu$-high} iff $\Shift(\pi)(\nu^M)>\nu^N$. In case $M,N$ are passive or type 1/2, we say $\pi$ is \emph{$\nu$-preserving}.

 The \emph{$n$-lifting}
embeddings (a weakening of the near $n$-embeddings
%conf
of \cite[Remark 4.3]{outline}), are introduced in
%conf
\cite[Definition 2.1]{premouse_inheriting} and their basic properties developed in \cite[\S2]{premouse_inheriting}.

Let $M,N$ be active
pre-ISC-premice and $\pi:M\to N$.\footnote{When we consider embeddings $\pi:M\to N$ of active pre-ISC-premice
at ``degree $-1$'', we have $\dom(\pi)=M$.
Note that this differs from the conventions
when $M,N$ are active type 3 premice,
where fine structural maps $\pi:M\to N$
are literally of the form $\pi:M^\sq\to N^\sq$.} We say $\pi$ is a \emph{$(-1)$-embedding} iff
 $\pi$ is $\Sigma_1$-elementary  and cofinal in $\rho_{(-1)}^N=\OR^N$. We say $\pi$ is a \emph{near $(-1)$-embedding} iff it is $\Sigma_1$-elementary, and
\emph{$(-1)$-lifting} iff $\pi$ is  $\Sigma_0$-elementary. (Note that, for example, a near $m$-embedding between type 3
premice induces an
$\rSigma_{m+2}$-elementary embedding between the unsquashed structures.
Likewise, a near
$(-1)$-embedding is $\rSigma_1$-elementary between the unsquashed structures.
Degree $(-1)$ sits immediately below degree $0$  in the squashed hierarchy.)

\subsubsection{Iteration trees and iterability}\label{sec:notation_iteration_trees}

%conf
See \cite[\S1.1.6]{premouse_inheriting}
and \cite[\S1.1.5]{iter_for_stacks}.
Let $\Tt$ be an iteration tree.
We write $\lh(\Tt)$ for the length  of $\Tt$,
 $\lh(\Tt)^-=\{\alpha\bigm|\alpha+1<\lh(\Tt)\}$, and
$<^\Tt$ for the tree order on $\lh(\Tt)$. We write $(M^\Tt_\alpha,\deg^\Tt_\alpha)$ for the $\alpha$th model and  degree of $\Tt$, for $\alpha<\lh(\Tt)$.
Given $\alpha+1<\lh(\Tt)$, $\beta=\pred^\Tt(\alpha+1)$ denotes
the $<^\Tt$-predecessor of $\alpha+1$,
$M^{*\Tt}_{\alpha+1}$ is the model $N\ins M^\Tt_\beta$ such that $M^\Tt_{\alpha+1}=\Ult_d(N,E^\Tt_\alpha)$,
where $d=\deg^\Tt_{\alpha+1}$,
and $i^{*\Tt}_{\alpha+1}:M^{*\Tt}_{\alpha+1}\to M^\Tt_{\alpha+1}$ is the  ultrapower map. We write $\mathscr{D}^\Tt$ for the set of nodes at which $\Tt$ drops in model, so $\alpha+1\in\mathscr{D}^\Tt$ iff $M^{*\Tt}_{\alpha+1}\pins M^\Tt_\alpha$;
and $\mathscr{D}^\Tt_{\deg}$ is likewise,
but for drops in model or degree.
If $\gamma\leq^\Tt\delta$ then $[\gamma,\delta]^\Tt$, $(\gamma,\delta]^\Tt$, etc, are the corresponding $<^\Tt$-intervals.  If $\gamma\leq^\Tt\delta$ and $(\gamma,\delta]^\Tt\cap\dropset^\Tt=\emptyset$
then $i^\Tt_{\gamma\delta}:M^\Tt_\gamma\to M^\Tt_\delta$ is the iteration map,
and if $\gamma=\alpha+1$
then $i^{*\Tt}_{\alpha+1,\delta}=i^\Tt_{\alpha+1,\delta}\com i^{*\Tt}_{\alpha+1}$.
If $\gamma<^\Tt\delta$ then $\succ^\Tt(\gamma,\delta)$ denotes the $<^\Tt$-successor of $\gamma$ in the interval $(\gamma,\delta]^\Tt$.
A tree of length $1$ is \emph{trivial} in that it uses no extenders. In the context of an iteration tree $\Tt$ of
successor length $\xi+1$, $\infty$ denotes $\xi$. Given such a tree, $b^\Tt$ denotes $[0,\infty]^\Tt$, and if $(0,\infty]^\Tt\cap\dropset^\Tt=\emptyset$ then $i^\Tt$ denotes $i^\Tt_{0\infty}$.
We say $\Tt$ is \emph{terminally-non-dropping} iff $\Tt$ has successor
length and $b^\Tt\cap\dropset^\Tt_{\deg}=\emptyset$; $\Tt$ is
\emph{terminally-non-model-dropping}
iff $\Tt$ has successor length and $b^\Tt\cap\dropset^\Tt=\emptyset$.
We will also use much of the foregoing terminology and slight variants thereof
for iteration trees on bicephali (see \S\ref{sec:bicephali} and \S\ref{sec:solidity}), without necessarily mentioning it explicitly.

Let $M$ be an $m$-sound premouse.
Recall that an iteration tree $\Tt$ on $M$ is \emph{$m$-maximal}
if (i) $\lh(E^\Tt_\alpha)\leq\lh(E^\Tt_\beta)$ for all
$\alpha<\beta$,
(ii) $(M^\Tt_0,\deg^\Tt_0)=(M,m)$,
(iii) for all $\alpha+1<\lh(\Tt)$, $\pred^\Tt(\alpha+1)$ is the least $\beta$ such that
$\crit(E^\Tt_\alpha)<\nu(E^\Tt_\beta)$, and
(iv)  for all such $\alpha+1,\beta$,
$(M^{*\Tt}_{\alpha+1},\deg^\Tt_{\alpha+1})$
is the largest $(N,n)\ins(M^\Tt_\beta,\deg^\Tt_\beta)$
such that $\Ult_n(N,E^\Tt_\alpha)$ is well-defined.
An $(m,\alpha)$-iteration strategy
is one for $m$-maximal trees of length $\leq\alpha$
(that is, player II wins if a tree of length $\alpha$ is produced).

For the definitions of \emph{$m$-maximal stacks}
see \cite[\S1.1.5]{iter_for_stacks},
and for the iteration game $\mathscr{G}_{\mathrm{fin}}(M,m,\omega_1+1)$
%conf
see \cite[Definition 1.1]{iter_for_stacks}.

\begin{rem}\label{rem:padding}
In comparison arguments, in which we have two trees $\Tt,\Uu$,
we will make the conventional use of padding, setting $E^\Tt_\alpha=\emptyset$
or $E^\Uu_\alpha=\emptyset$
at stages $\alpha$ of the comparison
at which we do not use an extender in $\Tt$ or $\Uu$ respectively. In some comparison arguments we will also have stages $\alpha$
at which $E^\Tt_\alpha=\emptyset=E^\Uu_\alpha$.
We will only explicitly define
the rules for determining
$\pred^\Tt(\alpha+1)$ for non-padded trees.  Consider a padded tree $\Tt$. If $E^\Tt_\alpha=\emptyset$ then we set $\pred^\Tt(\alpha+1)=\alpha$ (and $M^\Tt_{\alpha+1}=M^\Tt_\alpha$ etc).
If $E^\Tt_\alpha\neq\emptyset$,
we have
an interval $[\beta_0,\beta_1]$
of ordinals $\beta$ which might
serve as $\pred^\Tt(\alpha+1)$.
But here we will have  $E^\Tt_\beta=\emptyset$ for all $\beta\in[\beta_0,\beta_1)$,
and $M^\Tt_\beta=M^\Tt_{\beta_0}$
and $\deg^\Tt_\beta=\deg^\Tt_{\beta_0}$
for all $\beta\in[\beta_0,\beta_1]$.
So it does not actually matter
which $\beta\in[\beta_0,\beta_1]$
is used as $\pred^\Tt(\alpha+1)$.
But for specificity, we set $\beta=\beta_1$ (the unique ordinal $\beta'$ such that $E^\Tt_{\beta'}\neq\emptyset$ and the rules for non-padded trees are satisfied at $\beta'$).\end{rem}

\section{Proof outline}\label{sec:plan}

In the classical proof that sufficiently iterable countable premice $M$ possess the various standard fine structural properties,  iterability with respect to transfinite stacks of normal trees is used to show that there is a normal strategy for $M$ with the weak Dodd-Jensen property (cf., for example,
%conf
\cite[\S4.3]{outline}). Such a strategy is then used for
various comparison arguments, and the weak Dodd-Jensen property is of central importance to the analysis of those comparisons.

Since we assume only normal iterability,
we must prove the fine structural facts
without relying on  weak Dodd-Jensen (cf.~Footnote \ref{ftn:wDJ_necessary}).
 In this section we will give an outline of the main arguments in the paper, highlighting the methods for getting around the lack of weak Dodd-Jensen.

As discussed in \S\ref{sec:intro}, condensation (Theorem \ref{thm:condensation}) was almost already dealt with in \cite[Theorem 5.2]{premouse_inheriting}, and it will follow from that theorem and  solidity (Theorem \ref{thm:solidity}). The proofs of solidity (\ref{thm:solidity})
 and projectum-finite generation (Theorem \ref{thm:finite_gen_hull}) which we give here are very related to the methods used in \cite{premouse_inheriting}. So we won't say anything further directly regarding the details of
the proof of condensation.

\subsection{Solidity and universality}\label{sec:outline_solidity}
We begin  with the most central results of the paper, solidity and universality, primarily focusing on solidity.
Let $k<\om$ and let $M$ be a $k$-sound, $(k,\om_1+1)$-iterable premouse.
Say we want to prove that $M$ is $(k+1)$-solid. Let $\gamma\in p_{k+1}^M$ and let
\[ H=\cHull^M_{k+1}(\gamma\cup\{p_{k+1}^M\cut\{\gamma\},\pvec_k^M\}).\]
We want to see  $H\in M$. We may assume that $M$ is countable,
and that for $\gamma'\in p_{k+1}^M$
with $\gamma'>\gamma$,
the corresponding hull $H'$ is in $M$.
For the purposes of this outline, let us  also assume that $\gamma$ is an $M$-cardinal.

In the classical proof that $H\in M$ under these hypotheses, assuming also that there is an $(k,\om_1+1)$-iteration strategy $\Sigma$ for $M$ with the weak Dodd-Jensen property, the phalanx $\phP=((M,{<\gamma}),H,\gamma)$ is compared with $M$, producing iteration trees $\Uu$ on $\phP$ and $\Tt$ of $M$. The first extender $E$ used in $\Uu$ (if there is one) is taken from $\es_+^H$, with $\gamma<\lh(E)$,
and extenders $F$ with $\crit(F)<\gamma$ apply to $M$,
whereas if $\crit(F)\geq\gamma$
then $F$ applies to $H$ or a later model of $\Uu$, with tree order determined by the rules of normality. Suppose that $\Uu$ is formed by lifting to
a tree $\Uu'$ on $M$ via  $\Sigma$,
and also that $\Tt$
is via $\Sigma$.
Then
one shows that
 the comparison
terminates with  main branch $b^\Uu$
of $\Uu$ above $H$;
 let $M^\Uu_\infty$ be the final iterate.
 One shows that there is no drop in model or degree along the main branch $b^\Uu$ of $\Uu$,
 so we have an iteration map
 $i^\Uu:H\to M^\Uu_\infty$,
 which is a $k$-embedding,
 and $\crit(i^\Uu)\geq\gamma$, because of the rule that
 extenders $E$ with $\crit(E)<\gamma$ apply to $M$.
  (Possibly $\Uu$ is trivial, in which case $M^\Uu_\infty=H$
  and $i^\Uu$ is the identity.)
 One also shows that either $M^\Uu_\infty\pins M^\Tt_\infty$  or $b^\Tt$ drops in model and $M^\Uu_\infty=M^\Tt_\infty$. These  facts (together with fine structural analysis) lead to the conclusion that $H\in M$.
But the arguments involved (showing that $b^\Uu$ is above $H$, etc),
make significant use of the weak Dodd-Jensen property.

The first key to working without weak Dodd-Jensen is that we prove solidity and universality  by
induction
on a certain mouse order $<^{\para}_k$. (The top-down induction on the elements of $p_{k+1}^M$ made implicit earlier is a sub-induction of this global one.) This  order is related to  Dodd-Jensen, and the induction
provides a partial substitute for weak Dodd-Jensen, one which is provable from
normal iterability alone.
Given $k$-sound premice $P,Q$,
we set $P<^{\para}_kQ$ iff there is a
finite $k$-maximal stack $\vec{\Tt}=\left<\Tt_i\right>_{i<n}$ on $Q$,
consisting
of finite trees $\Tt_i$, non-dropping in model and degree
on their main branches,
and a $k$-embedding
$\pi:P\to M^{\vec{\Tt}}_\infty$ such that
$\pi(p_{k+1}^P)<i^{\vec{\Tt}}(p_{k+1}^Q)$.
This order is wellfounded within the $(k,\om_1+1)$-iterable
premice, so the induction makes sense (Lemma
\ref{lem:param_order_wfd}).

Since we are working by induction on $<^{\para}_k$,
for our particular $M$ introduced above,
we may assume that  all $P<^{\para}_kM$ are $(k+1)$-solid and $(k+1)$-universal. This has useful consequences
analogous to Dodd-Jensen, and  is  useful in showing that $M$ is $(k+1)$-solid and $(k+1)$-universal.
In particular, we will see that $H<^{\para}_kM$,  and so
$H$ is
$(k+1)$-solid and $(k+1)$-universal.
In certain cases, we will be able to use this to deduce that $H\in M$.
This deduction, however, seems to require the application of
Lemma \ref{lem:finite_gen_hull}
to premice $P<^{\para}_kM$;
this lemma is just like the projectum-finite generation
theorem \ref{thm:finite_gen_hull},
except that it has extra solidity and universality hypotheses
(which we know hold of $P<^{\para}_kM$).
\footnote{
The proof of  Lemma \ref{lem:finite_gen_hull}
relies on  Dodd-soundness, so this has a role in our proof of solidity and universality, albeit indirect.}
\footnote{Actually, it suffices
to apply Corollary \ref{cor:def_over_core_from_extra_hypos}, which is a less fine consequence of
 Lemma \ref{lem:finite_gen_hull}.}

The inductive hypothesis is also
useful in a second manner:
it has the consequence that degree $k$ iteration maps on $M$
preserve $p_{k+1}$ (Lemma \ref{lem:p-pres_for_premouse_from_para_order}). This preservation
fact is also used in establishing that $M$ is $(k+1)$-solid (whereas
 usually $p_{k+1}$-preservation
is taken as a consequence of $(k+1)$-solidity).
It helps ensure that certain comparison arguments  we perform (to be discussed below)
end in a useful fashion, and also helps in their analysis.

The fact that $<^{\para}_k$ is wellfounded within  $(k,\om_1+1)$-iterable
premice follows easily from
%conf
\cite[Theorem 9.6]{iter_for_stacks}. The main point of \cite{iter_for_stacks}
was the construction of an iteration strategy for
stacks
of normal trees from a normal iteration strategy which has (a certain kind of) condensation.
Here, we do not have any such strategy condensation assumption, but (as shown
in
\cite{iter_for_stacks}) such an assumption  is unnecessary here, because
we can prove that $<^{\para}_k$ is wellfounded by considering
stacks of length $\om$ consisting of normal trees of finite length
(and condensation is only relevant for stacks incorporating  infinite trees, where the
strategy's branch choices come into play).\footnote{A length $\om$ stack
of finite trees lifts, via the methods of \cite{iter_for_stacks},
into an infinite tree, of length some countable ordinal,
but no strategy condensation is needed for that calculation.}

Using the inductive hypothesis (with $<^{\para}_k$) and some fine structural calculations,
making use of condensation as in Facts \ref{fact:condensation}
and \ref{fact:om_condensation},
we will reduce to the case that \[ M=\Hull_{k+1}^M(\rho_{k+1}^M\cup\{\pvec_{k+1}^M\}), \]
along with some further easy restrictions which we will not detail here. And as mentioned above, the induction hypothesis leads to the fact that
$H$ is $(k+1)$-solid and $(k+1)$-universal.

After these initial steps, the second key to working without weak Dodd-Jensen enters the picture,
which is to make use of  \emph{\tu{(}generalized\tu{)} bicephalus comparisons}
in place of phalanx comparisons. They are motivated by both the classical bicephalus comparison arguments such as those used to establish the ``uniqueness of the next extender'' results in \cite[\S9]{fsit} (see Remark \ref{rem:traditional_bicephali}), and also by the classical phalanx comparison arguments. Recall that the bicephali considered in \cite{fsit}   have (or can be viewed as having) form
$(N_0,N_1)$ where $N_0$ and $N_1$ are active premice such that $N_0^\passive=N_1^\passive$, and in particular $N_0,N_1$ have the same universe. The classical bicephalus argument involves a comparison
of   $(N_0,N_1)$ with itself.
In our present context,
letting $\pi:H\to M$ be the uncollapse map,
we will consider $B=(\gamma,\gamma,\pi,H,M)$ as a  bicephalus,
consisting of two structures $H,M$, and some auxiliary information given by the ``exchange ordinal'' $\gamma$
and the map $\pi:H\to M$.
(The repetition of $\gamma$ is just for consistency with the more general case dealt with in \S\ref{sec:solidity}.)
Note that $H$ and $M$ have distinct universes, so $(H,M)$ is certainly not a bicephalus in the sense of \cite{fsit}.
We will form a comparison of $B$
with $M$, forming trees $\Tt$ on $B$ and $\Uu$ on $M$.
The first extender $E$ used in $\Tt$ will again have $\gamma<\lh(E)$.
If we use an extender $E^\Tt_\alpha$ in $\Tt$
with $\crit(E^\Tt_\alpha)<\gamma$ then we will have $\pred^\Tt(\alpha+1)=0$ and
$E^\Tt_\alpha$ will apply  to the entire bicephalus $B$, not just $M$ (nor just $H$). This will produce an ultrapower
\[ B^\Tt_{\alpha+1}=\Ult(B,E^\Tt_\alpha)=(\gamma',\gamma',\pi',H',M'),\]
with similar properties to those of $B$. Here $M'=\Ult_k(M,E)$,
$\gamma'=i^{M,k}_E(\gamma)$,
and $H',\pi'$ come from the hull of $M'$ at $\gamma'$, analogous to the manner in which $H,\pi$ come from the hull of $M$ at $\gamma$. See \S\ref{sec:solidity} for the precise definition.
Similarly, if we have a node $\beta<\lh(\Tt)$ and $B^\Tt_\beta=(\gamma^\Tt_\beta,\gamma^\Tt_\beta,\pi^\Tt_\beta,H^\Tt_\beta,M^\Tt_\beta)$
is a bicephalus, $\pred^\Tt(\xi+1)=\beta$, $\crit(E^\Tt_\xi)<\gamma^\Tt_\beta$
and $E^\Tt_\xi$ is total over $M^\Tt_\beta$,
then $B^\Tt_{\xi+1}$ will be the bicephalus $\Ult(B^\Tt_{\beta},E^\Tt_\xi)$.
Other  extenders will apply to just a single model, forming an ultrapower in the usual fashion.
For example, if $\gamma\leq\crit(E^\Tt_0)$, then $E^\Tt_0$ will apply to a segment of either $H$ or $M$, depending on the details of the situation.
We will show that the comparison terminates, by basically the usual argument.
Having a bicephalus $B^\Tt_\beta$
indexed at $\beta$ helps to keep the comparison going at stage $\beta$,
analogously to comparison arguments with classical bicephali. This helps to prevent the comparison from terminating for unuseful reasons, and substitutes partially for the fact, in a phalanx argument using weak Dodd-Jensen, that for the tree $\Tt'$ on the phalanx $\phP$, $b^{\Tt'}$ does not end up above $M$. So the comparison must terminate for ``useful'' reasons. But on the other hand,
assuming that $H\notin M$,
we will show that the iteration maps connecting $B$ to later bicephali $B^\Tt_\beta$ will preserve fine structure nicely enough that
the comparison \emph{cannot} terminate for such ``useful'' reasons. This contradiction will give that $H\in M$. The material in \S\ref{sec:fine_structure_prelim}
will help establish this preservation of fine structure.

This completes our discussion of
the plan for proving
solidity and universality.

\begin{rem}\label{rem:history_gen_biceph}
 Prior to the development of the general ideas for the proof of solidity and universality we give (in 2014--15), forms of generalized bicephali had also been used by
 %conf
 Schindler \cite[\S5, Definition 5.3]{cmfali} (on the core model), Woodin (on fine structure for mice with long extenders), and the author
 (for condensation \cite[Theorem 5.2]{premouse_inheriting} and the premouse inheriting strong cardinals \cite[\S6]{premouse_inheriting}).
 Around  the same time as the present work,
Steel used related methods in
fine structural proofs for strategy mice
\cite{ACPFMP} (when ``lifting a phalanx'' in comparison arguments involving iteration to a background construction).
\end{rem}

\subsection{Measures, Dodd-soundness, projectum-finite generation}\label{sec:meas_Dodd_prof_fin_plan}

We now give an outline of the kinds of arguments to be used in the main results in \S\ref{sec:mim},
\S\ref{sec:Dodd_proof}
and \S\ref{sec:finite_gen_hull}.
Each of these rely on a common kind of analysis of a comparison, but  in somewhat different contexts.

As a representative  case, which is also the simplest, let us focus on the setup for the argument for Theorem \ref{thm:measures_in_mice} in \S\ref{sec:mim}.
So let $M$ be a passive mouse and $\mu\in M$
be such that $M\sats$ ``$\mu$ is a countably complete ultrafilter''.
We want to see that there is a $0$-maximal tree $\Tt$ on $M$
with $M^\Tt_\infty=\Ult_0(M,\mu)$, $i^\Tt$ exists and $i^\Tt=i^{M}_\mu$,
and moreover, $\Tt$ has finite length and every extender used in $\Tt$ is finitely generated,
meaning that it is equivalent to its restriction to some finite set of generators.

Let $U=\Ult_0(M,\mu)$.
For the present outline, let us make some simplifying assumptions. Suppose  that $U$ is wellfounded and iterable. Compare $M$ with $U$,
with trees  $\Uu$ on $U$ and $\Tt$ on $M$. Suppose that we get $M^\Uu_\infty=Q$ where $Q=M^\Tt_\infty$, the final branches $b^\Uu,b^\Tt$ are non-dropping
and $k\com i^M_\mu=j$,
where $k=i^\Uu$ and $j=i^\Tt$. See Figure \ref{fgr:sketch} (some of the components of which are yet to be defined).
\begin{figure}
\centering
\begin{tikzpicture}
 [mymatrix/.style={
    matrix of math nodes,
    row sep=0.35cm,
    column sep=0.4cm}]
   \matrix(m)[mymatrix]{
  U&          {} &       {}& {} & Q& \\
   {} & {} \\
   {} & {} & \bar{Q}\\
   \\
 M\\};
 \path[->,font=\scriptsize]
%maps from left $M$
(m-5-1) edge[bend right] node[below] {$\ \ \ \ j$} (m-1-5)
(m-5-1) edge node[left] {$i^M_\mu$} (m-1-1)
(m-1-1)edge node[above] {$\sigma$} (m-3-3)
(m-5-1)edge node[above] {$\bar{j}$} (m-3-3)
(m-3-3)edge node[above] {$\somevarpi$} (m-1-5)
(m-1-1) edge node[above] {$k$} (m-1-5);
\end{tikzpicture}
\caption{The diagram commutes, where $k=i^\Uu$, $j=i^\Tt$,  $\bar{j}=i^{\bar{\Tt}}$.}
\label{fgr:sketch}
\end{figure}
We want to see then that $\Uu$ is trivial;
 this will give that $U=Q$ and $k=\id$ and $i^M_\mu=j$. We also want to see  $\Tt$ is finite and every extender used in $\Tt$ is finitely generated.

 Now since $k\com i^M_\mu=j$, we have $\crit(j)=\min(\crit(i^M_\mu),\crit(k))\leq\crit(k)$.
Because $\Tt,\Uu$ arise from comparison and by considerations of compatibility of extenders, it follows that $\crit(j)=\crit(i^M_\mu)<\crit(k)$. Let $x\in U$ be the seed of $\mu$, represented by the identity function.
The expectation is that $\Tt$ is the minimal tree which ``captures $x$''.
Using a routine kind of finite support computation, to be discussed in \S\ref{sec:sub_simple_embeddings}, we can find a tree $\bar{\Tt}$ of finite length with last model $\bar{Q}$, such that $b^{\bar{\Tt}}$
is non-dropping, and find a (suitably elementary) map $\somevarpi:\bar{Q}\to Q$
such that $\somevarpi\com \bar{j}=j$
where $\bar{j}=i^{\bar{\Tt}}$,
and with $k(x)\in\rg(\somevarpi)$.
Because of this and since $U=\{i^M_\mu(f)(x)\bigm|f\in M\}$,
we can define a map $\sigma:U\to \bar{Q}$
with $\sigma\com i^M_\mu=\bar{j}$ and $\somevarpi\com\sigma=k$.
By the commutativity, we have
$\crit(\bar{j})=\crit(j)<\crit(k)$.

Suppose for simplicity that the first extender used along $b^{\bar{\Tt}}$ is a normal measure $E^{\bar{\Tt}}_{\bar{\alpha}}$. By the commutativity, $E^{\bar{\Tt}}_{\bar{\alpha}}$ must be the normal measure derived from the first extender $E^\Tt_\alpha$ used along $b^\Tt$. And since $k$ is an iteration map with $\crit(\bar{j})<\crit(k)=\min(\crit(\sigma),\crit(\somevarpi))$,
we get that $U$, $\bar{Q}$ and $Q$
agree below $\crit(k)$,
and since $E^{\bar{\Tt}}_{\bar{\alpha}}\notin \bar{Q}$, therefore
 $E^{\bar{\Tt}}_{\bar{\alpha}}\notin Q$. Considering the ISC, it follows that $E^{\bar{\Tt}}_{\bar{\alpha}}=E^\Tt_\alpha$ is a normal measure. So we have made progress toward  our goals: the first extender used along $b^\Tt$  is finitely generated, and its generator is contained in the range of $k$. (If we get that all generators of all extenders used along $b^\Tt$ and in the range of $k$,
 then $U=Q$, so $\Uu$ is trivial.)

  The last paragraph is too much of a simplification, however, as will not be able to arrange that $E^{\bar{\Tt}}_{\bar{\alpha}}$  is a normal measure in general (neither the later extenders used along $b^{\bar{\Tt}}$). But  we will be able to arrange that all extenders used along $b^{\bar{\Tt}}$ (in fact in $\bar{\Tt}$ at all, in this case) are finitely generated, via a somewhat more careful finite support computation,
 detailed in \S\ref{sec:capturing_elements}. But this is not enough for the preceding argument to work, since
 we won't yet be able to rule out the possibility that $k\neq\id$ and $\crit(k)\leq\gamma$ for some generator $\gamma$ of $E^{\bar{\Tt}}_{\bar{\alpha}}$.
 In order to deal with this, we will need to be able to analyze the movement of certain key generators under the relevant kinds of embeddings (such as iteration embeddings).
 For this, we will need to know that all extenders in $\es^M$
 are Dodd-sound; the Dodd-solidity of the extenders will allow us to track images of Dodd parameters from Dodd-sound extenders (see \S\ref{sec:Dodd_prelim}).
 Using this, we will end up being able to show that $E^{\bar{\Tt}}_{\bar{\alpha}}=E^\Tt_\alpha$
 and this extender is generated by a finite set
 which is in the range of $k$.
 (This can in general take multiple steps of analysis.
 It might be that $E^{\bar{\Tt}}_{\bar{\alpha}}$ itself is non-Dodd-sound, but it will equivalent to the composition of  a finite sequence of Dodd-sound extenders, and each step of analysis will correspond to one of those Dodd-sound extenders.)
 This argument can be iterated all the way along $b^{\bar{\Tt}}$,
 and this will eventually give that $U=\bar{Q}=Q$ and $\Tt$ is as desired.

 Along with Theorem \ref{thm:measures_in_mice}, we will use variants of the argument outlined above to prove
 Theorems \ref{thm:super-Dodd-soundness}
 (super-Dodd soundness)
 and \ref{thm:finite_gen_hull}
 (projectum-finite generation),
 and also Theorem  \ref{tm:if_U_not_iterate_of_core}.
 In \S\ref{sec:fine_structure_prelim} and \S\S\ref{sec:bicephali}--\ref{sec:capturing_elements}
we cover in detail
the tools needed to complete the sketch outlined above,
and the background for their adaptations to the other theorems just mentioned.
(The results of the careful finite support computation in \S\ref{sec:capturing_elements} will
 only be used directly in \S\ref{sec:mim}.
 But this key material will be adapted to and applied in the other contexts
directly within \S\S\ref{sec:Dodd_proof},\ref{sec:finite_gen_hull} themselves.)
The proof of Theorem \ref{thm:finite_gen_hull}
in particular will involve a kind of bicephalus comparison (but the bicephali here will be different to those considered in the proof of solidity mentioned earlier). For this reason,
 much of the material in \S\S\ref{sec:bicephali}--\ref{sec:Dodd_prelim} deals with these bicephali. In order to motivate the role of bicephali within that material, let us also say a little about the  proof of  \ref{thm:finite_gen_hull}.

 Let $M$ satisfy the assumptions of Theorem \ref{thm:finite_gen_hull}
 and $C=\core_{m+1}(M)$. Supposing also that $M$ is $(m+1)$-universal
 and $C$ is $(m+1)$-solid
 (hence $C$ is $(m+1)$-sound),
 we will use a comparison argument
 to prove the conclusions of the lemma, comparing a certain kind of bicephalus with itself.
 Letting $\rho=\rho_{m+1}^M=\rho_{m+1}^C$, the bicephalus will be $B=(\rho,C,M)$,
 consisting of the two models $C$ and $M$, and an ``exchange ordinal'' $\rho$.
 In the comparison,
 we will form two iteration trees $\Tt,\Uu$,  both on $B$. At nodes $\alpha$ of $\Tt$, we will have either have a bicephalus $B^\Tt_\alpha=(\rho^\Tt_\alpha,C^\Tt_\alpha,M^\Tt_\alpha)$,
 with fine structural properties  much like those of $B$, or a single premouse $C^\Tt_\alpha$, or a single premouse $M^\Tt_\alpha$.
Much like in the solidity argument,
if we have a bicephalus $B^\Tt_\alpha$ and
 $\pred^\Tt(\beta+1)=\alpha$
 and $\crit(E^\Tt_\beta)<\rho^\Tt_\alpha$ and $E^\Tt_\beta$ is $C^\Tt_\alpha$-total (equivalently, $M^\Tt_\alpha$-total), then $B^\Tt_{\beta+1}$ will be a bicephalus $\Ult(B^\Tt_\alpha,E^\Tt_\beta)$.
 Other extenders will just apply to a single model.
 We will arrange the comparison such that it terminates in the following fashion: there is some node $\alpha$ of $\Uu$ or of $\Tt$, but let us assume it is $\Uu$, such that:
 \begin{enumerate}[label=(\roman*)]\item there is a bicephalus $B^\Uu_\alpha=(\rho^\Uu_\alpha,C^\Uu_\alpha,M^\Uu_\alpha)$ indexed at $\alpha$ in $\Uu$,
 \item\label{item:C^Tt_alpha_is_a_seg_of_model_of_Uu} $C^\Uu_\alpha$ is an initial segment of one of the models of $\Tt$ indexed at stage $\alpha$
 (there might be one or two of them),
 \item $\Uu\rest[\alpha,\infty]$ is trivial (meaning that it uses no extenders),
 and
 \item\label{item:structure_of_tail_of_Uu} $\Tt\rest[\alpha,\infty]$
 is equivalent to  an $m$-maximal ``tree'' $\Tt'$ on $C^\Uu_\alpha$\footnote{\label{ftn:wrinkles_1}Although $\Tt$ has wellfounded models, at this stage of the proof, we will have to allow the possibility that $\Tt'$  has illfounded models. (Note $\Tt'$ is  on $C^\Uu_\alpha$,
 as opposed to being a normal continuation of $\Tt\rest(\alpha+1)$, and we don't seem to know here that $C^\Uu_\alpha$ is itself iterable.) But that illfoundedness will be strictly above the part of the model relevant to the comparison (that is, least disagreements will always occur in the wellfounded part, and the last model of $\Tt'$ will just be $M^\Uu_\alpha$, which is wellfounded). Actually, instead of working with $\Tt'$, we will work with a (real) tree on a certain phalanx, avoiding any possibility of illfoundedness.} (recall point \ref{item:C^Tt_alpha_is_a_seg_of_model_of_Uu} in this connection), has finite length
 (so $\infty=\alpha+n$ for some $n<\om$),  uses only finitely generated extenders, and has last model  $M^\Uu_\alpha$.
 \end{enumerate}
 Point \ref{item:structure_of_tail_of_Uu} says that the theorem holds with respect to $C^\Uu_\alpha,M^\Uu_\alpha$
 in place of $C,M$ (except for the wrinkle mentioned in Footnote \ref{ftn:wrinkles_1}).
 From here, we show that the picture reflects back to $C,M$ themselves,
 proving the theorem (under the extra hypotheses of universality and solidity).%\setcounter{footnote}{0}
 \footnote{The possibility of illfoundedness mentioned in Footnote \ref{ftn:wrinkles_1}
 has no bearing on the tree on $C$ which we get, since $C$ is iterable, and this tree is only finite in length.} The details on how the comparison is formed, and the reflection back to $C,M$, will be discussed in \S\ref{sec:proj_fin_gen_theorem},
 and we won't say anything more about those things here.
  The analysis of the tail end $(\Tt,\Uu)\rest[\alpha,\infty]$ of the comparison
 is very much like the analysis of the comparison in the proof of Theorem \ref{thm:measures_in_mice}, which we outlined above.
 So in \S\S\ref{sec:simple_embeddings},\ref{sec:Dodd_prelim} we develop  tools needed for that argument
 both for iteration trees on premice
 and on (this kind of) bicephali,
 and
   within  \S\S\ref{sec:Dodd_proof},\ref{sec:finite_gen_hull}, we develop variants of the more careful finite support argument from \S\ref{sec:capturing_elements}.
The bicephali themselves and iteration trees on them are first described in
 \S\ref{sec:bicephali}.

\subsection{Structure of main results}

The logical structure of the proof of the central fine structural results is as follows:
\begin{enumerate}[label=--]
 \item In
\S\ref{sec:fine_structure_prelim}
and \S\S\ref{sec:bicephali}--\ref{sec:capturing_elements}
 we lay out basic definitions and prove various lemmas.
 Parts of this material are quite standard or a small variant of standard theory,
 but its inclusion  will hopefully make for  a better reading experience.  The material
 in \S\ref{sec:fine_structure_prelim}
 will be used throughout,
 the material in \S\S \ref{sec:bicephali}--\ref{sec:capturing_elements} used in \S\S\ref{sec:mim}--\ref{sec:finite_gen_hull}.
 \item In \S\ref{sec:mim},
 we prove Lemma \ref{lem:measures_in_mice},
 which is equivalent to Theorem \ref{thm:measures_in_mice} on measures in mice, except that it has the additional hypothesis that  all proper segments of $M$ are Dodd-sound.
 This lemma is not needed for the main fine structural results;
 it is only used to deduce the actual Theorem \ref{thm:measures_in_mice}
 later in \S\ref{sec:Dodd_proof},
 and Theorem \ref{thm:measures_in_mice} is itself not needed. However,
  parts of the proof of Lemma \ref{lem:measures_in_mice}
  are referred to in \S\S\ref{sec:Dodd_proof},\ref{sec:finite_gen_hull}.
 \item In \S\ref{sec:Dodd_proof}, we prove Lemma \ref{lem:super-Dodd-soundness},
 which is equivalent to the super-Dodd-soundness Theorem
\ref{thm:super-Dodd-soundness}, except that it has the additional hypothesis
that the 1st core $\core_1(M)$ is $1$-sound (this is not immediate
without $1$-universality and $1$-solidity).
\item In \S\ref{sec:finite_gen_hull}, we prove Lemma \ref{lem:finite_gen_hull},
which is equivalent to  projectum-finite generation
\ref{thm:finite_gen_hull}, except that it has the additional hypothesis
that $M$ is $(m+1)$-universal
and its $(m+1)$st core $\core_{m+1}(M)$ is $(m+1)$-solid.
The proof will rely on Lemma \ref{lem:super-Dodd-soundness}.
We also deduce Corollary \ref{cor:def_over_core_from_extra_hypos} from Lemma \ref{lem:finite_gen_hull}.
\item In \S\ref{sec:mouse_order}, we introduce the mouse
order $<^{\para}_k$, and prove some facts about it.
\item In \S\ref{sec:solidity}, we prove the solidity and universality Theorem
\ref{thm:solidity}, using Fact \ref{fact:condensation},
the material in \S\ref{sec:fine_structure_prelim}, \S\ref{sec:mouse_order},
and Corollary \ref{cor:def_over_core_from_extra_hypos}.\footnote{The
only proof that we know
of for Corollary \ref{cor:def_over_core_from_extra_hypos}
is via Lemma \ref{lem:finite_gen_hull},
the proof of which relies on significant fine structure,
including Dodd-solidity.
This is the only way in which
Dodd-solidity
comes up in the solidity and universality.
So a simpler proof of Corollary \ref{cor:def_over_core_from_extra_hypos} might simplify the proof of solidity and universality
overall.

The \emph{super} aspect of super-Dodd-solidity is not used the proof, however; only standard Dodd-solidity is relevant.}

\item In \S\ref{sec:conclusion}, we deduce that the full versions
of Theorems \ref{thm:super-Dodd-soundness}, \ref{thm:finite_gen_hull}
and \ref{thm:condensation} hold.
\end{enumerate}

If one takes Corollary \ref{cor:def_over_core_from_extra_hypos} as a black box,
the rest of the proof of solidity and universality is covered
by just
 \S\ref{sec:fine_structure_prelim},
 \S\ref{sec:mouse_order}
 and \S\ref{sec:solidity}.

 In \S\ref{sec:Dodd_proof} we also prove Theorem \ref{thm:ISC} (on
pseudomice and the ISC) outright; it does not depend
 on the other results.

\section{Fine structural preliminaries}
\label{sec:fine_structure_prelim}

We begin by laying out some general
fine structural facts which we will need throughout.
This extends somewhat the theory established in \cite{fsit} and
\cite{outline}.

\begin{rem}[Superstrongs]\label{rem:superstrong_diff}
Nearly all of the general MS theory, including
that of \cite{fsit} and \cite{outline}, goes through at the
superstrong level, with very little change to proofs.
(Actually, we only need part of it, since we will be
proving solidity etc here anyway.)
Likewise for the results in \cite{extmax},
%conf
excluding  \cite[\S5]{extmax}, which
is appropriately generalized in \cite{premouse_inheriting}.
We restate some of the theorems from those papers in this section,
but without
the superstrong
restriction.

The main ubiquitous new feature
is that if $\Tt$ is an $m$-maximal
 tree on an $m$-sound premouse, we can have  $\alpha+1<\lh(\Tt)$ with
$\lh(E^\Tt_\alpha)=\lh(E^\Tt_{\alpha+1})$, but only
under special circumstances (in particular that $E^\Tt_\alpha$
is superstrong); see
\cite[\S1.1.6]{premouse_inheriting}.
This implies that when forming a comparison
$(\Tt,\Uu)$ of premice, if $\gamma$ indexes the least disagreement
between $M^\Tt_\alpha,M^\Uu_\alpha$,
one should only set $E^\Tt_\alpha\neq\emptyset$
if $F(M^\Tt_\alpha|\gamma)\neq\emptyset$
and either $F(M^\Uu_\alpha|\gamma)=\emptyset$
or $\nu(F(M^\Tt_\alpha|\gamma))\leq\nu(F(M^\Uu_\alpha|\gamma))$,
and likewise symmetrically for $E^\Uu_\alpha$.
In some variant situations, such as comparisons involving a tree $\Uu$ on a phalanx, we may have some models $M^\Uu_\alpha$ which fail to be pre-ISC-premice;
but they will be pre-ISC-premice.
For any active pre-ISC-premouse $S$,
let $\nutilde(S)=\max(\lgcd(S),\nu(F^S))$.
Then we modify the decision procedure just mentioned (for determining whether $E^\Tt_\alpha\neq\emptyset$, etc) by using $\nutilde(M^\Tt_\alpha|\gamma)$ and $\nutilde(M^\Uu_\alpha|\gamma)$
in place of $\nu(M^\Tt_\alpha|\gamma)$
and $\nu(M^\Uu_\alpha|\gamma)$.
We proceed in such a manner
in all comparisons in this paper (whether mentioned explicitly or not).

One significant exception to the MS theory readily adapting
is that Steel's proof of Dodd-solidity does not immediately
generalize to the superstrong level. Zeman  proved
the analogue for Jensen-indexed mice (at the superstrong level)
in \cite{zeman_dodd}.
We establish the MS version in Theorem \ref{thm:super-Dodd-soundness}.

There is also one small tweak required in proofs of iterability of  certain phalanxes and/or bicephali involved in the proofs of solidity, Dodd-solidity, etc. We discuss this in the proof of Claim \ref{clm:B_is_almost_iterable} in the proof of Theorem \ref{thm:solidity},
and (toward the end of) the proof of Claim \ref{clm:Ds_phU_iterable} in the proof of Lemma \ref{lem:super-Dodd-soundness}.\footnote{The same issue arises in the proof of condensation. This was unfortunately overlooked by the author when writing the proof of \cite[Theorem 5.2]{premouse_inheriting}.
But it is easily corrected, as discussed in the last paragraph of the proof of Claim \ref{clm:B_is_almost_iterable} of  Theorem \ref{thm:solidity}.}

A premouse $M$ is \emph{superstrong-small}
if there is no $\alpha\leq\OR^M$ with $F^{M|\alpha}$ superstrong.
\end{rem}

\begin{lem}
The fine structure
  theory of \cite[\S4]{fsit} goes through
routinely \tu{(}allowing for superstrong extenders\tu{)}.
Likewise \cite[6.1.5]{fsit}: If $\Tt$ is a $k$-maximal tree
and  $\alpha+1<\lh(\Tt)$, then
 $E$ is close to $M^{*\Tt}_{\alpha+1}$.
\end{lem}
\begin{proof}
We leave the verification to the reader; the presence
 of superstrong extenders has no substantial impact. (But one should
 bear  Remark \ref{rem:superstrong_diff} in mind.)
\end{proof}

This paper relies heavily on a result from \cite{extmax}, the proof of which is
related to the
calculations used in the proof of \emph{Strong Uniqueness} in \cite[\S
6]{fsit}. We now explain this; the definition
%conf
below follows \cite[Definition 2.19]{extmax}, and is a variant of $(\rho_{m+1}^M,p_{m+1}^M)$:

\begin{dfn}\label{dfn:z,zeta}
 Let $M$ be an $m$-sound premouse and $x\in\core_0(M)$. Then
$(z_{m+1}^M(x),\zeta_{m+1}^M(x))$
 denotes the least $(z,\zeta)\in\widetilde{\OR}$
 (see \S\ref{sec:notation_general}) such that $\zeta\geq\om$ and
 \[ \Th_{\rSigma_{m+1}}^M(\zeta\cup\{z,\pvec_m^M,x\})\notin M.\]
 And $(z_{m+1}^M,\zeta_{m+1}^M)$ denotes
$(z_{m+1}^M(\emptyset),\zeta_{m+1}(\emptyset))$.
\end{dfn}
The relationship between $(z_{m+1}^M,\zeta_{m+1}^M)$, $(p_{m+1},\rho_{m+1}^M)$ and $(m+1)$-solidity is made clear in the following fact, which is an easy exercise
%conf
(or see \cite[Remark 2.20]{extmax}):

\begin{fact}\label{fact:z_zeta}
 Let $M$ be as above and $(z,\zeta)=(z_{m+1}^M,\zeta_{m+1}^M)$
 and $(p,\rho)=(p_{m+1}^M,\rho_{m+1}^M)$.
Then:
\begin{enumerate}[label=--]
 \item $(z,\zeta)\leq(p,\rho)$ and $\rho\leq\zeta$ and $z\leq p$,
 \item $M$ is $(m+1)$-solid\footnote{Recall here
 that, as in \cite{imlc} but not \cite{fsit},
 we do not incorporate $(m+1)$-universality
 into $(m+1)$-solidity.}$\iff (z,\zeta)=(p,\rho)\iff
z=p\iff \zeta=\rho$.
\item $M$ is non-$(m+1)$-solid $\iff [\rho<\zeta\text{ and
}z<p]\iff\rho<\zeta\iff z<p$.
\end{enumerate}
\end{fact}

The following definition
slightly generalizes the picture
of a sequence of extenders being applied along the branch of an iteration tree. In particular,
when iterating bicephali in  the proof of Lemma \ref{lem:finite_gen_hull} (toward the proof of  \ref{thm:finite_gen_hull},
projectum-finite generation),
we will have such situations in which it seems we can't assume that the extenders are close to the model to which they apply.

\begin{dfn}\label{dfn:abstract_iteration}
Let $N$ be a $k$-sound premouse. A \dfnemph{degree $k$ abstract
weakly amenable
iteration of $N$} is a pair
$\left(\left<N_\alpha\right>_{\alpha\leq\lambda},\left<E_\alpha\right>_{
\alpha<\lambda}\right)$
such that $N_0=N$,
for all $\alpha<\lambda$,
$N_\alpha$ is a $k$-sound premouse,
$E_\alpha$ is a short extender weakly amenable to $N_\alpha$,
$\crit(E_\alpha)<\rho_k^{N_\alpha}$,
$N_{\alpha+1}=\Ult_k(N_\alpha,E_\alpha)$,
and for all limits $\eta\leq\lambda$,
$N_\eta$ is the resulting direct limit.
We say the iteration is \dfnemph{wellfounded}
if $N_\lambda$ is also wellfounded (in which case $N_\lambda$ is a
$k$-sound premouse).
\end{dfn}

Lemma \ref{lem:amenability_pres}
and Fact \ref{fact:z,zeta_pres} below are proved
%confirmed number
essentially as is \cite[Lemma 2.21]{extmax}
(the presence of superstrong extenders has no relevance to the proof). Together with Lemma \ref{lem:basic_fs_pres}, they will help us  analyse the effect of  iteration maps on critical objects (such as standard parameters and projecta)
in comparison arguments such as the proof of solidity, Dodd-solidity, and projectum-finite-generation.

\begin{lem}\label{lem:amenability_pres}
 Let $M$ be a $k$-sound premouse and $\eta<\rho_k^M$
 be a limit ordinal. Let $S\sub\eta^{<\om}$ be $M$-amenable;
 that is, for each $\alpha<\eta$, we have $S\cap\alpha^{<\om}\in M$.
 Suppose $S\notin M$.
 Let $E$ be a short extender, weakly amenable to $M$.
 Let $U=\Ult_k(M,E)$ and $i=i^{M,k}_E$. Suppose $U$ is wellfounded.
 Let $S^U=\bigcup_{\alpha<\eta}i(S\cap\alpha^{<\om})$.
 Then $S^U\notin U$.

 Likewise, if
$\left(\left<M_\alpha\right>_{\alpha\leq\lambda},\left<E_\alpha\right>_{
\alpha<\lambda}\right)$ is a wellfounded degree $k$ abstract weakly amenable
iteration of $M_0=M$, then
$S^{M_\lambda}=\left(\bigcup_{\alpha<\eta}j_{0\lambda}(S\cap\alpha^{<\om}
)\right)\notin M_\lambda$.
\end{lem}

\begin{fact}[$(z,\zeta)$-preservation]\label{fact:z,zeta_pres}
 Let $N$ be $k$-sound and
$\left(\left<N_\alpha\right>_{\alpha\leq\lambda},\left<E_\alpha\right>_{
\alpha<\lambda}\right)$ a wellfounded degree $k$ abstract weakly amenable
iteration
of $N$. Let $j:N\to N_\lambda$ be the final iteration map
and $x\in\core_0(N)$.
Then
\[ z^{N'}_{k+1}=j(z^N_{k+1})\text{ and }\zeta^{N'}_{k+1}=\sup
j``\zeta^N_{k+1} \]
and  generally,
$z^{N'}_{k+1}(j(x))=j(z^N_{k+1}(x))\text{ and }\zeta^{N'}_{k+1}(j(x))=
\sup j``(\zeta_{k+1}^N(x))$.
\end{fact}

\begin{lem}[Preservation of fine structure]\label{lem:basic_fs_pres}
Let $N$ be a $k$-sound premouse\footnote{Note that we do not assume
any iterability, nor that $N$ is $(k+1)$-solid
or $(k+1)$-universal.}, $\rho=\rho_{k+1}^N$ and $p=p_{k+1}^N$. Let $E$ be a short extender weakly amenable to
$N$, with $\kappa=\crit(E)<\rho$.
Let $N'=\Ult_k(N,E)$ and $j=i^{N,k}_E$.
Suppose that $N'$ is wellfounded.
Let $x\in\core_0(N)$.
Let $z'_{j(x)}=j(z^N_{k+1}(x))$, $\zeta'_{j(x)}=\sup j``\zeta_{k+1}^N(x)$,
 $p'=j(p)$ and $\rho'=\sup j``\rho^N_{k+1}$.
Then:
\begin{enumerate}[label=\arabic*.,ref=\arabic*]
 \item\label{item:j_is_k-embedding}
$N'$ is $k$-sound and $j$ is a $k$-embedding.
\item\label{item:theory_missing} $\Th_{\rSigma_{k+1}}^{N'}(\rho'\un p')\notin
N'$.
\item\label{item:if_rho_k+1=rho'}  $\rho_{k+1}^{N'}\leq\rho'$ and if
$\rho_{k+1}^{N'}=\rho'$ then
$p_{k+1}^{N'}\leq p'$.
\item\label{item:kappa<rho_k+1_or_E_close} If $\kappa<\rho$ or $E$ is
close to $N$ then $\rho_{k+1}^{N'}=\rho'$ \tu{(}and so $p_{k+1}^{N'}\leq
p'$\tu{)}.
\item\label{item:characterize_N'_solid} $N'$ is $(k+1)$-solid
iff
\tu{[}$N$ is $(k+1)$-solid and $\rho_{k+1}^{N'}=\rho'$\tu{]}.
\item\label{item:p_pres_if_N'_solid} If $N'$ is $(k+1)$-solid then
$p_{k+1}^{N'}=p'$.
\item\label{item:characterize_N'_sound} $N'$ is $(k+1)$-sound
iff \tu{[}$N$ is $(k+1)$-sound and
$\kappa<\rho$\tu{]}.
\end{enumerate}
\end{lem}
\begin{proof}
Part \ref{item:j_is_k-embedding} is standard.
Part \ref{item:theory_missing} follows from Fact \ref{fact:z,zeta_pres}
applied with $x=p$;
and part \ref{item:if_rho_k+1=rho'} is an immediate consequence of part
\ref{item:theory_missing}.

Part \ref{item:kappa<rho_k+1_or_E_close}:
If $\kappa\geq\rho$ and $E$ is close to $N$,
the conclusion is standard. So suppose $\kappa<\rho$.
Let $\alpha<\rho'$ and $x\in
N'$. Let $\beta\in[\kappa,\rho)$ and $y\in
N$ be such that $i_E(\beta)\geq\alpha$ and $x=[a,f]^{N,k}_E$ for some function $f$
which is
$\rSigma_{k}^N(\{y\})$ (let $y=f$ if $k=0$). Let $t=\Th_{\rSigma_{k+1}}^N(\beta\un\{y\})$.
Then because $a\in[i_E(\kappa)]^{<\om}$ and using the
proof that generalized witnesses compute
%conf
solidity witnesses (see \cite[\S1.12 + p.~326]{imlc}) one can
show that $i_E(t)$ computes $t'=\Th_{\rSigma_{k+1}}^{N'}(\alpha\un\{x\})$,
so $t'\in N'$.

Part \ref{item:characterize_N'_solid}: Suppose first that
$N$ is $(k+1)$-solid and $\rho_{k+1}^{N'}=\rho'$.
So $p_{k+1}^{N'}\leq p'$ by part \ref{item:if_rho_k+1=rho'}.
But $(N',p')$ is $(k+1)$-solid, which implies
that $p_{k+1}^{N'}\geq p'$, so $p_{k+1}^{N'}=p'$,
so $N'$ is $(k+1)$-solid.

Now suppose that either $N$ is non-$(k+1)$-solid
or $\rho_{k+1}^{N'}<\rho'$.
Let $(z,\zeta)=(z_{k+1}^N,\zeta_{k+1}^N)$
and
$(z',\zeta')=(j(z),\sup j``\zeta)=(z_{k+1}^{N'},\zeta_{k+1}^{N'})$
(the second equality by Fact \ref{fact:z,zeta_pres}).
Then by Fact \ref{fact:z_zeta},
if $N$ is non-$(k+1)$-solid
 then $\rho<\zeta$, so $\rho_{k+1}^{N'}\leq\rho'<\zeta'$.
On the other hand, if $\rho_{k+1}^{N'}<\rho'$, then
since $\rho\leq\zeta$, we have
$\rho_{k+1}^{N'}<\rho'\leq\zeta'$.
So in either case, $\rho_{k+1}^{N'}<\zeta'=\zeta_{k+1}^{N'}$,
so again by Fact \ref{fact:z_zeta}, $N'$ is non-$(k+1)$-solid.

Part \ref{item:p_pres_if_N'_solid}: Suppose $N'$ is $(k+1)$-solid. Then by part \ref{item:characterize_N'_solid}, $N$ is also $(k+1)$-solid. So by Fact \ref{fact:z_zeta}, $p=z_{k+1}^N$
and
$p_{k+1}^{N'}=z_{k+1}^{N'}$.
But by Fact \ref{fact:z,zeta_pres}, $j(z_{k+1}^N)=z_{k+1}^{N'}$, so $p'=j(p)=p_{k+1}^{N'}$, as desired.

Part \ref{item:characterize_N'_sound}:
If $N$ is $(k+1)$-sound
and $\kappa<\rho$, then by the previous parts,
$\rho'=\rho_{k+1}^{N'}$ and $p'=p_{k+1}^{N'}$.
But $j(\kappa)<\rho'$, so
\[ N'=\Hull_{k+1}^{N'}(j(\kappa)\cup\rg(j))=\Hull_{k+1}^{N'}(\rho'\cup\{p'\}),\]
so $N'$ is sound.\footnote{This direction
(that $N$ being $(k+1)$-sound and $\kappa<\rho$ implies $N'$ is $(k+1)$-sound)
can also be proved more
directly, as opposed
to going through the previous parts, and was observed
earlier by Steve Jackson and the author.}

Conversely, suppose first that $N$ is non-$(k+1)$-sound
but $\kappa<\rho$. We may assume that $N'$ is $(k+1)$-solid,
so by the previous parts, $N$ is $(k+1)$-solid,
 $\rho'=\rho_{k+1}^{N'}$ and $p'=p_{k+1}^{N'}$;
 moreover,
there is some $x\in\core_0(N)$ such that
\[ x\notin\Hull_{k+1}^N(\rho\cup\{p,\pvec_k^N\})\]
But this statement is $\rPi_{k+1}$ in these parameters
(note that $\rho<\rho_0^N$, so $\rho\in\dom(j)$, as $N$ is non-$(k+1)$-sound), so
\[ j(x)\notin\Hull_{k+1}^{N'}(j(\rho)\cup\{p_{k+1}^{N'},\pvec_k^{N'}\}).\]
Since $\rho_{k+1}^{N'}=\rho'\leq j(\rho)$, it follows that $N'$ is
non-$(k+1)$-sound.

Finally suppose that $\rho\leq\kappa$. We may assume  $N'$ is $(k+1)$-solid,
so $\rho'=\rho=\rho_{k+1}^{N'}$ and $p'=p_{k+1}^{N'}$.
But then $\kappa\notin\Hull_{k+1}^{N'}(\rho'\cup\{p',\pvec_k^{N'}\})$,
so $N'$ is non-$(k+1)$-sound.
\end{proof}
The preceding lemma extends easily to  abstract weakly amenable iterations:
\begin{cor}\label{cor:basic_fs_pres}
Let $k<\om$ and $N$ be a $k$-sound premouse. Let
$\left(\left<N_\alpha\right>_{\alpha\leq\lambda},\left<E_\alpha\right>_{
\alpha<\lambda}\right)$ be a wellfounded degree $k$ abstract weakly amenable
iteration
on $N$, with iteration maps
$j_{\alpha\beta}:N_\alpha\to N_\beta$.
Then the obvious
generalizations of the conclusions of Lemma \ref{lem:basic_fs_pres}  hold
for the maps $j_{\alpha\beta}:N_\alpha\to N_\beta$,
where the hypothesis of part \ref{item:kappa<rho_k+1_or_E_close}
is modified to say ``if for each
$\alpha<\lambda$, either $\crit(E_\alpha)<\rho_{k+1}(N_\alpha)$
or $E_\alpha$ is close to $N_\alpha$'',
and the characterization of soundness
in part \ref{item:characterize_N'_sound} modified to say
``$N$ is $(k+1)$-sound and $\crit(E_\alpha)<\rho_{k+1}(N_\alpha)$
for each $\alpha<\lambda$''.
\end{cor}

\begin{lem}\label{lem:iteration_p_rho_pres}
 Let $N$ be a $k$-sound premouse and $\Tt$ be a terminally non-dropping
$k$-maximal tree on
$N$.
 Let $N'=M^\Tt_\infty$. Then $\rho_{k+1}^{N'}=\sup i^\Tt``\rho_{k+1}^N$
 and $p_{k+1}^{N'}\leq i^\Tt(p_{k+1}^N)$.
 Moreover, let $q\ins p_{k+1}^N$ be such that $q$ is
$(k+1)$-solid for $N$.
 Then $i^\Tt(q)\ins p_{k+1}^{N'}$, and $q=p_{k+1}^N$ iff
$i^\Tt(q)=p_{k+1}^{N'}$.
\end{lem}
\begin{proof}
We just consider the case $N'=\Ult_k(N,E)$ where $E$ is a short extender which is close to $N$ with $\kappa=\crit(E)<\rho_k^N$; the
generalization to the
full lemma is then easy.

If $\kappa<\rho_{k+1}^N$
then the fact that $\rho_{k+1}^{N'}=\sup i_E^{N,k}``\rho_{k+1}^N$
and $p_{k+1}^{N'}\leq i_E^{N,k}(p_{k+1}^N)$ follows from
\ref{lem:basic_fs_pres}.
If $\rho_{k+1}^N\leq\kappa$,
use the closeness of $E$ to $N$  and
\cite[\S4]{fsit} as usual.
The ``moreover'' clause follows routinely, using generalized witnesses.
\end{proof}

Finally, the following well-known notion will be relevant in \S\ref{sec:non-solid}, and also in the proof of solidity in \S\ref{sec:solidity}:
\begin{dfn}\label{dfn:hull_prop}
For a $n$-sound premouse $N$
and
$\rho_{n+1}^M\leq\rho<\rho_n^M$,
we say that $N$ has the \dfnemph{$(n+1,x)$-hull
property at $\rho$} iff
$\pow(\rho)^N\sub\cHull_{n+1}^N(\rho\un\{x,\pvec_n^N\})$
(note the hull is collapsed).
\end{dfn}

\section{Prologue: A non-solid premouse}\label{sec:non-solid}

Before really embarking on our endeavour of proving
that iterable premice are (among other things) solid,
it is maybe motivating to see that the goal cannot be
overly trivial, by finding examples of non-solid premice.
We do that in this section, assuming there is a measurable
cardinal.
The results here are not needed
in later sections, however. The basic motivation for
the construction is Lemma \ref{lem:basic_fs_pres}:
If we can a find $k$-sound premouse $M$
and an extender weakly amenable to $M$
such that $\rho_{k+1}^M\leq\crit(E)<\rho_k^M$ and $U=\Ult_k(M,E)$ is wellfounded
and
$\rho_{k+1}^U<\rho_{k+1}^M$,
 then $U$ is non-$(k+1)$-solid. In order to achieve
this,
$E$ must at least be non-close to $M$ (though it seems this might
not be enough by itself), and one might encode new information into some
component measure of $E$,
and this information can then be present in $U$. Actually we don't need
to work hard to encode new information; starting with some useful
fine structural circumstances, things essentially take care of themselves.

\begin{tm}\label{tm:non-solid}
If there is a measurable cardinal then there is a non-solid premouse.
\end{tm}
\begin{proof}
 Suppose there is a measurable cardinal.

\begin{clm}\label{clm:type_1_mouse_with_bounded_1-hulls} There is a $(0,\om_1+1)$-iterable $1$-sound
countable type 1  premouse $M$ such that $\rho_1^M=\om_1^M$,
$p_1^M=\emptyset$, and $\Hull_1^M(\{x\})$ is bounded
 in $M$, for each $x\in M$.
\end{clm}
\begin{proof}
Let $L[D]$ be Kunen's model for one measurable cardinal at $\kappa$.
So $D\in L[D]$ and $L[D]\sats$ ``$D$ is the unique normal measure on $\kappa$''.
Let $P$ be the rearrangement of $L[D]$
as a premouse.
Let $E\in\es^P$ be the extender given by $D$,
and let $N=P|\lh(E)$. Then
 $N$  is $1$-sound and $(0,\om_1+1)$-iterable,
with $\rho_1^N=\kappa^{+N}=\kappa^{+L[D]}$.

Given
$\alpha<\kappa^{+N}$, let $H_\alpha=\Hull_1^N(\alpha)$
(recall that this denotes the uncollapsed hull). Then  $H_\alpha\cap\kappa^{+N}$
is bounded in $\kappa^{+N}$,
because otherwise $\rho_1^N\leq\kappa$.
Since $N$ is type 1, we have the usual $\Sigma_1^N$ cofinal function $f:\kappa^{+N}\to\OR^N$. Therefore $H_\alpha\cap\OR^N$
is bounded in $\OR^N$, so
 $H_\alpha\in N$ and $N_\alpha=\cHull_1^N(\alpha)\in N$
and $\pi_\alpha\in N$ where $\pi_\alpha:N_\alpha\to H_\alpha$
is the uncollapse map. Note $N_\alpha$ is a premouse
and $\pi_\alpha$ is $\rSigma_1$-elementary,
and $N_\alpha$ is $\alpha$-sound.

Now say $\alpha$ is \emph{good} iff $\alpha<\om_1^N$
and $\alpha=\crit(\pi_\alpha)$, so $\alpha=\om_1^{N_\alpha}$. Let $\alpha$ be good.
Easily by condensation,  $N_\alpha\pins N$,
so $N_\alpha\ins Q$ where $Q$ is least with $\alpha=\om_1^{N_\alpha}\leq\OR^Q$
and $\rho_\om^Q=\om$. Therefore if $\alpha<\beta$ are both good,
then $N_\alpha\pins N_\beta$ and $N_\alpha\in H_\beta$,
and also
\[
\sup(H_\alpha\cap\OR)<\sup(H_\beta\cap\OR)\]
(because $\Th_{\rSigma_1}^{N}(\alpha)=\Th_{\rSigma_1}^{N_\alpha}(\alpha)$,
so the parameter $N_\alpha\in H_\beta$ can be used to define the object
$H_\alpha$ in a $\Sigma_1$ fashion).
Note that the good ordinals form a club $C\sub\om_1^N$,
and $C\in N$. Let $\eta$ be the supremum of the first $\om$-many
good ordinals, so $\eta$ is also good.

Let $M=N_\eta$. We claim that $M$ is as desired.
Let us verify that $\rho_1^{N_\eta}=\eta=\om_1^{N_\eta}$.
 Clearly $\rho_1^{N_\eta}\leq\eta$, so we show
 $\om_1^{N_\eta}\leq\rho_1^{N_\eta}$.
Let $p\in[\OR]^{N_\eta}$. Since
$H_\eta=\bigcup_{\alpha<\eta}H_\alpha$,
we can fix $\alpha<\eta$ with $\pi_\eta(p)\in H_\alpha$.
But then
\[ t\eqdef\Th_1^{N_\eta}(\{p\})=\Th_1^{H_\eta}(\{\pi_\eta(p)\})=\Th_1^{H_\alpha}(\{\pi_\eta(p)\})=\Th_1^{N_\alpha}(\{\pi_\alpha^{-1}(\pi_\eta(p))\}),\]
and since
$N_\alpha\pins N|\eta$, therefore $t\in N_\eta$, which suffices. It easily follows that $p_1^{N_\eta}=\emptyset$. And $\Hull_1^{N_\eta}(\{x\})$ is bounded in $N_\eta$ for each $x\in N_\eta$, by the
 properties of good $\alpha,\beta<\eta$, which reflect into
$M$. Note that $M=\Hull_1^M(C\cap\eta)$, so $M$ is countable. The rest is clear.
\end{proof}

For the rest of the construction, we just need a mouse $M$
as in Claim \ref{clm:type_1_mouse_with_bounded_1-hulls}. So fix such an $M$.\footnote{If
$M$ is beyond $0^\pistol$, then the proof
to follow will depend on Theorem \ref{thm:finite_gen_hull}.}
Let $U=\Ult_0(M,F^M)$ and $i:M\to U$ the ultrapower map. Then by standard
preservation of fine structure,
$\rho_1^U=\rho_1^M=\om_1^M=\om_1^U$.
Recall the notation $M\wr\xi$ from \S\ref{sec:notation}.

\begin{clm}\label{clm:unbounded_hull_of_M}  $\om_1^M\sub\Hull_1^M(X)$ for any
unbounded $X\sub\OR^M$,
and $\om_1^M\sub\Hull_1^{U}(X)$
for any unbounded $X\sub\OR^{U}$.
\end{clm}
\begin{proof}
For $\xi<\OR^M$, let $\eta_\xi$ be the least $\eta>\xi$ such that
\[
\Th_{\rSigma_1}^{M\wr\eta}(\om_1^M)\neq\Th_{\rSigma_1}^{M\wr\xi}(\om_1^M),\]
and let $\gamma_\xi$ be the least $\gamma<\om_1^M$ such that
\[
\Th_{\rSigma_1}^{M\wr\eta_\xi}(\gamma+1)\neq\Th_{\rSigma_1}^{M\wr\xi}(\gamma+1).
\]
Now let $X\sub\OR^M$ be unbounded.
Let $\gamma<\om_1^M$.
Then there is $\xi\in X$
such that $\Hull_1^M(\gamma)$ is bounded in $\xi$,
so
$\Th_{\rSigma_1}^M(\gamma)=\Th_{\rSigma_1}^{M\wr\xi}(\gamma)$.
But then note that $\gamma_\xi\geq\gamma$,
and $\gamma_\xi\in\Hull_1^M(X)$,
but then since $\gamma_\xi<\om_1^M$, we have $\gamma_\xi+1\sub\Hull_1^M(X)$.

For $U$ it is likewise, as $i:M\to U$ is
$\Sigma_1$-elementary and cofinal.
\end{proof}

Let $\mu=\crit(F^M)$.
Since $U=\Ult_0(M,F^M)$, we have
$U=\Hull_1^U(\rg(i)\cup\{\mu\})$,
and  $\rg(i)=\Hull_1^U(\om_1^U)$.
It follows that  $H'=\Hull_1^U(\{\mu,x\})$ is bounded in $U$ for every $x\in U$;
for otherwise, Claim \ref{clm:unbounded_hull_of_M} implies that $\om_1^U\sub H'$,
and hence $\rg(i)\sub H'$, and hence $U=H'$,
but then $\rho_1^U=\om$, a contradiction.

Recall that $M^\passive$ was defined in \S\ref{sec:premice_and_phalanxes}.
We will now find another embedding $j:M^\passive\to U^\passive$
with $\crit(j)=\mu$,
use it to lift $F^M$ to $F'$ over $U^\passive$,
defining a premouse $U'=(U^\passive,F')$,
and deduce from this that $K=\Hull^{U'}(\{\mu\})$
is unbounded in $U'$ and $U'$ is non-solid.
It seems to be the fusion of the (canonical) information
in $F^M$ together with the (non-canonical) information
in $j$ which  leads to the failure of solidity.

\begin{clm}\label{clm:generic_embedding_j}There is $j:M^\passive\to U^\passive$ such that $j$
 is
cofinal and
elementary
 (in the language of passive premice),
 $\crit(j)=\mu$,  $j$ is continuous at $\mu^{+M}$,
$U$ has universe $\{j(f)(\mu)\bigm|f\in M\}$, and $j\neq i$.\end{clm}
\begin{proof}Note that $i$ has all desired properties, except that
 $i=i$.
So let $G$ be  $(L[U^\passive],\PP)$-generic where $\PP=\Coll(\om,\OR^{U})$.
(We get such  a $G\in V$.
For let $M_1=U$ and $M_2=\Ult_0(M_1,F^{M_1})$.
Then $M_2$ is countable and $L[U^\passive]\sub L[(M_2)^{\passive}]$ and $L[(M_2)^{\passive}]$
has no segment projecting strictly across $\OR^{M_2}$ (consider the length $\OR$ iteration of $M$ using $F^M$ and its images). So $\pow(\OR^U)\cap L[U^\passive]\sub M_2$, and so is countable.)
Fix a surjection $g:\om\to U^\passive$ with $g\in L[U^\passive,G]$.
Let $T=T_g$ be the natural tree of attempts
to build a $j$ with the desired properties, excluding
that $j\neq i$, with cofinality/surjectivity
requirements
arranged by using $g$ for bookkeeping in the natural way.
Clearly $T$ is illfounded, since $i$
is such an embedding. But we have $T\in L[U^\passive,G]$,
whereas $i\notin L[U^\passive,G]$, so we get (many)  such $j\neq i$
with $j\in L[U^\passive,G]$.
\end{proof}

Fix any $j$ as in Claim \ref{clm:generic_embedding_j} (it need not be generic over $L[U^\passive]$).
Let $E=E_j$. Note that $U^\passive=\Ult_0(M^\passive,E)$
and $E$ is generated by $\{\mu\}$.
Let $U'=\Ult_0(M,E)$, i.e.~including the  active extender.
Then $j:M\to U'$ is cofinal $\Sigma_1$-elementary,
and  $U'$ is a premouse. We have (and let)
\begin{equation}\label{eqn:H_=_rg(j)} H=\rg(j)=\Hull_1^{U'}(\om_1^M)=\Hull_1^{U'}(\mu),\end{equation}
and  $M$ is the transitive collapse of $H$.
(Note that if $j\in L[U^\passive,G]$ where $G$
is set-generic over $L[U^\passive]$, then $U'\notin L[U^\passive,G]$,
because from $j$ and $U'$ one can recover $M$,
but
 $M\notin L[U^\passive,G]$.)

 \begin{clm}\label{clm:unbounded_hull_of_U'}
$\om_1^M=\om_1^{U'}\sub\Hull_1^{U'}(X)$
for any unbounded $X\sub\OR^{U'}$.
\end{clm}
\begin{proof}
 Like for $U$ in the proof of Claim \ref{clm:unbounded_hull_of_M}.
\end{proof}

Let $K=\Hull_1^{U'}(\{\mu\})$. The plan is to show
that $K$ is unbounded in $U'$, and hence (by Claim \ref{clm:unbounded_hull_of_U'}, line (\ref{eqn:H_=_rg(j)}) and choice of $j$ (i.e., as in Claim \ref{clm:generic_embedding_j}))
$K=U'$, so $\rho_1^{U'}=\om<\rho_1^M$,
and so by Lemma \ref{lem:basic_fs_pres}, $U'$ cannot be $1$-solid (actually
we will establish the failure of solidity  without
appeal to \ref{lem:basic_fs_pres}).

Let $\bar{U}$ be the transitive collapse of $K$ and $\pi:\bar{U}\to U'$
the uncollapse.
Write $\pi(\bar{\mu})=\mu$, $\pi(\bar{\lambda})=\lambda=\crit(F^{U'})$,
etc. Let $\bar{H}=\Hull_1^{\bar{U}}(\om_1^{\bar{U}})$
and $\bar{M}$ be its transitive collapse
and $\bar{j}:\bar{M}\to\bar{U}$ the uncollapse.

Recall the \emph{hull property} from Definition \ref{dfn:hull_prop}.

\begin{clm}\label{clm:p_1_leq_kappa}
We have:
\begin{enumerate}[label=\arabic*.,ref=\arabic*]
\item\label{item:Ubar_is_gend} $\bar{U}$ is a premouse and
$\bar{U}=\Hull_1^{\bar{U}}(\{\bar{\mu}\})$,
so $\rho_1^{\bar{U}}=\om$ and $p_1^{\bar{U}}\leq\{\bar{\mu}\}$.
\item\label{item:Hull_om_1=Hull_kappabar}
$\bar{H}=\Hull_1^{\bar{U}}(\om_1^{\bar{U}})=\Hull_1^{\bar{U}}(\bar{\mu})$.
\item $\bar{\lambda}$ is the least ordinal in
$\bar{H}\cut\bar{\mu}$, and therefore $\crit(\bar{j})=\bar{\mu}$ and
$\bar{j}(\bar{\mu})=\bar{\lambda}$.\label{item:crit(jbar)}
\item\label{item:hull_property_for_Ubar}
$\bar{U}$ has the $(1,\emptyset)$-hull
property at $\bar{\mu}$,
so $\bar{M}||(\bar{\mu}^+)^{\bar{M}}=\bar{U}||(\bar{\mu}^+)^{\bar{U}}$,
so
\[
\bar{M}=\cHull_1^{\bar{U}}(\om_1^{\bar{U}})=
\cHull_1^{\bar{U}}(\bar{\mu})\notin\bar{U}.\]
\end{enumerate}
\end{clm}
\begin{proof}
Part \ref{item:Ubar_is_gend}: This is routine.

Parts \ref{item:Hull_om_1=Hull_kappabar}, \ref{item:crit(jbar)}: These follow
easily from the analogous properties of $H,U'$ and $\Sigma_1$-elementarity.

Part \ref{item:hull_property_for_Ubar}: As $U'$ has the
$(1,\emptyset)$-hull property at $\mu$, this also reflects.
 (Given $A\sub\mu$
 with $A\in\rg(\pi)$, $U'\sats$ ``There is $\alpha<\om_1$
 and $A'\sub\lambda=\crit(\dot{F})$ such that $A'\in\Hull_1(\{\alpha\})$
 and $A'\cap\mu=A$''. So $\bar{U}$ satisfies this regarding
$\bar{A}=\pi^{-1}(A)$.
 This gives the $(1,\emptyset)$-hull property.)
\end{proof}

\begin{clm}\label{clm:D_j_not_on_M-seq}
 The normal measure $D_j$ derived from $j$ is not that of any $M$-total
$F\in\es_+^M$.
\end{clm}
\begin{proof}
We already know  $D_j\neq F^M$.
But if $F\in\es^M$ is $M$-total type 1 then $F\neq D_j$ since
\[ \Ult_0(M^\passive,F)\neq\Ult_0(M^\passive,F^M)=U^\passive=\Ult_0(M^\passive,
D_j).\qedhere\]
\end{proof}

\begin{clm}\label{clm:D_jbar_not_on_Mbar-seq}
The normal  measure $D_{\bar{j}}$ derived from $\bar{j}$
is not that of any $\bar{M}$-total  $\bar{F}\in\es_+^{\bar{M}}$.
\end{clm}
\begin{proof}
Let $\bar{F}\in\es_+^{\bar{M}}$ be $\bar{M}$-total type 1
with $\crit(\bar{F})=\bar{\mu}$.
Let $F=j^{-1}(\pi(\bar{j}(\bar{F})))$, or $F=F^M$
if $\bar{F}=F^{\bar{M}}$. Suppose for simplicity of notation that $\bar{F}\in\es^{\bar{M}}$ (so $\bar{F}\neq F^{\bar{M}}$); otherwise it is essentially the same.
So $F\in\es^M$ is $M$-total type 1 with $\crit(F)=\mu$.
By Claim \ref{clm:D_j_not_on_M-seq}, $E=E_j$ and $F$ have distinct derived normal
measures.
Fix $A\sub\mu$ with $A\in E_{\{\mu\}}\cut F_{\{\mu\}}$.
Then letting $D'=j(F_{\{\mu\}})\in\Hull_1^{U'}(\mu)$, we have
\[ U'\sats\exists A'\ [A'\sub\lambda\text{ and }A'\in\Hull_1^{U'}(\mu)\text{
and
}\mu\in A'\text{ but }A'\notin D']\]
(as witnessed by $A'=j(A)$). This is a $\Sigma_1$-assertion of the parameter
$(\mu,D',\lambda)$,
so reflects into $\bar{U}$ regarding $(\bar{\mu},\bar{D},\bar{\lambda})$
where $\pi(\bar{D})=D'$. But $\bar{j}(\bar{F}_{\{\bar{\mu}\}})=\bar{D}$,
and this shows that $\bar{F}_{\{\bar{\mu}\}}$ and $D_{\bar{j}}$
are distinct measures.\end{proof}

The following claim is the key thing to prove:

\begin{clm}\label{clm:H_is_cofinal}
$K$ is cofinal in $\OR^{U'}$, and hence $\bar{U}=K=U'$.
\end{clm}
\begin{proof}
 Suppose $\xi=\sup K\cap\OR^{U'}<\OR^{U'}$.
Let $\widetilde{U}=U'\wr\xi$. Then $\pi:\bar{U}\to \widetilde{U}$ is
$\Sigma_1$-elementary. But since $\widetilde{U}\in U'$
and $\pi$ is determined by $(\pi(\bar{\mu}),\widetilde{U})$,
therefore $\bar{U},\pi\in U'$ also. Also $\rho_1^{\bar{U}}=\om$
and in fact
$\bar{U}=\Hull_1^{\bar{U}}(\{\bar{\mu}\})$,
so by Theorem \ref{thm:finite_gen_hull}
and the elementarity of $j$, it follows that
$U'\sats$ ``$\bar{C}=\core_1(\bar{U})\pins U'|\om_1^{U'}$ and $\bar{U}$ is
an iterate of $\bar{C}$, via a $0$-maximal tree $\bar{\Tt}$ of finite length''.\footnote{\label{ftn:mice_below_0^pistol}If $M$ is below $0^\pistol$
then instead of  Theorem \ref{thm:finite_gen_hull},
we can just use the fact
that every mouse below $0^\pistol$ is an iterate of its core (see \S\ref{sec:iterates_of_cores}). Note that just because $\bar{U}=\Hull_1^{\bar{U}}(\{\bar{\mu}\})$,
$\bar{\Tt}$ must have finite length.}
That is, whenever we have
such a triple $(\widetilde{M},\pi^*,M^*)\in M$
(in particular, with $\pi^*:M^*\to\widetilde{M}=M\wr\zeta$ for some $\zeta$), then $M^*$ is an
iterate of $C^*=\core_1(M^*)\pins M|\om_1^M$
via a finite $0$-maximal tree
(by Theorem \ref{thm:finite_gen_hull} and because $M^*$
is $(0,\om+1)$-iterable, since we have $\pi^*$).
But the assertion that this happens for all such triples is $\rPi_1$,
so lifts to $U'$. (Note that the assertion
refers to $F^{U'}$, not just $\es^{U'}=\es^{U^\passive}=\es^U$.)

Since $\bar{C}$ is $1$-sound and $\bar{U}$ is an iterate thereof, $\bar{U}$ is $1$-solid and $1$-universal. By Claim \ref{clm:p_1_leq_kappa}
parts \ref{item:Ubar_is_gend} and \ref{item:hull_property_for_Ubar},
$p_1^{\bar{U}}\leq\{\bar{\mu}\}$ but
$\bar{M}=\cHull_1^{\bar{U}}(\om_1^{\bar{U}})\notin\bar{U}$,
so since $\rho_1^{\bar{U}}=\om$ and by $1$-solidity, either $p_1^{\bar{U}}=\emptyset$
or $p_1^{\bar{U}}=\{\alpha\}$ for some $\alpha<\om_1^{\bar{U}}$.
But in either case and by $1$-universality,
\[ \bar{C}=\core_1(\bar{U})=\cHull_1^{\bar{U}}(\om_1^{\bar{U}})=\bar{M},\]
and $\bar{j}$ is the core map, and $\bar{j}=i^{\bar{\Tt}}$ (where $\bar{\Tt}$ is the finite tree on $\bar{C}=\bar{M}$ with
 last model $\bar{U}$ mentioned above). So $\crit(i^{\bar{\Tt}})=\bar{\mu}$
and since $\bar{U}=\Hull_1^{\bar{U}}(\{\bar{\mu}\})$,
therefore $\bar{\Tt}$ uses only extender $\bar{E}$,
which is type 1, $\bar{E}\in\es_+(\bar{M})$, and $\crit(\bar{E})=\bar{\mu}$.
But we also have $\bar{j}=i^{\bar{\Tt}}$,
and this contradicts Claim \ref{clm:D_jbar_not_on_Mbar-seq}.

So $K$ is cofinal in $U'$,
so by Claim \ref{clm:unbounded_hull_of_U'},
$\om_1^{U'}=\om_1^M\sub K$, but $\mu\in K$,
so $K=U'$.
\end{proof}

\begin{clm} $\rho_1^{U'}=\om$ and $p_1^{U'}=\{\mu\}$
and $U'$ is non-$1$-solid.
\end{clm}
\begin{proof}
 By the previous claim,  $\rho_1^{U'}=\om<\om_1^M=\rho_1^M$.
Since $E$ is weakly amenable to $M$
 and by Lemma \ref{lem:basic_fs_pres} part \ref{item:characterize_N'_solid}, it follows that
  $U'$ is non-$1$-solid. But actually, we will argue
  directly, without using Lemma \ref{lem:basic_fs_pres}.
  It suffices to see $p_1^{U'}=\{\mu\}$,
  since $M=\cHull_1^{U'}(\mu)$ and $M\notin U'$.
For this, it suffices to see $p_1^{U'}\geq\{\mu\}$,
by Claim \ref{clm:p_1_leq_kappa} and since $U'=\bar{U}$.
So let $p\in[\mu]^{<\om}$.
Then  $A\eqdef\cHull_1^{U'}(\{p\})=\cHull_1^{M}(\{p\})$,
since $\crit(j)=\mu$, and since $\rho_1^M=\om_1^M$, $A\in M|\om_1^M\sub U'$,
so we are done.
 \end{proof}
 Since $U'$ is a non-solid premouse, this completes the proof.
\end{proof}

We finish this section with an observation
regarding uniqueness of
active extenders in $\lambda$-indexed premice,
which relates somewhat to the foregoing argument.

\begin{tm}\label{tm:uniqueness_of_premouse_extender}
Let $M$ be an active $\lambda$-indexed premouse
with $\cof(\OR^M)$ uncountable.
 Then $F^{M|\alpha}$ is the unique $E$ such that $(M||\OR^M,E)$
is a $\lambda$-indexed premouse.
\end{tm}
\begin{proof}
Let $N=M^\passive$. Let
 $M_0=(N,E_0)$ and $M_1=(N,E_1)$ be $\lambda$-indexed
 premice with $E_0\neq E_1$. We will derive a contradiction.
 Let $i_n=i^{N}_{E_n}$ and $\kappa_n=\crit(i_n)$ for $n=0,1$.
 We may assume that $\kappa_0\leq\kappa_1$.
 Let $\lambda$ be the largest cardinal of $N$.

 \begin{clmeight} Let
$\alpha\in[\lambda,\OR^N)\cap\rg(i_0)\cap\rg(i_1)$.
Let $i_n(\bar{\alpha}_n)=\alpha$ for $n=0,1$.
Then:
\begin{enumerate}[label=\arabic*.,ref=\arabic*]
\item $i_0``(\bar{\alpha}_0+1)\sub\rg(i_1)$
\item  if $\kappa_0=\kappa_1$ then
  $\bar{\alpha}_0=\bar{\alpha}_1$
and
$i_0\rest(\bar{\alpha}_0+1)=i_1\rest(\bar{\alpha}_1+1)$.
\end{enumerate}
\end{clmeight}
\begin{proof}
Let $\beta\geq\alpha$
be least with $\rho_\om^{N|\beta}=\lambda$.
By elementarity, $\beta\in\rg(i_0)\cap\rg(i_1)$.
Let $i_n(\bar{\beta}_n)=\beta$.
Then  $\kappa_n\leq\bar{\alpha}_n\leq\bar{\beta}_n<(\kappa_n^+)^N$ and
$\rho_\om^{N|\bar{\beta}_n}=\kappa_n$.
But then note that
\[\Th^{N|\bar{\beta}_0}(\kappa_0)=
\Th^{N|\beta}(\kappa_0)=\Th^{N|\bar{\beta}_1}(\kappa_0). \]
If $\kappa_0=\kappa=\kappa_1$, it follows that
$\bar{\beta}_0=\bar{\beta}_1$
and similarly, because $i_0\rest\kappa=\id=i_1\rest\kappa$,
it follows that $i_0\rest(N|\bar{\beta}_n)=i_1\rest(N|\bar{\beta}_n)$,
which suffices in this case. And if $\kappa_0<\kappa_1$,
it is similar, but we get
$i_0``(N|\bar{\beta}_0)\sub i_1``(N|\bar{\beta}_1)$.
\end{proof}

\begin{clmeight}
 $\rg(i_0)\cap\rg(i_1)$ is bounded in $N$.
\end{clmeight}
\begin{proof}
Suppose not. As $E_0\neq E_1$ and by the previous claim,
it follows that $\kappa_0<\kappa_1$
and so $\rg(i_0)\sub\rg(i_1)$.
But then letting $E$ be the $(\kappa_0,\kappa_1)$-extender
derived from $i_0$, note that $E$ is a whole segment
of $E_0$ with $\lambda(E)=\kappa_1$. So by the initial segment
condition for $\lambda$-indexed premice,
$E\in N$. But $E$ singularizes $(\kappa_1^+)^N$, a contradiction.
\end{proof}

Now
fix $\gamma<\OR^N$ above $\sup(\rg(i_0)\cap\rg(i_1))$.
So let $\alpha_0<(\kappa_0^+)^N$
be least with $i_0(\alpha_0)>\gamma$.
Let $\beta_0<(\kappa_1^+)^N$ be least
with $i_1(\beta_0)>i_0(\alpha_0)$.
Let $\alpha_1<(\kappa_0^+)^N$ be least
 with $i_0(\alpha_1)>i_1(\beta_0)$. Etc.
 Let $\alpha=\sup_{n<\om}\alpha_n$
 and $\beta=\sup_{n<\om}\beta_n$.

Since $\cof(\OR^N)=\cof(\kappa_n^+)^N>\om$,
we have $\alpha<(\kappa_0^+)^N$ and $\beta<(\kappa_1^+)^N$.
But then  $i_0\rest(\alpha+1)\in N$ and $i_1\rest(\beta+1)\in N$,
by weak
amenability. So the construction of $\left<\alpha_n,\beta_n\right>_{n<\om}$ can
be done inside
$N$, so $\cof^N(\alpha)=\om=\cof^N(\beta)$.
So $i_0$ is continuous
at $\alpha$ and $i_1$ continuous at $\beta$.
But then $i_0(\alpha)=i_1(\beta)\in\rg(i_0)\cap\rg(i_1)$,
a contradiction, completing the proof.
\end{proof}

Note the preceding
argument is reminiscent of the Zipper Lemma.
Given also the known methods of translating between extenders
and iteration strategies, there is in fact a significant connection there.
For the question of identifying the extender
sequence of $L[\es]$ when $V=L[\es]$ for a Mitchell-Steel
indexed premouse $L[\es]$ see \cite{extmax} and \cite{V=HODX_pub}.

\section{Standard bicephali}\label{sec:bicephali}

As sketched in \S\ref{sec:meas_Dodd_prof_fin_plan},
the proof of projectum-finite-generation (Theorem \ref{thm:finite_gen_hull},
via Lemma \ref{lem:finite_gen_hull})
 will use a
 comparison argument involving a
(generalized) bicephalus
of the form $B=(\rho,C,M)$,
where $M$ is $k$-sound and $(k+1)$-universal, $C=\core_{k+1}(M)$ is
$(k+1)$-solid and $\rho=\rho_{k+1}^M$. This bicephalus is therefore
like the exact bicephali of
%conf
\cite[Definition 3.1]{premouse_inheriting}, except
that $M$ need not be $\rho$-sound (and is not in the case of interest).
So the development in \cite{premouse_inheriting}
does not literally apply to $B$; however, it all adapts
in an obvious manner.\footnote{In fact, $M$ will be projectum-finitely-generated,
as in the statement of Theorem \ref{thm:finite_gen_hull}.
So we could augment the premouse language
with a new constant symbol $\dot{x}$, and interpret $\dot{x}$ as the least
$x\in[\rho_k^M]^{<\om}$ such that $M=\Hull_{k+1}^M(\rho\cup\{x\})$.
Then $M$ is  $\rho$-sound
(in fact $(k+1)$-sound) in the expanded language (with $p_{k+1}=\emptyset$ in
the new language),
and the iteration maps $i$ which act on bicephali will all easily
preserve the interpretation of the symbol (when shifting $\rho$
to $\sup i``\rho=\rho_{k+1}^{M'}$, where $M'$ is the iterate). Under this
interpretation, the theory of \cite{premouse_inheriting} goes through
regarding the segments of iteration trees which shift bicephali.
However, one obviously needs to adapt things for the parts
which shift a non-dropping iterate $M'$ of $M$ above $\rho_{k+1}^{M'}$.}
Iteration trees on $B$ will also be formed just as in
\cite{premouse_inheriting}.

For some self-containment, we give the basic definitions in this section, but
omit various calculations which appear in \cite{premouse_inheriting}.
In \S\S\ref{sec:simple_embeddings}--\ref{sec:capturing_elements},
we will introduce and deal with
with with embeddings between iteration trees on
premice and on bicephali,
finite support for such trees,
 and certain decompositions of iteration maps.
These will be important
in the main arguments in \S\S\ref{sec:mim},\ref{sec:Dodd_proof},\ref{sec:finite_gen_hull}, and for parts of this material (in \S\S\ref{sec:simple_embeddings},\ref{sec:Dodd_prelim}),
it will use the basic definitions and notation relating to bicephali.

In the proof of solidity in \S\ref{sec:solidity}, we will work with bicephali of a different
kind, and also the trees on those bicephali will be formed somewhat differently
to those we are about to describe.
We will wait until \S\ref{sec:solidity} to introduce those bicephali formally,
but the main idea and much of the setup and notation is the same.

\begin{dfn}
 Let $M$ be an $m$-sound premouse
 and $\rho<\rho_m^M$. Say that
 $M$ is \dfnemph{$(m+1,\rho)$-finitely generated}
 iff there is some $g\in[\OR^M]^{<\om}$
 such that
 $M=\Hull_{m+1}^M(\rho\cup\{g,\pvec_m^M\})$.
Let $g^M(\rho)$ be the least such $g$.
\end{dfn}

The following is a variant
%conf
of \cite[Definition 3.1]{premouse_inheriting}.

\begin{dfn}\label{dfn:3-simple_pre-exact_bicephalus}
A \dfnemph{3-simple pre-exact bicephalus}\footnote{But we normally
just write \emph{bicephalus}.} is a
structure
$B=(\rho,M,N)$ such that $M,N$ are premice,
$\rho<\min(\OR^M,\OR^N)$ and $\rho$
is a cardinal of both $M,N$,
 $M||\rho^{+M}=N||\rho^{+N}$,
for some $m,n\in\om$,
we have $\rho<\rho_m^M$
and $M$ is $m$-sound and $(m+1,\rho)$-finitely generated
(hence $\rho_{m+1}^M\leq\rho$)
and $\rho_{n+1}^N\leq\rho<\rho_n^N$ and $N$ is $n$-sound
and $(n+1,\rho)$-finitely generated.

We say $B$ has \dfnemph{degree} $(m,n)$.
We write $(\rho^B,M^{0B},M^{1B})=(\rho,M,N)$
and $B||\rho^{+M}=M||\rho^{+M}=N||\rho^{+N}$.
We say $B$ is \dfnemph{trivial} iff $M=N$.

Let $\widetilde{B}=(\widetilde{M},\widetilde{N},\widetilde{\rho})$ be a triple
such that $\widetilde{M},\widetilde{N}$ are structures for the premouse language satisfying the premouse axioms
and $\widetilde{\rho}$ is a linear order.
We say that $\widetilde{B}$ is \dfnemph{wellfounded} if $\widetilde{M},\widetilde{N}$ are both wellfounded, and hence premice, and $\widetilde{\rho}$
is a wellorder
  (and in this case we take $\widetilde{M},\widetilde{N},\widetilde{\rho}$
to be transitive).
\end{dfn}

\begin{rem}
 Note that in \cite[Definition 3.1]{premouse_inheriting},
 it is allowed for example that $m=-1$,
 which means that $M$ is type 3 with $\rho=\rho_0^M=\lgcd(M)$;
 likewise for $N,n$. Although we could have developed
 the material in this section also for that case,
 doing so would have added complications (like those
 %conf
 in \cite[Definition 3.14]{premouse_inheriting}, for example).
 We actually have no application for this case in this paper,
so we have instead restricted our attention to $m,n\geq 0$
(hence the adjective ``3-simple'').
 On the other hand, in \cite[Definition 3.1]{premouse_inheriting},
 it is also assumed that $M,N$ are $\rho$-sound,
whereas this assumption is weakened here.
Actually we only need to use these bicephali
under the added assumptions
that $m=n$, $\rho=\rho_{n+1}^N=\rho_{n+1}^M$,
$N$ is $(n+1)$-universal and
$M=\core_{n+1}(N)$ is $(n+1)$-solid
(hence $(n+1)$-sound),
but $N$ is not $(n+1)$-sound.

From now on in this section,
and until \S\ref{sec:solidity},
the only bicephali we consider will be 3-simple pre-exact,
so we drop these two adjectives and just write \emph{bicephalus}.
\end{rem}

\begin{dfn}\label{dfn:bicephalus_ult}
 Let $B$ be a bicephalus of degree $(m,n)$.
A short extender $E$ is
\dfnemph{weakly amenable} to
$B$ iff $\crit(E)<\rho^B$ and $E$ is weakly amenable to $M^B$ (equivalently, to
$N^B$). Suppose $E$ is weakly amenable to $B$. Define
\[ B'=\Ult_{m,n}(B,E)=(M',N',\rho') \]
where $M'=\Ult_m(M,E)$, $N'=\Ult_n(N,E)$ and $\rho'=\sup i_E``\rho$.
\end{dfn}

We generalize abstract iterations of premice in the obvious manner:
\begin{dfn}\label{dfn:bicephalus_abstract_it}
Let $B$ be a standard bicephalus
of degree $(m,n)$. A \dfnemph{degree $(m,n)$ abstract
weakly amenable
iteration of $B$} is a pair
$\left(\left<B_\alpha\right>_{\alpha\leq\lambda},\left<E_\alpha\right>_{
\alpha<\lambda}\right)$
such that $B_0=B$,
for all $\alpha<\lambda$,
$B_\alpha$ is a degree $(m,n)$ standard bicephalus,
$E_\alpha$ is a short extender weakly amenable to $B_\alpha$,
$\crit(E_\alpha)<\rho^{B_\alpha}$,
$B_{\alpha+1}=\Ult_{m,n}(B_\alpha,E_\alpha)$,
and for all limits $\eta\leq\lambda$,
$B_\eta$ is the resulting direct limit.
We say the iteration is \dfnemph{wellfounded}
if $B_\lambda$ is wellfounded.
\end{dfn}

\begin{lem}\label{lem:bicephalus_ult}
 With notation as in \ref{dfn:bicephalus_ult},
 suppose $B'=(\rho',M',N')$
 is wellfounded.
 Let $i:M\to M'$ and $j:N\to N'$ be the ultrapower maps.
 Then:
 \begin{enumerate}[label=\arabic*.,ref=\arabic*]
  \item\label{item:B'_is_biceph_deg_pres} $B'$ is a bicephalus of degree $(m,n)$.
  \item\label{item:B'_trivial_iff_B_trivial} $B'$ is trivial iff $B$ is trivial.
  \item\label{item:g_pres} $i(g^M(\rho))=g^{M'}(\rho')$
  and $j(g^N(\rho))=g^{N'}(\rho')$,
  \item\label{item:i,j_agree_cont_cof} $i\rest(B||\rho^{+B})\sub j$
  and $i,j$ are continuous/cofinal at $\rho^{+B}$.
 \end{enumerate}
 Likewise for $B_\lambda$ etc if the abstract iteration in
\ref{dfn:bicephalus_abstract_it}
 is wellfounded.
\end{lem}
\begin{proof}
This is mostly as in
%conf
\cite[Lemma 3.15]{premouse_inheriting}, but
the augmented bicephalus defined in
%conf
\cite[Definition 3.14]{premouse_inheriting}
is not relevant, and there are other small differences.
Let us first verify
that $M'$ is $(m+1,\rho')$-finitely generated and likewise for $N',n$,
and part \ref{item:g_pres} holds.
Let $g'=i(g^M(\rho))$. Then
\begin{equation}\label{eqn:M'_is_Hull_pres} M'=\Hull_{m+1}^{M'}(\rho'\cup\{g',\pvec_m^{M'}\}) \end{equation}
as usual. Moreover,
$g'$ is least such, because otherwise
there is $g''\in[\OR^{M'}]^{<\om}$
and $\beta<\rho$
with
\[ g'\in\Hull_{m+1}^{M'}(i(\beta)\cup\{g'',\pvec_m^{M'}\}),\]
but this is an $\rSigma_{m+1}$ assertion over $M'$
about $(i(\beta),g',\pvec_m^{M'})$, so
it pulls back to $M$ about $(\beta,g^M(\rho),\pvec_m^M)$,
a contradiction. Likewise for $N',n$.

 Parts \ref{item:B'_is_biceph_deg_pres} and \ref{item:i,j_agree_cont_cof} are quite routine consequences of what we have established above.
 For part \ref{item:B'_trivial_iff_B_trivial}, suppose $B$ is non-trivial but $m=n$. Then we can fix an $\rSigma_{m+1}$ formula $\varphi$ and $\vec{x}\in\rho^{<\om}$ such that
$M\sats\varphi(\vec{x},g^M(\rho))$ iff $N\sats\neg\varphi(\vec{x},g^N(\rho))$.
But then by part \ref{item:g_pres},
$M'\sats\varphi(i(\vec{x}),g^{M'}(\rho'))$ iff $ N'\sats\neg\varphi(j(\vec{x}),g^{N'}(\rho'))$,
and since $i(\vec{x})=j(\vec{x})$ by part \ref{item:i,j_agree_cont_cof}, it follows that $M'\neq N'$.
\end{proof}

We next define degree-maximal iteration trees $\Tt$ on bicephali $B=(M,N,\rho)$,
and associated notation and terminology.
%conf
This follows \cite[Definition 3.20]{premouse_inheriting}: as there,
these are much like iteration trees on premice, but
associated to each node $\alpha<\lh(\Tt)$ of $\Tt$,
 will be a structure $B_\alpha$,
 which is either a bicephalus
or a premouse. We  write $\mathscr{B}$
for the set of nodes $\alpha$ such that $B_\alpha$ is a bicephalus. If $\alpha\in\mathscr{B}$ then we write $B_\alpha=(M^0_\alpha,M^1_\alpha,\rho_\alpha)$ for this bicephalus.
  Associated to each node $\alpha$ will also be a set $\sides_\alpha\sub\{0,1\}$, with $\sides_\alpha\neq\emptyset$, indicating
which ``sides'' of the base bicephalus $B$ are associated to  $B_\alpha$.
We will have
 $\sides_\alpha=\{0,1\}$ iff $\alpha\in\mathscr{B}$.

 Note that for $\alpha<\lh(\Tt)$, we have $\alpha\notin\mathscr{B}$
 iff $\sides_\alpha=\{e\}$ for some $e$, and  here $e\in\{0,1\}$.
 If $\sides_\alpha=\{e\}$
 then $B_\alpha$ will be a premouse, and we also write $B_\alpha=M^e_\alpha$.
In general, $\mathscr{B}\cap[0,\alpha]^\Tt$ will be a closed initial segment of $[0,\alpha]^\Tt$. Let $\beta=\max(\mathscr{B}\cap[0,\alpha]^\Tt)$, so $B_\beta=(M^0_\beta,M^1_\beta,\rho_\beta)$ is a bicephalus. Suppose $\beta<\alpha$
 and $\sides_\alpha=\{e\}$. Then we can think of $B_\alpha=M^e_\alpha$ as being ``above $M^e_\beta$'', analogous to how models of an iteration tree on a phalanx are ``above'' a model of the phalanx.

 If $\alpha+1<\lh(\Tt)$
 then we will have an integer $\exitside_\alpha\in\sides_\alpha$
 which indicates the ``side''
from which the exit extender $E^\Tt_\alpha$ is taken, and in general this ``side'' is unrestricted. That is, we must just have $E^\Tt_\alpha\in\es_+(M^{e\Tt}_\alpha)$ where $e=\exitside_\alpha$.
So if $\sides_\alpha=\{e\}$ then  $\exitside_\alpha=e$.
But if  $\alpha\in\mathscr{B}$ then in general, we are free to choose any $e\in\{0,1\}$ as $\exitside_\alpha$, and choose $E^\Tt_\alpha\in\es_+(M^{e\Tt}_\alpha)$.

 Suppose  $\beta\in\mathscr{B}=\pred^\Tt(\alpha+1)$.
 If $\crit(E^\Tt_\alpha)<\rho_\beta$
 and $E^\Tt_\alpha$
 is total over $M^0_\beta$ (equivalently, total over $M^1_\beta$),
 then we will put $\alpha+1\in\mathscr{B}$ and set $B_{\alpha+1}=\Ult_{m,n}(B_\beta,E^\Tt_\alpha)$ where $B$ has degree $(m,n)$. If
  $\rho_\beta\leq\crit(E^\Tt_\alpha)$
  then $\alpha+1\notin\mathscr{B}$, so $B_{\alpha+1}$ will just be a premouse.
In the latter case it will be important that we set $\sides_{\alpha+1}=\{e\}$
  and $M^{*\Tt}_{\alpha+1}\ins M^e_\beta$
  where $\exitside^\Tt_\beta=e$.
  Otherwise things work mostly as for $k$-maximal iteration trees on $k$-sound premice. We now give all the  details.

\begin{dfn}\label{dfn:bicephalus_tree}
 Let $B=(M^0,M^1,\rho)$ be a bicephalus of degree $(m^0,m^1)$.
Let $\lambda\in\OR\cut\{0\}$.
A \dfnemph{degree-maximal iteration tree on  $B$} of
\dfnemph{length $\lambda$} is a  system
\[
\Tt=\left({<^\Tt},\curlyB,
\dropset,\dropset_{\deg},\left<B_\alpha,\rho_\alpha,\sides_\alpha\right>_{
\alpha<\lambda},
\left<M^{e}_{\alpha},\deg^e_\alpha,i^e_{\alpha\beta}\right>_{\alpha\leq
\beta<\lambda\text{ and }e<2}, \right. \]
\[ \left.\left<\exitside_\alpha,\exit_\alpha,E_\alpha,
B^*_{\alpha+1}\right>_{\alpha+1<\lambda},
\left<M^{e*}_{\alpha+1},
i^{e*}_{\alpha+1},i^{e*}_{\alpha+1,\beta}\right>_{\alpha+1\leq\beta<\lambda\text{ and }e< 2}\right), \]
with the following properties
for all $\alpha<\lambda$:
\begin{enumerate}[label=\arabic*.,ref=\arabic*]
 \item  ${<^\Tt}$ is an iteration tree order on $\lambda$,
 \item $\dropset\sub\lambda$ is the set of \dfnemph{dropping nodes},
 \item $\emptyset\neq\sides_\alpha\sub\{0,1\}$,
 \item $\alpha\in\mathscr{B}$
 iff $\sides_\alpha=\{0,1\}$,
\item $0\in\curlyB\sub\lambda$ and $\curlyB$ is closed
in $\lambda$ and closed downward under $<^\Tt$,
\item $B_0=B$ and $(\deg^0_0,\deg^1_0)=(m^0,m^1)$,
\item If
$\alpha\in\curlyB$
then $B_\alpha=(M^0_\alpha,M^1_\alpha,\rho_\alpha)$
is a degree $(m^0,m^1)$ bicephalus
and $(\deg^0_\alpha,\deg^1_\alpha)=(m^0,m^1)$.
\item If $\alpha\notin\curlyB$
and $\sides_\alpha=\{e\}$
then $B_\alpha=M^e_\alpha$ is a $\deg^e_\alpha$-sound premouse,
and $M^{1-e}_\alpha=\emptyset$,
\item If $\alpha+1<\lambda$
then letting $e=\exitside_\alpha$, we have
$e\in\sides_\alpha$ and $\exit_\alpha\ins M^e_\alpha$.
\item $E_\alpha=F^{\exit_\alpha}\neq\emptyset$.
\item If $\alpha+1<\beta+1<\lh(\Tt)$ then $\lh(E_\alpha)\leq\lh(E_\beta)$.
\item\label{item:it_tree_on_biceph_clause_succssor_step} If $\alpha+1<\lh(\Tt)$
then $\beta=\pred^\Tt(\alpha+1)$ is least
with $\crit(E_\alpha)<\nu(E_\beta)$. Moreover,
\begin{enumerate}
\item $\sides_{\alpha+1}\sub\sides_\beta$
\item\label{item:it_tree_on_biceph_clause_when_alpha+1_in_curlyB} $\alpha+1\in\curlyB$ iff
[$\beta\in\curlyB$ and $\crit(E_\alpha)<\rho_\beta$
and $E_\alpha$ is $B_\beta$-total].
\item If $\alpha+1\in\curlyB$ then $B^{*}_{\alpha+1}=B_\beta$,
$B_{\alpha+1}=\Ult_{m^0,m^1}(B_\beta,E_\alpha)$
and $i^{0*}_{\alpha+1},i^{1*}_{\alpha+1}$ are the associated
ultrapower maps
(as in \ref{dfn:bicephalus_ult}).
\item If $\alpha+1\notin\curlyB$
then
$\sides_{\alpha+1}=\{e\}$
where $e=\exitside_\beta$, and moreover,
 $M^{e*}_{\alpha+1}\ins M^e_\beta$ and $\deg^e_{\alpha+1}$
are determined as for degree-maximal trees on premice
(with $\deg^e_{\alpha+1}\leq m_e$ if
$(0,\alpha+1]^\Tt\cap\dropset=\emptyset$),
\[ M^e_{\alpha+1}=\Ult_{d}(M^{e*}_{\alpha+1},E_\alpha) \]
where $d=\deg^e_{\alpha+1}$,
and $i^{e*}_{\alpha+1}$ is the ultrapower map.
\item $\alpha+1\in\dropset$ iff $\sides_{\alpha+1}=\{e\}$
and $M^{e*}_{\alpha+1}\pins M^e_\beta$.
\item\label{item:it_tree_on_biceph_clause_when_alpha+1_in_dropset_deg} $\alpha+1\in\dropset_{\deg}$ iff
 either $\alpha+1\in\dropset$
 or ($\sides_{\alpha+1}=\{e\}$ and
$\deg^e_{\alpha+1}<\deg^e_\beta$).
\end{enumerate}
\item\label{item:biceph_tree_it_maps}The iteration map $i^e_{\alpha\beta}:M^e_\alpha\to
M^e_\beta$ is defined iff [$\alpha\leq^\Tt\beta$
and $\dropset\cap(\alpha,\beta]^\Tt=\emptyset$],
and is then defined as usual;
likewise for $i^{e*}_{\alpha+1}$ and $i^{e*}_{\alpha+1,\beta}=i^e_{\alpha+1,\beta}\com i^{e*}_{\alpha+1}$ if $\alpha+1<\lh(\Tt)$.
\item\label{item:bicephalus_tree_limit_alpha} Suppose $\alpha$ is a limit. Then $\dropset\cap(0,\alpha)^\Tt$
is finite and $\alpha\notin\dropset\cup\dropset_{\deg}$,
$\sides_\alpha=\lim_{\beta<^\Tt\alpha}\sides_\beta$,
and the models $M^e_\alpha$ are the resulting direct limits
under the iteration maps $i^e_{\beta\gamma}$
for $\beta\leq^\Tt\gamma<^\Tt\alpha$, etc,
determining $B_\alpha$, etc.
And $\deg^e_\alpha=\lim_{\beta<^\Tt\alpha}\deg^e_\beta$. Here if $\alpha\in\curlyB$
then $\rho_\alpha=\sup_{\beta<^\Tt\alpha}i^{e\Tt}_{\beta\alpha}``\rho_\beta$.
\end{enumerate}

Note that even if
$E_\beta\in\es_+(M^{0}_\beta)\cap\es_+(M^{1}_\beta)$,
we still specify $\exitside_\beta$, and this is used
to determine $\sides_{\alpha+1}$ when $\pred^\Tt(\alpha+1)=\beta$.

Superscript $\Tt$ denotes an object associated
to $\Tt$, so $B^\Tt_\alpha=B_\alpha$, $M^{e\Tt}_\alpha=M^e_\alpha$, etc.
We employ other notation associated to iteration trees (in particular as described in \S\ref{sec:notation_iteration_trees})
in the obvious manner; for example, if $\lh(\Tt)=\xi+1$ then $\infty$ may be used to denote $\xi$.

Also define
$\curlyM^{e\Tt}=\{\alpha<\lambda\bigm|\sides_\alpha=\{e\}\}$.
Given for example a bicephalus $(P,Q,\rho)$,
we  may also write  $\mathscr{P}^\Tt=\curlyM^{0\Tt}$
and
$P_\alpha=M^{0\Tt}_\alpha$
and $i_{\alpha\beta}=i^{0\Tt}_{\alpha\beta}$,
and likewise  $\mathscr{Q}^\Tt=\curlyM^{1\Tt}_\alpha$
and $Q_\alpha=M^{1\Tt}_\alpha$
and $j_{\alpha\beta}=i^{1\Tt}_{\alpha\beta}$.

A \dfnemph{putative degree-maximal iteration tree}
on $B$ is likewise, but if $\lambda=\alpha+1$,
then we do not demand that $B_\alpha$ is wellfounded
or that it be a bicephalus/premouse.
\end{dfn}

The following  lemma summarizes
some basic  facts.
The proof
is left to the reader,
but it is much as for trees on premice,
 using Lemmas \ref{lem:bicephalus_ult}
and \ref{cor:basic_fs_pres} (it is also like
%conf
\cite[Lemma 3.21]{premouse_inheriting}).

\begin{lem}\label{lem:deg-max_tree_biceph_basic_props}
 Let $\Tt$ be a putative degree-maximal tree on a
 degree $(m^0,m^1)$ bicephalus $(M^0,M^1,\rho)$ with wellfounded models. Then it is a degree-maximal
 tree, i.e.~satisfies all conditions of Definition \ref{dfn:bicephalus_tree}.
 Moreover, writing
$\mathscr{B}=\curlyB^\Tt$, etc, we have:

\begin{enumerate}[label=\arabic*.,ref=\arabic*]
\item\label{item:Ff-iteration_tree_agmt}
Let $\alpha<\beta<\lh(\Tt)$. Then:
 \begin{enumerate}[label=--]
  \item If $\beta\in\scB$ then $\lh(E_\alpha)\leq\rho_\beta$ and
$\lh(E_\alpha)$ is
a cardinal of $B_\beta$.

  \item $\exit_\alpha||\lh(E_\alpha)=B_\beta||\lh(E_\alpha)$.

  \item If $e\in\{0,1\}$ and  $\beta\in\scM^{e}$ then either:
  \begin{enumerate}[label=--]
   \item $\lh(E_\alpha)<\OR(M^{e}_\beta)$ and
   $\lh(E_\alpha)$ is a
cardinal of
$M^{e}_\beta$, or
\item $\lh(E_\alpha)=\OR(M^{e}_\beta)$ and
$\beta=\alpha+1$ and
$E_\alpha$ is superstrong
and $M^{e}_\beta$ is active type 2.
 \end{enumerate}
\end{enumerate}
 \item \label{item:closeness}
If $\alpha+1<\lh(\Tt)$ then $E_\alpha$ is weakly amenable to
$B^{*}_{\alpha+1}$ and if
$\alpha+1\notin\scB$
then $E_\alpha$ is close to $B^{*}_{\alpha+1}$.

\item Let $\alpha<\lh(\Tt)$ and $e\in\sides_\alpha$ and
$\delta\leq^\Tt\alpha$ with $(\delta,\alpha]^\Tt\cap\dropset_{\deg}=\emptyset$. Then
$i^{e}_{\delta\alpha}$ is a $\deg^{e}_\alpha$-embedding,
and if
$\delta$ is a
successor then
$i^{e*}_{\delta\alpha}$ is a $\deg^{e}_\alpha$-embedding.
\item\label{item:biceph_tree_basic_props_when_b^Tt_no_drop_in_model_or_degree} Suppose $b^\Tt$ does not drop in model or degree
and $e\in\sides_\infty$. Let $m=m^e$ and $M=M^e$.
Suppose $\rho_{m+1}^{M}=\rho$ and $M$ is $(m+1)$-sound.
Let $\alpha=\max(\curlyB\cap[0,\infty]^\Tt)$.
Then $M^{e}_\alpha$ is $(m+1)$-sound.
If $\alpha<^\Tt\infty$ then $M^{e}_\infty$
is $m$-sound, $(m+1)$-solid and $(m+1)$-universal
but not
$(m+1)$-sound, $\core_{m+1}(M^{e}_{\infty})=M^{e}_\alpha$
and $i^{e}_{\alpha\infty}$ is the core map,
which is $p_{m+1}$-preserving.
\item Suppose
$b^\Tt$
drops in
model or degree and let $\deg^{e}_\infty=k$ where $\sides_\infty=\{e\}$.
Let
$\alpha+1\in b^\Tt$ be least such that $(\alpha+1,\infty]^\Tt$ does not drop in
model or degree.
 Then  $M^{e}_\infty$ is $(k+1)$-universal
and $(k+1)$-solid,
but
not $(k+1)$-sound,
and $M^{e*}_{\alpha+1}=\core_{k+1}(M^e_\infty)$ and
$i^{e*}_{\alpha+1,\infty}$ is the core
embedding, which is $p_{k+1}$-preserving.
\end{enumerate}
\end{lem}

\begin{dfn}
The \dfnemph{$(\om_1+1)$-iteration game}
for a bicephalus $B$ is defined in the obvious manner,
with players I and II building a putative degree-maximal
tree on $B$, player I choosing the extenders $E^\Tt_\alpha$, and $\exitside^\Tt_\alpha$, when it is ambiguous, and player II
the branches, with player II having to ensure wellfoundedness.
\dfnemph{Iteration strategies}
and \dfnemph{iterability} associated
to the game are then as usual.
\end{dfn}

\section{Simple embeddings of iteration trees}\label{sec:simple_embeddings}

In this section we discuss the construction
of finite hulls $\bar{\Tt}$ of iteration trees $\Tt$,
capturing a given element $x\in M^\Tt_\infty$, in the sense that we have a natural  copy map $\somevarpi:M^{\bar{\Tt}}_\infty\to M^\Tt_\infty$ with  $x\in\rg(\somevarpi)$.
More generally, we will also consider hulls which are finite after some stage of $\Tt$,
allowing some control over the critical point of $\somevarpi$. The tree $\Tt$ will be on a premouse or a bicephalus,
and $\bar{\Tt}$ on the same structure. The methods will be used in the main arguments in \S\S\ref{sec:mim}--\ref{sec:finite_gen_hull},
as outlined in \S\ref{sec:meas_Dodd_prof_fin_plan}. They  mostly involve quite routine calculations like in the copying construction, but some of the details will be less routine, and in \S\ref{sec:projecta_cofinality}
and \S\ref{sec:essentially_degree-max} below we make some preparations for the less routine aspects.

\subsection{Projecta and definable cofinality}\label{sec:projecta_cofinality}

When considering a hull $\bar{\Tt}$ of a tree $\Tt$,
we will keep track of the correspondence of fine-structural degrees  $\deg^\Tt_\alpha$ and $\deg^{\bar{\Tt}}_{\bar{\alpha}}$  between nodes $\alpha$ of $\Tt$ and corresponding nodes of $\bar{\alpha}$ of $\bar{\Tt}$,
by keeping track of the relevant projecta of $M^\Tt_\alpha$ and $M^{\bar{\Tt}}_{\bar{\alpha}}$ and how they are shifted by iteration maps.\footnote{Around the  time the author worked these things out, Steel worked out some related calculations.} \footnote{One might  instead be able to use the methods of \cite{fs_tame}
to see that the copy maps associated to nodes at degree $n$ are near
$n$-embeddings.
But this would require some work also,
particularly as we do have to allow the possibility that for
trees on bicephali,
some ultrapowers $\Ult_n(M,E)$ will be formed without
$E$ being close to $M$.}  To keep track of the projecta, it is important to consider their definable cofinality, at the appropriate level of definability.

We are presently interested in $\rho_m^M$ where  $M$ is an $m$-sound premouse, and how $\rho_m^M$ shifts under degree $m$ iteration maps $M\to\Ult_m(M,E)$,
and also, how these things relate to a copied ultrapower map $M'\to\Ult_m(M',E')$.
The ``$\bfrSigma_m^M$-cofinality'' of $\rho_m^M$ is relevant to this.
There are 3 natural ways to define this notion of cofinality, and we will show that, in fact, they are all identical:
\begin{dfn}\label{dfn:wcof_etc}
 Let $M$ be an $m$-sound premouse and $\rho=\rho^M_m$.

  The \dfnemph{weak-$\bfrSigma^M_m$-cofinality
 of $\rho$}, denoted $\wcof^{\bfrSigma^M_m}(\rho)$, is the least
$\kappa\in[\om,\rho]$
 such that either (i) $\kappa=\rho$,
 or (ii) $\rho\in M$ and $\kappa=\cof^M(\rho)$,
 or (iiiw) $m>0$ and there is $x\in\core_0(M)$
 such that $\rho\cap\Hull_m^M(\kappa\cup\{x\})$ is unbounded in $\rho$.

 The
\dfnemph{$\bfrSigma^M_m$-cofinality of $\rho$},
denoted $\cof^{\bfrSigma_m^M}(\rho)$,
is likewise, but with (iiiw) replaced by (iiic)
 $m>0$ and there is a cofinal normal function $f:\kappa\to\rho$
  which is $\bfrSigma_m^M$.

  The
\dfnemph{amenable-$\bfrSigma^M_m$-cofinality of $\rho$},
denoted $\acof^{\bfrSigma_m^M}(\rho)$,
is likewise, but with (iiiw) replaced by (iiia) $m>0$ and there is a
cofinal normal function $f:\kappa\to\rho$
such that $f\rest\alpha\in M$ for each $\alpha<\kappa$,
and the function $F:\kappa\to M$
is $\bfrSigma_m^M$, where $F(\alpha)=f\rest\alpha$.

We say  $F$ is \dfnemph{$m$-good for $M$}
iff $F$ is as in (iiia),
or
letting $\kappa=\acof^{\bfrSigma_m^M}(\rho)$,
then $F\in M$, $F:\kappa\to M$,
$F(\alpha):\alpha\to\rho$ for each $\alpha<\kappa$, and $f:\kappa\to\rho$
is cofinal normal where $f=\bigcup_{\alpha<\kappa} F(\alpha)$.
\end{dfn}

\begin{rem}If $m=0$ and $M$ is non-type 3 then
$\rho_0^M=\OR^M=\cof^{\bfrSigma_0^M}(\rho_0^M)$.
\end{rem}

\begin{lem}\label{lem:when_m-good_F_exists}Let $M$ be an $m$-sound premouse.
 Then there is an $F$ which is $m$-good for $M$ iff either
$m>0$ or $M$ is active type 3.
 \end{lem}
 \begin{proof}
 Let $\rho=\rho_m^M$ and $\kappa=\acof^{\bfrSigma_m^M}(\rho)$.
 If $m=0$ and $M$ is active type 3,
 then (iiia) fails trivially as $m=0$,
 and $\rho=\nu(F^M)\in M$ (though $\rho\notin M^{\mathrm{sq}}$),
 so $\kappa=\cof^M(\rho)$,
 and it follows there there is $F\in M$
 as required.
 Now suppose $m>0$, but there is no $F$ as in (iiia) for $\kappa$. Then note that also (i) does not attain for $\kappa$, so (ii) attains instead. Therefore there is $F\in M$ as required. Conversely, if $m=0$ but $M$ is not active type 3, then (iiia) fails, and $\rho=\OR^M\notin M$, so there is no $F\in M$ as above, so there is no $F$ which is $m$-good for $M$.
\end{proof}

\begin{lem}
 Let $M$ be an $m$-sound premouse and $\rho=\rho^M_m$.
 Then
 \[
\kappa\eqdef\wcof^{\bfrSigma_m^M}(\rho)=\cof^{\bfrSigma_m^M}(\rho)=
\acof^{\bfrSigma_m^M}(\rho),\qedhere\]
and if $\kappa\in M$ then $\kappa$ is regular in $M$.
\end{lem}
\begin{proof}
If $m=0$ then the three ordinals are the same by definition, and clearly if $\kappa\in M$ then $M\sats$ ``$\kappa$ is regular''.

So suppose $m>0$. If $\rho\in M$ and $\rho$ is singular in $M$
then it is easy to see that all three values
agree with
$\cof^M(\rho)$.
So suppose that $\rho\notin M$, or $\rho\in M$
and $\rho$ is regular in $M$.
Clearly
\[ \wcof^{\bfrSigma_m^M}(\rho)\leq\cof^{\bfrSigma_m^M}(\rho)\leq\acof^{
\bfrSigma_m^M}(\rho)\leq\rho.\]
So suppose $\kappa=\wcof^{\bfrSigma_m^M}(\rho)<\rho$;
we must find a function $f:\kappa\to\rho$
witnessing that $\acof^{\bfrSigma_m^M}(\rho)=\kappa$.

Since $\kappa<\rho_m^M=\rho$,
we can easily find an $\bfrSigma_m^M$ function $g:\kappa\to\rho$
which is cofinal. (Take any $D\sub\rho$
and function $g':D\to\rho$ which is cofinal and $\bfrSigma_m^M$, so in particular, $D$ is $\bfrSigma_m^M$. Since $\kappa<\rho_m^M$, therefore $D\in M$.
By the minimality of $\kappa$,
$D$ has ordertype $\kappa$, so $g'$ is easily modified to produce $g$ as desired.)
Let $x\in\core_0(M)$
be such that $g$ is $\rSigma_m^M(\{x\})$.
By the minimality of $\wcof^{\bfrSigma_m^M}(\rho)$,
note that $g``\alpha$ is bounded in $\rho$
for each $\alpha<\kappa$.

We claim there is an $\bfrSigma_m^M$ function
$h:\rho_{m-1}^M\to\rho$ which is cofinal
and monotone increasing. To see this we stratify
%conf
$\rSigma_m^M$ truth as in \cite[Appendix to \S2]{fsit}. If $m=1$
and $M$ is passive and has no largest proper segment,
then we just look at the $\rSigma_1^M$ truths
witnessed by proper segments $t_\alpha=M|\alpha\pins M$,
for various $\alpha<\OR^M$. If $m=1$ and $M$ has a largest
proper segment $N$, it is likewise,
but letting $\OR^N=\eta$, then $t_{\eta+n}=M|(\eta+n)$
denotes $\Ss_{n}(N)$, for $n<\om$.
And if $m=1$ and $M$ is active, one
needs to augment proper segments of $M$ with proper segments
of the active extender of $M$.
If $m>1$, one considers theories of the form
\[ t_\alpha=\Th_{\rSigma_{m-1}}^M(\alpha\cup\{x,\pvec_{m-1}^M\}),\]
for various $\alpha<\rho_{m-1}^M$, as coding witnesses
to $\rSigma_m$ truths
(here $x$ is the parameter from which we defined $g$ above).
See \cite[\S2]{fsit} for more details.
Now given $\beta<\rho_{m-1}^M$,
let $A_\beta\sub\kappa$
be the set of all $\alpha<\kappa$
such that there is $\xi$ such that $t_\beta$ codes a witness
to the ($\rSigma_m(\{x\})$) statement
``$g(\alpha)=\xi$''. Let $g_\beta:A_\beta\to\rho$
be the resulting function (so $g_\beta\sub g$).
Then $g_\beta\in M$ since $t_\beta\in M$.
Since $\rho$ is regular in $M$,
therefore $\rg(g_\beta)$ is bounded in $\rho$.
Now let $h:\rho_{m-1}^M\to\rho$
be $h(\beta)=\sup(\rg(g_\beta))$.
So $h$ is $\bfrSigma_m^M$, and clearly
$h$ is cofinal and monotone increasing.

We now claim that for each $\alpha<\kappa$,
there is $\beta<\rho_{m-1}^M$ such that $g\rest\alpha\sub g_\beta$,
where $g_\beta$ is as above.
For otherwise, letting $\alpha$ be the least counterexample,
note that there is a cofinal $\bfrSigma_m^M$ function
$j:\alpha\to\rho_{m-1}^M$, but then $h\com j:\alpha\to\rho$
is cofinal, contradicting the minimality of $\kappa$.

It follows that the function $f':\kappa\to\rho$ defined
$f'(\alpha)=\sup g``\alpha$, is as desired,
except that $f'$ need not be strictly increasing.
But from $f'$ we can easily get an $f$ as desired.
\end{proof}

\begin{rem}\label{rem:i-hat}
 Recall $\widehat{i}=\Shift(i)$ was defined in \S\ref{sec:notation_embeddings}.  If $M,F$ are as in the lemma below and $M$ is type 3, it can be that $F\in M\cut M^\sq$, in which case $F\in\dom(\widehat{i})$ but $F\notin\dom(i)$,
 which is why we need to deal with $\widehat{i}$ here.
\end{rem}

\begin{lem}\label{lem:d-preserving}
Let $M$ be an $m$-sound premouse
and $E$ be an extender weakly amenable to $M$ with $\mu=\crit(E)<\rho_m^M$.
Suppose $U=\Ult_m(M,E)$  is wellfounded. Let $i=i^{M,m}_E$.
Write $\kappa^M=\cof^{\bfrSigma_m^M}(\rho_m^M)$ and
$\kappa^U=\cof^{\bfrSigma_m^U}(\rho_m^U)$.
Then:

\begin{enumerate}[label=\arabic*.,ref=\arabic*]
\item\label{item:kappa^U=sup} $\kappa^U=\sup i``\kappa^M$
\tu{(}recall also $\rho_m^U=\sup i``\rho_m^M$\tu{)}.
\end{enumerate}

Moreover, suppose either $m>0$ or $M$ is active type 3, and let $F:\kappa^M\to M$ be $m$-good for $M$ \tu{(}see Lemma \ref{lem:when_m-good_F_exists}\tu{)}.
Then:

\begin{enumerate}[resume*]
\item\label{item:mu_neq_kappa^M} Suppose $m>0$ and $\mu\neq\kappa^M$.
Then:
\begin{enumerate}[label=\tu{(}\alph*\tu{)}]
 \item If $\rho_m^M=\rho_0^M$ then $\rho_m^U=\rho_0^U$.
 \item If $\rho_m^M<\rho_0^M$ then $\rho_m^U=i(\rho_m^M)$.
 \item If $\kappa^M=\rho_0^M$ then $\kappa^U=\rho_0^U$.
 \item If $\kappa^M<\rho_0^M$ then $\kappa^U=i(\kappa^M)$.
\item  If $F\in M$ then $\Shift(i)(F):\kappa^U\to U$ is $m$-good for $U$.
\item  Given any  $\rSigma_m$ term $t$ and
$x\in\core_0(M)$ such that $F=t^M_x$,
we have that $t^U_{i(x)}:\kappa^U\to U$ is
$m$-good for $U$.
\end{enumerate}
\item\label{item:mu=kappa^M} Suppose $m>0$ and $\mu=\kappa^M$ \tu{(}so
$\kappa^M<\rho_m^M$\tu{)}. Then:
\begin{enumerate}[label=\tu{(}\alph*\tu{)}]
 \item If $\rho_m^M=\rho_0^M$ then $\rho_m^U<\rho_{m-1}^U=\rho_0^U$.
 \item If $\rho_m^M<\rho_0^M$ then $\rho_m^U<i(\rho_m^M)$.
 \item $\kappa^U=\mu=\kappa^M<i(\kappa^M)$.
\item  If $F\in M$ then $\Shift(i)(F)\rest\mu=i\com F$ is $m$-good
for $U$.
\item Given any  $\rSigma_m$ term $t$ and
$x\in\core_0(M)$ such that  $F=t^M_x$, we have that
$t^U_{i(x)}\rest\mu=i\com F$ is
$m$-good for $U$.
\end{enumerate}

\item\label{item:m=0_type_3_mu_neq_kappa^M} Suppose $m=0$ and $M$ is type 3 and
$\mu\neq\kappa^M$.
Then $i$ is $\nu$-preserving
and $\Shift(i)(\kappa^M)=\kappa^U$,
and  $\Shift(i)(F):\kappa^U\to U$
is $0$-good for $U$.

\item\label{item:m=0_type_3_mu=kappa^M} Suppose $m=0$ and $M$ is type 3 and
$\mu=\kappa^M$.
Then $i$ is $\nu$-high
and $\kappa^U=\mu=\kappa^M$,
and  $\Shift(i)(F)\rest\mu=i\com F$
is $0$-good for $U$.
\end{enumerate}
\end{lem}
\begin{proof}
Part \ref{item:mu_neq_kappa^M}:
If $M$ is type 3, then because $m>0$, $i$ is $\nu$-preserving.

Suppose first that $\mu<\kappa^M$.
Then since $\kappa=\wcof^{\bfrSigma_m^M}(\rho_m^M)$,
all $g:[\kappa]^{<\om}\to\rho_m^M$
 used in forming $\Ult_m(M,E)$
are all bounded in $\rho_m^M$.
All the parts now follow easily,
as does that $\kappa^U=\sup i``\kappa^M$ in this case.

Now suppose instead that $\kappa^M<\mu$,
so $\kappa^M<\mu<\rho_m^M$.
Let $g:[\mu]^{<\om}\to\rho_m^M$
be $\bfrSigma_m^M$ and $a\in[\nu_E]^{<\om}$.
Let $f=\bigcup\rg(F)$.
Note that
$h:[\mu]^{<\om}\to\kappa^M$ is $\bfrSigma_m^M$,
where $h(u)$ is the least $\alpha<\kappa^M$
such that $h(u)<f(\alpha)$.
Since $\kappa^M<\mu<\rho_m^M$,
therefore $h\in M$.
Therefore there is an $E_a$-measure one set
$A\sub[\mu]^{<\om}$ such that $h\rest A$ is constant.
Everything now follows easily, including again that $\kappa^U=\sup i``\kappa^M$.

Part \ref{item:mu=kappa^M}: This is straightforward.

Parts \ref{item:m=0_type_3_mu_neq_kappa^M},\ref{item:m=0_type_3_mu=kappa^M}:
%conf
By calculations as in \cite[\S9]{fsit} (one considers
the correspondence between the simple ultrapower $\Ult(M,E)$ (formed
without squashing) and $\Ult_0(M,E)$ (the unsquash of the ultrapower of the
squash), and the respective ultrapower embeddings).

Part \ref{item:kappa^U=sup}: If $m=0$ and $M$ is non-type 3, then $\rho_m^M=\rho_0^M=\OR^M=\kappa^M$
and $\rho_m^U=\rho_0^U=\OR^U=\kappa^U$,
and also $i$ is cofinal, which suffices. Otherwise see the discussion above.
\end{proof}

\begin{dfn}\label{dfn:m-preserving}
 We say that $\pi:\core_0(M)\to\core_0(N)$ is \dfnemph{$m$-preserving}
 iff $\pi$ is
$m$-lifting and:
 \begin{enumerate}[label=--]
  \item $\pi$ is $\nu$-preserving, and
  \item if $m>0$ then for each $i\in(0,m]$, we have:
  \begin{enumerate}[label=--]
  \item $\pi(p_i^M)=p_i^N$,
  \item if $\rho_i^M<\rho_0^M$
  then $\pi(\rho_i^M)=\rho_i^N$,
  \item if $\rho_i^M=\rho_0^M$
  then $\rho_i^N=\rho_0^N$,
  \item if $\cof^{\bfrSigma_i^M}(\rho_i^M)=\rho_0^M$
  then $\cof^{\bfrSigma_i^N}(\rho_i^N)=\rho_0^N$,
  \item if $\kappa=\cof^{\bfrSigma_i^M}(\rho_i^M)<\rho_0^M$
  then $\pi(\kappa)=\cof^{\bfrSigma_i^N}(\rho_i^N)$,
  and there is an $\rSigma_i$-term
  $t$ and $x\in\core_0(M)$ such that
  $f^M_{t,x}$
  is $i$-good for $M$ (where $f^M_{t,x}(u)=t^M(x,u)$)
  and
   $f^{N}_{t,\pi(x)}$ is $i$-good for $N$.\qedhere
 \end{enumerate}
 \end{enumerate}
\end{dfn}

\begin{lem}\label{lem:stronger_d-preserving}
 Let $\pi:\core_0(M)\to\core_0(N)$
 be $m$-preserving where $m>0$. Suppose that
$\kappa^M=\cof^{\bfrSigma_m^M}(\rho_m^M)<\rho_m^M$.
Let $t_1,x_1$ be such that
$f_{t_1,x_1}^M$
 is $m$-good for $M$. Then $f_{t_1,\pi(x_1)}^N$
 is $m$-good for $N$.
\end{lem}
\begin{proof}
 Fix $t,x$ such that $f^M_{t,x}$ is $m$-good
 for $M$ and $f^N_{t,\pi(x)}$ is $m$-good for $N$.
 If $\rho_m^M<\rho_0^M$ then let $\rho=\rho_m^M$,
 and otherwise let $\rho=0$.
 Let
 \[ T=\Th_{\rSigma_m}^M(\kappa^M\cup\{x,x_1,\rho,\kappa^M\}).\]
 Then $T\in\core_0(M)$, since $\kappa^M<\rho_m^M$.
 Note that the facts that
\begin{enumerate}[label=(\roman*)]
 \item\label{item:first_fact_t1_x1} $\all\alpha<\kappa^M\
[t_1^M(x_1,\alpha):\alpha\to\rho_m^M\text{ is a normal function}]$, and
\item $\all\alpha<\beta<\kappa^M\
[t_1^M(x_1,\alpha)\sub t_1^M(x_1,\beta)]$,
\end{enumerate}
are simply expressed facts about $T$.
Also, note that there is a fixed $\rSigma_m$ formula $\varphi$ such that
\[
\all\alpha,\beta<\kappa^M\ [t^M(x,\alpha+1)(\alpha)<t_1^M(x_1,
\beta+1)(\beta)\iff
\text{``}\varphi(x,x_1,\alpha,\beta)\text{''}\in T].\]
Therefore,  that
\begin{enumerate}[resume*]
\item\label{item:third_fact_t1_x1}
$\sup\Big(\rg\big(\bigcup(\rg(f^M_{t,x}))\big)\Big)=
\sup\Big(\rg\big(\bigcup(\rg(f^M_{t_1,x_1}))\big)\Big)$
\end{enumerate}
is also a simple fact about $T$.

Now let $\kappa^N=\cof^{\bfrSigma_m^N}(\rho_m^N)$.
 Since $\pi$ is $m$-preserving, if $\rho=0$ then $\rho_m^N=\rho_0^N$,
 and if $\rho>0$ then $\rho_m^N=\pi(\rho)$, and in any case,
 \[\pi(T)=\Th_{\rSigma_m}^N
(\kappa^N\cup\{\pi(x),\pi(x_1),\pi(\rho),\kappa^N\}).\]
But then \ref{item:first_fact_t1_x1}--\ref{item:third_fact_t1_x1}
about $T$
lift to $\pi(T)$. Since $f^N_{t,\pi(x)}$ is $m$-good for $N$ by assumption,
 \ref{item:first_fact_t1_x1}--\ref{item:third_fact_t1_x1} for $\pi(T)$
ensure that $f^N_{t_1,\pi(x_1)}$
is also $m$-good for $N$.
\end{proof}

\subsection{Essentially degree-maximal trees}\label{sec:essentially_degree-max}

The kinds of hulls of iteration trees we consider will lead to the following slight generalization of degree-maximal trees:
\begin{dfn}\label{dfn:ess-deg-max}
 Let $\Tt$ be an iteration tree on a premouse or bicephalus,
 satisfying the requirements of degree-maximality,
 except for the monotone increasing length condition
 (i.e.~that $\lh(E^\Tt_\alpha)\leq\lh(E^\Tt_\beta)$
 for $\alpha<\beta$).
 We say that $\Tt$ is \dfnemph{essentially-degree-maximal}
 iff $\nu(E^\Tt_\alpha)\leq\nu(E^\Tt_\beta)$
 for all $\alpha+1<\beta+1<\lh(\Tt)$.
 Likewise \dfnemph{essentially-$m$-maximal}.

 Let $\Tt$ be essentially-degree-maximal
 and $\alpha+1<\lh(\Tt)$. We say
 that $E^\Tt_\alpha$ is \dfnemph{$\Tt$-stable}
 iff $\lh(E^\Tt_\alpha)\leq\lh(E^\Tt_\beta)$
 for all $\beta\geq\alpha$.
\end{dfn}

\begin{dfn}\label{dfn:ess_strat}
Given an $m$-sound premouse $M$,  an \dfnemph{essentially-$(m,\om_1+1)$-iteration strategy for $M$}
is just like an $(m,\om_1+1)$-strategy for $M$,
except that it works with essentially-$m$-maximal trees instead of $m$-maximal ones.
\end{dfn}

To calibrate expectations, the reader might consider the natural example. Suppose $\Tt$ is essentially $m$-maximal on $M$, and uses just two extenders $E^\Tt_0$ and $E^\Tt_1$,
with $\nu(E^\Tt_0)\leq\nu(E^\Tt_1)<\lh(E^\Tt_1)<\lh(E^\Tt_0)$. Let $\bar{\Tt}$ be the $m$-maximal tree on $M$ that  uses just one extender, $E^{\bar{\Tt}}_0=E^\Tt_1$.
Then we claim $M^\Tt_2=M^{\bar{\Tt}}_1$.
For if $\crit(E^\Tt_1)<\nu(E^\Tt_0)$
then $\pred^\Tt(2)=0=\pred^{\bar{\Tt}}(1)$, and clearly $M^{*\Tt}_2=M^{*\bar{\Tt}}_1\ins M$
and $\deg^{\Tt}_2=\deg^{\bar{\Tt}}_1$,
so $M^{\Tt}_2=M^{\bar{\Tt}}_1$.
On the other hand, suppose $\nu(E^\Tt_0)\leq\crit(E^\Tt_1)$. Then $\pred^\Tt(2)=1$. Moreover, $2\in\dropset^\Tt$, with $M^{*\Tt}_2\pins M^\Tt_1|\lh(E^\Tt_0)$. For $\lh(E^\Tt_0)=\nu(E^\Tt_0)^{+M^\Tt_1}$
and \[ \nu(E^\Tt_0)\leq\crit(E^\Tt_1)<\crit(E^\Tt_1)^{+\exit^\Tt_1}<\lh(E^\Tt_1)<\lh(E^\Tt_0)=\nu(E^\Tt_0)^{+M^\Tt_1},\]
and so $E^\Tt_1$ is not $M^\Tt_1$-total.
In $\bar{\Tt}$,
we have $\pred^{\bar{\Tt}}(1)=0$,
and since $M||\lh(E^\Tt_0)=M^\Tt_1||\lh(E^\Tt_0)$,
therefore $M^{*\bar{\Tt}}_1=M^{*\Tt}_2$
and $\deg^{\bar{\Tt}}_1=\deg^{\Tt}_2$,
so $M^{\bar{\Tt}}_1=M^\Tt_2$.

The lemmas below generalize this example.

\begin{lem}\label{lem:ess-deg-max_stability}
 Let $\Tt$ be an essentially-degree-maximal tree
 and $\alpha+1<\lh(\Tt)$.
 Then $E^\Tt_\alpha$ is $\Tt$-stable
 iff either:
 \begin{enumerate}[label=--]
  \item  $\alpha+2=\lh(\Tt)$,\footnote{Note that $\alpha+2=\lh(\Tt)$ iff $M^\Tt_{\alpha+1}$ is the last model of $\Tt$ iff $E^\Tt_\alpha$ is the last extender of $\Tt$.} or
  \item $\lh(E^\Tt_\alpha)<\lh(E^\Tt_{\alpha+1})$, or
\item $\lh(E^\Tt_\alpha)=\lh(E^\Tt_{\alpha+1})$
 and either \tu{[}$\alpha+3=\lh(\Tt)$ or
$\lh(E^\Tt_{\alpha+1})<\lh(E^\Tt_{\alpha+2})$\tu{]}.
\end{enumerate}
\end{lem}
\begin{proof}
 Suppose $\alpha+2<\lh(\Tt)$, so $E^\Tt_{\alpha+1}$ exists. If $\lh(E^\Tt_\alpha)<\lh(E^\Tt_{\alpha+1})$
 then $\lh(E^\Tt_\alpha)$ is a cardinal
 of $\exit^\Tt_{\alpha+1}$,
 so $\lh(E^\Tt_\alpha)\leq\nu(E^\Tt_{\alpha+1})\leq\nu(E^\Tt_\beta)<\lh(E^\Tt_\beta)$
 for all $\beta>\alpha$,
 so $E^\Tt_\alpha$ is $\Tt$-stable.
 Now suppose
 $\lh(E^\Tt_\alpha)=\lh(E^\Tt_{\alpha+1})<\lh(E^\Tt_{\alpha+2})$
 (in particular, $E^\Tt_{\alpha+2}$ exists).
 Then $E^\Tt_{\alpha+1}$ is $\Tt$-stable
 as above, and it follows that $E^\Tt_\alpha$ is also $\Tt$-stable.
 Conversely, these are the only options for $E^\Tt_\alpha$ to be $\Tt$-stable,
 since it is impossible to have $\lh(E^\Tt_\alpha)=\lh(E^\Tt_{\alpha+1})=\lh(E^\Tt_{\alpha+2})$.
 (If $\lh(E^\Tt_\alpha)=\lh(E^\Tt_{\alpha+1})$ then $E^\Tt_\alpha$ is superstrong and $E^\Tt_{\alpha+1}$ is type 2,
 so $\lh(E^\Tt_{\alpha+1})<\lh(E^\Tt_{\alpha+2})$.)
\end{proof}

\begin{lem}\label{lem:essentially_to_normal}
Let $\Tt$ be an essentially-$m$-maximal tree on a premouse $M$
of length $\alpha+1$. Then there is an $m$-maximal
tree $\bar{\Tt}$ on $M$, with $\lh(\bar{\Tt})=\bar{\alpha}+1\leq\alpha+1$,
such that $M^{\bar{\Tt}}_{\bar{\alpha}}=M^\Tt_\alpha$,
and if $(0,\alpha]^\Tt$ is non-\tu{(}model, degree\tu{)}-dropping
then so is $(0,\bar{\alpha}]^{\bar{\Tt}}$
and $i^{\bar{\Tt}}=i^\Tt$.

Let $\Tt$ be an essentially-degree-maximal tree on a bicephalus
$B$. Then there is, analogously, a corresponding degree-maximal tree
$\bar{\Tt}$
on $B$.
\end{lem}
\begin{proof}
 This is a straightforward generalization of the example above; the extenders
 used in $\bar{\Tt}$ are  those $E$
 which are $\Tt$-stable, and
 the branches of $\bar{\Tt}$ are  those determined by $\Tt$
 in the obvious manner.
\end{proof}

These observations also yield the following lemma:
\begin{lem}\label{lem:ess_iter}
Let $M$ be an $m$-sound premouse and $\Sigma$ be an $(m,\om_1+1)$-strategy for $M$.
Then there is an essentially-$(m,\om_1+1)$-strategy $\Sigma'$, with $\Sigma\sub\Sigma'$, and the trees $\Tt$ on $M$ via $\Sigma'$ are exactly those which determine a tree $\Ttbar$ via $\Sigma$
as in Lemma \ref{lem:essentially_to_normal}
and its proof.
\end{lem}
\subsection{Simple embeddings}\label{sec:sub_simple_embeddings}
In this subsection we describe and construct the first kind of hulls of iteration trees to be used in \S\S\ref{sec:mim}--\ref{sec:finite_gen_hull}
(cf.~the discussion at the start of \S\ref{sec:simple_embeddings}).
 In some of the definitions/lemmas,
 we  state things literally only for iteration trees on bicephali;
 in those cases the version for trees on premice is just the obvious
simplification thereof.
\begin{dfn}\label{dfn:extcopy}
 Let $\somevarpi:M\to N$ be a non-$\nu$-high
  embedding between
premice.
 Let $E\in\es_+^M$. Then $\extcopy(\somevarpi,E)$ denotes $F$,
 where either:
 \begin{enumerate}[label=(\roman*)]
 \item $E\in\es^M$ and
$F=\Shift(\somevarpi)(E)$
(see  \S\ref{sec:notation_embeddings} and Remark \ref{rem:i-hat}), or
\item $E=F^M$, $\somevarpi$ is non-$\nu$-low and $F=F^N$, or
 \item $E=F^M$, $\somevarpi$ is $\nu$-low and $F=F^N\rest\Shift(\somevarpi)(\nu(F^M))$.\qedhere
\end{enumerate}
\end{dfn}

If $\somevarpi$ is a copy map arising in a typical copying construction, then $F=\extcopy(\somevarpi,E)$ is the
natural copy of $E$ to an extender in $F\in\es_+^N$.
It is important that $\somevarpi$ is non-$\nu$-high here,
since otherwise if $F\in\es^N\cut\core_0(N)$ then $\extcopy(\somevarpi,E)\notin\es_+^N$.
Note that in general, we have $\Shift(\somevarpi)(\nu(E))\geq\nu(\extcopy(\somevarpi,E))$.
We define a variant copying process in \ref{dfn:concopy}, which
handles $\nu$-high embeddings, assuming enough condensation.

\begin{dfn}[$\lambda$-simple embedding]\label{dfn:lambda-simple_embedding}
Let $M$ be an $m$-sound premouse,
 $\Tt$ an $m$-maximal tree on $M$
and $\Ttbar$  essentially-$m$-maximal on $M$;
or let $B=(\rho,M,N)$ be a degree $(m,n)$ bicephalus,
$\Tt$ be degree-maximal on $B$,
and
$\bar{\Tt}$
essentially-degree-maximal
on $B$.
Let $\lambda<\lh(\Ttbar)$.
We say that
  \[
\Phi=(\Ttbar,\varphi,\left<\sigma^0_\alpha,\sigma^1_\alpha\right>_{
i<\lh(\Ttbar)})\]
  is a \dfnemph{$\lambda$-simple embedding \tu{(}of $\Ttbar$\tu{)} into
$\Tt$}, written
$\Phi:\Ttbar\hookrightarrow_{\lambda\text{-}\simple}\Tt$,  iff the following conditions hold, where
we write $\widehat{\sigma}=\Shift(\sigma)$ for maps $\sigma$ in clauses \ref{item:drop_star_matches}--\ref{item:shift_lemma_gives_sigma'}:
 \begin{enumerate}[label=\arabic*.,ref=\arabic*]
  \item $\lambda+1\leq\lh(\Ttbar)<\lambda+\om$
  and $\lh(\Ttbar)\leq\lh(\Tt)$.
      \item $\Ttbar\rest(\lambda+1)=\Tt\rest(\lambda+1)$.
  \item $\varphi:\lh(\Ttbar)\to\lh(\Tt)$ is order-preserving with
  $\varphi\rest(\lambda+1)=\id$.
    \item If $\Ttbar,\Tt$ are on $B$, then
    $\varphi$ preserves bicephalus/model structure; that is:
\begin{enumerate} \item
$\sides^{\Ttbar}_\alpha=\sides^\Tt_{\varphi(\alpha)}$.
 \item if $\sides^{\Ttbar}_{\alpha}=\{n\}$
 then $\{\beta,\beta+1,\pred^\Tt(\beta+1)\}\sub\rg(\varphi)$ where $\beta$ is
least such that
  $\beta+1\leq^\Tt\varphi(\alpha)$ and $\sides^\Tt_{\beta+1}=\{n\}$.
\end{enumerate}
  \item  $\varphi$ preserves tree,  drop and
degree structure.
  More precisely:
  \begin{enumerate}
  \item $\alpha<^\Ttbar\beta$ iff $\varphi(\alpha)<^\Tt\varphi(\beta)$.
  \item If $\beta+1\in(0,\varphi(\alpha)]^\Tt\cap\dropset_{\deg}^\Tt$ then
$\{\beta,\beta+1,\pred^\Tt(\beta+1)\}\sub\rg(\varphi)$.
  \item $\beta+1\in\dropset^\Ttbar$/$\dropset_{\deg}^\Ttbar$ iff
$\varphi(\beta+1)\in\dropset^\Tt$/$\dropset^\Tt_{\deg}$,
  \item $\deg^{n\Ttbar}_\alpha=\deg^{n\Tt}_{\varphi(\alpha)}$.
  \end{enumerate}
  \item\label{item:varrho_i_maps}   If $n\in\sides^{\Ttbar}_\alpha$ then
$\sigma^n_\alpha:\core_0(M^{n\Ttbar}_\alpha)\to
\core_0(M^{n\Tt}_{\varphi(\alpha)})$  is a
 $\deg^{n\Ttbar}_\alpha$-preserving embedding.
\item $\sigma^n_\alpha=\id$ for $\alpha\leq\lambda$.
\item\label{item:copy_extenders} For $\alpha+1<\lh(\Ttbar)$, we have
$\exitside^{\Ttbar}_\alpha=\exitside^\Tt_{\varphi(\alpha)}$,
and letting $n=\exitside^{\Ttbar}_\alpha$, we have
 $E^\Tt_{\varphi(\alpha)}=\extcopy(\sigma^n_\alpha,
E^\Ttbar_\alpha)$.
\item Let $\beta=\pred^\Ttbar(\alpha+1)$. Then:
\begin{enumerate}
\item $\varphi(\beta)=\pred^\Tt(\varphi(\alpha)+1)$,
\item $\varphi(\alpha)+1\leq^\Tt\varphi(\alpha+1)$

(and note  that also
$\sides^\Tt_{\varphi(\alpha)+1}=\sides^\Tt_{\varphi(\alpha+1)}=\sides^{\Ttbar}
_{\alpha+1}$ and also
$(\varphi(\alpha)+1,\varphi(\alpha+1)]^\Tt\cap\dropset^\Tt_{\deg}=\emptyset$),
  \item
$i^{n\Tt}_{\varphi(\alpha)+1,\varphi(\alpha+1)}$ is
$i$-preserving for each $(n,i)$ with
$n\in\sides^{\Ttbar}_{\alpha+1}$
and $i\leq\deg^{n\bar{\Tt}}_{\alpha+1}$.

\end{enumerate}
\item\label{item:drop_star_matches} Let $\alpha+1\in\dropset^\Ttbar$
and $\beta=\pred^\Ttbar(\alpha+1)$ and $\{n\}=\sides^{\Ttbar}_{\alpha+1}$.
Then
$\widehat{\sigma^n_\beta}(M^{n*\Ttbar}_{\alpha+1})=
M^{n*\Tt}_{\varphi(\alpha)+1}$
 (as mentioned above,  $\widehat{\sigma}$ denotes $\Shift(\sigma)$ here and below).
\item\label{item:varrho_agmt} Let $\beta<\alpha<\lh(\Ttbar)$
and $n=\exitside^{\Ttbar}_\beta$ and $\ell\in\sides^{\Ttbar}_\alpha$.
Then:
 \begin{enumerate}[label=--]
\item  $\widehat{\sigma^n_\beta}\rest\nu^\Ttbar_\beta\sub\sigma^\ell_\alpha$,
and
\item
if $E^\Ttbar_\beta$ is type 1 or 2
then
$\widehat{\sigma^n_\beta}
\rest\lh(E^\Ttbar_\beta)\sub\sigma^\ell_\alpha$
\end{enumerate}
(and recall $\sigma\sub\widehat\sigma$ in general).
\item\label{item:shift_lemma_gives_sigma'} Let $\beta=\pred^\Ttbar(\alpha+1)$ and $n\in\sides^{\Ttbar}_{\alpha+1}$.
Then
$\sigma^n_{\alpha+1}=i^{n\Tt}_{\varphi(\alpha)+1,\varphi(\alpha+1)}\com\sigma'$,
where $\sigma'$ is defined as in the Shift Lemma. That is, let
$d=\deg^{n\Ttbar}_{\alpha+1}$, $\bar{M}^*=M^{n*\Ttbar}_{\alpha+1}$,
$M^*=M^{n*\Tt}_{\varphi(\alpha)+1}$, $\bar{E}=E^\Ttbar_\alpha$,
$E=E^\Tt_{\varphi(\alpha)}$,
 $\ell=\exitside^{\Ttbar}_\alpha$.
Then for each $a\in[\nu(\bar{E})]^{<\om}$, if  $d=0$, then for each $f\in\core_0(M^*)$ with $f:[\crit(\bar{E})]^{|a|}\to \core_0(M^*)$,
\[
\sigma'\left(\left[a,f\right]^{\bar{M}^*,0}_{\bar{E}}\right)=
\left[\widehat{\sigma^\ell_\alpha}(a),
\widehat{\sigma^n_\beta}(f)\right]^{M^*,0}_{E},\]
whereas if $d>0$, then for each
 $x\in\core_0(\bar{M}^*)$ and  degree $d$ term $t$ (see \S\ref{sec:notation_fine_structure}),
\[
\sigma'\left(\left[a,f_{t,x}^{\bar{M}^*}\right]^{\bar{M}^*,d}_{\bar{E}}\right)=
\left[\widehat{\sigma^\ell_\alpha}(a),
f_{t,\widehat{\sigma^n_\beta}(x)}^{M^*}\right]^{M^*,d}_{E}.\]

 \end{enumerate}

Actually conditions
 \ref{item:drop_star_matches} and \ref{item:varrho_agmt} follow automatically
from the others;
 we have stated them explicitly here as they help to make sense of the others.

Let $\Phi:\Ttbar\hookrightarrow_{\lambda\text{-}\simple}\Tt$ with
$\Phi=(\Ttbar,\varphi,\vec{\sigma})$ where $\vec{\sigma}=\left<\sigma^0_\alpha,\sigma^1_\alpha\right>_{\alpha<\lh(\Ttbar)}$.
We write
$\varphi^\Phi=\varphi$ and $(\sigma^n_\alpha)^\Phi=\sigma^n_\alpha$.
Given $\beta=\alpha+1\in(0,\lh(\Ttbar)]$, write
$\Phi\rest \beta=(\Ttbar\rest\beta,\varphi\rest \beta,\vec{\sigma}\rest \beta)$.
\end{dfn}

We now record a few more properties of $\lambda$-simple embeddings,
writing $\widehat{\sigma}=\Shift(\sigma)$:
\begin{lem}\label{lem:lambda-simple_embedding}
 Let $\Phi:\Ttbar\hookrightarrow_{\lambda\text{-}\simple}\Tt$,
  let $\varphi=\varphi^\Phi$ and
$\sigma_\alpha^e=\sigma_\alpha^{\Phi,e}$. Let
$\alpha,\beta,\xi<\lh(\Ttbar)=\gamma+1$. Then:
 \begin{enumerate}[label=\arabic*.,ref=\arabic*]
\item If $\alpha<^\Ttbar \beta$ and
$(\alpha,\beta]^\Ttbar\cap\dropset^\Ttbar=\emptyset$
and $e\in\sides^{\Ttbar}_\beta$ then
$\sigma^e_\beta\com
i^{e\Ttbar}_{\alpha\beta}=i^{e\Tt}_{\varphi(\alpha)\varphi(\beta)}
\com\sigma^e_\alpha$.

\item\label{item:drop_comm} If $\alpha+1\in\dropset^\Ttbar$
\tu{(}so $\varphi(\alpha+1)=\varphi(\alpha)+1$\tu{)} and
$\xi=\pred^\Ttbar(\alpha+1)$ and $\sides^{\Ttbar}_{\alpha+1}=\{e\}$
then
\[ \sigma^e_{\alpha+1}\com
i^{e*\Ttbar}_{\alpha+1}=
i^{e*\Tt}_{\varphi(\alpha+1)}\com\widehat{\sigma^e_\xi}\rest\core_0(M^{
e*\Ttbar }_{\alpha+1}).\]
\item\label{item:when_lh_dec} Suppose
$\bar{\iota}=\lh(E^{\Ttbar}_{\alpha+1})<\lh(E^{\Ttbar}_\alpha)$.
Then $\varphi(\alpha)+1<^\Tt\varphi(\alpha+1)$ and $E^{\Ttbar}_\alpha$ is type
3.
Let $\bar{\nu}=\nu(E^{\Ttbar}_\alpha)$,
$\nu=\nu(E^\Tt_{\varphi(\alpha)})$,
$d=\exitside^{\Ttbar}_\alpha$,
$e\in\sides^{\Ttbar}_{\alpha+1}$,
 $j^e=i^{e\Tt}_{\varphi(\alpha)+1,\varphi(\alpha+1)}$,
$\kappa=\crit(j^e)$, $\lambda=j^e(\kappa)$ and
$\iota=\lh(E^\Tt_{\varphi(\alpha+1)})$.
Then
\[ \widehat{\sigma^d_\alpha}(\bar{\nu})=\nu=
\kappa<\lambda=\sigma^e_{\alpha+1}(\bar{\nu})<
\iota<\lambda^{+M^{e\Tt}_{\varphi(\alpha+1)}}. \]
\item Suppose $\varphi(\alpha)+1<^\Tt\varphi(\alpha+1)$.
Let $\nu=\nu(E^\Tt_{\varphi(\alpha)})$
and $\delta=\sup_{\xi<\varphi(\alpha+1)}\nu(E^\Tt_\xi)$.
Let $\beta+1<\lh(\Ttbar)$ and $\kappa=\crit(E^{\Ttbar}_\beta)$
and $d=\exitside^{\Ttbar}_\beta$.
Then $\widehat{\sigma^d_\beta}(\kappa)\notin[\nu,\delta)$.
 \item\label{item:agmt_below_at_nu} Let $\alpha<\beta$
 and $d=\exitside^{\Ttbar}_\alpha$,
 $e\in\sides^{\Ttbar}_\beta$,
  $\nu=\nu(E^\Ttbar_\alpha)$,
 $\iota=\lh(E^\Ttbar_\alpha)$.
Then:
\begin{enumerate}[label=--]
 \item $\widehat{\sigma^d_\alpha}
\rest\nu=\sigma^e_\beta\rest\nu=
\widehat{\sigma^e_\beta}\rest\nu$,
\item $\widehat{\sigma^d_\alpha}(\nu)\leq
\sigma^e_\beta(\nu)=\widehat{\sigma^e_\beta}(\nu)$,
\item $\widehat{\sigma^d_\alpha}(\nu)=\nu(E^\Tt_{\varphi(\alpha)})$,
\item if $E^\Ttbar_\alpha$ is type 1/2 then
$E^\Ttbar_\alpha$ is $\Ttbar$-stable,
 $\widehat{\sigma^d_\alpha}\rest\iota=
\sigma^e_\beta\rest\iota$
and
$\sigma^e_\beta(\iota)=\lh(E^\Tt_{
\varphi(\alpha)})$
\end{enumerate}
\end{enumerate}
\end{lem}
\begin{proof}
We just discuss part \ref{item:when_lh_dec};
the rest is routine and left to the reader. Since
$\sigma^d_\alpha$ is $\nu$-preserving,
letting
$\sigma':\core_0(M^{e\bar{\Tt}}_{\alpha+1})\to
\core_0(M^{e\Tt}_{\varphi(\alpha)+1})$ be defined
via the Shift Lemma (as in condition \ref{item:shift_lemma_gives_sigma'} of \ref{dfn:lambda-simple_embedding}),
$\sigma'(\bar{\nu})=\nu$ and
 either \begin{enumerate}[label=(\roman*)]
\item $\lh(E^{\bar{\Tt}}_\alpha)<\rho_0(M^{e\bar{\Tt}}_{\alpha+1})$ and $\sigma'(\lh(E^{\bar{\Tt}}_\alpha))=\lh(E^{\Tt}_{\varphi(\alpha)})$
(so $\lh(E^{\Tt}_{\varphi(\alpha)})<\rho_0(M^{e\Tt}_{\varphi(\alpha)+1})$),
or
\item $\bar{\nu}$ is the largest cardinal of $M^{e\bar{\Tt}}_{\alpha+1}$,  $\lh(E^{\bar{\Tt}}_\alpha)=\rho_0(M^{e\bar{\Tt}}_{\alpha+1})$, $\nu$ is the largest cardinal of $M^{e\Tt}_{\varphi(\alpha)+1}$
and  $\lh(E^{\Tt}_{\varphi(\alpha)})=\rho_0(M^{e\Tt}_{\varphi(\alpha)+1})$.\footnote{
This holds iff $E^{\bar{\Tt}}_\alpha$ is of superstrong type and letting $\bar{\mu}=\crit(E^{\bar{\Tt}}_\alpha)$ and $\bar{M}^*=M^{e*\bar{\Tt}}_{\alpha+1}$ , either $\bar{M}^*$ is active type 2 with largest cardinal $\bar{\mu}$ or is active type 3 with largest cardinal $\bar{\mu}^{+\bar{M}^*}=\nu(F^{\bar{M}^*})$.}
\end{enumerate}

So  letting $\iota'=\sigma'(\bar{\iota})$, we have
$ \iota'<\lh(E^\Tt_{\varphi(\alpha)})\leq\lh(E^\Tt_{\varphi(\alpha)+1})$
as $\Tt$ is degree-maximal (not just essentially-degree-maximal).
Note that this is independent of $e\in\{0,1\}$.
So we may assume $e=\exitside^{\bar{\Tt}}_{\alpha+1}$.
But then if $\varphi(\alpha)+1=\varphi(\alpha+1)$
then $\sigma^e_{\alpha+1}=\sigma'$
and we must have that $E^\Tt_{\varphi(\alpha)+1}=\extcopy(\sigma',E^{\bar{\Tt}}_{\alpha+1})$, so $\lh(E^\Tt_{\varphi(\alpha)+1})=\iota'$,
contradicting that $\iota'<\lh(E^\Tt_{\varphi(\alpha)+1})$.
So $\varphi(\alpha)+1<\varphi(\alpha+1)$
and $\kappa<\iota'$
as $j^e(\iota')=\iota=\lh(E^\Tt_{\varphi(\alpha+1)})$.
So $\nu=\kappa$, and the rest follows.
\end{proof}

\begin{dfn}[$\reps$]
Let $M$ be a premouse. We define $\reps^M:M\to\pow(\core_0(M))$.
If $M$ is non-type 3, then $\reps^M(x)=\{x\}$.
If $M$ is type 3, then $\reps^M(x)$
is the set of pairs $(a,f)$
such that $x=[a,f]^M_{F^M}$.
For an iteration tree $\Tt$, $\reps^\Tt_\alpha$ denotes $\reps^{M^\Tt_\alpha}$.
\end{dfn}

A natural way to produce $\lambda$-simple embeddings is through is via finite support:
\begin{dfn}[Finite Support]\label{dfn:finite_support}
Let $\Tt$ be either an $m$-maximal tree on $m$-sound premouse $M$,
or a degree-maximal tree on a bicephalus
 $B=(\rho,M,N)$.

 A \dfnemph{finite selection of $\Tt$} is a finite sequence
$\FF=\left<\FF^n_\theta\right>_{(\theta,n)\in J}$ such that
$J\sub\lh(\Tt)\cross\{0,1\}$ and
for all $(\theta,n)\in J$, we have $n\in\sides^\Tt_\theta$, $\FF^n_\theta\sub\core_0(M^n_\theta)$ and $\FF^n_\theta$
is finite and non-empty.
Write $J_\FF=J$ and $I_\FF=\{\alpha\bigm|\exists n\ [(\alpha,n)\in J]\}$.

A \dfnemph{finite support of $\Tt$}
is a finite selection $\supp=\left<\supp_\alpha^n\right>_{(\alpha,n)\in J}$ of
$\Tt$
such that letting $J=J_\FF$ and $I=I_\FF$, we have:
\begin{enumerate}[label=\arabic*.,ref=\arabic*]
 \item $0\in I$.
 \item $J=\{(\alpha,n)\bigm|\alpha\in I\text{ and }n\in\sides^\Tt_\alpha\}$
\item\label{item:nu_in_rg}
For each $(\alpha,n)\in J$ we have:
\begin{enumerate}[label=--]
 \item $\reps^{n\Tt}_\alpha(\nu(M^{n\Tt}_\alpha))\cap
 \supp^{n}_\alpha\neq\emptyset$, and
\item letting
$d=\deg^{n\Tt}_\alpha$, if $\rho_d(M^{n\Tt}_\alpha)<\rho_0(M^{n\Tt}_\alpha)$
then $\rho_d(M^{n\Tt}_\alpha)\in\supp^n_\alpha$.
\end{enumerate}
\item Let $(\beta+1,e)\in J$ and $E=E^\Tt_\beta$
and $\ell=\exitside^\Tt_\beta$
and $\gamma=\pred^\Tt(\beta+1)$ and
$m=\deg^{e\Tt}_{\beta+1}$
and $M^*=M^{e*\Tt}_{\beta+1}\ins M^{e\Tt}_\gamma$.
Then:
\begin{enumerate}
\item $\beta,\gamma,\gamma+1\in I$.
\item\label{item:x_output_of_term_gens} For all $x\in\supp^e_{\beta+1}$ there
is $(t,a,y)$ such that $t$ is an $\rSigma_m$ term,
$a\in[\nu_E]^{<\om}$,
$\reps^{\ell\Tt}_\beta(a)\cap\supp^\ell_\beta\neq\emptyset$,
$y\in\core_0(M^*)$,
$\reps^{e\Tt}_\gamma(y)\cap\supp^e_\gamma\neq\emptyset$, and
\[ x=[a,f_{t,y}^{M^*}]^{M^*,m}_{E}.\]
\item If $E\neq F(M^\ell_\beta)$ then $\reps^\ell_\beta(E)\cap
\supp^\ell_\beta\neq\emptyset$.
\end{enumerate}
\item\label{item:support_limit_case} Let $(\alpha,e)\in J$ with
$\alpha$ a limit ordinal
(so
$0\in I\cap\alpha$) and $\beta=\max(I\cap\alpha)$. Then
$\beta<^\Tt\alpha$, $\beta$ is a successor,
$\sides^\Tt_\alpha=\sides^\Tt_\beta$,
$(\beta,\alpha)^\Tt$ does not drop
in model or degree, and
$\supp^e_\alpha\sub i^{e\Tt}_{\beta\alpha}``\supp^e_\beta$.
\end{enumerate}

Let $\supp,J,I$ be as above.
We write
\[ I'_\supp=\{\alpha\in I\mid \alpha=\max(I)\text{ or }\alpha+1\in I\}.\]
Given a finite selection $\FF$ of $\Tt$,
we say that $\supp$ is a \dfnemph{finite support of $\Tt$ for $\FF$}
iff $I_\FF\sub I'_\supp$ and
$\FF^e_\beta\sub\supp^e_\beta$ for each $(\beta,e)\in
J_\FF$.
\end{dfn}

\begin{rem}
 The above definition is mostly like that of
 %conf
 \cite[Definition 2.7]{hsstm},
with the
following small
differences. First, we use $\reps_\alpha$ in place of $\rep_\alpha$. Second, we
have added condition \ref{item:nu_in_rg}, in order to ensure that we get
degree-preserving copy maps.
Third, condition (g) of \cite[2.7]{hsstm}
has been modified because our mice can have superstrong extenders.
The point of (g)(iv) of \cite[2.7]{hsstm}  was to ensure that finite support
trees have the monotone length condition, which one cannot quite achieve here,
and which is why we must allow essentially-degree-maximal,
instead of degree-maximal trees.
\end{rem}

We can now state the basic lemma on production of $\lambda$-simple embeddings mapping into a given tree $\Tt$:

\begin{lem}\label{lem:simple_embedding_exists}
Let $M,m,\Tt$ or $B=(\rho,M,N),\Tt$ be as in Definition
\ref{dfn:finite_support}. Then:
\begin{enumerate}[label=\arabic*.,ref=\arabic*]
\item\label{item:finite_support_exists} For every finite selection $\FF$  of
$\Tt$, there is a finite support of
$\Tt$ for $\FF$.
\item\label{item:lambda-simple_embedding_exists}
Let $\supp=\left<\supp_\alpha^n\right>_{(\alpha,n)\in J}$
be a finite support
of $\Tt$, and $\lambda\in I'_\FF$.
Then there is a unique pair $(\bar{\Tt},\Phi)$
with
\[ \Phi:\Ttbar\hookrightarrow_{\lambda\text{-}\simple}\Tt
\text{ and }\rg(\varphi^\Phi)=(\lambda+1)\cup I'_\FF.\] Moreover,
for each $\alpha<\lh(\bar{\Tt})$
and $e\in\sides^{\bar{\Tt}}_\alpha$,
we have $\SS_{\varphi^\Phi(\alpha)}^e\sub\rg(\sigma^{e\Phi}_\alpha)$.
\end{enumerate}
\end{lem}
\begin{proof}[Proof Sketch]
Part \ref{item:finite_support_exists}: This is a straightforward construction;
a very similar argument is
%conf
given in \cite[Lemma 2.8]{hsstm}. (Here it is actually slightly
easier, because of the
changes mentioned in
the remark above.)

Part \ref{item:lambda-simple_embedding_exists}: This is mostly a routine
copying
argument, maintaining
the properties in \ref{dfn:lambda-simple_embedding} by induction,
and using the properties in \ref{lem:lambda-simple_embedding},
and noting that every step is uniquely determined by the set $I_\FF$,
and that the closure of a finite support given by \ref{dfn:finite_support} is
enough to keep
things going. However, we will discuss the maintenance
of the fact that $\sigma^{d}_\alpha$ is
a $\deg^{d\bar{\Tt}}_\alpha$-preserving embedding,
since this is not completely
standard (though straightforward).
(Note that this property ensures that the degrees in $\bar{\Tt}$
match those in $\Tt$.)

So let $\alpha\geq\lambda$ and $\beta=\pred^{\bar{\Tt}}(\alpha+1)$.
Let $e\in\sides^{\bar{\Tt}}_{\alpha+1}$
and $\bar{M}^*=M^{e*\bar{\Tt}}_{\alpha+1}$
and $M^*=M^{e*\Tt}_{\varphi(\alpha)+1}$
and $\pi:\core_0(\bar{M}^*)\to\core_0(M^*)$
be $\pi=\Shift(\sigma^e_\beta)\rest\core_0(\bar{M}^*)$.
Let $d=\deg^{\bar{\Tt}}_{\alpha+1}$.
Just because $\pi$ is $d$-lifting,
if $\alpha+1\notin\dropset_\deg^{\bar{\Tt}}$
then $\varphi(\alpha)+1\notin\dropset_\deg^{\Tt}$.
If $\alpha+1\in\dropset_\deg^{\bar{\Tt}}$,
then by induction, $\pi$ is $(d+1)$-preserving,
 which ensures $\varphi(\alpha)+1\in\dropset^\Tt_\deg$
 and $\deg^{e\Tt}_{\varphi(\alpha)+1}=d$.
Let $\bar{M}'=M^{e\bar{\Tt}}_{\alpha+1}$
and $M'=M^{e\Tt}_{\varphi(\alpha)+1}$.
Let
$i_{\bar{\Tt}}=i^{e*\bar{\Tt}}_{\alpha+1}$,
$i_\Tt=i^{e*\Tt}_{\varphi(\alpha)+1}$
and
$\sigma':\core_0(\bar{M}')\to\core_0(M')$
be defined via the Shift
Lemma. So
$\sigma'$ is $d$-lifting and
$\sigma'\com i_{\bar{\Tt}}=i_\Tt\com\pi$.

We claim that $\sigma'$ is $d$-preserving. For
let $\mu=\crit(E^{\bar{\Tt}}_\alpha)$.
Then $\pi(\mu)=\crit(E^\Tt_{\varphi(\alpha)})$.
We have $\mu<\rho_d^{\bar{M}^*}$ and $\pi(\mu)<\rho_d^{M^*}$.
Let $\kappa^{\bar{M}^*}=\cof^{\bfrSigma_d^{\bar{M}^*}}(\rho_d^{\bar{M}^*})$
and
$\kappa^{M^*}=\cof^{\bfrSigma_d^{{M}^*}}(\rho_d^{{M}^*})$.
Suppose first that $\mu\neq\kappa^{\bar{M}^*}$.
Then since $\pi$ is
$d$-preserving, $\pi(\mu)\neq\kappa^{M^*}$.
By Lemma \ref{lem:d-preserving} and some easy calculations then,
$i_{\bar{\Tt}}$
and $i_\Tt$ are $d$-preserving.
Since $\sigma'\com i_{\bar{\Tt}}=
i_\Tt\com\pi$,
together with Lemma \ref{lem:stronger_d-preserving},
it follows
that $\sigma'$ is $d$-preserving, as desired.
Now suppose instead that
$\mu=\kappa^{\bar{M}^*}$.
Then since $\pi$ is
$d$-preserving,
$\pi(\mu)=\kappa^{M^*}$
and moreover
letting $f^{\bar{M}^*}_{t,x}$ be
$d$-good for $\bar{M}^*$
(here $t$ is an $\rSigma_d$-term and $x\in\core_0(\bar{M}^*)$),
then $f^{M^*}_{t,\pi(x)}$ is also $d$-good for $M^*$.
So by Lemma \ref{lem:d-preserving},
$f^{\bar{M}'}_{t,i_{\bar{\Tt}}(x)}\rest\kappa^{\bar{M}^*}$ is $d$-good
for $\bar{M}'$ and $f^{M'}_{t,i_{\Tt}(\pi(x))}\rest\kappa^{M^*}$ is $d$-good
for $M'$. But $\mu=\kappa^{\bar{M}^*}$
and $\pi(\mu)=\kappa^{M^*}$, so $\sigma'(\kappa^{\bar{M}^*})=\kappa^{M^*}$,
so with the natural term $t'$ and letting
$x'=(i_{\bar{\Tt}}(x),\kappa^{\bar{M}^*})$ we get
that $f^{\bar{M}'}_{t',x'}$ is $m$-good for $\bar{M}'$
and
\[
f^{M'}_{t',\sigma'(x')}=f^{M'}_{t,\sigma'(i_{\bar{\Tt}}(x))}
\rest\sigma'(\kappa^{\bar{M}^*})
 =f^{M'}_{t,i_{\Tt}(\pi(x))}\rest\kappa^{M^*},\]
which is $d$-good for $M'$.  The rest follows easily.

So $\sigma'$ is $d$-preserving. If $\varphi(\alpha+1)=\varphi(\alpha)+1$
then $\sigma^e_{\alpha+1}=\sigma'$, so we are done.
Suppose instead that $\varphi(\alpha)+1<^\Tt\varphi(\alpha+1)$.
Note then that by \ref{dfn:finite_support}, $\varphi(\alpha+1)=\eta$ is a limit.
So using properties \ref{item:nu_in_rg}
and \ref{item:support_limit_case} of \ref{dfn:finite_support},
if $\rho=\rho_d(M^{e\Tt}_{\eta})<\rho_0(M^{e\Tt}_\eta)$
then $\rho\in\rg(i^{e\Tt}_{\varphi(\alpha)+1,\eta})$,
which by \ref{lem:d-preserving} easily implies
that $i^{e\Tt}_{\varphi(\alpha)+1,\eta}$
is $d$-preserving.
Since
$\sigma^e_{\alpha+1}=i^{e\Tt}_{\varphi(\alpha)+1,\eta}\com\sigma'$,
this suffices.
\end{proof}

\section{Super-Dodd structure}\label{sec:Dodd_prelim}

 As sketched in \S\ref{sec:meas_Dodd_prof_fin_plan},
 the analysis of comparisons
  in \S\S\ref{sec:mim}--\ref{sec:finite_gen_hull}
  will rely on keeping track of how embeddings such as iteration maps shift certain critical generators of extenders. The key to understanding this is the analysis of the Dodd structure of extenders, which is the topic of this section.
  We will actually describe a slight refinement,  \emph{super-Dodd} structure;
  this is relevant if the mice in question have extenders of superstrong type on their sequence.

Thee \emph{Dodd parameter} and \emph{projectum}  of an active premouse $M$ were introduced in
%conf
\cite[\S3]{combin} and
%conf
\cite[\S4]{deconstructing}.
The definitions we give for these objects below are stated differently,
but they are equivalent. 	 The \emph{super-Dodd parameter} and \emph{projectum} are refinements of these notions.

\begin{dfn}\label{dfn:super_Dodd_param_proj} Let $M$
	be an active premouse, $F=F^M$ and $\mu=\crit(F)$.

	Recall that  $(t^M,\tau^M)$,
	the \dfnemph{Dodd parameter} and \dfnemph{projectum} of $M$, are
	the least
	$(t,\tau)\in\widetilde{\OR}$ such that $\mu^{+M}\leq\tau$ and $F$ is generated by
	$t\un\tau$ (see \S\ref{sec:notation_general} for the notation $\widetilde{\OR}$ and ordering thereof).

	We define $(\ttilde^M,\tautilde^M)$,
	the \dfnemph{super-Dodd parameter} and \dfnemph{projectum} of $M$, to be the
	least
	$(\ttilde,\tautilde)\in\widetilde{\OR}$ such that $F$ is generated by $\ttilde\un\tautilde$.
\end{dfn}

\begin{rem}\label{rem:Dodd_parameter}Note that the difference between the definition of $(t^M,\tau^M)$ and that of  $(\widetilde{t}^M,\widetilde{\tau}^M)$ is that the former includes the clause ``$\mu^{+M}\leq\tau$'', whereas the latter does not.

	Let $(\ttilde,\tautilde)=(\ttilde^M,\tautilde^M)$. Then
	$\ttilde\sub[\mu,\nu_F)$ and if $\ttilde\neq\emptyset$ then
	$\tautilde\leq\min(\ttilde)$.
	Either $\tautilde=0$ or $\tautilde>\mu^{+M}$. Let
	$(t,\tau)=(t^M,\tau^M)$.
	Then $t=\ttilde\cut\{\mu\}$ and
	$\tau=\max(\tautilde,\mu^{+M})$.
	Steel observed that $\tau$
	is a cardinal of $M$, so $\tautilde$ is also
	(in fact, either $0$ or an infinite cardinal).

	We consider the super-Dodd parameter and projectum  as, assuming the existence of mice with sufficient large cardinals, it is possible
	that
	$\{\mu\}\psub\ttilde^M$.
	This information is recorded in $\ttilde^M$, but not in $(t^M,\tau^M)$.
	So in this case, $(\ttilde^M,\tautilde^M)$
	records more information,
	and also in this case, super-Dodd-soundness, to be defined below, is more demanding than  Dodd-soundness. In case
	$\{\mu\}\psub\ttilde^M$,  $F$ has a superstrong proper
	segment (see
	%conf
	\cite[Remark 2.5]{extmax}, in the context of which, \emph{premice} are
	superstrong-small). So
	$\ttilde^M$ is just $t^M$ in case $M$ is  superstrong-small
	(but even in this case, if $\tautilde^M=0$ then $\tautilde^M\neq\tau^M$). In
	Zeman
	\cite{zeman_dodd} (which uses $\lambda$-indexing), premice can have
	superstrong extenders, but the notion of \emph{Dodd parameter} used there is
	analogous to
	\cite{combin}; it never includes the critical point.\end{rem}

\begin{dfn}\label{dfn:E_rest_x_cup_alpha}
 Let $\pi:M\to N$ be a $\Sigma_0$-elementary embedding between
premice,\footnote{\label{ftn:assume_passive}This definition does not depend on $F^M,F^N$, so we may assume that $M,N$ are passive. We may therefore also apply it to structures $M',N'$
for a language extending that of passive premice,
which are not themselves premice, but whose reducts $M,N$ to the language of passive premice are premice.}
 where $M||\kappa^{+M}=N||\kappa^{+N}$
and $\kappa=\crit(\pi)$.
Let
$\alpha\leq\beta\leq\pi(\kappa)$ and
$x\in[\beta]^{<\om}$ and let $E$ be the $(\kappa,\beta)$-extender derived from
$\pi$. Then
$E\rest(x\un\alpha)$
denotes the set $F$ of pairs $(A,y)$ such that $A\in M|\kappa^{+M}$ and
$y\in[x\un\alpha]^{<\om}$ and $y\in\pi(A)$ (this differs from the
 notation
in \cite{combin} when $\alpha<\kappa^{+M}$). So letting $F=E\rest(x\cup\alpha)$, $F$ is (or can be
treated as) an extender
over $M$. Let $U=\Ult(M,F)$ and $x'=[\id,x]^M_F$ and:
\begin{enumerate}[label=--]
\item if $\xi=\max(x'\cup\alpha)<i_F(\kappa)$ then let
$\gamma=\xi^{+U}$, and
\item otherwise let $\gamma=i_F(\kappa)$
 (note in this case $x'\sub\alpha=i_F(\kappa)$).
 \end{enumerate}
Then the \dfnemph{trivial completion} $\trivcom(F)$ of $F$ is the
$(\kappa,\gamma)$-extender derived from $i_F$ (this agrees with the standard
notion
when $x\sub\alpha$). The \dfnemph{transitive collapse} of $F$ is
$\trivcom(F)\rest(x'\un\alpha)$.

We may identify $F$, $\trivcom(F)$, and the transitive collapse of
$F$.
\end{dfn}

\begin{dfn}\label{dfn:super_Dodd-sound}
	Let $M$ be an active pm, $F=F^M$,
	$\alpha\in\OR^M\cap\lambda^M$ and
	$x\in[\OR^M\cap\lambda^M]^{<\om}$. The
	\dfnemph{Dodd-witness} for $(M,(x,\alpha))$ is the extender
	\[ \Dw^M(x,\alpha)=F\rest x\un\alpha. \]
	We say that $M$ is
	\dfnemph{Dodd-solid} iff for each $\alpha\in t^M$ with $\alpha>\kappa_F$,
	letting
	$x=t^M\cut(\alpha+1)$, we have $\Dw^M(x,\alpha)\in M$. We say that $M$ is
	\dfnemph{super-Dodd-solid} iff $M$ is Dodd-solid and, if $\kappa_F\in \widetilde{t}^M$,
	then
	letting
	$x=\widetilde{t}^M\cut(\kappa_F+1)$, we have the component
	measure $F_x\in M$ (or in notation as for the other
	witnesses, $\Dw^M(x,\kappa_F)\in M$). We say that $M$ is
	\dfnemph{Dodd-amenable} iff either
	$\tau^M=\mu^{+M}$ or
	$\Dw^M(t^M,\alpha)\in M$ for every $\alpha<\tau^M$.

	For
	passive premice $P$ we say that $P$ is (trivially) \dfnemph{Dodd-solid} and
	\dfnemph{super-Dodd-solid} and \dfnemph{Dodd-amenable}.
	A premouse $M$ is \dfnemph{Dodd-sound} iff $M$ is
	Dodd-solid and
	Dodd-amenable, and $M$ is \dfnemph{super-Dodd-sound} iff $M$ is
	super-Dodd-solid and
	Dodd-amenable.
\end{dfn}

\begin{rem}\label{rem:super-Dodd-solidity}
	By Remark \ref{rem:Dodd_parameter}, for superstrong-small premice $M$, super-Dodd-solidity
	(respectively, -soundness) is equivalent
	to Dodd-solidity (respectively, -soundness). But in case $\{\mu\}\psub\ttilde^M$ where $\mu=\crit(F^M)$,
	super-Dodd-solidity demands that $F^M\rest(\ttilde^M\cut\{\mu\})\in M$,
	which is not demanded by Dodd-solidity alone. We consider super-Dodd-solidity simply because it strikes the author as the more natural notion in the non-superstrong-small setting, and anyway, its proof (under appropriate hypotheses) is just slightly different from that for usual Dodd-solidity. In our application of Dodd structure, we need only Dodd-solidity, not super-Dodd-solidity.

	%conf
	Steel \cite[Theorem 4.1]{deconstructing}\footnote{The same result was claimed earlier in
	%conf
	\cite[Theorem 3.2]{combin}, but the supposed proof there had a gap, which was filled in \cite{deconstructing}. Recall that in \cite{combin} and \cite{deconstructing}, \emph{premice}
	are by definition superstrong-small.} proved that
	every $(0,\om_1,\om_1+1)^*$-iterable $1$-sound superstrong-small premouse is Dodd-sound (hence all of its proper segments are also).
	In light of the previous paragraph,
	Steel's result immediately implies super-Dodd-soundness for such premice.
	%conf
	Zeman \cite[Theorems 1.1, 1.2]{zeman_dodd} then proved the corresponding fact for $\lambda$-indexed premice (with corresponding iterability hypothesis), but without the ``superstrong-small'' restriction. Zeman also proved some other related facts in that context.

	We will establish in Theorem \ref{thm:super-Dodd-soundness} super-Dodd-soundness for  $(0,\om_1+1)$-iterable $1$-sound premice (of course, here as elsewhere in the paper, this means with Mitchell-Steel indexing and without the superstrong-smallness restriction). Thus, in comparison with Steel's result, we will assume only normal iterability, will not assume superstrong-smallness, and establish the stronger conclusion of super-Dodd-soundness. Actually, the main extra work required beyond Steel's argument will be in handling superstrong extenders; the rest requires only minor modifications. It seems highly likely that the argument we give will be almost contained in the combination of Steel's and Zeman's. But the proof we give will follow more the lines of Steel's.

	Toward the proof of super-Dodd-soundness, and also toward \S\S\ref{sec:mim},\ref{sec:finite_gen_hull}, we will  develop various properties of  Dodd and super-Dodd parameters and projecta.
	Some of the notions and facts come from or are slight variants of material from \cite{combin} and
	\cite{extmax}, and
	many properties of Dodd
	parameters and projecta established in those papers
	carry over to the present context; that is, they generalize to premice with
	superstrongs, and to super-Dodd parameters and projecta.
	We will also establish, in a certain context, an analogue (Lemma \ref{lem:Dodd_param_p_1}) of Zeman's analysis in  \cite{zeman_dodd} of the relationship between the Dodd parameter $t^M$ and $p_1^M$. We begin by summarizing some generally useful notation and facts regarding weak  hulls of type 2 premice, in Definition \ref{dfn:type_2_factor_embedding} and Lemma \ref{lem:type_2_factor_embedding} below;
	this is as in \cite[Claim 1, proof of Theorem 3.2(A), p.~176]{combin} and similar calculations in the proof of \cite[Lemma 2.15]{extmax}.
	%conf
\end{rem}
\begin{dfn}\label{dfn:type_2_factor_embedding}
Let $N$ be a type $2$ premouse, $F=F^N$, and $X=\alpha\cup x$
where $\alpha\leq\nu=\nu_F$
and $x\in[\nu]^{<\om}$. Then define
$\Gend^N(X)$ to be the set of all $x\in N$
such that $x$ is $F$-generated by some $t\in
[X]^{<\om}$, or equivalently,
\[
\Gend^N(X)=N\cap\big\{i^N_F(f)(t)\bigm|f\in N\wedge t\in [X]^{<\om}\big\}
\]
(cf.~\S\ref{sec:notation_extenders_and_ultrapowers}). Let $\kappa=\crit(F)$,  $G=F\rest X$ and
 $U_X=\Ult(N|\kappa^{+N},G)$.
Let $\pi:U_X\to U_{\nu}$
 be the natural factor map (here $U_\nu$ is just the special case of $U_X$ when $X=\nu$;
 equivalently, $U_\nu=\Ult(N|\kappa^{+N},F)$). That is, let $\eta$ be the ordertype
 of $\Gend^N(X)\cap\OR$ and $\sigma:\eta\to \Gend^N(X)$ be the isomorphism.
 Then
 \[ \pi([a,f]^{N|\kappa^{+N}}_G)=[\sigma(a),f]^{N|\kappa^{+N}}_{F}.\]
 So
$\Gend^N(X)=\rg(\pi)\cap N$.

Now suppose further that $\max(x)+1=\nu$ and let
$\pi(\nu_X)=\nu$. Then we will define a structure $N_X$ for the language of premice excluding the constant symbol $\dot{F}_{\downarrow}$, and an embedding
$\pi_X:N_X\to N$
with $\rg(\pi_X)=\Gend^N(X)$.  Let $F_X$
be the trivial completion of (the transitive collapse of)
$G$. So $\nu(F_X)=\nu_X$. We then define
\[ N_X=(U_X|\nu_X^{+U_X},\widetilde{F_X}), \]
where $\widetilde{F_X}$ is the amenable code of $F_X$
%conf
(as in \cite[between 2.9 and 2.10]{outline}), and let
$\pi_X:N_X\to N$
be the restriction $\pi_X=\pi\rest N_X$.
\end{dfn}
\begin{lem}\label{lem:type_2_factor_embedding}
Let $N,\nu,X,x,\kappa$, etc, be as in \ref{dfn:type_2_factor_embedding} with $\max(x)+1=\nu$.
Then $\crit(\pi_X)\geq\kappa$.
Suppose  $\crit(\pi_X)>\kappa$.\footnote{It can
be that $\crit(\pi_X)=\kappa$ if
there are superstrong extenders on $\es^N$. If $\crit(\pi_X)=\kappa$ then
$\rg(\pi_X)\cap\widetilde{F^N}=\emptyset$, and so $\pi_X$ is not
$\Sigma_1$-elementary in the
language with $\dot{\in},\dot{F}$.} Then:
\begin{enumerate}[label=\arabic*.,ref=\arabic*]
 \item $N_X$ is a pre-ISC-premouse.
 \item $\pi_X$ is cofinal in $\OR^N$ and is
$\Sigma_1$-elementary in the language of premice excluding the constant symbol
$\dot{\Fseg}$.\footnote{Recall that this includes
	constant symbols for the sup of generators of the active extender, so $\pi_X(\nu(F^{N_X}))=\nu(F^N)$.
(It also includes a constant symbol for the critical point, but by hypothesis, $\crit(\pi_X)>\kappa=\crit(F^{N_X})=\crit(F^N)$,
so $\pi_X(\kappa)=\kappa$ already anyway.)}
\item $F_X$ has largest
 generator $\nu_X-1>\kappa^{+N_X}$,
and is not type Z.
\item\label{item:pi_X_rSigma_1-elem_iff_etc} The following are equivalent:
\begin{enumerate}[label=--]
\item $N_X$ can be expanded to a structure $N'_X$ for the full premouse language such that $\pi_X:N_X'\to N$ is  $\rSigma_1$-elementary \tu{(}for this language\tu{)}.
\item $\Fseg^N\in\rg(\pi_X)$.
\item $N_X$ satisfies the weak ISC.\footnote{That is,
the largest non-type Z segment of $F^{N_X}\rest(\nu_X-1)$
is in $N_X$.}
\item $N_X$ satisfies the ISC \tu{(}so $N_X$ is a premouse\tu{)}.
\end{enumerate}
\end{enumerate}
\end{lem}
\begin{proof}
This is mostly as in
%conf
	\cite[Claim 1, proof of Theorem 3.2(A), p.~176]{combin}.
But part \ref{item:pi_X_rSigma_1-elem_iff_etc} was not covered there, so we just discuss that. It is straightforward
to see that if
$\Fseg^N\in\rg(\pi_X)$
then $N_X$ satisfies the ISC and $\pi_X(\Fseg^{N_X})=\Fseg^N$ (and so $\pi_X$
is
$\rSigma_1$-elementary). Now suppose that $N_X$ satisfies the weak ISC,
as witnessed by $G\in N_X$. If $F^{N_X}\rest(\nu_X-1)$ is non-type Z,
so $G$ is the trivial completion of that extender,
then it is straightforward
see that $\pi_X(G)=\Fseg^N$
(and this is also non-type Z).
Now suppose that $G'=F^{N_X}\rest(\nu_X-1)$ is type Z,
with largest generator $\gamma$,
so $G$ is the trivial completion of $G'\rest\gamma$,
 $\gamma$ is a limit of generators of $F^{N_X}$,
and $\gamma^{+U}=\gamma^{+U'}$
where $U=\Ult(N_X,G)$ and $U'=\Ult(N_X,G')$.
The elementarity of $\pi_X$
gives that $\pi_X(\gamma)$ is a generator
of $F$ and $\pi_X(G\rest\gamma)=F\rest\pi_X(\gamma)$,
so $\pi_X(\gamma)$ is also a limit of generators of $F$.
So it suffices to see that $\lh(G)=\gamma^{+U}=\nu_X-1$, because
then
\[ \lh(\pi_X(G))=\pi_X(\gamma)^{+\Ult(N,\pi_X(G))}=\nu-1,\]
which easily implies $F\rest(\nu-1)$ is type
Z and $\pi_X(G)=\Fseg^N\in\rg(\pi_X)$.
But we have the factor map
 $\sigma:U'\to U''$ where $U''=\Ult(N_X,F^{N_X})$,
and note $\crit(\sigma)=\nu_X-1$.
So if $\gamma^{+U'}<\nu_X-1$
then
\[ \gamma^{+U}=\gamma^{+U'}=\gamma^{+U''}=\gamma^{+N_X},\]
but as $G\in N_X$, in fact $\gamma^{+U}<\gamma^{+N_X}$.
So $\gamma^{+U}=\gamma^{+U'}=\nu_X-1$, as desired.
\end{proof}

The following lemma is a slight generalization
%conf
of \cite[Corollary 2.17]{extmax}, which was an
%conf
improvement of \cite[Lemma 4.4]{combin}.
\begin{lem}\label{lem:tau=rho_1}
	Let $M$ be a type 2 premouse, $\mu=\crit(F^M)$,
	and
	$\gamma=\max(\rho_1^M,\mu^{+M})$.
	Then $\tau^M\geq\gamma$.
	If $M$ is either Dodd-amenable
	or $1$-sound then
	$\tau^M=\gamma$.
\end{lem}
\begin{proof}
	To see that $\rho_1^M\leq\tau^M$,
	observe  $F^M\rest(\tau^M\cup t^M)\notin M$.
	Therefore $\gamma\leq\tau^M$ by definition.
	If $M$ is Dodd-amenable, we get $\tau^M=\gamma$
	by using Lemma \ref{lem:type_2_factor_embedding} as in  the proof of
	%conf
	\cite[2.17]{extmax}.
	If $M$ is $1$-sound, note that by
	Lemma \ref{lem:type_2_factor_embedding},
	$\rho_1^M\cup t$ generates $F^M$ for some $t\in[\OR^M]^{<\om}$,
	and deduce that $\tau^M\leq\gamma$.
\end{proof}

The simple lemma below is a condensation-based variant of the initial segment condition, which will be useful
in \S\ref{sec:capturing_elements}, where we will develop
some  key technical setup for our proofs of super-Dodd-soundness and Projectum-finite-generation  (Theorem \ref{thm:finite_gen_hull}).
\begin{lem}\label{lem:sub-extender} Let $M$ be a
$(0,\om_1+1)$-iterable $1$-sound type 2 premouse, $\mu=\crit(F^M)$
and $\nu=\nu(F^M)$.
Suppose $\mu^{+M}<\tau^M$.
Then for every
$r\in\nu^{<\om}$ and $\theta<\tau^M$ there is
$g\in\nu^{<\om}$ and $\Mbar\pins M$ such that $r\un\{\mu,\nu-1\}\sub
g$, $F^M_\downarrow$ is generated by $g$, $\Mbar$ is active type 2 and
$F^{\Mbar}$ is the trivial completion of $F^M\rest g\un\theta$.
\end{lem}
\begin{proof}
By Lemma \ref{lem:tau=rho_1},
 $\rho_1^M=\tau^M>\mu^{+M}$. We may assume
$\theta\geq\mu^{+M}$ and $\theta$ is
%conf
a cardinal in $M$. Applying \cite[Lemma 2.3]{V=HODX_pub} with $\theta$ and
$r$, we get an $\Mbar,g$ as required.
\end{proof}

\begin{rem} Let $M,\theta$, etc be as in
\ref{lem:sub-extender},
with $\mu^{+M}\leq\theta$. Then
$\rho<\OR^\Mbar<\rho^{+M}$, and because $\kappa,\nu,F^M_\downarrow$ are
generated by $g$
and by Lemma \ref{lem:type_2_factor_embedding}, the natural
$\pi:\Mbar\to M$ is a
$0$-embedding with
$\crit(\pi)\geq\theta$.\end{rem}
Under mild large cardinal hypotheses, it can be that
every extender in  the sequence of a mouse $M$ is Dodd-sound, and yet iteration trees on $M$ use non-Dodd-sound extenders. We next analyse the nature of such ``Dodd-unsoundness'', using
the preservation of Dodd-solidity parameters and projecta described in Lemma \ref{lem:type_2_s,sigma_pres}.

We now analyze the
relationship between
$p_1^M$ and $t^M$, in a certain context;
this is very similar to that in Jensen indexing,
due to Zeman, in \cite{zeman_dodd}.
We will use the analysis in the proof of Dodd-solidity.

\begin{lem}[Dodd parameter characterization]\label{lem:Dodd_param_p_1} Let $M$
be a type 2 premouse
with $F^M$ finitely generated;
so $\rho_1^M\leq\kappa^{+M}$ where $\kappa=\crit(F^M)$.
Suppose $M$ is $\kappa^{+M}$-sound.
Let  $p=p_1^M\cut\kappa^{+M}$, $\ttilde=\ttilde^M$ and $\xi\in\OR$.
 Then:
 \begin{enumerate}[label=\arabic*.,ref=\arabic*]
   \item\label{item:p_sub_t}  $p\sub\ttilde$,
   \item\label{item:t_charac} $\ttilde$ is the least tuple $t'$ such that
$t'$ generates \tu{(}that is, $F^M$-generates\tu{)}
$(p,F_\downarrow^M)$,
\item $u=\ttilde\cut p$ is the least tuple $u'$ such that $u'\cup p$
generates $F^M_\downarrow$
\tu{(}therefore if $\xi\in u$ then $\xi\cup(\ttilde\cut(\xi+1))$ does not
generate $F^M_\downarrow$\tu{)}.
\item\label{item:if_zeta_in_p}
if $\xi\in p$ then:
\begin{enumerate}[label=\tu{(}\alph*\tu{)}]
\item $(\ttilde\cut p)\cut(\xi+1)=s\cut(\xi+1)$
where $s$ is least such that $s\cup(p\cut(\xi+1))$
generates $F^M_\downarrow$.
\item $\ttilde\cut(\xi+1)\in\Hull_1^M(\{p\cut(\xi+1)\})$.
 \end{enumerate}
 \end{enumerate}
\end{lem}
\begin{proof}
Part \ref{item:t_charac}:
Certainly $\ttilde$ generates $(p,F_\downarrow^M)$,
since $\ttilde$ generates every element of $M$.
Now let $t'\leq \ttilde$ and suppose $t'$ generates $(p,F_\downarrow^M)$.
Let $X=t'$ and
 $\pi_X:M_X\to M$  be as in Definition \ref{dfn:type_2_factor_embedding}, so $\rg(\pi_X)=\Gend^M(X)$.
Then $p,F_\downarrow^M\in\rg(\pi_X)$.
Since $\kappa=\crit(F_\downarrow^M)$,
therefore $\kappa\in\rg(\pi_X)$,
which easily gives that $\kappa^{+M}\sub\rg(\pi_X)$.
So by Lemma \ref{lem:type_2_factor_embedding}, $M_X$ is a premouse and $\pi_X$
is $\rSigma_1$-elementary (in the full premouse language)
with
\[ \kappa^{+M}\cup\{p\}\sub \Gend^M(X)=\rg(\pi_X).\]
So $p_1^M\in\rg(\pi_X)$, and since $M$ is $1$-sound
with $\rho_1^M\leq\kappa^{+M}$,
therefore
$M\sub\rg(\pi_X)$, so
$t'$ generates $F^M$.
Therefore $t'=\ttilde$.

We now prove the remaining parts together.
Let
\[ p=p_1^M\cut\kappa^{+M}=\{\eta_0>\eta_1>\ldots>\eta_{n-1}\}\]
and $\eta_n=0$ (here $n=0$ if $p=\emptyset$).

\begin{clm*} For each $i\leq n$, we have:
\begin{enumerate}[label=\arabic*.,ref=\arabic*]
 \item $p\rest i=\{\eta_0,\ldots,\eta_{i-1}\}\sub\ttilde$,
 \item\label{item:t_cut_p_rest_i_charac}$(\ttilde\cut(p\rest
i))\cut(\eta_i+1)=s\cut(\eta_i+1)$,
 where $s$ is the least tuple such that $s\cup(p\rest i)$ generates
 $F^M_\downarrow$.
\end{enumerate}\end{clm*}
\begin{proof}The proof is by induction on $i\leq n$.
Trivially $p\rest0\sub\ttilde$.

Now fix $i\leq n$ and suppose $p\rest i\sub\ttilde$,
and let $s$ be the least tuple as in clause \ref{item:t_cut_p_rest_i_charac}.
Since $\ttilde$ generates $F^M$,
clearly $s'=\ttilde\cut(p\rest i)$ is such that $s'\cup(p\rest i)$
generates $F^M_\downarrow$.
So suppose $s\cut(\eta_i+1)<(\ttilde\cut(p\rest i))\cut(\eta_i+1)$.
Then note that $s\cup p$ generates $(F^M_\downarrow,p_1^M)$
(here $s\cup p$ generates $\kappa$ because it generates $F^M_\downarrow$,
and hence generates all points ${\leq\kappa^{+M}}$).
Therefore $s\cup p$ generates $F^M$, but note $s\cup p<\ttilde$, a
contradiction.

Now suppose  $i<n$, i.e.~$\eta_i>\kappa^{+M}$.
We show  $\eta_i\in \ttilde$.
Suppose not. Then
\[ \eta_i\in\Hull_1^M(\eta_i\cup\{\ttilde\cut(\eta_i+1)\}).\]
But  by part \ref{item:t_cut_p_rest_i_charac} of the claim,
$\ttilde\cut(\eta_i+1)\in\Hull_1^M(\{p\rest i\})$.
(Use the standard trick
for minimization.
That is, we clearly get some $s_0\in\Hull_1^M(\{p\rest i\})$
such that $s_0\cup(p\rest i)$ generates
$F^M_\downarrow$. But if $s_0$ is not the least,
then we also get some $s_1<s_0$ with this property, and so on.
Eventually some $s_m$ is the least.)
Therefore
\[ \eta_i\in\Hull_1^M(\eta_i\cup\{p\rest i\}),\]
contradicting the minimality of $p_1^M$.
\end{proof}

The remaining parts of the lemma follow easily from the claim
and its proof.
\end{proof}

\begin{rem}\label{rem:Dodd-appropriate}
	In  \S\S\ref{sec:mim}--\ref{sec:finite_gen_hull}, we will need to understand the action of iteration maps on Dodd parameters and projecta, and variants thereof. Suppose
	 $M$ is a Dodd-sound type 2 premouse
	and $E$ is a weakly amenable extender over $M$,
	with $\tau^M\leq\crit(E)$. Suppose that $N=\Ult_0(M,E)$ is wellfounded. Then $(t^N,\tau^N)$ relates
	to $(t^M,\tau^M)$, but the relationship depends heavily on $E$, as, for example, all generators of $E$ are generators of $F^N$. (So, for example, letting $j=i^{M,0}_E:M\to N$, if $E$ is a normal measure and $\kappa=\crit(E)$ then $t^N=j(t^M)\cup\{\kappa\}$ and $\tau^M=\tau^N$;
	if $E$ is type 3 then $t^N=j(t^M)\cut\nu(E)$ and $\tau^N=\nu(E)$.) A certain variant of the Dodd parameter and projectum, considered in \cite{extmax}, is preserved in a fashion more analogous to the standard parameter and projectum. We recall that next, and define the ``super-'' variant thereof.

	Let $M$ be a premouse and $N$ a structure for the premouse language.\footnote{\label{ftn:assume_passive_2}Like in  \ref{dfn:E_rest_x_cup_alpha},
	$F^M,F^N$ are not relevant here, so we may assume $M,N$ are passive, and we may also apply the definition to structures $M',N'$ for a larger language, as long as the reduct $M$ to the language of passive premice is a premouse.}
	%conf
	Recall from \cite[Definition 2.7]{extmax}
	that an embedding $j:M\to N$ is \emph{Dodd-appropriate} iff $j$ is $\Sigma_0$-elementary,
	cardinal preserving, $M||\mu^{+M}=N|\mu^{+N}$ where  $\mu=\crit(j)$ (so $\mu$ is inaccessible in $M$), and there is $\lambda\in\OR\cap\wfp(N)$ such that $\lambda\leq j(\mu)$
	and $E_j\rest\lambda\notin N$. Note that this does not require that $N$ be wellfounded.

	Suppose $j:M\to N$ is Dodd-appropriate and let $E=E_j\rest\xi$ where $\xi\leq j(\mu)$
	and $E_j$ has no generators in $[\xi,j(\mu))$, so $N|j(\mu)=\Ult(M,E)|i^M_E(\mu)$.
	%conf
	Also essentially from \cite[Definition 2.7]{extmax},
	the \emph{Dodd-solidity parameter and projectum}
	of $j$ (or of $E$), denoted $(s_j,\sigma_j)$, are the least $(s,\sigma)\in\widetilde{\OR}$ such that
	$\sigma\geq\mu^{+M}$ and $E\rest s\cup \sigma\notin N$. We also write $(s_E,\sigma_E)=(s_j,\sigma_j)$.
	\end{rem}

\begin{dfn}\label{dfn:super-Dodd-solidity_parameter}
Let $M$ be a premouse and $j:M\to N$ be Dodd-appropriate.\footnote{Remarks analogous to those in  Footnote \ref{ftn:assume_passive} hold here also.}
Let $\mu=\crit(j)$ and let $E=E_j\rest \xi$ where $\xi\leq j(\mu)$ and $E$ has no generators in $[\xi,j(\mu))$.

 %conf
Almost as in \cite[Definition 2.29]{extmax}, the \dfnemph{Dodd-solidity core} of $E$,\footnote{In
\cite{extmax},
the  \emph{Dodd-core} was formally defined to be the transitive collapse of $E\rest\sigma\cup s$ (as opposed to its trivial completion, but these are essentially equivalent). That terminology was not so good,
as the \emph{Dodd-core} of extender $E$ should, by analogy
with the standard projectum and parameter,
be defined by restricting the extender to $\tau'\cup t'$,
where $\tau'$ is least such that for some $t''$,
$E\rest(\tau'\cup t'')\notin\Ult(M,E)$,
and $t'$ is then the minimal witness for $\tau'$. So we opted here for \emph{Dodd-solidity core} instead.}
denoted $\core_{\ds}(E)$, is the extender $\trivcom(E\rest \sigma_E\cup s_E)$.
And if $E=F^P$ for some active premouse $P$, then $(s^P,\sigma^P)$ denotes $(s_E,\sigma_E)$
and
 $\core_{\ds}(P)$ denotes the pre-ISC-pm (Definition \ref{dfn:pre-ISC-premouse}) $\bar{P}$ such that $\bar{P}|\mu^{+\bar{P}}=P|\mu^{+P}$, $F^{\bar{P}}=\core_{\ds}(E)$
 and $\nu(F^{\bar{P}})\geq\delta$, where $\delta$ is the largest cardinal of $\bar{P}$ (this works via calculations as in Definition \ref{dfn:type_2_factor_embedding}
 and Lemma \ref{lem:type_2_factor_embedding}).

The \dfnemph{super-Dodd-solidity
parameter} and
\dfnemph{projectum} of $j$ (or of $E$), denoted
$(\stilde_j,\sigmatilde_j)=(\stilde_E,\sigmatilde_E)$, are the least
$(\stilde,\sigmatilde)\in\widetilde{\OR}$ such that
$E\rest \stilde\un\sigmatilde\notin N$.
And if $E=F^P$ for some active premouse $P$, then we write
$(\stilde^P,\sigmatilde^P)=(\stilde_E,\sigmatilde_E)$.
\end{dfn}

Note that if $j:M\to N$ is Dodd-appropriate and $E,\mu$ as in \ref{dfn:super-Dodd-solidity_parameter} then
either
$\sigmatilde_E>\mu^{+M}$ or $\sigmatilde_E=0$.
Various facts
regarding
Dodd-solidity
parameters  and projecta established in \cite{extmax}
generalize to mice with superstrongs, and to super-Dodd-solidity
parameters and
projecta.
Note also that the (super-)Dodd-solidity parameter and projectum is an analogue
of $(z_{k+1},\zeta_{k+1})$ from Definition \ref{dfn:z,zeta}.
 The following facts are clear:

\begin{fact}\label{fact:Dodd-sound_characterization}
Let $M$ be a type 2 premouse.
Then:
\begin{enumerate}[label=--]\item  $(s^M,\sigma^M)\leq(t^M,\tau^M)$,
	\item  $M$ is Dodd-sound iff $(s^M,\sigma^M)=(t^M,\tau^M)$,
	\item $(\widetilde{s}^M,\widetilde{\sigma}^M)\leq(\widetilde{t}^M,\widetilde{\tau}^M)$,
\item  $M$ is super-Dodd-sound iff
$(\widetilde{s}^M,\widetilde{\sigma}^M)=(\widetilde{t}^M,\widetilde{\tau}^M)$.
\end{enumerate}
\end{fact}
We now give a couple of lemmas which describe how the (super-)Dodd-solidity parameter and projectum are shifted by iteration maps.
\begin{lem}\label{lem:s_j,sigma_j_pres}
 Let $M,U$ be premice, $i:M\to U$ be Dodd-appropriate and
 $\kappa=\crit(i)$.
Let $u<\om$ be such that $U$ is $u$-sound. Let  $F$ be weakly amenable to $U$
with
$\kappa<\crit(F)<\rho_u^U$.
Suppose $W=\Ult_u(U,F)$  is wellfounded.
Let $j=i^{U,u}_F:U\to W$. Then
\begin{equation}\label{eqn:s_sigma_pres_j_com_i} s_{j\com i}=j(s_i) \text{ and }\sigma_{j\com i}=\sup
j``\sigma_i, \end{equation}
\begin{equation}\label{eqn:stilde_sigmatilde_pres_j_com_i} \widetilde{s}_{j\com i}=j(\widetilde{s}_i)\text{ and
}\widetilde{\sigma}_{j\com i}=\sup
j``\widetilde{\sigma}_{j\com i}.\end{equation}

Moreover, if $j':U\to W'$ is the iteration map of some wellfounded abstract
degree $u$
weakly amenable iteration on $U$ \tu{(}Definition \ref{dfn:abstract_iteration}\tu{)}, via extenders $F'$ each with
$\crit(F')>\kappa$, then $j'$ preserves
$(s,\sigma)$ and $(\widetilde{s},\widetilde{\sigma})$  analogously; that is, lines \tu{(}\ref{eqn:s_sigma_pres_j_com_i}\tu{)} and \tu{(}\ref{eqn:stilde_sigmatilde_pres_j_com_i}\tu{)} hold after replacing $j$ with $j'$ throughout.
\end{lem}
\begin{proof}
We literally just discuss $(\widetilde{s},\widetilde{\sigma})$-preservation for $j:U\to W$.
For $(s,\sigma)$ it is almost the same,
and the last paragraph of the lemma is an easy corollary.

Let $E=E_i\rest i(\kappa)$, so $E\notin U$.
If $\crit(j)>i(\kappa)$ then (by weak amenability) the conclusion is
immediate,
so suppose $\crit(j)\leq i(\kappa)$. Let $G=E_{j\com i}\rest j(i(\kappa))$.
If $X\sub i(\kappa)$ and $E\rest X\in U$, then
$j(E\rest X)=G\rest j(X)$,
as $\kappa^{+M}<\crit(j)$. So we just need to see
that $W$ does not contain any
fragments of $G$ which are too large.

Suppose that $\widetilde{\sigma}_i>\kappa^{+M}$. It suffices to see
that
\[ G\rest(j(\widetilde{s}_i)\cup\sup j``\widetilde{\sigma}_i)\notin W.\]
But this follows from Lemma \ref{lem:amenability_pres},
applied to the set
$E\rest(\widetilde{s}_i\cup\widetilde{\sigma}_i)$,
 coded as a subset of $\widetilde{\sigma}_i$ (which is  appropriately amenable
to
$U$,
 but not in $U$).

If instead $\widetilde{\sigma}_i=0$, then the measure
$E\rest\widetilde{s}_i\notin
U$,
but this is a subset of $\kappa^{+M}$, and $\pow(\kappa^{+M})\cap
U=\pow(\kappa^{+M})\cap W$, so we are done.
\end{proof}

We also need the following slight variant of the preceding lemma:

\begin{lem}\label{lem:type_2_s,sigma_pres}
 Let $M$ be a type 2 premouse.
Let $d\in\{0,1\}$ be such that $M$ is $d$-sound,
and if $d=1$, suppose $M$ is Dodd-sound.
Let $E$ be weakly amenable to $M$ with $\crit(E)<\rho_d^M$.
 Suppose that $M'=\Ult_d(M,E)$ is wellfounded.
Let $i=i^{M,d}_E$.
 Let $F=F^M$ and $F'=F^{M'}$.
 Then:
 \[ s_{F'}=i(s_F)\text{ and }\sigma_{F'}=\sup i``\sigma_F, \]
 \[ \widetilde{s}_{F'}=i(\widetilde{s}_F)\text{ and
}\widetilde{\sigma}_{F'}=\sup
i``\widetilde{\sigma}_F.\]
\end{lem}

\begin{proof}[Proof sketch]
%conf
A complete argument is given in \cite[\S2]{extmax} (formally below
superstrong), but we give a sketch here.
We literally just discuss $(\widetilde{s},\widetilde{\sigma})$-preservation.

If $d=0$, the proof is basically as in the proof of Lemma
\ref{lem:s_j,sigma_j_pres}.
(Maybe $\crit(E)\leq\crit(F)$, but that is fine,
considering the definition of $F'$. If $\widetilde{\sigma}_F=0$
then the measure $F\rest\widetilde{s}_F\notin M$,
and this is coded amenably as a subset of $\mu^{+M}$ where $\mu=\crit(F)$.)

Suppose $d=1$. So $\crit(E)<\rho_1^M$.
Suppose $\mu^{+M}<\rho_1^M$,
so by Lemma \ref{lem:tau=rho_1}
and Fact \ref{fact:Dodd-sound_characterization},
$\rho_1^M=\tau^M=\sigma^M=\tautilde^M=\sigmatilde^M$
and $t^M=s^M=\ttilde^M=\stilde^M$.
 We have  $p_1^{M'}=i(p_1^{M})$
and $\rho_1^{M'}=\sigma'$ where $\sigma'=\sup i``\sigmatilde^M$, so
\[ M'=\Hull_1^{M'}(i(p_1^{M'})\cup\sigma').\]
But because $M$ is Dodd-sound,
$i(F^M_\downarrow,p_1^M)=(F^{M'}_\downarrow,p_1^{M'})$
is generated
by $F'\rest s'\cup\sigma'$ where $s'=i(\stilde^M)$.
But then by Lemma \ref{lem:type_2_factor_embedding},
$F'\rest(s'\cup\sigma')$ generates all of $M'$,
so $F'\rest(s'\cup\sigma')\notin M'$.

Now suppose  $\rho_1^M\leq\mu^{+M}$. Since $M$ is Dodd-sound,
$\tau^M=\sigma^M=\mu^{+M}$
and $t^M=s^M$ and $\sigmatilde^M=0$ and either $\stilde^M=s^M$
or $\stilde^M=s^M\cup\{\mu\}$.
Either way, $F^M\rest\stilde^M\notin M$,
which is an $\rPi_2^M$ statement
about the parameter $\stilde^M$,
and since $i$ is $\rSigma_2$-elementary,
therefore $F^{M'}\rest s'\notin M$
where $s'=i(\stilde^M)$. It follows
that $\stilde^{M'}=s'$ and $\sigmatilde^{M'}=0$.
\end{proof}

If $\Tt$ is a $k$-maximal iteration tree on a $k$-sound premouse $M$,
the extenders $E^\Tt_\alpha$ need not be Dodd-sound, even if every extender in $\es_+^M$ is Dodd-sound. But often in cases of interest (always, if every extender in $\es_+^M$ is Dodd-sound) $E^\Tt_\alpha$ has a natural decomposition into a sequence of Dodd-sound extenders.
That sequence can be derived in a fairly simple manner from extenders used in $\Tt$. We discuss this next.

\begin{dfn}\label{dfn:Dodd-nice}
Let $M$ be an $m$-sound  premouse all of whose proper segments are Dodd-sound,
and $\Tt$ be  $m$-maximal  on $M$.
Let $\beta+1<\lh(\Tt)$. We say $\beta$ is
\dfnemph{Dodd-nice (for $\Tt$)}
iff
either
 $(0,\beta]^\Tt$ drops in model,
 or $E^\Tt_\beta\in\es(M^\Tt_\beta)$,
 or $F^M$ is Dodd-sound.

Similarly, let $\Tt$ a degree-maximal tree on
a bicephalus $B=(\rho,M^0,M^1)$ where
all proper segments of $M^0,M^1$
are Dodd-sound. Let $\beta+1<\lh(\Tt)$.
We say $\beta$ is \dfnemph{Dodd-nice (for $\Tt$)}
iff either
 $(0,\beta]^\Tt$ drops in model, or letting $e=\exitside^\Tt_\beta$,
 we have
   $E^\Tt_\beta\in\es(M^{e\Tt}_\beta)$ or
 $M^e$ is Dodd-sound.

Let $\Tt$ be a tree as in one of the cases above.
Let $\lambda<\lh(\Tt)$. We say that $\Tt$ is
\dfnemph{${\geq\lambda}$-Dodd-nice} iff $\beta$
is Dodd-nice for
every $\beta\geq\lambda$ with $\beta+1<\lh(\Tt)$.
If $b^\Tt$ exists, we
say that $b^\Tt$ is \dfnemph{${\geq\lambda}$-Dodd-nice} iff
$\beta$ is Dodd-nice for every
$\beta\geq\lambda$ with $\beta+1\in b^\Tt$.
\end{dfn}

\begin{rem}\label{rem:Dodd-nice_not_implies_Dodd-sound}
The Dodd-niceness of $\beta$ does \emph{not} imply
that $E^\Tt_\beta$ is Dodd-sound (assuming some large cardinals).
For example, suppose that $M$ is active type 2, $(0,\om_1+1)$-iterable,
all initial segments of $M$ are Dodd-sound,
and there is
 $\kappa$ such that $\tau^M\leq\kappa<\kappa^{+M}<\OR^M$,
and $\kappa$ is $M$-measurable via some $E\in\es^M$.
Let $\Tt$ be the $0$-maximal tree on $M$ with $E^\Tt_0=E$ and $E^\Tt_1=F(M^\Tt_1)$.
Then $1$ is Dodd-nice but $E^\Tt_1$ is not Dodd-sound, by Fact \ref{fact:Dodd-sound_characterization} and Lemma \ref{lem:type_2_s,sigma_pres}. Indeed, letting
$F=F^M$ and $F'=F(M^\Tt_1)=E^\Tt_1$, by \ref{fact:Dodd-sound_characterization} and \ref{lem:type_2_s,sigma_pres},  $\tau_F=\sigma_F=\sigma_{F'}\leq\kappa=\crit(i^\Tt_{01})$ and $s_{F'}=i^\Tt_{01}(s_F)=i^\Tt_{01}(t_F)$. But then it easily follows that
the elements of $M^\Tt_1$ which are $F'$-generated by $\sigma_{F'}\cup s_{F'}$ are precisely those in $\range(i^\Tt_{01})$, and so $E^\Tt_1$ is not Dodd-sound (since Dodd-soundness would require that all elements of $M^\Tt_1$ were $F'$-generated by $\sigma_{F'}\cup\{s_{F'}\}$).

The next lemma analyzes failures of Dodd-soundness in general; is a routine variant of results in \cite[\S2]{extmax} (particularly
%conf
\cite[Remark 2.30]{extmax}), so we leave the direct adaptation to the reader.
\end{rem}

\begin{lem}\label{lem:Dodd_cores_appear}
Let either
\begin{enumerate}[label=--]
\item  $m<\om$, $M$ be an $m$-sound
premouse all of whose proper segments are Dodd-sound,
and $\Tt$   be
$m$-maximal
on $M$, or
\item
 $B=(\rho,M^0,M^1)$ be a  bicephalus
such that all proper segments
of $M^0,M^1$  are Dodd-sound,
and $\Tt$  degree-maximal on $B$.
\end{enumerate}
Let $\beta+1<\lh(\Tt)$ be such that
$\beta$ is Dodd-nice for $\Tt$. Let $e=\exitside^\Tt_\beta$ and
 $E=E^\Tt_\beta$. Then $E$ fails to be Dodd-sound iff we have:
\begin{enumerate}[label=\tu{(}\roman*\tu{)}]
\item $M^{e\Tt}_\beta$ is active type 2 and $E=F(M^{e\Tt}_\beta)$, and
\item there is $\vareps+1\leq^\Tt\beta$
such that $(\vareps+1,\beta]^\Tt\cap\dropset^\Tt=\emptyset$  \tu{(}so $M^{e*\Tt}_{\vareps+1}$ is active type 2\tu{)}
and $\crit(E^\Tt_\vareps)\geq\sigma_{\bar{E}}$ where $\bar{E}=F(M^{e*\Tt}_{\vareps+1})$.
\end{enumerate}

Moreover, suppose $E$ is non-Dodd-sound and $\vareps$ is least as above. Let $P=M^{e\Tt}_\beta$, $\bar{P}=M^{e*\Tt}_{\varepsilon+1}$, $\bar{E}=F^{\bar{P}}$ and $\delta=\pred^\Tt(\vareps+1)$. Then:
\begin{enumerate}[label=\arabic*.,ref=\arabic*]
 \item\label{item:E-bar=Dodd-core_and_P-bar=Dodd-core} $\bar{E}=\core_{\ds}(E)$ and $\bar{P}=\core_{\ds}(P)$.
  \item $\bar{E}=F^{\bar{P}}$ is Dodd-sound.
 \item $\delta$ is the unique $\delta'<\lh(\Tt)$ such that $\bar{P}\ins M^{e\Tt}_{\delta'}$.
\item
$\deg^\Tt(\eps+1)=0$.
\item\label{item:tau_E-bar=sigma_E-bar=sigma_E} $\rho_1^{\bar{P}}\leq\tau_{\bar{E}}=\sigma_{\bar{E}}=\sigma_E\leq\crit(E^\Tt_\varepsilon)=\crit(i^{e*\Tt}_{\vareps+1,\beta})$.

 \item\label{item:s_E_is_image_of_s_E-bar} $s_E=i^{e*\Tt}_{\eps+1,\beta}(s_{\bar{E}})=i^{e*\Tt}_{\eps+1,\beta}(t_{\bar{E}})$.
\item\label{item:sigma_E<nu(E_delta)<lh(E_delta)<=OR^P-bar}
$\sigma_{E}<\nu(E^\Tt_\delta)<\lh(E^\Tt_\delta)\leq\OR^{\bar{P}}$.
 \item\label{item:delta_least_with_sigma_E<lh(E_delta)} $\delta$ is the least $\delta'$ such that $\sigma_E<\lh(E^\Tt_{\delta'})$.
\end{enumerate}
\end{lem}

Note that if $E=E^\Tt_\beta\in\es_+(M^\Tt_\beta)$ is Dodd-sound, then
there is also a unique
$\xi$ such that
$\core_\ds(E)\in\es_+(M^{e\Tt}_\xi)$, because $\core_\ds(E)=E$; in
fact, $\xi=\beta$.

\begin{proof}
 The characterization of Dodd-soundness
 and parts \ref{item:E-bar=Dodd-core_and_P-bar=Dodd-core}--\ref{item:s_E_is_image_of_s_E-bar} in the ``Moreover'' clause follow by a straightforward induction, using Lemma
\ref{lem:type_2_s,sigma_pres}.
Let us just discuss parts \ref{item:sigma_E<nu(E_delta)<lh(E_delta)<=OR^P-bar}
and \ref{item:delta_least_with_sigma_E<lh(E_delta)},
assuming the rest. Part \ref{item:sigma_E<nu(E_delta)<lh(E_delta)<=OR^P-bar}:
We have $\sigma_{\bar{E}}<\crit(E^\Tt_\varepsilon)<\nu(E^\Tt_\delta)$ by hypothesis
and normality, and $\sigma_E=\sigma_{\bar{E}}$ by part \ref{item:tau_E-bar=sigma_E-bar=sigma_E}. And $\nu(E^\Tt_\delta)<\lh(E^\Tt_\delta)\leq\OR^{\bar{P}}$ since $\bar{P}=M^{e*\Tt}_{\vareps+1}$, so $E^\Tt_\vareps$ measures exactly $\pow(\kappa)\cap\bar{P}$
where $\kappa=\crit(E^\Tt_\vareps)$,
but $\rho_1^{\bar{P}}\leq\sigma_{\bar{E}}\leq\kappa$, so $\lh(E^\Tt_\delta)\leq\OR^{\bar{P}}$.
Part \ref{item:delta_least_with_sigma_E<lh(E_delta)}:
Let
$\delta'$ be least
with $\sigma_E<\lh(E^\Tt_{\delta'})$. So $\delta'\leq\delta$ by part  \ref{item:sigma_E<nu(E_delta)<lh(E_delta)<=OR^P-bar};
we need to see $\delta'=\delta$. This is like part of the proof of Closeness \cite[Lemma 6.1.5]{fsit}.
Suppose that $\delta'<\delta$ and let $\gamma$ be least
such that $\gamma\geq\delta'$ and $\gamma+1\leq^\Tt\delta$. So $\sigma_E<\lh(E^\Tt_{\delta'})\leq\lh(E^\Tt_\gamma)$.
We have $\bar{P}\ins M^{e\Tt}_{\delta}$ and $\rho_1^{\bar{P}}\leq\sigma_{\bar{E}}=\sigma_E<\lh(E^\Tt_{\delta'})$. It follows from the rules of degree-maximal trees that $\bar{P}=M^{e\Tt}_\delta$ (that is, otherwise $\bar{P}=M^{e*\Tt}_{\vareps+1}\pins M^{e\Tt}_\delta$ and $\lh(E^{\Tt}_{\delta'})\leq\OR^{\bar{P}}$, but then $\lh(E^{\Tt}_{\delta'})$ is a cardinal in $M^{e\Tt}_{\delta}$, contradicting the fact that $\rho_1^{\bar{P}}<\lh(E^{\Tt}_{\delta'})\leq\OR^{\bar{P}}$). So by the minimality of $\vareps$,
$M^{e*\Tt}_{\gamma+1}$ is Dodd-sound (and active type 2)
and $\crit(E^\Tt_\gamma)<\sigma(F(M^{e*\Tt}_{\gamma+1}))$. But then by
Lemma \ref{lem:type_2_s,sigma_pres}
and straightforward elementarity considerations in  case $\deg^{e\Tt}_{\gamma+1}>1$, it follows that \[\sigma_{F(M^{e\Tt}_{\gamma+1})}\geq\sup i^{e*\Tt}_{\gamma+1}``\sigma_{F(M^{e*\Tt}_{\gamma+1}))}\geq\lh(E^\Tt_\gamma)\geq\lh(E^\Tt_{\delta'})>\sigma_E.\]Likewise for all $\gamma'$ with $\gamma+1<^\Tt\gamma'+1\leq^\Tt\delta$,
giving that \[\sigma_{\bar{E}}=\sigma_{F(M^{e\Tt}_{\delta})}\geq\sup i^{e*\Tt}_{\gamma+1,\delta}``\sigma_{F(M^{e*\Tt}_{\gamma+1})}>\sigma_E,\] contradicting
part \ref{item:tau_E-bar=sigma_E-bar=sigma_E}.
\end{proof}

We now adapt Dodd ancestry and Dodd
%conf
decomposition, described in \cite[2.31--2.36]{extmax}, to our  context. They provide a decomposition of non-Dodd-sound extenders used in degree-maximal iteration trees into linear sequences of Dodd-sound extenders,
and also, corresponding decompositions of iteration maps. The \emph{Dodd ancestry} relation $\alpha<^\Tt_{\Da}\beta$, introduced first, indicates when an extender $E^\Tt_\alpha$ of $\Tt$  ``hereditarily contributes generators'' to an extender $E^\Tt_\beta$ which are not absorbed into  the Dodd-solidity-core of $E^\Tt_\beta$.

\begin{dfn}\label{dfn:Dodd_ancestry}
Let $(m,M,\Tt)$ or $(B,\Tt)$ be as in \ref{lem:Dodd_cores_appear}. We define the \dfnemph{Dodd ancestry} relation $<^{\Tt}_{\Da}$ on $\lh(\Tt)^-$. Let
$\alpha+1,\beta+1<\lh(\Tt)$. Say $\alpha<^*\beta$ iff
$\beta$ is Dodd-nice for $\Tt$,
$E^\Tt_\beta$ is non-Dodd-sound and  $\vareps+1\leq^\Tt\alpha+1\leq^\Tt\beta$
where $\vareps$
is  least  as in
\ref{lem:Dodd_cores_appear}. Then
$<^\Tt_{\Da}$ denotes
the
transitive closure of $<^*$,
and $\leq^{\Tt}_{\Da}$ the reflexive closure of $<^{\Tt}_{\Da}$.

Given $\alpha\leq^{\Tt}_{\Da}\beta$,
let the \dfnemph{trace} of $(\beta,\alpha)$
be the sequence $(\beta_0,\beta_1,\ldots,\beta_n)$
where $\beta_0=\beta$, $\beta_n=\alpha$
and $\beta_{i+1}<^*\beta_i$ for $i<n$
(so  $\alpha=\beta$ iff $n=0$).
\end{dfn}

\begin{rem}\label{rem:second_example}
It is easy to see that if $\alpha<^\Tt_{\Da}\beta$ (so $\beta$ is Dodd-nice by definition) then $\alpha$ is also
Dodd-nice. In the example
discussed in Remark \ref{rem:Dodd-nice_not_implies_Dodd-sound}, $0<^*1$, so
 $0<^\Tt_{\Da}1$. And as described in the
 %conf
 following lemma (and cf.~\cite[2.27, 2.28]{extmax}), $E^\Tt_1$ is equivalent to the composition $E^\Tt_0\com F^M$, and so
 \[ M^\Tt_2=\Ult_0(M,E^\Tt_1)=\Ult_0(\Ult_0(M,F^M),E^\Tt_0),\]
 and the ultrapower maps agree;
 this decomposes the Dodd-unsound extender $E^\Tt_1$ into composition of two Dodd-sound extenders.
 Let us consider a slightly more complex example.
 With the hypotheses as in the same example,
 suppose further that $E$ is type 2 and there is $\mu$
 which is $M|\lh(E)$-measurable via some $D\in\es^M$,
 and $\tau_E\leq\mu<\mu^{+M|\lh(E)}<\lh(E)$.
 Let $\Uu$ be the $0$-maximal tree with $\lh(\Uu)=4$, $E^\Uu_0=D$,
 $E^\Uu_1=F(M^\Uu_1)$ and $E^\Uu_2=F(M^\Uu_2)$.
 (Note that $M^{*\Uu}_1=M|\lh(E)$
 and $\deg^\Uu_1=0$, since $\rho_1^{M|\lh(E)}\leq\tau_E\leq\mu$, so $M^\Uu_1=\Ult_0(M|\lh(E),D)$ is indeed active with the image of $E$,  $\kappa=\crit(E)=\crit(E^\Uu_1)<\mu$ and $E^\Uu_1=F(M^\Uu_1)$ is $M$-total and $\pred^\Uu(2)=0$ and $M^{*\Uu}_2=M$,
 so $M^\Uu_2=\Ult_0(M,E^\Uu_1)$, and then likewise,
 $M^\Uu_3=\Ult_0(M,E^\Uu_2)$ is active with the image of $F^M$,
 and $\crit(E^\Uu_2)=\crit(F^M)$, so $M^\Uu_3=\Ult_0(M,E^\Uu_2)$.) Then $0<^*1<^*2$,
 and so $0<^{\Uu}_{\Da}1<^\Uu_{\Da}2$ and $0<^\Uu_{\Da}2$. And we have
 \[ M^\Tt_3=\Ult_0(M,E^\Tt_2)=\Ult_0(\Ult_0(\Ult_0(M,F^M),E),D) \]
 and the ultrapower maps agree,
 and again, $E^\Tt_2$ is Dodd-unsound,
 but $F^M,E,D$ are each Dodd-sound.

 The following lemma  generalizes this kind of analysis, decomposing any extender $E^\Tt_\xi$ such that $\xi$ is Dodd-nice for $\Tt$ into a sequence $\left<G_\gamma\right>_{\gamma<\lambda}$ of Dodd-sound extenders, whose composition is equivalent to $E^\Tt_\xi$.
 The sequence will have strictly increasing critical points, but need not be ``normal'', in that we can have $\gamma<\delta<\lambda$ with $\crit(G_\delta)<\nu(G_\gamma)$.
 By composing such sequences,  we in fact analyze a full iteration map $i^\Tt_{\alpha\beta}$ (assuming all the relevant ordinals are Dodd-nice).
\end{rem}

\begin{lem}[Dodd decomposition] \label{lem:Dodd_decomposition}
 Let $(m,M,\Tt)$ or $(B,\Tt)$ be as in \ref{lem:Dodd_cores_appear}.
Let $\alpha<^\Tt\beta$ and $e\in\sides^\Tt_\beta$
be such that
$(\alpha,\beta]^\Tt$ does not drop
 in model or degree, $k=\deg^{e\Tt}_\alpha=\deg^{e\Tt}_\beta$,
 and suppose
 $\xi$ is Dodd-nice for $\Tt$
 for every  $\xi+1\in(\alpha,\beta]^\Tt$.
 Let $\vec{G}=\left<G_\gamma\right>_{\gamma<\lambda}$
 enumerate, in order of increasing critical point, the set of all extenders of
the form $\core_{\ds}(E^\Tt_\delta)$
 where $\delta\leq^\Tt_{\Da}\xi$ for some $\xi+1\in(\alpha,\beta]^\Tt$.
 Then:
 \begin{enumerate}[label=\arabic*.,ref=\arabic*]
 \item Let $\xi+1,\zeta+1\in(\alpha,\beta]^\Tt$
 and $\theta\leq^\Tt_{\Da}\xi$ and $\delta\leq^\Tt_{\Da}\zeta$  with $\theta\neq\delta$. Then $\crit(E^\Tt_\theta)\neq\crit(E^\Tt_\delta)$,
 so the enumeration of $\vec{G}$ mentioned in the lemma statement is well-defined. In fact, let $\vec{\theta}=(\theta_0,\ldots,\theta_m)$ be the trace of $(\xi,\theta)$
 and $\vec{\delta}=(\delta_0,\ldots,\delta_n)$ the trace of $(\zeta,\delta)$. Then:
 \begin{enumerate}[label=\tu{(}\alph*\tu{)}]
 \item If $m<n$ and $\vec{\theta}=\vec{\delta}\rest m$ then
 $\delta<^\Tt_{\Da}\theta$, so $\delta<\theta$ and \[\crit(E^\Tt_\theta)<\sigma_{E^\Tt_\theta}\leq\crit(E^\Tt_\delta)<\nu(E^\Tt_\delta)<\nu(E^\Tt_\theta).\]
 \item If $n<m$ and $\vec{\delta}=\vec{\theta}\rest n$
 then likewise \tu{(}with $\delta,\theta$ exchanged, so $\theta<^\Tt_{\Da}\delta$, etc\tu{)}.
 \item Suppose there is $k<\min(m,n)$ such that $\vec{\theta}\rest k=\vec{\delta}\rest k$ but $\theta_k\neq\delta_k$. If $\theta_k<\delta_k$ then $\theta<\delta$ and \[ \crit(E^\Tt_\theta)<\nu(E^\Tt_\theta)\leq\crit(E^\Tt_\delta)<\nu(E^\Tt_\delta);\]
 if $\delta_k<\theta_k$ then likewise.
 \end{enumerate}
  \item
There is an abstract degree $k$ weakly amenable iteration
 of $M^{e\Tt}_\alpha$ of the form
$\left(\left<N_\gamma\right>_{\gamma\leq\lambda},\left<G_\gamma\right>_{
\gamma<\lambda}\right)$, with abstract iteration maps
$j_{\gamma_1\gamma_2}:N_{\gamma_1}\to N_{\gamma_2}$.
\item $N_\lambda=M^\Tt_\beta$
and $j_{0\lambda}=i^{e\Tt}_{\alpha\beta}$.
\item For $\gamma\in[0,\lambda)$
and $\delta$ with
$G_\gamma=\core_{\ds}(E^\Tt_\delta)$
and $\mu=\crit(E^\Tt_\delta)=\crit(G_\gamma)$, we have:
\begin{enumerate}[label=--]
\item  $E^\Tt_\delta\rest\nu(E^\Tt_\delta)$
is the $(\mu,\nu(E^\Tt_\delta))$-extender
derived from $j_{\gamma\lambda}$.
\item For any $\beta'<\lh(\Tt)$ such that $\beta<^\Tt\beta'$ and $(\beta,\beta']^\Tt\cap\dropset^\Tt=\emptyset$,
 $E^\Tt_\delta\rest\nu(E^\Tt_\delta)$
is the $(\mu,\nu(E^\Tt_\delta))$-extender
derived from $i^\Tt_{\beta\beta'}\com j_{\gamma\lambda}$.
\item Let $j=j_{\gamma\lambda}$ and
$N'=\cHull_{k+1}^{N_\lambda}(\rg(j)\cup\sigma_j\cup\{s_j\})$
and $j':N'\to N_\lambda$ be the uncollapse map.
Then $N'=N_{\gamma+1}$ and $j'=j_{\gamma+1,\lambda}$.
\end{enumerate}
\end{enumerate}
Moreover, if also, $\alpha$ has form $\vareps+1$
where $\vareps$ is Dodd-nice for $\Tt$,
and $\vec{G}^*$ enumerates, in order of increasing critical point, all those extenders either in $\vec{G}$
or of form $\core_{\ds}(E^\Tt_\delta)$
where $\delta\leq_{\Da}^\Tt\vareps$,
then the analogous facts hold of $\vec{G}^*$,
$M^{e*\Tt}_{\vareps+1}$, $i^{e*\Tt}_{\vareps+1,\beta}$ \tu{(}replacing $\vec{G}$, $M^{e\Tt}_{\vareps+1}$, $i^{e\Tt}_{\vareps+1,\beta}$ respectively\tu{)}.
\end{lem}
\begin{proof}
%conf
This is as in \cite[2.32, 2.35, 2.36]{extmax}.
Note that the proof uses the associativity of extenders described in \cite[2.27, 2.28]{extmax};
these facts continue to hold, with the same proofs, without the superstrong-smallness assumption  formally present in \cite{extmax}.
\end{proof}

\begin{dfn}[Dodd decomposition]
In the context of the lemma above,
the sequence
$\vec{G}$
is called the \dfnemph{Dodd decomposition}
of $i^{e\Tt}_{\alpha\beta}$,
and in the context of the ``Moreover'' clause, $\vec{G}^*$ is that of $i^{e*\Tt}_{\vareps+1,\beta}$.
\end{dfn}

\section{Capturing with strongly finite trees}\label{sec:capturing_elements}

In this section we discuss
 constructions of
finite (of finite after some stage)
support trees $\bar{\Tt}$ of iteration trees $\Tt$,
capturing a given $x\in M^\Tt_\infty$
with a copy map $\somevarpi:M^{\bar{\Tt}}_\infty\to M^\Tt_\infty$ having $x\in\rg(\somevarpi)$
like in \S\ref{sec:simple_embeddings},
but  with more specialized properties than those considered there.
In particular, under some Dodd-soundness assumptions, we will show how to arrange that every extender used
in $\bar{\Tt}$ is finitely generated
(and also slight variations thereof).
Recall from  \S\ref{sec:meas_Dodd_prof_fin_plan} that we want such methods for the  argument
in \S\ref{sec:mim}, so the results in this section will be used there. Although we won't literally use the results  after \S\ref{sec:mim}, their proofs  will be used.
This is because will use slight variants of the results in \S\S\ref{sec:Dodd_proof},\ref{sec:finite_gen_hull}.
These variants will be described  in those sections,
but the proofs, being only embellishments on  those presented here, will only be sketched there.

\begin{dfn}\label{dfn:captures}
Let $m<\om$, let $M$ be $m$-sound and let $\bar{\Tt},\Tt$ be terminally non-dropping $m$-maximal trees on $M$.
Suppose $\Ttbar,\Tt$ have successor lengths $\xibar+1,\xi+1$ respectively
and let $x\in\core_0(M^\Tt_\xi)$ and $\alpha<\rho_0(M^\Tt_\xi)$.
Then we say that $\Ttbar$ \dfnemph{captures $(\Tt,x,\alpha)$}
iff there is $\lambda\in b^\Tt\cap b^\Ttbar$ and $\sigma$ such that:
\begin{enumerate}[label=--]
 \item $\Ttbar\rest(\lambda+1)=\Tt\rest(\lambda+1)$,
\item   $\lambda\leq^\Ttbar\xibar$ and $\lambda\leq^\Tt\xi$, and
\item $\sigma:M^{\Ttbar}_{\xibar}\to M^\Tt_\xi$ is an $m$-embedding
with $\{x\}\cup\alpha\sub\rg(\sigma)$
and $\sigma\com i^{\Ttbar}_{\lambda\xibar}=i^{\Tt}_{\lambda\xi}$.
\end{enumerate}
We say that $\Ttbar$ \dfnemph{captures} $(\Tt,x)$ iff $\Ttbar$ captures
$(\Tt,x,0)$.
\end{dfn}

The following definition and lemma will be used in \S\ref{sec:mim}:
\begin{dfn}[Strongly finite]\label{dfn:strongly_finite}
Let $M$ be an $m$-sound premouse
all of whose initial segments
are Dodd-sound (including $M$ itself), and
 $\Tt$ be a terminally-non-dropping successor length $m$-maximal tree on $M$.
We say that $\Tt$ is \dfnemph{strongly finite}
iff
$\lh(\Tt)<\om$
and for each $\alpha+1\in b^\Tt$
and each $\gamma\leq_{\Da}^\Tt\alpha$, $\core_{\ds}(E^\Tt_\gamma)$ is finitely
generated
(so $\sigma_{E^\Tt_\gamma}=\crit(E^\Tt_\gamma)^{+\exit^\Tt_\gamma}$).
\end{dfn}

It is not completely obvious from the definition, but in strongly finite trees $\Tt$, all extenders $E^\Tt_\chi$ used either
 feed generators into the eventual iteration map, in that $\chi\leq^\Tt_{\Da}\vareps$
 for some $\vareps+1\leq^\Tt\infty$,
 or $E^\Tt_\chi$ is used for
a rather
trivial reason
to do with indexing.  In this trivial case, basically $E^\Tt_\chi$
is type 2 and just has to be used in order to reveal the next measure which does contribute generators into the eventual iteration map, and the critical point of that measure is $\lgcd(\exit^\Tt_\chi)$.
(If we were working with $\lambda$-indexing instead of Mitchell-Steel indexing,
and $\Tt'$ were the natural version of $\Tt$ in that hierarchy,
then the $\lambda$-indexed version $E$ of $E^\Tt_\chi$ would not be used in  $\Tt'$, as we could produce the next measure without using $E$.) We make this precise with the following definition and lemma:

\begin{dfn}
Let $M$ be  $m$-sound and $\Tt$ be $m$-maximal  on $M$.
Let $\chi<\chi'<\lh(\Tt)$. We say that \dfnemph{$\chi$ is $\chi'$-transient \tu{(}with respect to $\Tt$\tu{)}}
iff for some $\eta$, we have:
\begin{enumerate}[label=--]\item $M^{\Tt}_{\chi'}$ is active type 2,
\item  $\chi=\pred^\Tt(\eta+1)<^\Tt\eta+1\leq^\Tt\chi'$,
\item $(\eta+1,\chi']^\Tt$ does not drop, and
 \item $\crit(i^{*\Tt}_{\eta+1,\chi'})=\lgcd(M^{*\Tt}_{\eta+1})$.\qedhere
 \end{enumerate}
\end{dfn}
\begin{lem}\label{lem:non-dropping-strongly_finite_Tt_structure}
Let $m,M$ be as in Definition \ref{dfn:strongly_finite}.
Let $\Tt$ be a strongly finite terminally-non-dropping
$m$-maximal tree on $M$.
Then for each $\chi+1<\lh(\Tt)$,  one of the following options holds,
where $\xi+1=\lh(\Tt)$:
\begin{enumerate}[label=\tu{(}\roman*\tu{)}]
\item\label{item:feeds_in_strongly_finite_Tt_structure} there are $\vareps,\chi'$ such that   $ \eps+1\leq^\Tt\xi$ and
$\chi'\leq^\Tt_{\Da}\eps$, and either:
\begin{enumerate}
\item\label{item:feeds_in_strongly-finite} $\chi=\chi'$, or
\item\label{item:feeds_in_extra_strongly-finite} $\chi<^\Tt\chi'$ and $\chi$ is $\chi'$-transient,
\end{enumerate}
or

\item\label{item:chi_along_main_branch_transient_strongly_finite_Tt_structure} $\chi<^\Tt\xi$ and $\chi$ is $\xi$-transient.
\end{enumerate}
Moreover, the three options \ref{item:feeds_in_strongly-finite}, \ref{item:feeds_in_extra_strongly-finite}, \ref{item:chi_along_main_branch_transient_strongly_finite_Tt_structure} are mutually exclusive.
\end{lem}

\begin{proof}[Proof sketch]
The proof is like that
for Subclaim \ref{sclm:Tt_structure_mim} in
the proof of Lemma \ref{lem:non-dropping-strongly_finite} below,
and so we leave it to the reader
to extract it from there.
(Below superstrong,
the argument appeared
within the proof of
%conf
\cite[Theorem 4.8]{mim}.)
\end{proof}

Although we do not need it, it is natural to observe that if $M$ is as above and non-type 2 then strongly finite trees on $M$ produce models which can also be produced via linear (but possibly non-normal) iterations,
via Dodd-sound measures.
It is similar if $M$ is type 2, but there is a slight wrinkle.
%conf
Cf.~\cite[Theorem 4.8(d)]{mim}:
\begin{prop}\label{prop:strongly_finite_linear_it}
Let $M,m$ be as in  \ref{dfn:strongly_finite}
and let $\Tt$ be a strongly finite terminally non-dropping $m$-maximal tree on $M$. Then there is a \tu{(}possibly non-normal, possibly non-$m$-maximal\tu{)} tree $\Ll$ on $M$
which has no drops in model or degree \tu{(}anywhere\tu{)},
$\deg^{\Ll}_\alpha=m$ for all $\alpha<\lh(\Ll)$,
with
$M^{\Ll}_\infty=M^\Tt_\infty$ and $i^\Ll=i^{\Tt}$ and every extender used
along the main branch $b^{\Ll}$ of $\Ll$ is finitely generated and Dodd-sound,
 $\crit(E^{\Ll}_\alpha)<\crit(E^{\Ll}_\beta)$ for $\alpha+1,\beta+1\in b^\Ll$ with $\alpha<\beta$, and moreover:
\begin{enumerate}[label=--]\item
if $M$ is non-type 2 then $\Ll$ is linear,
so $\alpha+1\in b^{\Ll}$ for every $\alpha+1<\lh(\Ll)$,
\item if $M$ is type 2,
then for each $\alpha+1<\lh(\Ll)$,
if $\alpha+1\notin b^{\Ll}$
then $\alpha\in b^{\Ll}$
and $\crit(i^{\Ll}_{\alpha\infty})=\lgcd(M^{\Ll}_\alpha)$
and $E^{\Tt}_\alpha=F(M^{\Ll}_\alpha)$.
\end{enumerate}
\end{prop}
\begin{proof}
 Let $\vec{G}=\left<G_\alpha\right>_{\alpha<\lambda}$ be the Dodd decomposition of $i^\Tt$
and $M_\alpha=\Ult_m(M,\vec{G}\rest\alpha)$.
Then for each $\alpha<\lambda$,
either:
\begin{enumerate}[label=(\roman*)]
\item $\crit(G_\alpha)^{+M_\alpha}<\OR^{M_\alpha}$ and $G_\alpha\in\es_+^{M_\alpha}$, or
\item $\crit(G_\alpha)=\lgcd(M_\alpha)$,
$M$ is active type 2, and $G_\alpha\in\es_+(\Ult(M^{\alpha},F^{M_\alpha}))$.
\end{enumerate}
This is similar to the proof of Subclaim \ref{sclm:Tt_structure_mim} in
the proof of Lemma \ref{lem:non-dropping-strongly_finite} below, so we omit further discussion. The proposition follows easily from this observation;
for example, if $M$ is non-type 2
then
$\Ll$ is just the degree $m$ linear iteration with $E^{\Ll}_\alpha=G_\alpha$.
\end{proof}
\begin{lem}\label{lem:non-dropping-strongly_finite}
Let $m,M$ be as in Definition \ref{dfn:strongly_finite}.
Let $\Uu$ be a terminally-non-dropping
$m$-maximal tree on $M$ and $x\in M^\Uu_\infty$.
Then there is a  strongly finite
terminally-non-dropping $m$-maximal tree
$\Tt$ on $M$
capturing $(\Uu,x)$.
\end{lem}
In particular, if $M$ is not type 2,
then $M^\Tt_\infty$ is a linear iterate of $M$ (but possibly via a non-normal tree).

The proof below the superstrong level, which is essentially identical to the one allowing superstrongs below,\footnote{As mentioned earlier, the proof that iterable sound mice with Mitchell-Steel indexing are Dodd-sound is not quite the direct translation of the proof below the superstrong level. But the present lemma simply assumes the relevant Dodd-soundness, so this is not an issue here.} appeared
as part of the proof of the author's thesis \cite[Theorem 4.8]{mim}.

\begin{proof}
 By Lemma \ref{lem:simple_embedding_exists},
 we already know
 there is a terminally non-dropping $m$-maximal tree $\Tt$ on $M$ capturing $(\Uu,x)$ with $\lh(\Tt)<\om$ (but we don't claim that $\Tt$ is strongly finite).
Call such a tree $\Tt$
a \dfnemph{candidate}. We will show that the ``least'' candidate is strongly finite,
for a certain natural notion of ``leastness'', which we now formulate. Let $\Tt$ be a candidate.
Note that $\Tt$ is ${\geq 0}$-Dodd-nice (Definition \ref{dfn:Dodd-nice}),
since every initial segment of $M$ is Dodd-sound.
Let $\xi+1=\lh(\Tt)<\om$.
Let $A$ be the set of ordinals $\beta$ such that
$\beta\leq^\Tt_{\Da}\alpha$ for some $\alpha+1\in b^\Tt$.
Let
$\left<\kappa_i\right>_{i<\ell}$ enumerate
$\{\crit(E^\Tt_\beta)\}_{\beta\in A}$ in decreasing order. So $\crit(i^{\Tt})=\kappa_{\ell-1}$.
Let $\beta_i\in A$ be such that $\kappa_i=\crit(E^\Tt_{\beta_i})$,
and
$\gamma_i=\lh(\core_\ds(E^\Tt_{\beta_i}))$. Then the \dfnemph{index} of
$\Tt$ is $\left<\gamma_i\right>_{i<\ell}$.
Let $\Tt$ be the candidate of
lexicographically minimal index (noting that there is one), and adopt the notation just described for this $\Tt$ (that is, $\xi$, $\gamma_i$, etc).
The following claim completes the
proof:

\begin{clmtwo}\label{clm:Tt_strongly_finite_mim}
 $\Tt$ is strongly finite.
\end{clmtwo}
\begin{proof}
Suppose not. We
will construct a candidate $\Ttbar$ with smaller index than $\Tt$, a
contradiction, basically by replacing the core $\core_{\ds}(E^\Tt_\beta)$
 of a certain extender $E^\Tt_\beta$ used in $\Tt$, whose core $\core_{\ds}(E^\Tt_\beta)$ is not finitely generated, with a sub-extender whose core is finitely generated. So let
$a<\ell$ be least (so $\kappa_a$ largest) such that $\core_{\ds}(E^\Tt_{\beta_a})$
is not finitely
generated. Let $\kappa=\kappa_a$, $\beta=\beta_a$,
$Q=\core_\ds(\exit^\Tt_{\beta})$ and
$F=F^Q$. So
$\kappa^{+Q}<\sigma=\sigma_F=\tau_F$.
By \ref{lem:Dodd_cores_appear} there is a unique $\theta\leq\beta$ with
$Q\ins M^\Tt_\theta$; also by \ref{lem:Dodd_cores_appear}, $\theta\leq^\Tt\beta$ and $\theta$ is the least ordinal with
$\sigma<\lh(E^\Tt_\theta)$.

Let $\thetabar\leq\theta$ be least with
$\lh(E^\Tt_\thetabar)>\kappa^{+Q}$.
Let $\thetabar+k=\theta$ (so $k<\om$) and
$\xibar+k=\xi$.
Let $\chi$ be least with $\chi\geq\theta$ and $\chi+1\leq^\Tt\xi$.
Let $\bar{\chi}+k=\chi$.
Our
plan
is to select some
$g\in\nu_F^{<\om}$
and $R\pins M^\Tt_\thetabar$ with $F^R\approx F\rest
g$, and such that we can define an $m$-maximal
 tree $\Ttbar$ on $M$ such that:
\begin{enumerate}[label=--]
 \item   $\Ttbar\rest\thetabar+1=\Tt\rest\thetabar+1$ and
  $\lh(\Ttbar)=\xibar+1$ and
 $(0,\bar{\xi}]^{\bar{\Tt}}\cap\dropset^{\bar{\Tt}}_{\deg}=\emptyset$.
 \item Replacing the role of $Q$ in
$\Tt$ with $R$
in $\Ttbar$, we perform a kind of reverse copy construction,
 much like in the proof of Lemma \ref{lem:simple_embedding_exists},
so that
$\Tt\rest[\theta,\xi]$ will be a ``copy'' of
$\Ttbar\rest[\thetabar,\xibar]$.
Moreover, $\bar{\chi}$ is least such that $\bar{\chi}\geq\thetabar$ and $\bar{\chi}+1\leq^{\Ttbar}\xibar$,
and for $\alpha\in[\bar{\chi}+1,\xibar]^\Ttbar$, the copying process will yield a copy map
\[ \pi_{\alpha}:M^{\bar{\Tt}}_{\alpha}\to M^{\Tt}_{\alpha+m} \]
with $\crit(\pi_\alpha)>\kappa^{+R}=\kappa^{+Q}$.
\item
The final copy map
$\pi_{\bar{\xi}}:M^{\Ttbar}_{\xibar}\to M^{\Tt}_\xi$ is an $m$-embedding
with
$\pi_{\xibar}\com i^\Ttbar=i^\Tt$.
\item
$\tau^{-1}(x)\in\rg(\pi_{\bar{\xi}})$,
where $\tau:M^{\Tt}_{\xi}\to M^{\Uu}_\infty$
witnesses that $\Tt$ captures $(\Uu,x)$.
\end{enumerate}
It will follow that
$\tau\com\pi_{\bar{\xi}}:M^{\Ttbar}_{\xibar}\to M^\Uu_\infty$ witnesses that $\Ttbar$ captures $(\Uu,x)$, and
therefore $\Ttbar$ will be a candidate. From the construction it will also be clear that $\Ttbar$ has index strictly less than does $\Tt$,
which will be a contradiction.

So we need to select $g$ and $R$ and build $\bar{\Tt}$.
Now $F=F^Q$ is either type 2 or  3. In either case, we will choose
a type 2 premouse $Q'\ins Q$ with
$F^{Q'}\sub F$,
a type 2  premouse $R\pins Q'$,
and a $0$-embedding $\pi_{RQ'}:R\to Q'$ with
\[ \crit(\pi_{RQ'})>\kappa^{+R}=\kappa^{+Q'}=\kappa^{+Q}.\]
If $F$ is type 2 then there are
cofinally many $g\in[\nu_F]^{<\om}$
(that is, cofinally many with respect to $\sub$) satisfying the conclusion of  Lemma
\ref{lem:sub-extender} with $\theta=0$, and
producing $R\pins M^\Tt_\theta$ with $F^R\approx F\rest g$.
In this case we will set $Q'=Q$ and $\pi_{RQ}:R\to Q$
as in Definition \ref{dfn:type_2_factor_embedding}
and Lemma \ref{lem:type_2_factor_embedding}.
If $F$ is type 3 then there are
cofinally many type 2 segments $Q'\pins Q|\nu(F^Q)$
with $F^{Q'}\sub F^Q$, and after choosing some appropriate such $Q'$, we will
 obtain $R$ and $\pi_{RQ'}:R\to Q'$
 from $Q'$ like in the case that $F$ is type 2.
Note that in any case, we will have
$\rho_\om^R=\rho_1^R=\kappa^{+R}=\kappa^{+Q}$,
and therefore by coherence etc, $R\pins M^{\Tt}_\thetabar$ (whereas $Q\ins M^\Tt_\theta$).

So, we will choose $g,R,Q'$ as in the previous paragraph,
with $Q',g$ large
enough
that the desired reverse copying
construction is
supported.
We will soon describe the copying process more precisely, but before doing so, it is useful to analyse the  structure of $\Tt\rest[\theta,\xi]$:

\begin{sclmtwo}\label{sclm:Tt_structure_mim} Let $\chi\in[\theta,\xi)$.
Then $\crit(E^\Tt_\chi)\notin(\kappa,\sigma)$, so
$\pred^\Tt(\chi+1)\notin(\thetabar,\theta)$. In fact,  one of the
following options
holds:
\begin{enumerate}[label=\tu{(}\roman*\tu{)}]
\item\label{item:feeds_in_or_feeds_in_extra_mim} there are
$\vareps,\chi'$ such that  $\eps+1\leq^\Tt\xi$ and
$\chi'\leq^\Tt_{\Da}\eps$, and either:
\begin{enumerate}
\item\label{item:feeds_in_mim} $\chi=\chi'$, or
\item\label{item:feeds_in_extra_mim} $\chi<^\Tt\chi'$
and $\chi$ is $\chi'$-transient
and $E^\Tt_{\chi'}=F(M^{\Tt}_{\chi'})$,
\end{enumerate}
or
 \item\label{item:branch_extra_mim}
 $\chi<^\Tt\xi$ and  $\chi$ is $\xi$-transient.
\end{enumerate}
Moreover, the three options \ref{item:feeds_in_mim},
\ref{item:feeds_in_extra_mim}
and \ref{item:branch_extra_mim}
are mutually exclusive.
\end{sclmtwo}

Recall here that if \ref{item:feeds_in_extra_mim}
holds and $\eta+1=\succ^\Tt(\chi,\chi')$, then $E^\Tt_{\chi}=F(M^{*\Tt}_{\eta+1})$
and $\crit(i^{*\Tt}_{\eta+1,\chi})=\lgcd(M^{*\Tt}_{\eta+1})$,
so $\crit(E^\Tt_\chi)=\crit(E^\Tt_{\chi'})$. Likewise, if \ref{item:branch_extra_mim}
holds and $\eta+1=\succ^\Tt(\chi,\xi)$, then $E^\Tt_\chi=F(M^{*\Tt}_{\eta+1})$, so $\crit(E^\Tt_\chi)=\crit(F(M^{\Tt}_\xi))$.
\begin{proof}
The mutual exclusivity is easy to see.
We prove that one of the options holds at each $\chi\in[\theta,\xi)$
by induction on $\chi$; it is straightforward to see that this implies
$\crit(E^\Tt_\chi)\notin(\kappa,\sigma)$ (use Lemma \ref{lem:Dodd_cores_appear}).

\begin{case}\label{case:chi=theta_mim}
$\chi=\theta$.

If $\theta=\beta$ then $E^\Tt_\theta=E^\Tt_\beta$ is Dodd-sound and \ref{item:feeds_in_mim} holds. So suppose
$\theta<^\Tt\beta$, so $E^\Tt_\beta$ is non-Dodd-sound.
We have $\kappa_a=\kappa=\crit(F)<\lh(E^\Tt_\theta)$,
and since $E^\Tt_\beta$ is non-Dodd-sound, in fact $a>0$ and $\kappa_{a-1}$ is the critical point of the iteration map $Q\to\exit^\Tt_\beta$, so $\kappa_{a-1}<\lh(E^\Tt_\theta)$.
Let $i<\ell$ be least (so $\kappa_i$ largest) such that $\kappa_i<\lh(E^\Tt_\theta)$
(so $i\leq a-1$ and $\kappa_i\geq\kappa_{a-1}$). Then $\exit^\Tt_\theta||\kappa_i^{+\exit^\Tt_\theta}=\core_{\ds}(\exit^{\Tt}_{\beta_i})||\kappa_i^{+\core_{\ds}(\exit^{\Tt}_{\beta_i})}$.
By Lemma \ref{lem:Dodd_decomposition},  $\beta_i<^\Tt_{\Da}\beta$
and $\sigma\leq\kappa_i$. (We have $\kappa_i<\nu(E^\Tt_\beta)$, since $E^\Tt_\beta$ is non-Dodd-sound,
hence type 2, and $\kappa_i<\lh(E^\Tt_\theta)\leq\lh(E^\Tt_\beta)$.)

\begin{scase}\label{scase:kappa_i<lgcd(exit)_mim}
 $\kappa_i<\lgcd(\exit^\Tt_\theta)$,
 so $\kappa_i^{+\exit^\Tt_\theta}<\lh(E^\Tt_\theta)$.

Then $\theta$ is the unique $\theta'$ such that $\core_{\ds}(\exit^\Tt_{\beta_i})\ins M^{\Tt}_{\theta'}$.
To see this,  by Lemma \ref{lem:Dodd_cores_appear},
it suffices to see that $\theta$ is the least $\theta'$ such that
 $\sigma(E^\Tt_{\beta_i})<\lh(E^\Tt_{\theta'})$. But this holds because (i)
  $\theta$ is the least $\theta'$ such that $\sigma<\lh(E^\Tt_{\theta'})$,
(ii) $\sigma\leq\kappa_i$, and (iii)  by choice of $a$ and the subcase hypothesis,
 \[ \sigma(E^\Tt_{\beta_i})=\tau(\core_{\ds}(E^\Tt_{\beta_i}))=\kappa_i^{+\exit^\Tt_{\beta_i}}=\kappa_i^{+\exit^\Tt_\theta}<\lh(E^\Tt_\theta).\]

Now $E^\Tt_{\beta_i}$ is Dodd-sound,
for otherwise,  $i>0$
and  $\kappa_{i-1}$ is the critical point of the iteration map $\core_{\ds}(\exit^\Tt_{\beta_i})\to\exit^\Tt_{\beta_i}$,
 but then by the previous paragraph, $\kappa_{i-1}<\nu(E^\Tt_\theta)<\lh(E^\Tt_\theta)$,
contradicting the choice of $i$.
It follows that $\beta_i=\theta$ and  \ref{item:feeds_in_mim} holds, completing this subcase.
\end{scase}

\begin{scase}\label{scase:kappa_i=lgcd(exit)_mim}$\kappa_i=\lgcd(\exit^\Tt_\theta)$.

Since $\lgcd(\exit^\Tt_\theta)=\kappa_i=\crit(E^\Tt_{\beta_i})$ is  a limit cardinal of $\exit^\Tt_{\beta_i}$,
therefore $\exit^\Tt_\theta$ is type 2 or 3.

Suppose $\exit^\Tt_\theta$ is type 2. Since $\kappa_i=\lgcd(\exit^\Tt_\theta)$, $\theta=\pred^\Tt(\beta_i+1)$. So if $\beta_i+1\leq^\Tt\xi$
then  \ref{item:branch_extra_mim} holds.
 Otherwise, there is $\vareps$ such that $\vareps+1\leq^\Tt\xi$ and $\beta_i<^\Tt_{\Da}\vareps$.
 But then  \ref{item:feeds_in_extra_mim} holds,
 with some $\chi'\leq^\Tt_{\Da}\vareps$ and $\beta_i+1=\succ^\Tt(\theta,\chi')$.

 Now suppose $\exit^\Tt_\theta$
 is type 3, so $\kappa_i=\nu(E^\Tt_\theta)$. We will reach a contradiction (recall we are presently assuming $\theta<^\Tt\beta$). Note that  $\pred^\Tt(\beta_i+1)=\theta+1$.

 We claim there is $\delta+1\leq^\Tt\xi$
 such that $\theta\leq^\Tt_{\Da}\delta$.
For fix $\vareps+1\leq^\Tt\xi$ such that $\beta_i\leq^\Tt_{\Da}\vareps$.
 If $\beta_i=\vareps$
 then
 \[ \theta+1=\pred^\Tt(\beta_i+1)<^\Tt\beta_i+1\leq^\Tt\xi,\]
 so $\delta=\theta$ works.
 So suppose $\beta_i<^\Tt_{\Da}\vareps$. We claim that $\delta=\vareps$ works, and that  $\theta<^{\Tt}_{\Da}\vareps$. In fact,
 let $\eta$ be least such that $\beta_i<^\Tt_{\Da}\eta\leq^{\Tt}_{\Da}\vareps$.
 Then \[\theta+1=\pred^\Tt(\beta_i+1)<^\Tt\beta_i+1\leq^\Tt\eta,\]
and note that
  $E^\Tt_{\beta_i}$
 is total over $M^{\Tt}_{\theta+1}=M^{*\Tt}_{\beta_i+1}$. We claim that also $\theta<^\Tt_{\Da}\eta$. For otherwise,  $F(M^{\Tt}_{\theta+1})=\core_{\ds}(E^\Tt_\eta)$ is Dodd-sound. But in the latter case, $\crit(E^\Tt_\theta)<\tau(F(M^{*\Tt}_{\theta+1}))$,
  so
  \[ \crit(E^\Tt_{\beta_i})=\kappa_i=\nu(E^\Tt_\theta)\leq i_{E^\Tt_\theta}(\crit(E^\Tt_\theta))<\tau(F(M^{\Tt}_{\theta+1}))=\sigma(F(M^{\Tt}_{\theta+1})),\]
  contradicting the fact that $\beta_i<^\Tt_{\Da}\eta$.

  So there is $\delta+1\leq^\Tt\xi$ with $\theta\leq^\Tt_{\Da}\delta$.
  Since $E^\Tt_\theta$
  is type 3, and by the choice of $a$,
  therefore $\crit(E^\Tt_\theta)\leq\kappa=\kappa_a=\crit(E^\Tt_\beta)$.
  There is also $\vareps+1\leq^\Tt\xi$ with $\beta\leq^\Tt_{\Da}\vareps$,
  and recall that (we are presently assuming) $\theta\neq\beta$.
  So by Lemma \ref{lem:Dodd_decomposition},
  it follows that
   $\sigma(E^\Tt_\theta)\leq\kappa=\crit(E^\Tt_\beta)$
  (otherwise, by that lemma,
  we must have $\sigma=\sigma(E^\Tt_\beta)\leq\crit(E^\Tt_\theta)$,
  contradicting that $\crit(E^\Tt_\theta)\leq\kappa$).
 But since $E^\Tt_\theta$ is type 3, $\sigma(E^\Tt_\theta)=\nu(E^\Tt_\theta)\geq\sigma>\kappa$, a contradiction.
\end{scase}
\end{case}
\begin{case}\label{case:chi>theta_(i)_holds_chi-1_and_mim} $\chi>\theta$,
\ref{item:feeds_in_mim} holds at stage $\chi-1$, but $\chi\not<^\Tt\xi$.

Then there is $\iota$ such that
$\chi-1<^{\Tt}_{\Da}\iota$
and $\iota\leq^\Tt\vareps$
for some $\vareps+1\leq^\Tt\xi$;
let
 $\iota$ be least such.
Then $\chi\leq^{\Tt}\iota$.
If $\chi=\iota$ then \ref{item:feeds_in_mim} holds (at stage $\chi$), suppose $\chi<^\Tt\iota$.
Let $b<\ell$ with $\kappa_b=\crit(i^\Tt_{\chi\iota})$,
so $\beta_b+1=\succ^\Tt(\chi,\iota)$.
Then $\kappa_b<\lh(E^\Tt_\chi)$.
Let $i$ be least such that $\kappa_i<\lh(E^\Tt_\chi)$.
\begin{scase} $\kappa_i<\lgcd(\exit^\Tt_\chi)$,
so $\kappa_i^{+\exit^\Tt_\chi}<\lh(E^\Tt_\chi)$.

Then $\core_{\ds}(\exit^\Tt_{\beta_i})\ins M^\Tt_\chi$, $E^\Tt_{\beta_i}$ is Dodd-sound and $\beta_i=\chi$,
all as in Subcase \ref{scase:kappa_i<lgcd(exit)_mim} of Case \ref{case:chi=theta_mim},
so \ref{item:feeds_in_mim} holds.
\end{scase}
\begin{scase}
$\kappa_i=\lgcd(\exit^\Tt_\chi)$.

Then
$E^\Tt_\chi$ is type 2 and
 \ref{item:feeds_in_extra_mim} holds
 (since
\ref{item:branch_extra_mim} is ruled out by case hypothesis).
This is established as in Subcase \ref{scase:kappa_i=lgcd(exit)_mim}
of Case \ref{case:chi=theta_mim}.
\end{scase}
\end{case}
\begin{case}\label{case:chi>theta_feeds_in_at_chi-1_and_chi<^Txi_mim} $\chi>\theta$ and
\ref{item:feeds_in_mim} holds at stage $\chi-1$, and $\chi<^\Tt\xi$.

Let $b<\ell$ with $\kappa_b=\crit(i^\Tt_{\chi\xi})$,
so $\beta_b+1=\succ^\Tt(\chi,\xi)$.
Then $\kappa_b<\lh(E^\Tt_\chi)$.
From here we proceed just as in Case \ref{case:chi>theta_(i)_holds_chi-1_and_mim}.
\end{case}
\begin{case}\label{case:chi>theta_branch_extra_at_chi-1_mim} $\chi>\theta$
and \ref{item:branch_extra_mim} holds at stage $\chi-1$.

This is basically like Case \ref{case:chi>theta_feeds_in_at_chi-1_and_chi<^Txi_mim}, but letting $\kappa_b=\lgcd(\exit^\Tt_{\chi-1})$,
then $\chi-1<^\Tt\xi$
and $\beta_b+1=\succ^\Tt(\chi-1,\xi)$
(instead of $\chi<^\Tt\xi$
and $\beta_b+1=\succ^\Tt(\chi,\xi)$).
But as before, consider the least $i$
such that $\kappa_i<\lh(E^\Tt_\chi)$,
so $i\leq b$ and $\kappa_b\leq\kappa_i$, etc.
\end{case}
\begin{case} $\chi>\theta$
and
\ref{item:feeds_in_extra_mim} holds at stage $\chi-1$.

This is a slight variant
of Case \ref{case:chi>theta_branch_extra_at_chi-1_mim}, completing all cases and hence the proof of the subclaim.\qedhere
\end{case}
\end{proof}

 By Subclaim \ref{sclm:Tt_structure_mim},
$\Tt\rest[\theta,\xi]$
can be viewed as a tree $\Tt'$ on the phalanx \[ \Phi(\Tt\rest(\bar{\theta}+1))\conc\left<(Q,0)\right>,\] where
in $\Tt'$, if $\crit(E^{\Tt'}_\alpha)\leq\kappa$ then $\pred^{\Tt'}(\alpha+1)\leq\bar{\theta}$
(so $E^{\Tt'}_\alpha$ applies to a model of $\Tt\rest(\bar{\theta}+1)$),
whereas if $\crit(E^{\Tt'}_\alpha)>\kappa$ then $\pred^{\Tt'}(\alpha+1)=\bar{\theta}+1$, with $(M^{*\Tt'}_{\alpha+1},\deg^{\Tt'}_{\alpha+1})\ins (Q,0)$.
In fact, by the subclaim,
if $\crit(E^{\Tt'}_\alpha)>\kappa$ then $\crit(E^{\Tt'}_\alpha)\geq\sigma$, and recall that $\exit^\Tt_\theta\ins Q$ and $\rho_1^Q\leq\sigma$,
so requiring $(M^{*\Tt'}_{\alpha+1},\deg^{\Tt'}_{\alpha+1})\ins (Q,0)$ corresponds with what happens in $\Tt$. In particular, if $\alpha\geq\theta$ and $\crit(E^\Tt_\alpha)>\kappa$
and $\pred^\Tt(\alpha+1)=\theta$
then $(M^{*\Tt}_{\alpha+1},\deg^{\Tt}_{\alpha+1})\ins (Q,0)$.

Let $Q',g,R$ and $\pi_{RQ'}:R\to Q'$ be as discussed earlier.
Then $R\pins M^{\Tt}_{\bar{\theta}}$ and $\rho_1^R=\kappa^{+Q}$. We attempt to
form $\bar{\Tt}$, extending $\bar{\Tt}\rest(\bar{\theta}+1)=\Tt\rest(\bar{\theta}+1)$,
viewing $\bar{\Tt}\rest[\bar{\theta},\xibar]$ as a tree $\bar{\Tt}'$ on the
corresponding phalanx
$\Phi(\Tt\rest(\bar{\theta}+1))\conc\left<(R,0)\right>$,  requiring that it copies in essentially the usual manner to $\Tt'$, with initial copy maps given by identity maps and $\pi_{RQ'}$. As long as $Q',g$ are chosen large enough, this succeeds. We discuss a little further below why this makes sense, particularly in the case that $Q$ is type 3, and hence $Q'\neq Q$.

Suppose first that $Q$ is type 2, so $Q'=Q$. Then reverse copying proceeds in the usual fashion, as long as
the copy maps  $\pi_{\bar{\delta}}:M^{\bar{\Tt}}_{\bar{\delta}}\to M^{\Tt}_{\bar{\delta}+m}$ (for $\bar{\delta}<\xi$) are $\nu$-preserving, and $E^\Tt_{\bar{\delta}+m}\in\rg(\pi_{\bar{\delta}})$.
Just as in the proof of \ref{lem:simple_embedding_exists},
by choosing  $g$ large enough,
we can ensure that these conditions are indeed maintained.

Now suppose instead that $Q$ is type 3.  Here  $Q'$ is type 2 with $F^{Q'}\sub F^Q$ and $Q'\pins Q$, and (we may arrange that) $\pi_{RQ'}:R\to Q'$ is a $0$-embedding. In this case we may not be able to ``reverse copy'' ultrapowers of $Q$ to ultrapowers of $R$, for example, since $Q'\neq Q$.
So we want to point out why this does not cause a problem. Because $Q$ is type 3, we have $\beta=\theta$ and $\sigma=\nu(F^Q)$, so $E^\Tt_\theta=F^Q$. At stage $\bar{\theta}$ in $\Ttbar$, we ``reverse copy'' in the sense that we set
$E^{\bar{\Tt}}_{\bar{\theta}}=F^{R}$.
This then yields the copy map(s) at stage $\bar{\theta}+1$  via essentially the Shift Lemma:
Given an $n$-sound $N$ with $N||\kappa^{+N}=Q|\kappa^{+Q}$ and $\kappa<\rho_n^N$, we get a  map
\[ \pi:\Ult_n(N,F^R)\to\Ult_n(N,F^{Q}) \]
by setting $\pi([a,f]^{N,n}_{F^{R}})=[\pi_{RQ'}(a),f]^{N,n}_{F^Q}$. By taking $Q',g$ large enough,
we get $\sigma\in\rg(\pi)$ here
(and $\nu(F^{Q'})\in\rg(\pi)$
as $\pi_{RQ'}$ is a $0$-embedding);
let $\pi(\bar{\sigma})=\sigma$.
Note that then $\nu(F^R)<\bar{\sigma}$, as $\pi_{RQ'}(\nu(F^R))=\nu(F^{Q'})<\nu(F^Q)=\sigma$.
Letting $\alpha\in[\theta,\xi)$,
since $\crit(E^{\Tt}_\alpha)\notin(\kappa,\sigma)$ and
 $\sigma=\nu(F^Q)=\nu(E^\Tt_\theta)$,
if $\crit(E^{\Tt}_\alpha)>\kappa$
then $\pred^{\Tt}(\alpha+1)>\theta$.
It will follow that in $\Ttbar$, for  all stages $\alpha\geq\bar{\theta}$, $\crit(E^{\bar{\Tt}}_\alpha)\notin(\kappa,\bar{\sigma})$,
and
if $\alpha>\bar{\theta}$ then $\root^{\bar{\Tt}}(\alpha)\leq\bar{\theta}$ (that is, $\alpha$ is above a node of $\Phi(\bar{\Tt}\rest(\bar{\theta}+1))$, not above $R$). So
in $\Ttbar$,
we never need to take an ultrapower of any $R'\ins R$, nor in fact of any $R'\ins M^{\bar{\Tt}}_{\bar{\theta}}$. So assuming that we have taken $Q',g$ large enough, the copying process succeeds, as desired.
We leave the remaining (straightforward) details to the reader.

Now $\Ttbar$ is a candidate with index less than that of $\Tt$,
a contradiction which completes the proof that $\Tt$ is strongly finite.
\end{proof}

This completes the proof of the theorem.
\end{proof}

\section{Measures in mice}\label{sec:mim}

We pause briefly to use the machinery of the previous
sections to analyse measures in mice. The argument is mostly
that given in
\cite[Theorem 4.8]{mim}, though that was proven only below
superstrong. It does not immediately generalize to the superstrong level, because its proof appeals to  to super-Dodd-solidity \ref{thm:super-Dodd-soundness}, and the proof of that result involves new features at superstrong. But otherwise, the argument from \cite[Theorem 4.8]{mim} adapts immediately to  superstrongs.
The version in \cite{mim} is also more restrictive
in that it assumes that $N\sats\KP^*+\text{``}\kappa^{++}$
exists, where the measure is over $\kappa$''. For the version here, this assumption is reduced to that of  $N\sats$ ``$\kappa^+$
exists, where the measure is over $\kappa$'' (in particular, without $\KP^*$).
Although we do not literally
need the theorem itself later, we will adapt its
proof
in the proofs of super-Dodd-solidity \ref{thm:super-Dodd-soundness}
and projectum-finite generation \ref{thm:finite_gen_hull},
so it serves as a good warm-up. We will
also omit some of the details of those proofs
when they are similar enough to those in this section.
So like  \S\ref{sec:capturing_elements}, the proofs of the main theorems of the paper do depend
on the arguments in this section. However,
it should be noted that the details
relating to the coding of finite iteration trees
definably over mice, from \cite{extmax},
which come up at the beginning of the proof, are \emph{not}
relevant to later proofs in the paper.
If one weakens the theorem below
by adding the assumption that
$N$ has no largest cardinal, then
one can avoid the use of any coding apparatus,
as we indicate during the proof.

\begin{tm}\label{thm:measures_in_mice}
 Let $N$ be a $(0,\om_1+1)$-iterable premouse and $D\in N$
 such that $N\sats$ ``$D$ is a countably complete ultrafilter''.
 Then there is a strongly finite, terminally-non-dropping, $0$-maximal  tree
$\Tt$
on $N$
 with $\Ult_0(N,D)=M^\Tt_\infty$ and $i^{N,0}_D=i^\Tt$.\footnote{Of course,
since $\Tt$ is $0$-maximal,
we must have $E^\Tt_\alpha\in\es_+(M^\Tt_\alpha)$, so the theorem is
non-trivial.}
\end{tm}

\begin{rem}\label{rem:measure_decomp_UA}
Note that since $D\in N$, it is implicit that $N\sats$ ``$\kappa^+$ exists, where $D$ is on $\kappa$''. Combined with Proposition \ref{prop:strongly_finite_linear_it} and Theorem \ref{thm:super-Dodd-soundness}, it follows that $D$ is equivalent to a finite linear iteration of $N$ which uses only Dodd-sound measures (the iteration need not be normal).
In more detail, letting $\Ll$ be as in the proposition and $n+1=\lh(\Ll)$, we have $\Ult_0(N,D)=M^\Ll_{n}$,
$i^{N,0}_D=i^\Ll$, and
$E^\Ll_i\in\es(M^\Ll_i)$  and $M^\Ll_i\sats$ ``$E^\Ll_i$ is a Dodd-sound measure'' for each $i< n$. (The other option for extenders $E^\Ll_i$ mentioned in Proposition \ref{prop:strongly_finite_linear_it} can't attain, because $D\in N$.) This  was clear below superstrong from \cite[Theorem 4.8]{mim}
and its proof (see in particular
 \cite[Theorem 4.8(d)]{mim}), but some of the aforementioned features of $\Ll$ were not explicitly stated there.

 More recently, in
 %conf
 \cite[Theorem 5.3.13]{UA}, Gabriel Goldberg proved from ZFC + UA (the Ultrapower Axiom) that all countably complete measures can be decomposed into a finite linear  iteration of Dodd-sound measures (as defined in
 %conf
 \cite[Definition 4.3.23]{UA};
 the iteration here is coarse, as the context is just a model of ZFC, not $L[\es]$). He also showed in
 %conf
 \cite[\S2.3.1 and Theorem 2.3.10]{UA} that UA holds in mice $M$ whose universe $\univ{M}$ models ZFC + ``$V=\HOD$'' + ``There is an ordinal $\alpha$
 such that $V_\alpha\preccurlyeq_2 V$ and $V_\alpha\sats\ZFC$''
 (this includes mice at the superstrong level). Note here that if we apply these results in a mouse $M$ (where $\univ{M}\sats\ZFC$, etc), it does not itself give us an $\Ll$ which is an iteration on $M$ in the usual fine sense, since we do not directly know that $E^\Ll_i\in\es(M^\Ll_i)$.  But if $M$ is below superstrong, then by \cite[Theorem 2.9]{mim}
 %conf
 (or \cite[Theorem 2.17]{V=HODX_pub}),   the Dodd-soundness of $E^\Ll_i$ in $M^\Ll_i$ ensures that $E^\Ll_i\in\es(M^\Ll_i)$ after all.
 The  author expects that \cite[Theorem 2.9]{mim}
 %conf
 (and \cite[Theorem 2.17]{V=HODX_pub}) also goes through routinely at the superstrong level, using Theorem \ref{thm:super-Dodd-soundness}, and in fact, by Theorem \ref{thm:measures_in_mice}, for the case of mice $N$ and  ultrafilters $D\in N$ as in that theorem, it does.

In any case, the proof we give here for Theorem \ref{thm:measures_in_mice} just follows the methods used for \cite[Theorem 4.8]{mim}.
\end{rem}
We cannot prove the theorem until we know that all proper segments of
$N$  are Dodd-sound, which will be established in
\S\ref{sec:Dodd_proof}.
For now we reduce it to this issue:

\begin{lem}\label{lem:measures_in_mice}
Adopt the hypotheses of Theorem \ref{thm:measures_in_mice},
and assume that all proper segments of $N$ are Dodd-sound.
Then the conclusion of Theorem \ref{thm:measures_in_mice} holds.
\end{lem}

\begin{proof}
We sketched the main plan for the proof  in the first part of \S\ref{sec:meas_Dodd_prof_fin_plan}. That plan should be read before starting the details below. In the first part of the proof, we will get into the kind of situation assumed for the sketch in \S\ref{sec:meas_Dodd_prof_fin_plan}.
We do this by taking some $M\pins N$ over which there is an $N$-ultrafilter $D$ definable which is a counterexample to the lemma,
taking a nice enough countable hull $\bar{M}$ of $M$, over which we can define a preimage $\bar{D}$ of some such $D$, and showing that the relevant fine structural ultrapower $\bar{U}=\Ult_m(\bar{M},\bar{D})$ is iterable.
At that point, we will basically have arranged the assumptions for the  plan,
and so we then execute that plan by  comparing $\bar{M}$ with $\bar{U}$ and analysing the outcome.

So suppose the lemma fails. Let $\theta\geq\om$ be such that $D$ is an ultrafilter over
$\theta$. Then $\pow(\theta)^N\in N$ (as $D\in N$),
so $\theta^{+N}<\OR^N$.
We have $D\sub N|\theta^{+N}$,
so can let $M\pins N$ with
$\rho_\om^M=\theta^{+N}$
 and $D\in\J(M)$.

 Let $0<n<\om$ with $\rho_n^M=\theta^{+N}=\rho_\om^M$.
 We claim that there is a strongly finite,
terminally-non-dropping
$n$-maximal tree $\Tt$ on $M$
such that $M^\Tt_\infty=\Ult_n(M,D)$
and $i^\Tt=i^{M,n}_D$.
This yields a tree $\Tt''$ on $N$ witnessing the lemma.
For using the regularity of $\theta^{+N}$ in $N$,
 one gets that $n$-maximal trees $\Uu$ on $M$ are equivalent to
$\om$-maximal
trees $\Uu'$ on $M$. That is, every function
$f:\alpha\to\core_0(M)$ with $\alpha\leq\theta$
such that $f$ is definable from parameters over $M$,
is actually $\bfrSigma_n^M$-definable,
and this fact is preserved for non-degree-dropping $n$-maximal,
hence $\om$-maximal, iterates.
Similarly, $\om$-maximal trees $\Uu'$ on $M$ correspond to $0$-maximal
trees $\Uu''$ on $N$ such that $\Uu''$ is based on $M$, as long as $\Uu''$ has wellfounded models,
and moreover, if $b^{\Uu'}$ exists and is non-dropping in model or degree,
then $b^{\Uu''}$ exists and is non-dropping, and $i^{\Uu'}=i^{\Uu''}\rest
M$. (For some ordinals $\alpha$ we have $M^\Tt_\alpha\pins M^{\Tt'}_\alpha$,
in which case $(0,\alpha]^{\Tt}$ and $(0,\alpha]^{\Tt'}$
do not drop in model or degree;
for some $\alpha$ we have $M^\Tt_\alpha=M^{\Tt'}_\alpha$.)
So given a tree $\Tt$ on $M$ as claimed,
the corresponding tree $\Tt''$ on $N$
witnesses the lemma.

So suppose the claim fails.
The failure is a first order
statement about $M$ (without parameters).
For if $D'\sub M|\theta^{+N}$
is definable (from parameters) over $M$, and $M\sats$ ``$D'$ is
a countably complete ultrafilter over $\theta$'',
then $D'\in N$ and $N\sats$ ``$D'$ is a countably complete ultrafilter over
$\theta$'', since $\pow(\theta)^N\sub M$. And
%conf
\cite[3.15--3.20]{extmax}
gives a coding of finite $n$-maximal
iteration trees on $n$-sound $(n,\om)$-iterable premice,
which is uniformly definable over such premice.
Note that the definability  depends on the restriction to $n$-maximal for a particular $n$. As usual, the superstrong version
of this is almost identical with that literally in \cite{extmax}.\footnote{In
the notation of \cite[3.15--3.20]{extmax},
the ordinal $\lambda^\Tt_\alpha$ is somewhat inconveniently
defined in the case that $E^\Tt_\alpha$ is superstrong
and $E^\Tt_\alpha\neq F(M^\Tt_\alpha)$ (in this case
it is the index of $E^\Tt_\alpha$). Here, if $\mu=\crit(E^\Tt_\alpha)$,
then instead of considering functions $f:[\mu]^{<\om}\to\mu$
when representing ordinals $<\lambda^\Tt_\alpha$,
one must of course instead consider
functions $f:[(\mu^+)^N]^{<\om}\to(\mu^+)^N$. Alternatively,
one might redefine $\lambda^\Tt_\alpha$
in this case as $\lambda(E^\Tt_\alpha)$,
and allow the possibility that
$\lambda^\Tt_\beta=\lambda^\Tt_{\beta+1}$.} However, for our present purposes, it suffices to consider only a coding of strongly finite trees $\Tt$, and for such $\Tt$, every extender $E^\Tt_\alpha$ is type 1 or type 2 and hence not superstrong. So the coding apparatus in \cite{extmax} is essentially sufficient for our purposes.
As remarked earlier,
if $N$ has
no largest cardinal, then we can avoid
appealing to any coding of iteration trees at all.
For in this case, instead of using $M$ as above,
we can take
$M=N|\theta^{++N}$,
and then the relevant iteration trees can be defined
directly over $M$.\footnote{In \cite{mim},
an argument more along these lines was given,
though even there there is not quite that much space available.}
We continue literally
with the general case (using a coding apparatus), but the remaining
details are essentially the same in the simplified version just mentioned.

Using these two things, we can fix $k\in[n,\om)$ and $m\gg k$ and an $\rSigma_m$ statement
\begin{equation}\label{eqn:psi}\psi
\end{equation}
asserting over $M$ (and also over similar premice)
 ``I am $k$-sound and there is a countably
complete ultrafilter $D'$ over $\theta$
which is $\rSigma_{k}$-definable from parameters over me,
such that for every strongly finite
terminally-non-dropping  $n$-maximal tree $\Tt$ on me
and every $\alpha<i^\Tt_{0\infty}(\theta)$,
if $\alpha$ generates $M^\Tt_{0\infty}$, i.e.~if
\[
M^\Tt_{0\infty}=\Hull_{\rSigma_n}^{M^\Tt_{0\infty}}(\rg(i^\Tt_{0\infty})\cup\{
\alpha\}),\]
then $D'$ is not the ultrafilter derived from $i^\Tt_{0\infty}$
with seed $\alpha$''.

Working in $N$, let $\bar{M}=\cHull_{m+1}^M(\{\vec{p}^M_n\})$
 and $\pi:\bar{M}\to M$ be the uncollapse
map.
Let $\pi(\thetabar)=\theta$ etc. Then $\bar{M}$ is $(m+1)$-sound
with
\[
\rho_{m+1}^{\bar{M}}=\om<
\thetabar^{+\bar{M}}=
\rho_{m}^{\bar{M}}=
\rho_{n}^{\bar{M}} \]
and $\pi(p_i^{\bar{M}})=p_i^M$ for $i\leq m+1$
(note $p_i^M=\emptyset$ for $i>n$) and
\[
\bar{M}=\Hull_{m+1}^{\bar{M}}(\pvec_n^{\bar{M}})=\Hull_{m+1}^{\bar{M}}
(\emptyset), \]
as $\pvec_n^{\bar{M}}$ is simply enough definable over $\bar{M}$
anyway.
As $m$ is large enough,  $\bar{M}\sats\psi$, so we can
fix an analogous
ultrafilter $\bar{D}$ over $\bar{M}$,
and may assume that its definition lifts to give $D$ over $M$.
Note  that $n$-maximal trees on $\bar{M}$ are equivalent to $m$-maximal.
And because all initial segments of $M$ are Dodd-sound,
so are all initial segments of $\bar{M}$.

Let $X=\bigcap(\rg(\pi)\cap D)$.
Then $X\neq\emptyset$ by countable completeness in $N$.
Let $x\in X$. Let $\bar{U}=\Ult_n(\bar{M},\bar{D})=\Ult_m(\bar{M},\bar{D})$
(and note $i^{\bar{M},n}_{\bar{D}}=i^{\bar{M},m}_{\bar{D}}$
is an $m$-embedding and is continuous at $\thetabar^{+\bar{M}}$).
Like for $\bar{M}$, $n$-maximal trees on $\bar{U}$
are equivalent to  $m$-maximal.
The realization map\label{page:map_sigma_defined}
$\sigma:\bar{U}\to M$, defined
\[ \sigma([f]^{\bar{M},n}_{\bar{D}})=\pi(f)(x), \]
is an $n$-lifting embedding.
So $\bar{U}$ is $(n,\om_1+1)$-iterable (in $V$),
hence $(m,\om_1+1)$-iterable, as is $\bar{M}$.
Note $\bar{U}=\Hull_{m+1}^{\bar{U}}(\{\bar{x}\})$
where $\sigma(\bar{x})=x$.

We are now basically in the situation we assumed when describing the plan for the proof in \S\ref{sec:meas_Dodd_prof_fin_plan},
 which we will now execute in detail.
So consider the $m$-maximal comparison $(\Uu,\Tt)$ of $(\bar{U},\bar{M})$. We will show that $\Uu$ is trivial, $\Tt$ is strongly finite,
 $\bar{U}=M^\Tt_\infty$
 and $i^{\bar{M},m}_{\bar{D}}=i^\Tt$,
 which will be a contradiction.

Well, because $\bar{M}$ is $(m+1)$-sound with $\rho_{m+1}^{\bar{M}}=\om$,
etc, we get that
 $b^\Uu,b^\Tt$ do not drop, $M^\Tt_\infty=Q=M^\Uu_\infty$,
 $j,k$ are $m$-embeddings where
$j=i^\Tt$ and $k=i^\Uu$, as is  $i=i^{\bar{M},m}_{\bar{D}}$,
and $k\com i=j$ (see Figure \ref{fgr:first_commuting}).

\begin{figure}
\centering
\begin{tikzpicture}
 [mymatrix/.style={
    matrix of math nodes,
    row sep=0.35cm,
    column sep=0.4cm}]
   \matrix(m)[mymatrix]{
  \bar{U}&          {} &       {}&Q& \\
   {} & {} \\
 \bar{M}\\};
 \path[->,font=\scriptsize]
%maps from left $M$
(m-3-1) edge node[below] {$\ \ \ \ j=i^\Tt$} (m-1-4)
(m-3-1) edge node[left] {$i$} (m-1-1)
(m-1-1) edge node[above] {$k=i^\Uu$} (m-1-4);
\end{tikzpicture}
\caption{The diagram commutes.}
\label{fgr:first_commuting}
\end{figure}

Now since all initial segments of $\bar{M}$ are Dodd-sound, by Lemma
\ref{lem:non-dropping-strongly_finite},
we can fix a strongly finite $m$-maximal tree $\bar{\Tt}$ capturing
$(\Tt,k(\bar{x}))$.
We will eventually show that $\bar{\Tt}=\Tt$.
Let $\bar{Q}=M^{\bar{\Tt}}_\infty$ and $\somevarpi:\bar{Q}\to Q$ be the
capturing map.
So $\somevarpi$ is an $m$-embedding
and
$\somevarpi\com
i^{\bar{\Tt}}=i^\Tt=j$. Let $\bar{j}=i^{\bar{\Tt}}$.
Let $\bar{k}:\bar{U}\to\bar{Q}$ be $\bar{k}=\somevarpi^{-1}\com k$;
this makes sense as $k(\bar{x})\in\rg(\somevarpi)$.
Then all maps are $m$-embeddings and the full resulting
diagram commutes (see Figure \ref{fgr:second_commuting}).

We will analyze $\Tt$ and its relationship to $\bar{\Tt}$,
by analyzing the Dodd decompositions of $j=i^\Tt$ and $\bar{j}=i^{\bar{\Tt}}$,
eventually showing that these two Dodd decompositions are identical.

\begin{figure}
\centering
\begin{tikzpicture}
 [mymatrix/.style={
    matrix of math nodes,
    row sep=0.35cm,
    column sep=0.4cm}]
   \matrix(m)[mymatrix]{
  \bar{U}&          {} &       {}& {} & Q& \\
   {} & {} \\
   {} & {} & \bar{Q}\\
   \\
 \bar{M}\\};
 \path[->,font=\scriptsize]
%maps from left $M$
(m-5-1) edge[bend right] node[below] {$\ \ \ \ j=i^\Tt$} (m-1-5)
(m-5-1) edge node[left] {$i$} (m-1-1)
(m-1-1)edge node[above] {$\bar{k}$} (m-3-3)
(m-5-1)edge node[above] {$\bar{j}$} (m-3-3)
(m-3-3)edge node[above] {$\somevarpi$} (m-1-5)
(m-1-1) edge node[above] {$k=i^\Uu$} (m-1-5);
\end{tikzpicture}
\caption{The diagram commutes.}
\label{fgr:second_commuting}
\end{figure}

\begin{clmsix}\label{clm:crit(k)<crit(j)} Let $\kappa=\crit(j)$.
Then:
\begin{enumerate}[label=\arabic*.,ref=\arabic*]
\item\label{item:crit(k)>crit(j)}
$\kappa=\crit(i)=\crit(\bar{j}
)=\crit(j)<\kappa^{+\bar{U}}<\crit(k)=\min(\crit(\bar{k}) ,
\crit(\somevarpi))$.
\item\label{item:kappa^+_agmt}
$\bar{M}||\kappa^{+\bar{M}}=\bar{U}||
\kappa^{+\bar{U}}=\bar{Q}||\kappa^{+\bar{Q}}=Q||\kappa^{+Q}$.
\end{enumerate}
\end{clmsix}
\begin{proof}
By commutativity and since $k$ is an
iteration map, it  suffices to see
that $\crit(k)>\kappa$.
Also by commutativity, $\crit(k)\geq\kappa$, so suppose
$\crit(k)=\kappa$. We have $\crit(i)\geq\kappa$.
So for each $A\in\pow(\kappa)\cap\bar{M}$,
we have $i(A)\cap\kappa=A$, so
\[ k(A)=k(i(A)\cap\kappa)=k(i(A))\cap k(\kappa)=j(A)\cap k(\kappa).\]
Therefore $j,k$ are compatible through $k(\kappa)\leq j(\kappa)$.
But because $j,k$ are iteration maps arising from comparison, this is
impossible.
\end{proof}

\begin{clmsix}\label{clm:s_jbar=kbar(s_i)} We have:
\begin{enumerate}[label=\arabic*.,ref=\arabic*]
\item\label{item:s_j,sigma_j}
 $s_{k\com i}=s_j=k(s_i)\text{ and }\sigma_{k\com i}=\sigma_{j}=\sup
k``\sigma_i$.
 \item \label{item:s_j-bar,sigma_j-bar}
 $s_{\bar{k}\com i}=s_{\bar{j}}=\bar{k}(s_i)$ and
$\sigma_{\bar{k}\com i}=
\sigma_{\bar{j}}=\sup\bar{k}``\sigma_{i}$.
\item\label{item:s_j,sigma_j_as_rho_com_j-bar} $s_{\somevarpi\com\bar{j}}=s_j=\somevarpi(s_{\bar{j}})$
and $\sigma_{\somevarpi\com\bar{j}}=\sigma_j=\sup\somevarpi``\sigma_{\bar{j}}$.
\item\label{item:sigma_i=sigma_j-bar=sigma_j} $\kappa^{+\bar{M}}=\sigma_{i}=\sigma_{\bar{j}}=\sigma_j$.
\end{enumerate}
\end{clmsix}
\begin{proof}
Part \ref{item:s_j,sigma_j}: This holds because $k$ is an iteration map and by Claim \ref{clm:crit(k)<crit(j)} and Lemma
\ref{lem:s_j,sigma_j_pres}.
Lemma \ref{lem:s_j,sigma_j_pres} applies because
$E_i\rest i(\kappa)\notin
\bar{U}$.
For if $E_i\rest i(\kappa)\in\bar{U}$ then note that by
commutativity and Claim \ref{clm:crit(k)<crit(j)},
we have
\[ k(E_i\rest i(\kappa))=E_j\rest j(\kappa)\in Q,\]
which is impossible, as $j$ is an iteration map.

Part \ref{item:s_j-bar,sigma_j-bar}:
We have $\bar{k}\com i=\bar{j}$, so $s_{\bar{k}\com i}=s_{\bar{j}}$
and $\sigma_{\bar{k}\com i}=\sigma_{\bar{j}}$.
We also have
$\kappa^{+\bar{M}}<\crit(k)=\min(\crit(\bar{k}),\crit(\somevarpi))$.
Therefore by commutativity, $\bar{k}$
maps fragments of $E_i$ to fragments of $E_{\bar{j}}=E_{\bar{k}\com i}$.
So we just need to see we don't get too large
a fragment of $E_{\bar{j}}$  appearing in $\bar{Q}$. (Note that we don't know (yet) that $\bar{k}$
is an iteration map,  so we can't just use Lemma \ref{lem:s_j,sigma_j_pres} for
this.) But
if
 \[ E_{\bar{j}}\rest(\bar{k}(s_i)\cup\sup\bar{k}``\sigma_i)\in\bar{Q},\]
 then applying $\somevarpi$, we would have
 \[
E_{j}\rest(\somevarpi(\bar{k}(s_i))\cup\somevarpi(\sup\bar{k}``\sigma_i))\in Q,\]
since, much as for $\bar{k}$, $\somevarpi$ maps fragments of $E_{\bar{j}}$ to fragments of $E_j=E_{k\com i}$.
But as $\somevarpi(\bar{k}(s_i))=k(s_i)$ and
$\somevarpi(\sup\bar{k}``\sigma_i)\geq\sup k``\sigma_i$, this contradicts
part \ref{item:s_j,sigma_j}.

Part \ref{item:s_j,sigma_j_as_rho_com_j-bar}:  $\somevarpi\com\bar{j}=j$, so $s_{\somevarpi\com\bar{j}}=s_j$
and $\sigma_{\somevarpi\com\bar{j}}=\sigma_{j}$. But by part \ref{item:s_j,sigma_j}, commutativity and  \ref{item:s_j-bar,sigma_j-bar},
\[ s_j=k(s_i)=\somevarpi(\bar{k}(s_i))=\somevarpi(s_{\bar{j}}),\]
\[ \sigma_j=\sup k``\sigma_i=\sup \somevarpi\com \bar{k}``\sigma_i=\sup \somevarpi``\sigma_{\bar{j}}.\]

Part \ref{item:sigma_i=sigma_j-bar=sigma_j}:
This follows immediately from the preceding parts
together with the strong finiteness of $\bar{\Tt}$
(in particular using that the Dodd core of the first
extender used along $b^{\bar{\Tt}}$ is finitely generated).
\end{proof}

The following claim shows that the first extenders in  the Dodd decompositions of $\bar{j}$ and $j$ are identical.
 Let $\bar{\alpha}$ be least such that $\bar{\alpha}+1\in
b^{\bar{\Tt}}$, and $\alpha$ likewise for $\Tt$.

\begin{clmsix}\label{clm:first_Dodd_cores_match}
$\core_{\ds}(E^{\bar{\Tt}}_{\bar{\alpha}})=\core_{\ds}(E^\Tt_\alpha)$.
\end{clmsix}
\begin{proof}
 By the previous claim, we have
$\sigma_{\bar{j}}=\sigma_j=\kappa^{+\bar{M}}$,
and $\somevarpi(s_{\bar{j}})=s_j$. But then $E_{\bar{j}}\rest(\sigma_{\bar{j}}\cup
s_{\bar{j}})$ is equivalent to $E_j\rest(\sigma_j\cup s_j)$,
and these extenders are equivalent to the Dodd cores mentioned in the claim.
\end{proof}

So we have shown that $\bar{j}=i^{\bar{\Tt}}$ and $j=i^\Tt$ yield the same first extenders in their Dodd decompositions. We now want to show that they have the same second extenders, etc, proceeding all the way through.

Let
$G_0=\core_{\ds}(E^{\bar{\Tt}}_{\bar{\alpha}})=\core_{\ds}(E^\Tt_\alpha)$,
and let
 $\bar{M}_0=\bar{M}$ and $\bar{M}_1=\Ult_m(\bar{M},G_0)$.
 Let $j_{01}:\bar{M}\to\bar{M}_1$ and $j_{1\infty}:\bar{M}_1\to Q$ and
$\bar{j}_{1\infty}:\bar{M}_1\to\bar{Q}$
 be the Dodd decomposition maps (Lemma \ref{lem:Dodd_decomposition});
 in particular, $j_{01}$ is the ultrapower map.
So
 \[ \bar{M}_1=\cHull_{m+1}^{\bar{Q}}(\kappa^{+\bar{M}}\cup\{s_{\bar{j}}\}),
\]
 $\bar{j}_{1\infty}$ is the uncollapse map and
$\bar{j}_{1\infty}(t_{G_0})=s_{\bar{j}}$,
and likewise regarding $j_{1\infty}:\bar{M}_1\to Q$.
Since
$\somevarpi(s_{\bar{j}})=s_j$ and
$\sigma_{\bar{j}}=\sigma_j=\kappa^{+\bar{M}}$,
 we have $\somevarpi\com\bar{j}_{1\infty}=j_{1\infty}$.

By the claims, $\kappa^{+\bar{M}}\cup\{s_{\bar{j}}\}\sub\rg(\bar{k})$,
 so we can define $i_1:\bar{M}_1\to \bar{U}$  by
$i_1=\bar{k}^{-1}\com\bar{j}_{1\infty}$.
We get an extended commuting diagram (Figure \ref{fgr:third_commuting}).

\begin{figure}
\centering
\begin{tikzpicture}
 [mymatrix/.style={
    matrix of math nodes,
    row sep=0.35cm,
    column sep=0.4cm}]
   \matrix(m)[mymatrix]{
  \bar{U}&    &{}   &{}   {} &       {}& {} & Q& \\
   {} &{}& {} \\
   {} & {}&{} &{} & \bar{Q}\\
   {} & {} &\bar{M}_1\\
 \bar{M}\\};
 \path[->,font=\scriptsize]
%maps from left $M$
(m-5-1) edge[bend right=45] node[above,pos=0.3] {$j=i^\Tt\ \ $} (m-1-7)
(m-5-1) edge node[left] {$i$} (m-1-1)
(m-1-1)edge node[above] {$\bar{k}$} (m-3-5)
(m-4-3) edge node[left] {$i_1$} (m-1-1)
(m-5-1) edge[above] node {$j_{01}$} (m-4-3)
(m-4-3) edge[above] node[pos=0.4] {$\bar{j}_{1\infty}\ $} (m-3-5)
(m-4-3) edge[bend right=35] node[below,pos=0.25] {$j_{1\infty}$} (m-1-7)
(m-3-5)edge node[above] {$\somevarpi$} (m-1-7)
(m-1-1) edge node[above] {$k=i^\Uu$} (m-1-7);
\end{tikzpicture}
\caption{The diagram commutes. Also
$\bar{j}=\bar{j}_{1\infty}\com\bar{j}_{01}$.}
\label{fgr:third_commuting}
\end{figure}

\begin{clmsix}\label{clm:M-bar_1_neq_Q-bar}$\bar{M}_1\neq\bar{Q}$.\end{clmsix}
\begin{proof}
Suppose $\bar{M}_1=\bar{Q}$,
so $\left<G_0\right>$ is the full Dodd decomposition of $\bar{j}$,
so $\bar{j}_{1\infty}=\id$.
Then $i_1:\bar{Q}\to\bar{U}$ and
$\bar{k}\com i_1=\id$ (because all the relevant
generators
are in $\rg(\bar{k})$). Therefore
$\bar{Q}=\bar{U}$ and $\bar{k}=i_1=\id$ and $\bar{U}$ is a normal iterate
of $\bar{M}$, via $\bar{\Tt}$,
and $i=i^{\bar{\Tt}}=\bar{j}_{0\infty}$.
 (Therefore the comparison of $\bar{M}$
with $U$ actually yields $\Tt=\bar{\Tt}$ and $\Uu$ is trivial.)
Moreover, $\bar{\Tt}$ is strongly finite.
But we arranged that $M\sats\psi$
(from line (\ref{eqn:psi})), which
ensured that  no such $\bar{\Tt}$ exists, a contradiction.
\end{proof}

Since $\bar{M}_1\neq\bar{Q}$ and $\bar{M}_1$ appears along the Dodd decomposition of $\bar{j}=i^{\bar{\Tt}}$,
$\crit(\bar{j}_{1\infty})$ exists.
By commutativity, $\crit(j_{1\infty})$ also exists (so $\bar{M}_1\neq Q$ also). We now get the following slight variant of Claim \ref{clm:crit(k)<crit(j)}:

\begin{clmsix}\label{clm:crit(k)<crit(j1)} Let $\kappa_1=\crit(j_{1\infty})$.
Then:
\begin{enumerate}[label=\arabic*.,ref=\arabic*]
\item\label{item:crit(k)>crit(j1)}
$\kappa_1=\crit(i_1)=\crit(\bar{j}
_{1\infty})=\crit(j_{1\infty})<\kappa_1^{+\bar{U}}<\crit(k)=\min(\crit(\bar{k}) ,
\crit(\somevarpi))$.
\item\label{item:kappa^+_agmt1}
$\bar{M}_1||\kappa_1^{+\bar{M}_1}=\bar{U}||
\kappa_1^{+\bar{U}}=\bar{Q}||\kappa_1^{+\bar{Q}}=Q||\kappa_1^{+Q}$.
\end{enumerate}
\end{clmsix}
\begin{proof}
Let $\bar{\alpha},\alpha$
be as before; that is,
$\bar{\alpha}$ is least such that $\bar{\alpha}+1\in b^{\bar{\Tt}}$,
and $\alpha$ likewise for $\Tt$.
If $E^\Tt_{\alpha}$ is Dodd-sound, or equivalently, $G_0=E^\Tt_\alpha$,
then
$\bar{M}_1=M^\Tt_{\alpha+1}$ and
and $j_{1\infty}=i^\Tt_{\alpha+1,\infty}$, and we can argue as in the proof of Claim \ref{clm:crit(k)<crit(j)} to show that $\crit(k)>\crit(j_{1\infty})$, and hence
$\crit(i_1)=\crit(\bar{j}_{1\infty})=\crit(j_{1\infty})$, etc.
Suppose instead that $E^\Tt_\alpha$ is non-Dodd-sound.
Let $\gamma$ be such that $\gamma<^{\Tt}_{\Da}\alpha$
with $\crit(E^\Tt_\gamma)$ minimal.
Then $E^\Tt_\gamma\rest\nu(E^\Tt_\gamma)$,
which is used in the comparison, is derived
from $j_{1\infty}$.
So we can still argue as before to show that
$\crit(k)>\crit(j_{1\infty})$, etc.
\end{proof}

We next adapt Claim \ref{clm:s_jbar=kbar(s_i)};
the proof is essentially identical, so we leave it to the reader:

\begin{clmsix}\label{clm:s_jbar1=kbar(s_i1)} We have:
\begin{enumerate}[label=\arabic*.,ref=\arabic*]
\item\label{item:s_j1,sigma_j1}
 $s_{k\com i_1}=s_{j_{1\infty}}=k(s_{i_1})\text{ and }\sigma_{k\com i_1}=\sigma_{j_{1\infty}}=\sup
k``\sigma_{i_1}$.
 \item\label{item:s_j-bar1,sigma_j-bar1}
 $s_{\bar{k}\com i_1}=s_{\bar{j}_{1\infty}}=\bar{k}(s_{i_1})$ and
$\sigma_{\bar{k}\com i_1}=
\sigma_{\bar{j}_{1\infty}}=\sup\bar{k}``\sigma_{i_1}$.
\item\label{item:s_j1,sigma_j1_as_rho_com_j-bar1} $s_{\somevarpi\com\bar{j}_{1\infty}}=s_{j_{1\infty}}=\somevarpi(s_{\bar{j}_{1\infty}})$
and $\sigma_{\somevarpi\com\bar{j}_{1\infty}}=\sigma_{j_{1\infty}}=\sup\somevarpi``\sigma_{\bar{j}_{1\infty}}$.
\item\label{item:sigma_i1=sigma_j-bar1=sigma_j1} $\kappa^{+\bar{M}_1}=\sigma_{i_1}=\sigma_{\bar{j}_{1\infty}}=\sigma_{j_{1\infty}}$.
\end{enumerate}
\end{clmsix}

Let $\bar{\alpha}_1$ be least such that $\bar{\alpha}_1+1\in
b^{\bar{\Tt}}\cut\{\bar{\alpha}+1\}$, and $\alpha_1$ likewise for $\Tt$
(with respect to $\alpha$). We now adapt Claim \ref{clm:first_Dodd_cores_match};
again the proof is just like before:

\begin{clmsix}\label{clm:second_Dodd_cores_match}
$\core_{\ds}(E^{\bar{\Tt}}_{\bar{\alpha}_1})=\core_{\ds}(E^\Tt_{\alpha_1})$.
\end{clmsix}

So $\bar{j}$ and $j$ also
have the same second extenders in their respective Dodd decompositions.

We are now in a position to define $G_1=\core_{\ds}(E^{\bar{\Tt}}_{\bar{\alpha}_1})=\core_{\ds}(E^\Tt_{\alpha_1})$, $\bar{M}_2$, $i_2:\bar{M}_2\to\bar{U}$, etc, and we get another commuting diagram,
with $\bar{M}_2$ situated between $\bar{M}_1$ and $\bar{Q}$.

This process extends directly  to all finite stages of the Dodd decomposition of $\bar{j}$.
But $\bar{\Tt}$ is finite, so we  reach some stage $\ell<\om$
with
$\bar{M}_{\ell}=\bar{Q}$.
However, the proof of Claim \ref{clm:M-bar_1_neq_Q-bar}  adapts directly (with ``$\ell$'' replacing ``$1$''
and ``$\left<G_0,\ldots,G_{\ell-1}\right>$'' replacing ``$\left<G_0\right>$'') to show that in fact, $\bar{M}_{\ell}\neq\bar{Q}$.
This contradiction completes the proof.
\end{proof}

\section{Super-Dodd-soundness}\label{sec:Dodd_proof}

The following theorem is basically due to the combination
of work of Steel
%conf
\cite[Theorem 3.2]{combin},
%conf
\cite[Theorem 4.1]{deconstructing} and Zeman
%conf
\cite[Theorems 1.1, 1.2]{zeman_dodd}. The proof we
give here is somewhat different, however.
It will involve the methods of the proof of Lemma \ref{lem:measures_in_mice},
with which the reader should probably be familiar.

\begin{tm}[Super-Dodd-soundness]
\label{thm:super-Dodd-soundness}
Let $M$ be an active, $(0,\om_1+1)$-iterable premouse, let $\kappa=\kappa^M$,
and suppose that
either:
\begin{enumerate}[label=\tu{(}\alph*\tu{)}]
 \item\label{item:M_1-sound} $M$ is $1$-sound, or
 \item\label{item:M_crit-sound} $M$ is $\kappa^{+M}$-sound.
 \end{enumerate}
 Then $M$ is super-Dodd-sound.
\end{tm}

In this section we prove the following lemma, which is the central
argument for the proof of super-Dodd-soundness:

\begin{lem}[Super-Dodd-soundness]
\label{lem:super-Dodd-soundness}
Assume the hypotheses of Theorem \ref{thm:super-Dodd-soundness}.
Suppose further that  $\core_1(M)$ is $1$-sound.
Then the conclusion of Theorem \ref{thm:super-Dodd-soundness} holds.
\end{lem}

\begin{rem}
 The hypothesis that $\core_1(M)$ is $1$-sound is actually superfluous,
 by Theorem \ref{thm:solidity},
but we will have to use the lemma in order to prove this fact.

The proof of super-Dodd-soundness  to follow
proceeds by first reducing to
the case in
which normal
iterability yields an iteration strategy with the weak Dodd-Jensen property,
and
then approximately
follows Steel's proof. Actually this reduction is very easy,
and just takes a few lines of argument. However, Steel's
proof makes significant use of the
assumption that $M$ has no extenders of superstrong type on its sequence, so
does not fully suffice
for our purposes. There is some more work involved in adapting things
to the superstrong level. There are also some minor modifications
for the \emph{super-}
aspect of super-Dodd-soundness.
But none of this adaptation work  relates directly to proving fine
structure
from normal iterability. Moreover, Zeman \cite{zeman_dodd} already proved
(standard) Dodd-soundness for $1$-sound mice with $\lambda$-indexing,
at the superstrong level
(from $(0,\om_1,\om_1+1)^*$-iterability). So after the reduction to the case
that we have  a strategy with weak Dodd-Jensen,
one could presumably  follow Zeman's proof closely, translating
it to Mitchell-Steel indexing,
at least for standard Dodd-soundness. The proof we give presumably does have
significant overlap with Zeman's, but is instead based on Steel's
and an elaboration of the proof of Theorem \ref{thm:measures_in_mice}.%\setcounter{footnote}{0}
\footnote{\label{ftn:other_Dodd}Alternatively, it
seems one might
prove Dodd-soundness (if not super-Dodd-soundness) at the level of
superstrongs in Mitchell-Steel indexing,
by quoting Zeman,  together with Fuchs' translation
process \cite{fuchs_translating_models}, \cite{fuchs_translating_strategies},
along with unpublished work of the author
on translation of iteration trees between $\lambda$-iteration rules
and Mitchell-Steel rules (see \cite{rule_conversion_v2}),
but the
author
has not
considered this option seriously, nor tried to
understand all of the fine structure involved. Anyway,
that path would presumably be significantly less efficient than the one we take
here.}
\end{rem}

Before we start, note the following, which will be useful, since $M$ is active:

\begin{dfn}\label{dfn:Tt-special}
Let $\Tt$ be a $0$-maximal
iteration tree.
 For $\alpha+1<\lh(\Tt)$, say
 $\alpha$ is \dfnemph{$\Tt$-special}
iff $(0,\alpha]^\Tt\cap\dropset^\Tt=\emptyset$
and $E^\Tt_\alpha=F(M^\Tt_\alpha)$.
\end{dfn}

\begin{lem}\label{lem:Tt-special}
Let $\Tt$ be $0$-maximal,  $\alpha$ be $\Tt$-special and
 $\beta=\pred^\Tt(\alpha+1)$.
Then $\beta\leq^\Tt\alpha$.
If $\beta<\alpha$ then
$(\beta,\alpha]^\Tt$ is non-dropping,
 $\crit(i^\Tt_{\beta\alpha})>\crit(E^\Tt_\alpha)$,
and $\gamma$ is non-$\Tt$-special
for every $\gamma+1\in(\beta,\alpha]^\Tt$.
Hence, if all proper segments of $M^\Tt_0$  are Dodd-sound
then $\gamma$ is Dodd-nice for each $\gamma+1\in(\beta,\alpha]^\Tt$.
\end{lem}

\begin{proof}[Proof of Lemma \ref{lem:super-Dodd-soundness}]
The plan is as follows. We first reduce to  case \ref{item:M_1-sound}, that
$M$ is $1$-sound, and then reduce
further to the case in which
normal iterability is sufficient to provide an iteration strategy with the weak
Dodd-Jensen
property. From there we proceed with a comparison argument similar to Steel's
original proof, augmented with
further analysis which is necessary
to handle the presence of superstrong extenders.
We will give the rough idea of this comparison argument  after we perform the reductions just mentioned.

So let us first reduce
to case \ref{item:M_1-sound}.
Suppose that \ref{item:M_crit-sound} holds.
Let $C=\core_1(M)$ and let $\pi:C\to M$ be the core map. So (by
our added assumption for the lemma)
$C$ is $1$-sound,
so assuming the result in case \ref{item:M_1-sound},
$C$ is
super-Dodd-sound. So if $\kappa=\kappa^M<\rho_1^M$ then $M=C$ is
super-Dodd-sound.
And if $\kappa=\rho_1^M$ then note that we still have $M=C$,
since $\pi$ is a $0$-embedding (and $M$ is type 2),
so $\rg(\pi)$ is
cofinal
in $\kappa^{+M}$.
So we may assume  $\rho_1^M<\kappa$.
Since
$\pi$
is
$\rSigma_1$-elementary, it suffices to see  $\pi(\ttilde^C)=\ttilde^M$ and
$\tautilde^M=0$.
Let
$t'\in[\nu(F^M)]^{<\om}$
be least generating
$(\Fseg^M,p_1^M)$ (with respect to $F^M$). Since $\Fseg^M,p_1^M\in\rg(\pi)$, we
have $t'\in\rg(\pi)$.
And $\kappa$ is
generated by $t'$, because $\Fseg^M$ is.

We claim  $t'=\ttilde^M$ and $\tautilde^M=0$. In fact, because $M$ is
$\kappa^{+M}$-sound and
$p_1^M,\Fseg^M$ are generated by $t'$, by
 hypothesis \ref{item:M_crit-sound} and Lemma
\ref{lem:type_2_factor_embedding}, $t'$ generates
$F^M$. It follows that
$\tautilde^M=0$ and $\ttilde^M\leq t'$. But then the minimality of $t'$ implies
that $\ttilde^M=t'$.

So $\ttilde^M\in\rg(\pi)$ and $\tautilde^M=0$. But then it is easy to see that
$\pi(\ttilde^C)=\ttilde^M$, completing the proof.

Now assume that \ref{item:M_1-sound} holds; that is, $M$ is $1$-sound. We will show
$M$ is
super-Dodd-sound. We may and do assume that
$M$ is type 2 and all its proper segments are Dodd-sound.
Moreover, $M$ is Dodd-amenable. For let $\kappa=\kappa^M$ and suppose that $\kappa^{+M}<\tau^M$ (otherwise Dodd-amenability is immediate).
Then by
Lemma \ref{lem:tau=rho_1}, $\kappa^{+M}<\tau^M=\rho_1^M$.
But then for every $\alpha<\tau^M$,
we have $F\rest\alpha\cup\{t^M\}\in M$,
since this is coded as an $\bfrSigma_1^M$ subset of $(\kappa^{+M}\cross\alpha)^{<\om}$, and $\kappa^{+M},\alpha<\rho_1^M$.
So it suffices to prove that $M$ is
super-Dodd-solid.

We now  take the $\Sigma_1$-hull over $M$
of a singleton $\{q\}$, with $q$ selected capturing the relevant objects
and such that the resulting transitive collapse is sound. For this, we apply
%conf
\cite[Definition 2.2, Lemma 2.3]{V=HODX_pub}:
Let $(q,\om)\in\Dd$ be $1$-self-solid for $M$ with
$\ttilde^M,\tautilde^M\in\Hull_1^M(\{q\})$, let
$\Mbar=\cHull_1^M(\{q\})$
and $\pi:\Mbar\to M$ be the uncollapse (such a pair $(q,\om)$ exists by
\cite[Lemma 2.3]{V=HODX_pub}; in the notation there, the ``$q$'' and ``$\om$'' are written in the other order). So $\rho_1^{\bar{M}}=\om$ and $\bar{M}$ is sound.

\begin{clmone}\label{clm:Mbar_sDs_implies_M_sDs} If $\Mbar$ is super-Dodd-solid then so is $M$.\end{clmone}

\begin{proof} Note
$\pi(\ttilde^{\Mbar}\rest\lh(\ttilde^M))=\ttilde^M$,
so if $\vec{E}\in M$ witnesses super-Dodd-solidity for $\Mbar$, then
$\pi(\vec{E}\rest\lh(t^M))$
witnesses super-Dodd-solidity for $M$.
\end{proof}

Also, all proper segments of $\Mbar$ are Dodd-sound. So resetting notation, we may assume:

\begin{ass}\label{ass:rho_1^M=om} $\rho_1^M=\om$.\end{ass}

Recall $M$ is also $1$-sound. We can therefore
use these fine structural circumstances to substitute
for any appeals to weak Dodd-Jensen in the proofs.
(Moreover, the unique $(0,\om_1+1)$-strategy for $M$ has  weak Dodd-Jensen.)
Thus, we have completed the portion of the proof of super-Dodd-solidity which
is directly relevant to the main theme of the paper
(that is, proving fine structural facts from normal iterability).
From here on we can just approximate Steel's proof of
Dodd-solidity (the conjunction of \cite{combin} and \cite{deconstructing},
augmented with observations in
\cite{operator_mice_v3}),
modified to deal with superstrong extenders and to prove super-Dodd-solidity.
We
will, however,
give a full argument (diverging somewhat from Steel's argument).

\label{page:outline_Dodd-soundness}
In the main argument (which we will only need to use under some further assumptions, specified in Assumption \ref{ass:v_not_empty}), we will consider a comparison of $M$ with a phalanx $\mathfrak{U}$. This phalanx could, for example,  be of form $\mathfrak{U}=((M,{<\zeta}),U,\zeta)$ where
$U=\Ult_0(M,G)$ for some Dodd-solidity witness $G$ corresponding to $\zeta\in t^M$, where,
for example, $\zeta$ is an $M$-cardinal
and $\kappa<\zeta<\lgcd(M)$.\footnote{This is the direct analogue of the proof of ISC in \cite[\S10]{fsit}. In  Steel's proof of Dodd-solidity, say in case $\zeta$ is an $M$-cardinal and $\zeta<\lgcd(M)$,
instead of comparing with $\phU$, $M$ is actually compared with the phalanx $((M,{<\zeta}),H,\zeta)$,
where $H$ is as described just before Claim \ref{clm:H_props} below. But these two comparisons are just slight translations of one another. We find it a little more convenient to work with $\phU$ instead.}
 Using the fine structure,
 we will see that neither $b^\Uu$ nor $b^\Tt$ drops, $b^\Uu$ is above $U$, and $i^\Uu\com i^{M,0}_G=i^\Tt$. This will give that $G$ is a sub-extender of that derived from $i^\Tt$, using a set of generators of form $\zeta\cup x$, for some finite $x$.
 By analyzing the extenders used along $b^\Tt$,
 we will see that this sub-extender can in fact be computed working inside $M$. The main difficulty will be in handling  the possibility that some $\Tt$-special extender $E^\Tt_\alpha$
 is used along $b^\Tt$ (that is,
 $E^\Tt_\alpha$ is a non-dropping image of $F^M$). The special case of this in which the first extender used along $b^\Tt$ is $\Tt$-special, is somewhat easier to handle.
 The other case (the first extender used along $b^\Tt$ is non-$\Tt$-special, but some later one is $\Tt$-special) is readily ruled out in Steel's proof, because it implies that there are superstrong extenders in $\es^M$.
 But that argument does not suffice for us here,
 so we need to handle it directly, which takes some more work
 (see especially  Claim \ref{clm:varepsilon_non-Tt-special} of Case \ref{case:Ds_xi_is_non-Tt_special}).
The measure analysis argument
from the proof of Theorem \ref{thm:measures_in_mice}
will be useful in this regard. (Another variant
of the measure analysis argument
will also be needed later in the proof
of Theorem
\ref{thm:finite_gen_hull}.)

So let $\zeta\in\ttilde^M$, let $w=\ttilde^M\cut(\zeta+1)$, and let
$G=\Dw^M(w,\zeta)$. We must see $G\in M$.
We first dispense (in Claims \ref{clm:type_z_in}
and \ref{clm:if_zeta_in_p_1_cut_kappa^+_then_G_in_M} below) with some easy cases.

\begin{clmone}\label{clm:F_rest_t-tilde_not_in_M}
$F^M$ is generated by $\widetilde{t}^M$, so
$F^M\rest\widetilde{t}^M\notin M$.
\end{clmone}
\begin{proof}
We have $\tau^M=\kappa^{+M}$,
by Lemma \ref{lem:tau=rho_1},
Assumption \ref{ass:rho_1^M=om}
and since $M$ is $1$-sound.
So $\widetilde{\tau}^M=0$,
which suffices.
\end{proof}

\begin{clmone}\label{clm:type_z_in} If $w=\emptyset$, or equivalently, $\nu(F^M)=\zeta+1$,
then $G\in
M$.\end{clmone}
\begin{proof}If  $G$ is non-type Z, this is
by the ISC. Suppose $G$ is type Z. We argue like in
%conf
\cite[Theorem 2.7]{deconstructing};
here is a sketch.
Since $G$ is type Z, it has a largest generator $\xi$, which is a limit of generators (of both $G$ and $F^M$).
So $\bar{G}=G\rest\xi$
is type 3,
and by the ISC, there is $R\pins M$ with $F^R=\bar{G}$.
Let
\[ \bar{U}=\Ult_0(M,\bar{G})\text{ and } U=\Ult_0(M,G)\text{ and }W=\Ult_0(M,F^M),\]
and let
$\bar{\pi}:\bar{U}\to U$ and $\pi:U\to W$ be the factor maps.
So $\crit(\bar{\pi})=\xi$ and $\crit(\pi)=\zeta>\xi$.
In fact, $\crit(\pi)=\zeta=\xi^{+U}$. For since $\crit(\pi)>\xi$, we have $\crit(\pi)\geq\xi^{+U}$. But $\xi^{+U}$ is not a $W$-cardinal, since $\bar{G}\in W$ (by coherence), and therefore $\crit(\pi)\leq\xi^{+U}$.

\begin{sclmone}\label{sclm:xi=lgcd(M)}$\xi$ is the largest cardinal of $M$, and $\xi$ is inaccessible in $M$.\end{sclmone}
\begin{proof} Since $R\pins M$, we have
$\bar{G}\in M$. It follows that $\card^M(\zeta)\leq\xi$, and since $\zeta\geq\lgcd(M)$, therefore $\xi\geq\lgcd(M)$.
So we just need to see that $\xi$ is inaccessible in $M$.
As $G$ is type Z,
 $\lh(F^R)=\xi^{+\bar{U}}=\xi^{+U}$.
By  condensation for $\om$-sound mice
(Fact \ref{fact:om_condensation}) with $\bar{\pi}$,
$\bar{U}||\xi^{+\bar{U}}=U||\xi^{+U}$,
and $\xi$ is inaccessible in both $\bar{U}$ and $U$.
But $\pi(\xi)=\xi$, so $\xi$ is inaccessible in $W$, and so by coherence, $\xi$ is also inaccessible in $M$, as desired.
\end{proof}

Now  consider the phalanx
\[\ph=((M,{<\xi}),(R,\xi),U,\zeta=\xi^{+U}).\]
Note $\ph$ is $((0,-1,0),\om_1+1)$-iterable\footnote{\label{ftn:anomalous_footnote}The degree $-1$ is used
because $R$ is active type 3 with $\rho_0^R=\nu(F^R)=\xi$, and extenders $E$ only apply to $R$ with $\crit(E)=\xi$. The ultrapower $\Ult_{-1}(R,E)=\Ult(R,E)$ is just the usual ultrapower, formed without squashing. Iterability just requires that all models produced are wellfounded, not that they be premice. If $\crit(E^\Uu_\alpha)=\xi$
then $M^\Uu_{\alpha+1}$ is not a premouse, even if it is wellfounded,
because it fails the ISC.},
because we can lift $(0,-1,0)$-maximal trees on it to trees on the $(0,0,0)$-iterable phalanx
\[ \ph'=((M,{<\xi}),(M,\xi),W,\OR^M=\xi^{+W}),\]
using the identity map $M\to M$, the inclusion map $R\to R\pins M$, and $\pi:U\to W$ as lifting maps. The execution of the lifting process is much like (but somewhat simpler than) that in the proof of Claim \ref{clm:Ds_phU_iterable} below, so we omit further details here.
We compare $\ph$ with $M$, resulting in trees $\Uu,\Tt$ respectively.

\begin{sclmone} $b^\Uu$ above $U$,  $b^\Uu,b^\Tt$ are non-dropping,
$M^\Uu_\infty=M^\Tt_\infty$
and $i^\Uu\com i^{M,0}_{G}=i^\Tt$.\end{sclmone}\begin{proof} This is because $M$ is $1$-sound and $\rho_1^M=\om$ and by standard calculations. So we leave the most of the proof to the reader, but just make a couple of remarks on some details.

Suppose
$M^\Uu_\alpha$ is above $R$
without drop.
Then $M^\Uu_\alpha$ is not a premouse, because it fails
the ISC with respect to $F(M^\Uu_\alpha)\rest\xi$.
In particular, $M^\Uu_\alpha\nins M^\Tt_\alpha$.

And if $\beta+1<\lh(\Uu)$ and $R\neq M^{*\Uu}_{\beta+1}$
then $E^\Uu_\beta$ is close
to $M^{*\Uu}_{\beta+1}$;
the proof is very much like \cite[6.1.5]{fsit},
and also we give related arguments later in the
current proof, and also in the proof of solidity.
So the iteration maps of $\Uu$ preserve fine structure
enough that the usual arguments work,
and in particular,
$i^\Uu\com i^{M,0}_G=i^\Tt$
by $p_1$-preservation.\end{proof}

But now note that $E^\Uu_0=F^R=\bar{G}$,
and $M^\Uu_1$ has a total type 1 extender $F$ on its sequence
such that $G=F\com\bar{G}$. Since $\bar{G}\in M$, it follows that $G\in M$, establishing the claim.
\end{proof}

\begin{clmone}\label{clm:if_zeta_in_p_1_cut_kappa^+_then_G_in_M}If $\zeta\in p_1^M\cut\kappa^{+M}$ then
$G\in M$.
\end{clmone}
\begin{proof}
Let $p=p_1^M\cut(\zeta+1)$ and
$N=\cHull_1^M(\zeta\cup\{p\})$ and $\pi:N\to M$ the uncollapse. By $1$-solidity, $N\in M$.
By Lemma \ref{lem:Dodd_param_p_1},
$w=\ttilde^M\cut(\zeta+1)\in\rg(\pi)$.
But then $G=\Dw^M(w,\zeta)$ is equivalent
to $F^N\in M$.
\end{proof}

By Claims \ref{clm:type_z_in} and \ref{clm:if_zeta_in_p_1_cut_kappa^+_then_G_in_M}, Lemma \ref{lem:Dodd_param_p_1} and induction,
and noting that $\crit(F^M)\notin p_1^M$, we may assume:

\begin{ass}\label{ass:v_not_empty}\
\begin{enumerate}[label=--]
 \item $w\neq\emptyset$,
\item
 $\Dw^M(w\cut(\alpha+1),\alpha)\in M$ for all $\alpha\in w$, and
\item  $\zeta\notin p_1^M$,
so $F^M_\downarrow$ is not generated
by $\zeta\cup w$.
\end{enumerate}
\end{ass}

Under these assumptions we will deal with the main case, which will involve a comparison similar to that in the proof of Claim \ref{clm:type_z_in}. This will be a comparison between a phalanx $\mathfrak{U}$, defined below, and $M$; we introduced an example case of this comparison in the outline given just after Assumption \ref{ass:rho_1^M=om}.
Let $U=\Ult_0(M,G)$, $W=\Ult_0(M,F^M)$, and $\pi:U\to W$ be the factor map. So $\crit(\pi)=\zeta$.
Let $X=w\un\zeta$, $H=M_X$ and $\pi^-=\pi_X:H\to M$ (as in
Definition \ref{dfn:type_2_factor_embedding}). Note that $F^H=\trivcom(G)$, $H^{\passive}=U|\max(\bar{w})^{+U}$ where $\pi(\bar{w})=w$,
 $\pi^-=\pi\rest H$,  and
$\crit(\pi^-)=\crit(\pi)=\zeta$. So $\zeta$ is an $H$-cardinal and a $U$-cardinal.

\begin{clmone}\label{clm:H_props}
If  $\zeta>\kappa$ then
 $H$ fails the ISC.
\end{clmone}
\begin{proof}
By Assumption \ref{ass:v_not_empty} and Lemma
\ref{lem:type_2_factor_embedding}.
\end{proof}

If $\zeta$ is an $M$-cardinal, define the phalanx
\[ \phU=((M,{<\zeta}),U,\zeta). \]
And if $\card^M(\zeta)=\delta<\zeta$ (in which case $\zeta=\delta^{+U}$), define the phalanx
\[ \phU=((M,{<\delta}),(R,\delta),U,\zeta),\]
where $R\pins M$,
 $\zeta=\delta^{+R}$, $r\in\{-1\}\cup\om$ and $\rho_{r+1}^R=\delta<\rho_r^R$, where $\rho_{-1}^R=\OR^R$ (which is only relevant in case $R$ is type 3).

 \begin{clmone}\label{clm:Ds_phU_iterable}$\phU$ is $((0,0),\om_1+1)$-iterable,
or $((0,r,0),\om_1+1)$-iterable, as appropriate.\footnote{For the meaning in case $r=-1$, see Footnote \ref{ftn:anomalous_footnote}.}\end{clmone}

The iterability proof is mostly a routine copying process, and we won't need the details of the process later. For these reasons we postpone it to later,
beginning on page \pageref{page:Ds_iterability_proof}.
(However, there is a new wrinkle which appears at the superstrong level, as we explain there.)

We now compare $\mathfrak{U}$ with $M$.
Recall that we use the slight tweak of the usual comparison process described in Remark \ref{rem:superstrong_diff}, and note that in the present comparison, there can be $\alpha<\lh(\Uu)$ such that $M^\Uu_\alpha$ fails to be a premouse. This happens just if  $R$ exists and is active type 3 with $\delta=\lgcd(R)$,
 $\alpha$ is above $R$,
 and $(R,\alpha]^\Uu\cap\dropset^\Uu=\emptyset$; here the iteration map $i^{\Uu}_{R,M^\Uu_\alpha}:R\to M^\Uu_\alpha$ has critical point $\delta$ and $F^{M^\Uu_\alpha}$ fails the ISC with respect to $F^R\notin M^\Uu_\alpha$.

\begin{clmone}\label{clm:comparison_succeeds} There is a successful comparison $(\Uu,\Tt)$ of $(\phU,M)$.\end{clmone}
\begin{proof}
This is like in the classical fine structure proofs. (Because $U$ is a premouse, the usual arguments with the ISC work to show that an attempted comparison cannot last through length $\om_1+1$.  Note that the failures of the ISC in models $M^\Uu_\alpha$ mentioned above
do not interfere with comparison termination, because  $\crit(F^{M^\Uu_\alpha})=\crit(F^R)$,
so $F^{M^\Uu_\alpha}$ cannot be used as the ``typical'' extender along a branch of length $\om_1$.)
\end{proof}
We will now analyze the comparison.

\begin{clmone}\label{clm:super-D-s_closeness}
Let $\alpha+1<\lh(\Uu)$. Then:
\begin{enumerate}[label=\arabic*.,ref=\arabic*]
\item \label{item:super-D-s_closeness}
$E^\Uu_\alpha$
is close to $M^{*\Uu}_{\alpha+1}$.
\item\label{item:super-D-s_measures_in_when} If $0<^\Uu\alpha+1$
and $(0,\alpha+1]^\Uu\cap\dropset^\Uu=\emptyset$ and
$\crit(E^\Uu_\alpha)<i^\Uu_{0\beta}(\crit(F^U))$ where
$\beta=\pred^\Uu(\alpha+1)$
then $(E^\Uu_\alpha)_a\in M^\Uu_\beta$ for every $a\in[\nu(E^\Uu_\alpha)]^{<\om}$.
\end{enumerate}
\end{clmone}
\begin{proof}
Part \ref{item:super-D-s_closeness}: This is almost by the argument in \cite[6.1.5]{fsit},
but
there is a subtlety in the case that
$R$ is defined and $M^{*\Uu}_{\alpha+1}=R$,
so assume this.
Then $\crit(E^\Uu_\alpha)=\delta<\zeta=\delta^{+R}$,
and $\kappa^{+M}<\delta$.
A straightforward induction  on $\alpha$
shows that $E^\Uu_\alpha\neq F(M^\Uu_\alpha)$,
so $\mu_a=(E^\Uu_\alpha)_a\in M^\Uu_\alpha$
for every $a\in[\nu(E^\Uu_\alpha)]^{<\om}$,
so in fact $\mu_a\in U||\zeta^{+U}$.
But  $\pi:U\to W$
with $\crit(\pi)=\zeta$, so by  condensation for $\om$-sound mice (Fact \ref{fact:om_condensation}),
either (a) $U||\zeta^{+U}=W||\zeta^{+U}=M||\zeta^{+U}=R||\zeta^{+U}$,
or (b) $R=M|\zeta=W|\zeta$ is active with extender $F=F^R$,
and $U||\zeta^{+U}=\Ult(R,F^R)||\zeta^{+U}$.
In either case, $\mu_a$ is close to $R$.

Part \ref{item:super-D-s_measures_in_when}: This follows
easily from an inspection of the proof of \cite[6.1.5]{fsit}.
\end{proof}

Recall that we are working under hypothesis \ref{item:M_1-sound} (that $M$ is $1$-sound)
and by Assumption \ref{ass:rho_1^M=om}, $\rho_1^M=\om$.
So with Claim \ref{clm:super-D-s_closeness}, it follows that fine
structure is
preserved by the
iteration maps in the usual way, $b^\Uu$ is
above $U$, neither $b^\Tt$ nor $b^\Uu$ drops,
 $M^\Tt_\infty=Q=M^\Uu_\infty$,
and letting $j=i^\Tt$, $k=i^\Uu$ and $i_G=i^{M,0}_G$,
we have $k\com i_G=j$, and $\crit(k)\geq\zeta$.

\begin{clmone}\label{clm:Ds_j,k_com_i_G_short}
We have:
\begin{enumerate}[label=\arabic*.,ref=\arabic*]
\item\label{item:Ds_i^Tt_short} $j$ is short; that is,
for each $\beta\in(0,\infty)^\Tt$,
we have $\crit(i^\Tt_{\beta\infty})\leq i^\Tt_{0\beta}(\kappa)$.
\item\label{item:Ds_i^Uu_com_i_G_short} $k\com i_G$ is short; that is, for each $\beta\in[0,\infty)^\Uu$, we have $\crit(i^\Uu_{\beta\infty})\leq i^\Uu_{0\beta}(i_G(\kappa))$.\footnote{In fact,
arguments similar to those given for part \ref{item:Ds_i^Tt_short} give strict inequality here; that is,
for each $\beta\in[0,\infty)^\Uu$,
we have $\crit(i^\Uu_{\beta\infty})<i^\Uu_{0\beta}(i_G(\kappa))$. Therefore $i^\Uu$ is continuous at $i_G(\kappa)$.

If $M$ is below superstrong,
this has the consequence that
the same holds for $j$
for $\beta\in(0,\infty)^\Tt$;
that is, for each such $\beta$,
we have $\crit(i^\Tt_{\beta\infty})<i^\Tt_{0\beta}(\kappa)$.
For suppose $\beta\in(0,\infty)^\Tt$
is such that $\crit(i^\Tt_{\beta\infty})=i^\Tt_{0\beta}(\kappa)$.
Then the short extender derived from $i^\Tt$
has a whole segment of length $i^\Tt_{0\beta}(\kappa)$,
so $\{i^\Tt(f)(\kappa)\bigm|f\in M\}$ is bounded by $i^\Tt_{0\beta}(\kappa)$.
But if $M$ is below superstrong then $\{i_G(f)(\kappa)\bigm|f\in M\}$ is unbounded in $i_G(\kappa)$
%conf
(see the proof of \cite[Lemma 4.4]{deconstructing}).
And $\crit(i^\Uu)>\kappa$;
this is clear as $M$ is below superstrong, which implies $\crit(i^\Uu)\geq\zeta>\kappa$,
but actually holds more generally by Claim \ref{clm:Ds_crit(k)>crit(j)} below. But then since $i^\Uu$ is continuous at $i_G(\kappa)$,  it follows that $\{i^\Uu(i_G(f))(\kappa)\bigm|f\in M\}$
is unbounded in $i^\Uu(i_G(\kappa))$. But $i^\Uu\com i_G=i^\Tt$,
so this just says that $\{i^\Tt(f)(\kappa)\bigm|f\in M\}$ is unbounded in $i^\Tt(\kappa)$, a contradiction.

It follows that if $M$ is below superstrong and $\vareps+1\in b^\Tt$ and $\vareps$ is $\Tt$-special, then $\pred^\Tt(\vareps+1)=0$. This gives a much simpler proof of Claim \ref{clm:varepsilon_non-Tt-special} under these hypotheses, and this Steel's argument basically uses this kind of simpler argument at that point.}

\end{enumerate}
\end{clmone}
\begin{proof}
Part \ref{item:Ds_i^Tt_short}:
Suppose otherwise and let $\beta$ be the least counterexample.
Let $\kappa'=j(\kappa)$;
so $\kappa'=i^\Tt_{0\beta}(\kappa)$ and since $Q=M^\Tt_\infty$,
\[ M^\Tt_\beta=\cHull_1^{Q}(\kappa'\cup\rg(j))=\cHull_1^Q(\kappa'\cup\{p_1^Q\}).\]
Let $\gamma\in[0,\infty)^\Uu$
be least such that either $\gamma+1=\lh(\Uu)$ or  $i^\Uu_{0\gamma}(i_G(\kappa))<\crit(i^\Uu_{\gamma\infty})$. Then $\kappa'=k(i_G(\kappa))=i^\Uu_{0\gamma}(i_G(\kappa))$, and since $Q=M^\Uu_\infty$ and $k``(\zeta\cup\{w^U\})\sub\kappa'$,
\[ M^\Uu_\gamma=\cHull_1^Q(\kappa'\cup\rg(k))=\cHull_1^Q(\kappa'\cup\{p_1^Q\}),\]
so $M^\Uu_\gamma=M^\Tt_\beta$. But this implies that $M^\Uu_\gamma=Q=M^\Tt_\beta$, so the comparison terminates there, a contradiction.

Part \ref{item:Ds_i^Uu_com_i_G_short}
is similar.
\end{proof}
\begin{clmone}\label{clm:Ds_crit(k)>crit(j)}
$\crit(i_G)=\crit(j)=\kappa<\kappa^{+M}=\kappa^{+U}=\kappa^{+Q}<\crit(k)$.
\end{clmone}
\begin{proof}
If $\zeta>\kappa$, this is already clear. Otherwise
use the proof of Claim
\ref{clm:crit(k)<crit(j)} in the proof of Lemma
\ref{lem:measures_in_mice}.
\end{proof}

We have $\pi:U\to W$. Let $\pi(w^U)=w$,
so
$U=\Hull_1^U(\zeta\cup\{p_1^U,w^U\})$.

Let $\xi+1\in b^\Tt$ with $\pred^\Tt(\xi+1)=0$.  For \emph{$\Tt$-special} see Definition  \ref{dfn:Tt-special}.

\begin{caseone}\label{case:xi_Tt-special} $\xi$ is $\Tt$-special.

Let $M'=M^\Tt_\xi$, $s'=s_{F^{M'}}$ and $\sigma'=\sigma_{F^{M'}}$.
We have
\[ k(s_{i_G})=s_{k\com i_G}=s_j=s'=i^\Tt_{0\xi}(s_{F^M}),\]
\[ \sup k``\sigma_{i_G}=\sigma_{k\com i_G}= \sigma_j=\sigma'=\sup
i^\Tt_{0\xi}``\sigma_{F^M},\]
using Lemma \ref{lem:s_j,sigma_j_pres} for the first equality in each line, the fact that $k\com i_G=j$ for the second equalities, that $F^{M'}$ is derived from $j$ for the third equalities, and Lemma \ref{lem:type_2_s,sigma_pres} for the
final equalities.
By induction, $w\ins s_{F^M}$, so $w'=i^\Tt_{0\xi}(w)\ins s'$,
so $w'\ins k(s_{i_G})$ and $\lh(s_{i_G})\geq\lh(w)$.

\begin{clmone}\label{clm:sDs_s_i_G_rest_lh(w)<w^U}$s_{i_G}\rest\lh(w)< w^U$.\end{clmone}
\begin{proof}
If $s_{i_G}\rest\lh(w)>w^U$ then the extender derived from $i_G$
 with support $w^U\cup\zeta$ belongs to $U$, which is impossible
 as this is equivalent to $G$ itself.

So suppose $w^U\ins s_{i_G}$.
 Then $\max(s_{i_G})=\max(w^U)$.
So if $\zeta>\kappa$ then $H$ satisfies weak ISC and hence ISC
 (by Lemma \ref{lem:type_2_factor_embedding}),
 contradicting Claim \ref{clm:H_props}.
 So $\zeta=\kappa$, so  $U=\Hull_1^U(\{w^U\cup p_1^U\})$
 and $i^\Uu(w^U)=w'$.
 But
then $\crit(k)=\kappa$,
since the fact that $\zeta=\kappa$
implies that $\kappa$ is not $F^M$-generated by $w$, hence not $F^{M'}$-generated by $w'$,
and hence $\kappa\notin\Hull_1^Q(\{p_1^Q,w'\})$. This contradicts  Claim \ref{clm:Ds_crit(k)>crit(j)}.
\end{proof}

\begin{scaseone}
 $\kappa=\zeta$.

In this case we just need to see that the measure
$(F^M)_w\in M$.
But since $j=k\com i_G$ and $\crit(k)>\kappa$ and $k(s_{i_G}\rest\lh(w))=w'$,
we have
\[
(F^M)_w=(F^{M'})_{w'}=G_{s_{i_G}\rest\lh(w)}=(F^M)_{\pi(s_{i_G}\rest\lh(w))},\]
and since $s_{i_G}\rest\lh(w)<w^U$, we have
$\pi(s_{i_G}\rest\lh(w))<\pi(w^U)=w$.
By induction, the Dodd-solidity witnesses
at each $\alpha\in w$ belong to $M$,
and note now that $(F^M)_w$ is a component measure of one of these,
and hence in $M$.
\end{scaseone}
\begin{scaseone}\label{scase:kappa<zeta_Ds} $\kappa<\zeta$.

In this case we will show  $F^M\rest\widetilde{t}^M\in M$,
contradicting Claim
\ref{clm:F_rest_t-tilde_not_in_M}.
By Claim \ref{clm:H_props}, $H$ fails the ISC, so
 $\max(s_{i_G})<\max(w^U)$.
Note that $k(\max(w^U))$
is a generator of some extender $E$
used along $b^\Tt$, and since $s'=k(s_{i_G})$ and $\nu(F^{M'})=\max(s')+1$,
we have $k(\max(s_{i_G}))<\crit(E)$. It follows that $U$
 has an inaccessible
 cardinal $\chi$ with $\max(s_{i_G})<\chi<\max(w^U)$;
 take $\chi$ least such.
 Let $\theta\in b^\Uu$ be least such that either $\theta+1=\lh(\Uu)$
 or $i^\Uu_{0\theta}(\chi)\leq\crit(i^\Uu_{\theta\infty})$, so
 in fact by the minimality of $\chi$, $i^\Uu_{0\theta}(\chi)<\crit(i^\Uu_{\theta\infty})$ and $i^\Uu_{0\theta}(\chi)=k(\chi)$
 and $i^\Uu_{0\theta}(s_{i_G})=k(s_{i_G})=s'$.

 Let $t'=i^\Tt_{0\xi}(\widetilde{t}^M)$. We have $\max(\widetilde{t}^M)=\max(s_{F^M})$, so $\max(t')=\max(s')<k(\chi)$.
 By a finite support argument
 like those in the proof of Lemma \ref{lem:simple_embedding_exists},
 there is $(\Vv,\somevarpi)$ such that $\Vv$ is  a $(0,0)$- or $(0,r,0)$-maximal tree  of finite length on $\phU$,  $b^{\Vv}$
 is above $U$ and does not drop,
 $\somevarpi:M^{\Vv}_\infty\to M^\Uu_\theta$ is a $0$-embedding,
 $\somevarpi\com i^{\Vv}_{0\infty}=i^\Uu_{0\theta}$
 and $t'\in\rg(\somevarpi)$.
 By Claim \ref{clm:super-D-s_closeness} part \ref{item:super-D-s_measures_in_when},
 for every $\alpha+1\in b^{\Vv}$, every component measure of $E^{\Vv}_\alpha$ is in $M^{\Vv}_\beta$,
 where $\beta=\pred^{\Vv}(\alpha+1)$. Let $\left<\alpha_i\right>_{i<\ell}$
 enumerate those $\alpha_i$
 such that $\alpha_i+1\in b^{\Vv}$. Then with some more  finite support calculations,
 we can find $(\bar{\Vv},\bar{\somevarpi})$ such that   $\bar{\Vv}=\big(\left<U_\alpha\right>_{\alpha\leq\ell},\vec{E}=\left<E_\alpha\right>_{\alpha\leq\ell}\big)$ is an abstract $0$-maximal iteration of $U_0=U$ (so $U_{\alpha+1}=\Ult_0(U_\alpha,E_\alpha)$ for $\alpha<\ell$) such that \[ E_\alpha\in U_\alpha|i^{U,0}_{\vec{
 E}\rest\alpha}(\chi) \]
 for each $\alpha<\ell$, each $E_\alpha$
 is a finitely generated extender, and $\bar{\somevarpi}:U_\ell\to M^{\Vv}_\infty$ is a $0$-embedding
 such that $\bar{\somevarpi}\com i^{U,0}_{\vec{E}}=i^{\Vv}_{0\infty}$
 and $\somevarpi^{-1}(t')\in\rg(\bar{\somevarpi})$.
 Note that $F^M\rest\widetilde{t}^M$
 is equivalent to the measure $\mu$ derived from $i^{U,0}_{\vec{E}}\com i_G$
 with seed $\bar{\somevarpi}^{-1}(\somevarpi^{-1}(t'))$, so it suffices to see that $\mu\in M$.

 Note that $\vec{E}\in U|\chi$. So since $\chi<\max(w^U)$,
we have $\vec{E}\in U'=\Ult_0(M,G')$ where $G'=E_{i_G}\rest\max(w^U)$
is the extender derived from $i_G$
of length $\max(w^U)$.
And $\mu$ is just the measure derived from $i^{U',0}_{\vec{E}}\com i^{M,0}_{G'}$ with the same seed
$\bar{\somevarpi}^{-1}(\somevarpi^{-1}(t'))$ (note this seed is below $\chi$).

Let $F'=F^M\rest\max(w)$ and $U''=\Ult_0(M,F')$.
Let $\somevarsigma:U'\to U''$
be the factor map induced by $\pi\rest\max(w^U)$.
Then $\somevarsigma\com i^{M,0}_{G'}=i^{M,0}_{F'}$,
so
 $\mu$ is  the measure derived from $i^{U'',0}_{\somevarsigma(\vec{E})}\com i^{M,0}_{F'}$ with seed $\somevarsigma(\bar{\somevarpi}^{-1}(\somevarpi^{-1}(t'))$. But because $F'\in M$ and $\somevarsigma(\vec{E})\in\Ult_0(M|\kappa^{+M},F')$, it follows that $\mu\in M$, as desired.
\end{scaseone}
\end{caseone}

\begin{caseone}\label{case:Ds_xi_is_non-Tt_special} $\xi$ is non-$\Tt$-special.

In this case we will use an argument
like that in the proof of Lemma \ref{lem:measures_in_mice}. For this, we must see that
we can capture the relevant fragment of the extender derived from $i^\Tt$  with a tree $\Ttbar$
which is, modulo $\zeta$, strongly finite. We make this precise as follows.
Let us say that a $0$-maximal
tree $\Ss$ on $M$
 is \dfnemph{nicely-$\zeta$-strongly finite}
iff
\begin{enumerate}[label=\arabic*.,ref=\arabic*]
\item\label{item:nicely-zeta-sf_length} $\lh(\Ss)<\om$,
 \item $\Ss$ is non-trivial, $b^\Ss$ does not drop and
$\crit(i^\Ss)=\kappa=\crit(F^M)$
(so $\kappa^{+M}<\lh(E^\Ss_0)$),
\item $\zeta\leq\lh(E^\Ss_0)$,
\item either
$\zeta=\delta$
or $\zeta=\delta^{+\exit^\Ss_0}$, where $\delta=\card^M(\zeta)$,
\item\label{item:nicely-zeta-sf_xi'_non-special}
 $\xi'$ is non-$\Ss$-special,
where  $\xi'$ is least such that $\xi'+1\in b^\Ss$, and
\item\label{item:nicely-zeta-sf_fin_gend}
for each $\alpha+1\in b^\Ss$ and $\beta=\pred^\Ss(\alpha+1)$, we have:
\begin{enumerate}
 \item if $\beta=0$ (so $\crit(E^\Ss_\alpha)=\kappa$ and $\alpha$ is
non-$\Ss$-special) then:
 \begin{enumerate}
  \item if $\kappa=\zeta$ then
$\core_{\ds}(E^\Ss_\alpha)$
  is finitely generated (so $\sigma(E^\Ss_\alpha)=\kappa^{+M}$),
  \item if $\kappa<\zeta$ then
 $\sigma(E^\Ss_\alpha)\leq\zeta$ (so note
$\sigma(E^\Ss_\alpha)\in\{\delta,\zeta\}$),
  \item $\core_{\ds}(E^\Ss_\gamma)$
is finitely generated for each $\gamma<^\Ss_{\Da}\alpha$ (and note
$\delta\leq\crit(E^\Ss_\gamma)$),
 \end{enumerate}
 \item if $\beta>0$ and $\alpha$ is non-$\Ss$-special
then $\core_{\ds}(E^\Ss_\gamma)$
is finitely generated  for each $\gamma\leq^\Ss_{\Da}\alpha$,
\item if $\beta>0$ and $\alpha$ is $\Ss$-special
(so Lemma \ref{lem:Tt-special} applies) then
$\core_{\ds}(E^\Ss_\theta)$ is finitely generated
 for each $\theta\leq^\Ss_{\Da}\gamma$, for each
$\gamma+1\in(\beta,\alpha]^\Ss$.\end{enumerate}
\end{enumerate}

\begin{clmone}\label{lem:nicely-zeta-strongly_finite}
 Let
$x\in\core_0(M^\Tt_\infty)$.
Then there is a  nicely-$\zeta$-strongly finite $0$-maximal tree $\Ss$ on $M$ which
captures $(\Tt,x,\zeta)$ (see Definition \ref{dfn:captures}).
\end{clmone}

\begin{proof}
 This is a straightforward variant of the proof of Lemma
\ref{lem:non-dropping-strongly_finite}.
Say a $0$-maximal tree $\Ss$ on $M$ is a \emph{candidate} if
it satisfies conditions \ref{item:nicely-zeta-sf_length}--\ref{item:nicely-zeta-sf_xi'_non-special} of nice-$\zeta$-strong-finiteness
and $\Ss$ captures $(\Tt,x,\zeta)$.
By the properties of $\Tt$, we get candidates from straightforward finite support (condition \ref{item:nicely-zeta-sf_xi'_non-special} uses the present case hypothesis that $\xi$ is non-$\Tt$-special).
Define the \emph{index} of a candidate like before, but excluding the lengths of $\Ss$-special extenders.
That is, let $A$ be the set of all ordinals $\beta$ such that there is $\alpha+1\in b^\Ss$ such that either $\beta<^{\Ss}_{\Da}\alpha$ or $\beta=\alpha$ is non-$\Ss$-special. Now proceed as before: let $\left<\kappa_i\right>_{i<\ell}$ enumerate
$\{\crit(E^\Ss_\beta)\bigm|\beta\in A\}$ in decreasing order,
let $\beta_i\in A$ be such that $\kappa_i=\crit(E^\Ss_{\beta_i})$,
and let $\gamma_i=\lh(E^\Ss_{\beta_i})$.
Then the index of $\Ss$ is $\left<\gamma_i\right>_{i<\ell}$.
Let $\Ss$ be the candidate of lexicographically least index.

 It suffices to see that $\Ss$ is nicely-$\zeta$-strongly finite,
 so suppose not. So condition \ref{item:nicely-zeta-sf_fin_gend} fails.
As in the proof of \ref{lem:non-dropping-strongly_finite}, we
will use this to construct a candidate $\Ssbar$ with smaller index than $\Ss$, a
contradiction.

Adopt the notation introduced in the definition of \emph{index} (for this $\Ss$). As condition \ref{item:nicely-zeta-sf_fin_gend} fails, we can fix the least
$a<\ell$  such that $\core_{\ds}(E^\Ss_{\beta_a})$
is  not of the desired form.
a  perusal of the clauses in condition \ref{item:nicely-zeta-sf_fin_gend} shows that this just means that $\sigma(E^\Ss_{\beta_a})>\zeta$ and $\core_{\ds}(E^\Ss_{\beta_a})$ is not finitely generated. Let $\kappa=\kappa_a$, $\beta=\beta_a$,
$Q=\core_\ds(\exit^\Ss_{\beta})$ and
$F=F^Q$. So $\max(\kappa^{+Q},\zeta)<\sigma=\sigma_F=\tau_F=\sigma(E^\Ss_\beta)$. And like before,
by \ref{lem:Dodd_cores_appear} there is a unique $\theta\leq\beta$ such that
$Q\ins M^\Ss_\theta$; moreover, $\theta\leq^\Ss\beta$ and $\theta$ is the least ordinal such that
$\sigma<\lh(E^\Ss_\theta)$.

Let $\thetabar$ be least such that
$\lh(E^\Ss_\thetabar)>\max(\kappa^{+Q},\zeta)$. (Recall that $\lh(E^\Tt_0)\geq\zeta$, and $\lh(E^\Tt_0)=\zeta$ iff $M|\zeta$ is active.
It easily follows that the same holds for $\Ss$. So if $\kappa^{+Q}<\zeta$
then $\thetabar\in\{0,1\}$.)
Let $\thetabar+m=\theta$ (so $m<\om$) and
$\varthetabar+m+1=\vartheta+1=\lh(\Ss)$.
Much as in the proof of   \ref{lem:non-dropping-strongly_finite},
we will find
$g\in\nu_F^{<\om}$
and $R\pins M^\Ss_\thetabar$ with $F^R\approx F\rest
\zeta\cup g$, such that we can define a candidate $\Ssbar$  such that  $\Ssbar\rest\thetabar+1=\Ss\rest\thetabar+1$,
  $\lh(\Ssbar)=\varthetabar+1$,
 $(0,\varthetabar]^{\bar{\Ss}}\cap\dropset^{\bar{\Ss}}=\emptyset$, and
  $\bar{\Ss}\rest[\bar{\theta},\varthetabar]$ is built as a reverse copy of $\Ss\rest[\theta,\vartheta]$.
This will
result in a  final copy map
$\pi_{\varthetabar}:M^{\Ssbar}_{\varthetabar}\to M^{\Ss}_\vartheta$,
which will be a $0$-embedding
with
$\pi_{\varthetabar}\com i^\Ssbar_{0\varthetabar}=i^\Ss_{0\vartheta}$ and
$\zeta\cup\{\tau^{-1}(x)\}\sub\rg(\pi_{\varthetabar})$,
where $\tau:M^{\Ss}_{\vartheta}\to M^{\Tt}_\infty$
witnesses that $\Ss$ captures $(\Uu,x,\zeta)$.
It will follow that
$\tau\com\pi_{\varthetabar}:M^{\Ssbar}_{\varthetabar}\to M^\Uu_\infty$ witnesses that $\Ssbar$ captures $(\Uu,x,\zeta)$.

If $\kappa^{+Q}\geq\zeta$, this will be done essentially like before.
Suppose instead that $\kappa^{+Q}<\zeta$. Then  $\pred^\Ss(\beta+1)=0$.
(For  $\zeta\leq\lh(E^\Ss_0)$
and either $\delta=\zeta$ is an $M$-cardinal, or $\zeta=\delta^{+\exit^\Ss_0}$. Note then that
$\nu(E^\Ss_\gamma)\geq\delta$ for each $\gamma$, and so if $\pred^\Ss(\beta+1)\neq 0$
then $\delta\leq\kappa$,
so $\kappa^{+Q}\geq\zeta$.)
So $\beta$ is non-$\Ss$-special
and $\kappa=\crit(F^M)$.
Suppose $Q$ is type 2.
Then $\max(\kappa^{+Q},\zeta)<\tau_F=\rho_1^{Q}$,
$\zeta$ is a $Q$-cardinal,
and $Q$ is $1$-sound.
So for cofinally many $g\in[\nu(F)]^{<\om}$,
there is a type 2 segment $R\pins Q$ with $F^R$ equivalent to $F\rest\zeta\cup\{g\}$
and $\rho_1^R=\tau(F^R)=\zeta$,
so we choose an appropriate $R$ of this form, much as before.
If instead $Q$ is type 3,
then we again choose some sufficiently large type 2 segment $Q'\pins Q$
with $F^{Q'}\sub F$,
and then if still $\rho_1^{Q'}>\zeta$, proceed as we did in the type 2 case with $Q'$ replacing $Q$.

To see that we can arrange the reverse copying, we establish the analogue of Subclaim \ref{sclm:Tt_structure_mim} from the proof of  \ref{lem:non-dropping-strongly_finite}, classifying the various extenders used in $\Ss\rest[\theta,\vartheta)$:

\begin{sclmone}\label{sclm:Tt_structure_Ds} Let $\chi\in[\theta,\vartheta)$.
Then $\crit(E^\Ss_\chi)\notin(\kappa,\sigma)$, so
$\pred^\Ss(\chi+1)\notin(\thetabar,\theta)$. In fact,  exactly one of the
following
holds:
\begin{enumerate}[label=\tu{(}\roman*\tu{)}]
\item\label{item:feeds_in_Ds} there are $\vareps,\chi'$ such that $\vareps$ is non-$\Ss$-special,  $ \eps+1\leq^\Ss\vartheta$ and
$\chi'\leq^\Ss_{\Da}\eps$, and either:
\begin{enumerate}
\item\label{item:feeds_in_Ds_feeds_in} $\chi=\chi'$, or
\item\label{item:feeds_in_Ds_extra} $\chi<^\Ss\chi'$ and $\chi$ is $\chi'$-transient and $E^\Ss_{\chi'}=F(M^\Ss_{\chi'})$,
\end{enumerate}
or
\item\label{item:Tt-special_primary_option_Ds} there are $\vareps,\vareps',\chi'$ such that $\vareps$ is  $\Ss$-special, $\vareps+1\leq^\Ss\vartheta$ and either:
\begin{enumerate}[label=--]
\item $\chi'=\vareps'=\vareps$, or
\item letting $\iota=\pred^\Ss(\vareps+1)$, then $\iota<^\Ss\vareps'+1\leq^\Ss\vareps$ and $\chi'\leq^\Ss_{\Da}\vareps'$,
\end{enumerate}
and either:
\begin{enumerate}\item\label{item:Tt-special_primary_option_Ds_feeds_in} $\chi=\chi'$, or
\item\label{item:Tt-special_primary_option_Ds_extra} $\chi<^\Ss\chi'$
and $\chi$ is $\chi'$-transient
and $E^\Ss_{\chi'}=F(M^\Ss_{\chi'})$,\footnote{The requirement that $E^\Ss_{\chi'}=F(M^\Ss_{\chi'})$ is redundant in the case that $\chi'=\vareps$, since $\vareps$ is $\Ss$-special.}\end{enumerate}
or

\item\label{item:chi_along_main_branch_transient} $\chi<^\Ss\vartheta$ and $\chi$ is $\vartheta$-transient.
\end{enumerate}
Moreover,
the 5 options
\ref{item:feeds_in_Ds_feeds_in},
\ref{item:feeds_in_Ds_extra},
\ref{item:Tt-special_primary_option_Ds_feeds_in},
\ref{item:Tt-special_primary_option_Ds_extra},
\ref{item:chi_along_main_branch_transient}
are mutually exclusive.
\end{sclmone}
\begin{proof}
This is just a slight variant of the proof of Subclaim \ref{sclm:Tt_structure_mim} of Claim \ref{clm:Tt_strongly_finite_mim}
in the proof of  \ref{lem:non-dropping-strongly_finite}.
We leave the details to the reader.
\end{proof}

 Like in the proof of \ref{lem:non-dropping-strongly_finite},
 Subclaim \ref{sclm:Tt_structure_Ds} ensures that we can  carry out the reverse copying. This contradicts the minimality of the index of $\Ss$, as desired.
\end{proof}

By the proof of Subclaim \ref{sclm:Tt_structure_Ds}, we have:

\begin{clmone}\label{clm:Tt-structure_nicely-zeta-strongly-finite}
Let $\Ss$ be nicely-$\zeta$-strongly finite and $\chi+1<\vartheta+1=\lh(\Ss)$.
Then exactly one of the 5 options \ref{item:feeds_in_Ds_feeds_in},
\ref{item:feeds_in_Ds_extra},
\ref{item:Tt-special_primary_option_Ds_feeds_in},
\ref{item:Tt-special_primary_option_Ds_extra},
\ref{item:chi_along_main_branch_transient} of Subclaim \ref{sclm:Tt_structure_Ds}
of the proof of
Claim \ref{lem:nicely-zeta-strongly_finite} holds with respect to $\Ss$.
\end{clmone}

\begin{goal}\label{goal:(*)} We will prove ($*$), which is the  conjunction of the following statements:
  \begin{enumerate}[label=--]
    \item $\Uu$ is trivial, $U=M^\Tt_\infty$ and $i_G=i^\Tt$,
    \item $\Tt$ is nicely-$\zeta$-strongly finite, and
    \item
$\alpha$ is non-$\Tt$-special for every $\alpha+1\in b^\Tt$.
  \end{enumerate}
\end{goal}

 \begin{clmone}If ($*$) holds then $G\in M$.\end{clmone}
\begin{proof} By Claim \ref{clm:Ds_j,k_com_i_G_short}, $j=i^\Tt$ is short. So if $\kappa=\delta=\zeta$ then $(\kappa^{+\om})^M<\OR^M$ and
$\Tt$ is based on $M||\kappa^{++M}$, so easily $G\in M$.
Now suppose $\kappa^{+M}\leq\delta<\zeta$.
Because $w\neq\emptyset$, we have
$\zeta<\lgcd(M)$.

We claim that for each $\chi+1<\lh(\Tt)$, option \ref{item:feeds_in_Ds}
of Claim \ref{clm:Tt-structure_nicely-zeta-strongly-finite}
holds.
For option \ref{item:chi_along_main_branch_transient} is ruled out by the shortness of $i^\Tt$,
and  option \ref{item:Tt-special_primary_option_Ds} is ruled out since by ($*$), there is no $\Tt$-special $\vareps$ with $\vareps+1\in b^\Tt$.
It follows that
we can consider $\Tt$ as a tree $\Tt'$ on $M|\lgcd(M)$. In fact, letting $\alpha+1=\succ^\Tt(0,\infty)$, then $\Tt=\Tt_0\conc\Tt_1$
where $\Tt_0=\Tt\rest(\alpha+2)$
and $\Tt_1$ is based on $M^{\Tt_0}_\infty|\mu^{++M^{\Tt_0}_\infty}$,
where $\mu=i^{\Tt_0}_{0,\alpha+1}(\kappa)$.
Since
$M^{\passive}\sats\ZFC^-$, $\Tt'\in M$, so $G\in M$.
\end{proof}

So it suffices to prove ($*$).
Fix
 a nicely-$\zeta$-strongly finite
tree $\bar{\Tt}$ capturing $(\Tt,i^\Uu(w^U),\zeta)$ (exists by Claim \ref{lem:nicely-zeta-strongly_finite}).
We will proceed like in the proof of Lemma \ref{lem:measures_in_mice}
to show that $\Tt=\bar{\Tt}$
and $\Uu$ is trivial,
with further embellishments to rule out $\Tt$-special extenders along the main
branches of $\Tt,\bar{\Tt}$.

Let $\bar{Q}=M^{\bar{\Tt}}_\infty$
and $\somevarpi:\bar{Q}\to Q$ be the capturing map.
Let $\bar{j}=i^{\bar{\Tt}}$.
So $\somevarpi\com\bar{j}=j$.
Let $\bar{k}:U\to\bar{Q}$ be $\bar{k}=\somevarpi^{-1}\com k$.
Then the diagram
in Figure \ref{fgr:Dodd_first_commuting} commutes.
Because $\Tt$ is short, and by Claim \ref{clm:Ds_crit(k)>crit(j)},
we have:

\begin{clmone}\label{clm:j-bar_short_Ds}$\bar{j}=i^{\bar{\Tt}}$ is short, the models $M,U,\bar{Q},Q$ agree through their common value for $\kappa^+$, and \[ \crit(i_G)=\crit(\bar{j}
)=\crit(j)=\kappa<\kappa^{+M}<\crit(k)=
\min(\crit(\bar{k}),\crit(\somevarpi)).\]
\end{clmone}

\begin{figure}
\centering
\begin{tikzpicture}
 [mymatrix/.style={
    matrix of math nodes,
    row sep=0.35cm,
    column sep=0.4cm}]
   \matrix(m)[mymatrix]{
  U&          {} &       {}& {} & Q& \\
   {} & {} \\
   {} & {} & \bar{Q}\\
   \\
 M\\};
 \path[->,font=\scriptsize]
%maps from left $M$
(m-5-1) edge[bend right] node[below] {$\ \ \ \ j=i^\Tt$} (m-1-5)
(m-5-1) edge node[left] {$i=i_G$} (m-1-1)
(m-1-1)edge node[above] {$\bar{k}$} (m-3-3)
(m-5-1)edge node[above] {$\bar{j}$} (m-3-3)
(m-3-3)edge node[above] {$\somevarpi$} (m-1-5)
(m-1-1) edge node[above] {$k=i^\Uu$} (m-1-5);
\end{tikzpicture}
\caption{The diagram commutes.}
\label{fgr:Dodd_first_commuting}
\end{figure}

Let $i=i_G$ and $s=s_{i}$ and $\sigma=\sigma_{i}$.
Let $s'=s_{j}$ and $\sigma'=\sigma_{j}$.

\begin{clmone}\label{clm:s,sigma_pres_in_Ds_Case_2}We have:
\begin{enumerate}[label=\arabic*.,ref=\arabic*]
\item\label{item:s'=k(s)_sigma'=sup_varrho_sigma}
$s'=s_j=s_{k\com i}=k(s_i)\text{ and }\sigma'=\sigma_j=\sigma_{k\com
i}=\sup
k``\sigma_i$.
 \item \label{item:s_k-bar_com_i=s_j-bar=k-bar(s_i)_etc_Ds_Case_2}
 $s_{\bar{k}\com i}=s_{\bar{j}}=\bar{k}(s_i)$ and
$\sigma_{\bar{k}\com i}=
\sigma_{\bar{j}}=\sup\bar{k}``\sigma_{i}$.
\item\label{item:s_varrho_com_j-bar=s_j=varrho(s_j-bar)_etc_Ds_Case_2} $s_{\somevarpi\com\bar{j}}=s_j=\somevarpi(s_{\bar{j}})$
and $\sigma_{\somevarpi\com\bar{j}}=\sigma_j=\sup\somevarpi``\sigma_{\bar{j}}$.
\item\label{item:sigma_i=sigma_j-bar=sigma_j_in_kappa^+,delta,zeta_Ds_Case_2}
$\max(\delta,\kappa^{+M})\leq\sigma_{i}=
\sigma_{\bar{j}}=\sigma_j\in\{\kappa^{+M},\delta,\zeta\}$.
\end{enumerate}
\end{clmone}
\begin{proof}
Part \ref{item:s'=k(s)_sigma'=sup_varrho_sigma} is by Lemma
\ref{lem:s_j,sigma_j_pres}.
Parts \ref{item:s_k-bar_com_i=s_j-bar=k-bar(s_i)_etc_Ds_Case_2}
and \ref{item:s_varrho_com_j-bar=s_j=varrho(s_j-bar)_etc_Ds_Case_2} are as in the proof of Claim \ref{clm:s_jbar=kbar(s_i)}
of Lemma \ref{lem:measures_in_mice}.
For part \ref{item:sigma_i=sigma_j-bar=sigma_j_in_kappa^+,delta,zeta_Ds_Case_2},
the fact that
$\max(\delta,\kappa^{+M})\leq\sigma_{\bar{j}}\in\{\kappa^{+M},\delta,\zeta\}$
follows from the
nice-$\zeta$-strong finiteness
of $\bar{\Tt}$. Since $\max(\zeta,\kappa^{+M})\leq\crit(k)\leq\crit(\bar{k})$,
the preceding parts then give that $\sigma_i=\sigma_{\bar{j}}=\sigma_{j}$.
\end{proof}

 Let $\bar{\xi}$ be least such that $\bar{\xi}+1\in
b^{\bar{\Tt}}$ (recall $\xi$ is likewise for $\Tt$).

\begin{clmone}
$\core_{\ds}(E^{\bar{\Tt}}_{\bar{\xi}})=\core_{\ds}(E^\Tt_\xi)$.
\end{clmone}
\begin{proof}
As in the proof of Claim \ref{clm:first_Dodd_cores_match} of Lemma
\ref{lem:measures_in_mice}
(recall $\bar{\xi}$ is non-$\bar{\Tt}$-special
and $\xi$ is non-$\Tt$-special).
\end{proof}

Now going on with the analysis of the Dodd-decompositions of
$E^{\Ttbar}_{\xibar}$
and $E^\Tt_\xi$ as in the proof of Lemma \ref{lem:measures_in_mice}
(and since $\xi,\xibar$ are non-$\Tt$-special and non-$\Ttbar$-special
respectively), we get that $\xi=\xibar$ and
$\Ttbar\rest(\xi+2)=\Tt\rest(\xi+2)$ (and $\xi\in b^{\Ttbar}\cap b^\Tt$).
Let $\beta\in b^{\bar{\Tt}}\cap b^\Tt$ be largest such that:
\begin{enumerate}[label=--]\item
$\Ttbar\rest(\beta+1)=\Tt\rest(\beta+1)$,  and
\item $\vareps$ is non-$\Tt$-special
 for every $\vareps+1\leq^\Tt\beta$.
 \end{enumerate}
 So
 $\xi+1\leq\beta$. Let $\jbar_{\beta\infty}:M^{\Ttbar}_\beta\to\bar{Q}$ and
$j_{\beta\infty}:M^{\Tt}_\beta\to Q$ be the iteration maps.
Then  $\somevarpi\com\jbar_{\beta\infty}=j_{\beta\infty}$
and as in the
proof of Lemma \ref{lem:measures_in_mice}, we
 can define $i_\beta:M^{\Ttbar}_\beta\to U$
such that $\kbar\com i_\beta=\jbar_{\beta\infty}$.
Let $M^*=M^{\Ttbar}_\beta=M^\Tt_\beta$.

\begin{clmone}
If $M^*=\bar{Q}$ then ($*$) holds (see Goal \ref{goal:(*)}), so $G\in M$.\end{clmone}
\begin{proof}
If $M^*=\bar{Q}$ then  $U=\bar{Q}$, $i^{\Ttbar}=i_G$
and $\Ttbar=\Tt$, all as in the proof of Lemma \ref{lem:measures_in_mice},
so ($*$) holds.
\end{proof}

So suppose $M^*\neq\bar{Q}$
(and hence $M^*\neq Q$).
Let $\bar{\varepsilon}+1=\succ^{\Ttbar}(\beta,\infty)$
and $\varepsilon+1=\succ^\Tt(\beta,\infty)$.

\begin{clmone}\label{clm:Ds_if_M^*_neq_Q-bar_then_eps_T-spec_or_T-bar-spec}
 If $M^*\neq\bar{Q}$ (hence $M^*\neq Q$) then either $\varepsilon$ is
$\Tt$-special
 or $\bar{\varepsilon}$ is $\Ttbar$-special.
\end{clmone}
\begin{proof}
 Otherwise we can carry on with the analysis of Dodd decomposition,
 and reach a contradiction to the maximality of $\beta$.
\end{proof}

So we may assume, for a contradiction:
\begin{ass}\label{ass:speciality}
$\bar{Q}\neq M^*\neq Q$
and either $\varepsilon$ is $\Tt$-special or $\bar{\varepsilon}$
is $\Ttbar$-special.
\end{ass}

Let
\[ i^*:M^*\to U\text{ and } \bar{j}^*:M^*\to\bar{Q}\text{ and }
j^*:M^*\to Q \]
be the maps resulting from the
analysis above. So $\bar{k}\com i^*=\bar{j}^*$
and $\somevarpi\com\bar{j}^*=j^*$ and $k\com i^*=j^*$.
We have $\crit(j^*)=\crit(E^\Tt_\varepsilon)\geq\delta$,
since $\nu(E^\Tt_0)\geq\delta$.
(Note that $E^\Tt_0$ might be superstrong
with $\nu(E^\Tt_0)=\delta$, in which case
$\crit(E^\Tt_\varepsilon)=\delta$, and then if $\delta<\zeta$,
we have $\crit(j^*)<\zeta$.) Let
$\mu^*=\crit(j^*)=\crit(F^{M^*})$. As usual we have:

\begin{clmone}\label{clm:crit(k)>crit(j^*)}
$\crit(i^*)=\crit(\bar{j}^*)=\crit(j^*)=\mu^*<\crit(k)=
\min(\crit(\bar{k}),\crit(\somevarpi))$.
\end{clmone}

Most of the rest of the proof is devoted to the following claim:

\begin{clmone}\label{clm:varepsilon_non-Tt-special}
 $\varepsilon$ is non-$\Tt$-special.
\end{clmone}
\begin{proof}
Suppose otherwise. We will show that the measure $\mu=
(F^M)_{\ttilde_{F^M}}\in M$,
contradicting Claim \ref{clm:F_rest_t-tilde_not_in_M}.
This will be similar to the argument used in Subcase \ref{scase:kappa<zeta_Ds} of Case \ref{case:xi_Tt-special}, by deriving $\mu$ from  measures already known to be in $M$.

Let $s^*=s_{i^*}$ and $\sigma^*=\sigma_{i^*}$.
Let $M'=M^\Tt_\varepsilon$.
Let $s'=s_{F^{M'}}$ and $\sigma'=\sigma_{F^{M'}}$.
By Claim \ref{clm:crit(k)>crit(j^*)} and the usual calculation, $s'=k(s^*)$ and $\sigma'=\sup k``\sigma^*$.
Let $w'=i^\Tt_{0\varepsilon}(w)$, so $w'\ins s'$.
 Let $\zeta'=i^\Tt_{0\varepsilon}(\zeta)$.

\begin{sclmone}$s^*\rest\lh(w)<w^U$.\end{sclmone}
\begin{proof}
 Much as before, we must have $s^*\rest\lh(w)\leq w^U$.
Suppose $s^*\rest\lh(w)=w^U$.
Note that
\[ \zeta'\notin\Hull_1^Q(\zeta'\cup\{p_1^Q,w'\}) \]
(as
$Q=M^\Tt_\infty$ and $F^{M'}$ is used along $b^\Tt$).
Since $k(w^U,p_1^U)=(w',p_1^Q)$ and $\crit(k)\geq\zeta$,
it follows that $U\neq Q$, so $\Uu$ is non-trivial, and in fact, there is
$\iota+1\in b^\Uu$
such that $\crit(E^\Uu_\iota)\leq\zeta'<\nu(E^\Uu_\iota)$.

Now $F_\downarrow^{M'}\in\rg(k)$ (basically because $k(s^*)=s'$).
Therefore
\[ F_\downarrow^{M'}\in\Hull^Q_1(\crit(E^\Uu_\iota)\cup\{p_1^Q,w'\}) \]
(see \S\ref{sec:premice_and_phalanxes} for $F_\downarrow$).
But by Assumption \ref{ass:v_not_empty}, $F_\downarrow^{M'}$ is not generated
by $\zeta'\cup w'$, and therefore
\[ F_\downarrow^{M'}\notin\Hull_1^Q(\zeta'\cup\{p_1^Q,w'\}), \]
whereas $\crit(E^\Uu_\iota)\leq\zeta'$, a contradiction, proving the subclaim.
\end{proof}

So let $m<\lh(w)$ be  such that $s^*\rest m=w^U\rest m$
but
$(s^*)_m<(w^U)_m$.

\begin{sclmone}\label{sclm:inaccessible_xi}
$U\sats$ ``There is an inaccessible
$\chi$ with $(s^*)_m<\chi< (w^U)_m$''.
\end{sclmone}
\begin{proof}
 We have $w'\ins s'=k(s^*)$, so
 \begin{enumerate}[label=(\roman*)]\item\label{item:Ds_m>0>s^*_vs_w^U} if $m>0$ then
$ k((s^*)_m)=(w')_m<k((w^U)_m)<k((s^*)_{m-1})=(w')_{m-1}$
and \item\label{item:Ds_m=0>s^*_vs_w^U}  if $m=0$ then
$k((s^*)_0)=(w')_0=\nu(F^{M'})-1<k((w^U)_0)$.
\end{enumerate}
But
\[ (w^U)_m\notin\Hull^U_1((w^U)_m\cup\{p_1^U,w^U\rest m\}),\]
and $k$ lifts this fact, and therefore
in case \ref{item:Ds_m>0>s^*_vs_w^U}, $k((w^U)_m)$ is a $(w'\rest m)$-generator of $F^{M'}$,
and in case \ref{item:Ds_m=0>s^*_vs_w^U}, $\Tt$ uses an extender $E$ along $b^\Tt$,
after $F^{M'}$, for which $k((w^U)_0)$ is a generator.
In case \ref{item:Ds_m=0>s^*_vs_w^U}, we get
\begin{equation}\label{eqn:xi'_and_others}
k((s^*)_m)<\chi'=\crit(E)\leq
k((w^U)_m)\end{equation}
and $\chi'$ is a limit of inaccessibles in $Q$,
which easily suffices.
In case \ref{item:Ds_m>0>s^*_vs_w^U}, $\Tt$ uses an extender $E$ on the branch
leading to $M'$, for which $k((w^U)_m)$ is a generator,
such that line (\ref{eqn:xi'_and_others}) also holds,
and again this easily suffices.
\end{proof}

Let $\chi$ be least as in Subclaim \ref{sclm:inaccessible_xi}.
Let $t=\ttilde^M$ and $t'=i^\Tt_{0\varepsilon}(t)$.
Recall $k(s^*\rest\lh(w))=w'\ins t'$.
Let $\theta\in b^\Uu$ be least such that
either $\theta+1=\lh(\Uu)$
or $i^\Uu_{0\theta}(\chi)\leq\crit(i^\Uu_{0\theta})$,
so in fact by the minimality of $\chi$, $i^\Uu_{0\theta}(\chi)<\crit(i^\Uu_{\theta\infty})$
and $i^\Uu_{0\theta}(\chi)=k(\chi)$
and $i^\Uu_{0\theta}((s^*)_m)=k((s^*)_m)=(t')_m<k(\chi)$. Note that $t'\in\rg(i^\Uu_{\theta\infty})$.
Let $i^\Uu_{\theta\infty}(\widehat{t})=t'$.

 As in Subcase \ref{scase:kappa<zeta_Ds} of Case \ref{case:xi_Tt-special},
 there is $(\Vv,\somevarpi)$ such that $\Vv$ is  a $(0,0)$- or $(0,r,0)$-maximal tree  of finite length on $\phU$,  $b^{\Vv}$
 is above $U$ and does not drop,
 $\somevarpi:M^{\Vv}_\infty\to M^\Uu_\theta$ is a $0$-embedding,
 $\somevarpi\com i^{\Vv}_{0\infty}=i^\Uu_{0\theta}$
 and $\widehat{t}\in\rg(\somevarpi)$.
 Also as there,
 we can find $\ell<\om$ and $(\bar{\Vv},\vec{U},\vec{E},\bar{\somevarpi})$ such that $\vec{U}=\left<U_\alpha\right>_{\alpha\leq\ell}$, $\vec{E}=\left<E_\alpha\right>_{\alpha<\ell}$ and $\bar{\Vv}=\big(\vec{U},\vec{E})$ is an abstract $0$-maximal iteration of $U_0=U$,  $E_\alpha\in U_\alpha|\chi$
 and $E_\alpha$
 is a finitely generated extender, and $\bar{\somevarpi}:U_\ell\to M^{\Vv}_\infty$ is a $0$-embedding
 with $\bar{\somevarpi}\com i^{U,0}_{\vec{E}}=i^{\Vv}_{0\infty}$
 and $\somevarpi^{-1}(\widehat{t})\in\rg(\bar{\somevarpi})$.
 Note that $\mu$
 is  the measure  derived from $i^{U,0}_{\vec{E}}\com i_G$
 with seed $\bar{\somevarpi}^{-1}(\somevarpi^{-1}(\widehat{t}))$.

 We have $\vec{E}\in U|\chi$. So since $\chi<(w^U)_m$,
we have $\vec{E}\in U'=\Ult_0(M,G')$ where \[ G'=E_{i_G}\rest (\{w^U\rest m\}\cup(w^U)_m).\]
So letting $\bar{\bar{\somevarpi}}:\Ult_0(U',\vec{E})\to\Ult_0(U,\vec{E})=U_\ell$ be the  map
given by the Shift Lemma applied to the factor map $U'\to U$,
 $\mu$ is  the measure derived from $i^{U',0}_{\vec{E}}\com i^{M,0}_{G'}$ with  seed
$\bar{\bar{\somevarpi}}^{-1}(\bar{\somevarpi}^{-1}(\somevarpi^{-1}(\widehat{t})))$.

Let $F'=F^M\rest (\{w\rest m\}\cup w_m)$ and $U''=\Ult_0(M,F')$.
Let $\somevarsigma:U'\to U''$
be the factor map induced by $\pi\rest(\{w^U\rest m\}\cup (w^U)_m)$.
Then $\somevarsigma\com i^{M,0}_{G'}=i^{M,0}_{F'}$,
so
 $\mu$ is  the measure derived from $i^{U'',0}_{\somevarsigma(\vec{E})}\com i^{M,0}_{F'}$ with seed $\somevarsigma(\bar{\bar{\somevarpi}}^{-1}(\bar{\somevarpi}^{-1}(\somevarpi^{-1}(\widehat{t}))))$. Since $F'\in M$ and $\somevarsigma(\vec{E})\in\Ult_0(M|\kappa^{+M},F')$, it follows that $\mu\in M$, as desired.

This completes the proof
that $\varepsilon$ is non-$\Tt$-special
(Claim  \ref{clm:varepsilon_non-Tt-special}).
\end{proof}

The following claim will now give the desired contradiction:

\begin{clmone}\label{clm:epsilon-bar_is_non-Tt-special} $\bar{\varepsilon}$ is non-$\Tt$-special.\footnote{Note
that this does not seem to follow immediately from the construction
of $\Ttbar$, as we have not yet shown that $\Tt$ has no $\Tt$-special extender
along $b^\Tt$, but only that $E^\Tt_\varepsilon$
is non-$\Tt$-special.
And we have to allow the possibility
that $b^\Tt$ has ordertype $>\om$
and  infinitely many extenders are used along $b^\Tt$
prior to the first $\Tt$-special
extender there.
So $E^{\bar{\Tt}}_{\bar{\varepsilon}}$
might ostensibly be the reflection of such a $\Tt$-special
extender in $\Tt$.} \end{clmone}
\begin{proof}
Because $\varepsilon$ is non-$\Tt$-special and as in the proof of
Closeness \cite[6.1.5]{fsit}, for every $a\in[\nu(E^\Tt_\varepsilon)]^{<\om}$,
we have $(E^\Tt_\varepsilon)_a\in M^*$. Suppose
$E^{\Ttbar}_{\bar{\varepsilon}}$
is $\Ttbar$-special and let
$\tbar=i^{\Ttbar}_{0\bar{\varepsilon}}(\ttilde_{F^M})$.
Then the measure $D=(E^{\Ttbar}_{\bar{\varepsilon}})_{\tbar}\notin
M^{\Ttbar}_{\bar{\varepsilon}}$,
and $\crit(i^{\Ttbar}_{\beta\bar{\varepsilon}})>\mu^*$ (recall
$M^*=M^{\Ttbar}_\beta=M^\Tt_\beta$),
so $D\notin M^*$. But
note that by  Claim \ref{clm:crit(k)>crit(j^*)} and
$s$-preservation, we have $\somevarpi(\max(\tbar))+1=\nu(E^\Tt_\varepsilon)$,
so $D=(E^\Tt_\varepsilon)_{\somevarpi(\tbar)}\in M^*$, a contradiction.
\end{proof}

So we have shown that $\varepsilon$ is non-$\Tt$-special and
$\bar{\varepsilon}$ is non-$\Ttbar$-special,
contradicting Assumption \ref{ass:speciality}, completing the analysis of Case \ref{case:Ds_xi_is_non-Tt_special}.
\end{caseone}

This completes everything other than
the postponed verification of the iterability of $\phU$, which we finally tend to:

\begin{proof}[Proof of Claim \ref{clm:Ds_phU_iterable}]\label{page:Ds_iterability_proof}
We will show that $\phU$ is $((0,0),\om_1+1)$-iterable
or $((0,r,0),\om_1+1)$-iterable respectively.\footnote{See Footnote \ref{ftn:anomalous_footnote} for the case $r=-1$.}
Toward this iterability,
let $\gamma$ be any $M$-cardinal with $\kappa<\gamma$ and let
\[ \phW_\gamma=((M,{<\gamma}),W,\gamma). \]
We will show that $\phW_\gamma$ is $((0,0),\om_1+1)$-iterable,
and that we can reduce trees on $\phU$ to trees on $\phW_{\pi(\zeta)}$.

\begin{sclmone}\label{sclm:phW_iterable}$\phW_\gamma$ is
$((0,0),\om_1+1)$-iterable.\end{sclmone}
\begin{proof}The proof is just a simple case of normalization
calculations. Write $\phW=\phW_\gamma$.
First, instead of considering (unpadded) $(0,0)$-maximal trees $\Ww$ on $\phW$ themselves, we will consider padded $(0,0)$-maximal trees $\Ww$
of a  special form. We write $-1,0$ for the two roots of $\Ww$
(so $M^{\Ww}_{-1}=M$ and $M^\Ww_0=W$).
Say that a padded $(0,0)$-maximal tree $\Ww$ on $\phW$ is
\emph{nicely padded} iff:
\begin{enumerate}[label=\arabic*.,ref=\arabic*]
\item For each $\alpha+1<\lh(\Ww)$,
if $0\leq^\Ww\alpha$
and $(0,\alpha]^\Ww\cap\dropset^\Ww=\emptyset$ and $E^\Ww_\alpha\neq\emptyset$
then $\lh(E^\Ww_\alpha)<i^\Ww_{0\alpha}(\OR^M)$.
 \item For each $\alpha+1<\lh(\Ww)$, if $E^\Ww_\alpha=\emptyset$
then:
\begin{enumerate}
\item $0\leq^\Ww\alpha$ and $(0,\alpha]^\Ww\cap\dropset^\Ww=\emptyset$,
\item
$\pred^\Ww(\alpha+1)=-1$,
$M^\Ww_{\alpha+1}=M^\Ww_\alpha$,
$M^{*\Ww}_{\alpha+1}=M$,
and $i^{*\Ww}_{\alpha+1}=i^\Ww_{0\alpha}\com i^{M,0}_{F^M}$,
\item $\alpha+2<\lh(\Ww)$,
\item $i^\Ww_{0\alpha}(\OR^M)<\lh(E^\Ww_{\alpha+1})$.
\end{enumerate}
\item For each $\alpha+1<\lh(\Ww)$, if $E^\Ww_\alpha\neq\emptyset$ then $\pred^\Ww(\alpha+1)$
is the least $\beta$ such that $\crit(E^\Ww_\alpha)<\nu^\Ww_\beta$,
where $\nu^\Ww_\beta=\nu(E^\Ww_\beta)$ if
$E^\Ww_\beta\neq\emptyset$, and
and $\nu^\Ww_\beta=i^\Ww_{0\beta}(\nu(F^M))$
if $E^\Ww_\beta=\emptyset$.
\item Everything else is as for $(0,0)$-maximal trees.
\end{enumerate}
Clearly nicely padded $(0,0)$-maximal trees on $\phW$
are equivalent to unpadded $(0,0)$-maximal trees on $\phW$.

Now let $\phM$ be the phalanx $((M,{<\gamma}),M,\gamma)$.
Note that
$(0,0)$-maximal trees $\Tt$ on $\phM$ are equivalent
to $0$-maximal trees on $M$ with $\gamma<\lh(E^\Tt_0)$. So let $\Sigma_{\phM}$ be the iteration strategy for $\phM$ equivalent to some $(0,\om_1+1)$-strategy for $M$.

Let $\Ww$ be a nicely padded $(0,0)$-maximal tree on $\phW$ and $\Tt$ be $(0,0)$-maximal on $\phM$.
We say that $(\Ww,\Tt)$ is a \emph{nice pair}
iff $\lh(\Tt)=\lh(\Ww)$, ${<^\Tt}={<^{\Ww}}$ and
 for each $\alpha+1<\lh(\Ww)$, we have:
 \begin{enumerate}[label=--]\item if $E^\Ww_\alpha\neq\emptyset$ then $E^\Tt_\alpha=E^\Ww_\alpha$, and
     \item if $E^\Ww_\alpha=\emptyset$ then $E^\Tt_\alpha=F(M^\Tt_\alpha)$.
     \end{enumerate}

 We get a $(0,0)$-maximal iteration strategy $\Sigma_{\phW}$ for $\phW$ producing nicely padded trees  $\Ww$ such that for some $(0,0)$-maximal tree $\Tt$ on $\phM$ via $\Sigma_{\phM}$,
 $(\Ww,\Tt)$ is a nice pair.
To see this, observe the following points regarding nice pairs $(\Ww,\Tt)$:
\begin{enumerate}[label=\Alph*.]\item  $\dropset^\Tt=\dropset^\Ww$
and
$\deg^\Tt=\deg^\Ww$.
 \item For each $\alpha<\lh(\Ww)$, if either  $-1\leq^\Ww\alpha$  or  [$0\leq^\Ww\alpha$
 and $(0,\alpha]^\Ww\cap\dropset^\Ww\neq\emptyset$] then:
\begin{enumerate}[label=(\roman*)]
\item $M^\Tt_{\alpha}=M^\Ww_\alpha$,
\item for each $\beta\leq^\Tt\alpha$ with
$(\beta,\alpha]^\Tt\cap\dropset^\Tt=\emptyset$,
we have $i^\Tt_{\beta\alpha}=i^\Ww_{\beta\alpha}$ and
 if $\beta$ is a successor then
$M^{*\Tt}_\beta=M^{*\Ww}_\beta$ and $i^{*\Tt}_{\beta}=i^{*\Ww}_\beta$.
\end{enumerate}
\item\label{item:derived_extender}  For each $\alpha<\lh(\Ww)$, if $0\leq^\Ww\alpha$ and $(0,\alpha]^\Ww\cap\dropset^\Ww=\emptyset$
then:
\begin{enumerate}
\item $M^\Ww_\alpha=\Ult_0(M,F^{M^\Tt_\alpha})$,
\item $i^\Tt_{0\alpha}=i^\Ww_{0\alpha}\rest M$,
\item\label{item:derived_extender_sub_part} $F^{M^\Tt_\alpha}$ is just the $(\kappa,\OR^{M^\Tt_\alpha})$-extender
derived from $i^\Ww_{0\alpha}\com i^{M,0}_{F^M}$.

 \item Suppose $\alpha$ is a successor $\beta+1<\lh(\Ww)$ and let $\vareps=\pred^\Ww(\beta+1)$.
 Then $\kappa=\crit(F^M)<\gamma\leq\crit(E^\Tt_\beta)<\OR(M^\Tt_\vareps)$, and
 because  $\OR^{M^\Tt_\vareps}$ is a successor
 cardinal of $M^\Ww_\vareps=\Ult_0(M,F^{M^\Tt_\vareps})$,
 the same functions are used in forming
 $\Ult_0(M^\Tt_\vareps,E^\Tt_\beta)$ as in forming
$i^\Ww_{\vareps,\beta+1}(M^\Ww_\vareps||\OR^{M^\Tt_\vareps})$.
\end{enumerate}
\end{enumerate}
With these points in mind, it is easy to see that $\Sigma_{\phW}$ is indeed a $((0,0),\om_1+1)$-strategy for $\phW$.
\end{proof}

We now reduce trees on $\phU$ to trees  on $\phW_{\pi(\zeta)}$,
giving the desired iterability of $\phU$.
 Recall (from directly after Assumption \ref{ass:v_not_empty})  that  $\pi:U\to W$ is the factor map.
 We have $\crit(\pi)=\zeta\geq\kappa$, so $\pi(\zeta)$ is a $W$-cardinal
 such that $\kappa<\pi(\zeta)<\OR^M$, and hence an $M$-cardinal.
 So by the previous claim, $\phW_{\pi(\zeta)}$ is $((0,0),\om_1+1)$-iterable.
 But  we can reduce trees $\Uu$ on $\phU$ to trees
 $\Ww$ on $\phW_{\pi(\zeta)}$ via  standard copying, except for one possibly non-standard detail which occurs at the superstrong level. We describe this below, along with some more detail in general.

 If $\zeta$ is an $M$-cardinal then
 $\phU=((M,{<\zeta}),U,\zeta)$,
 so we can lift directly to $\phW_{\pi(\zeta)}=((M,{<\pi(\zeta)}),W,\pi(\zeta))$
 using $\id:M\to M$ and $\pi:U\to W$
 as initial copy maps. This case is routine, so we move on.

 Suppose instead that $\card^M(\zeta)=\delta<\zeta=\delta^{+U}$, so  $R\pins M$ is defined, $\delta^{+U}=\delta^{+R}$ and
\[ \phU=((M,{<\delta}),(R,\delta),U,\delta^{+U}).\]
Then $\pi(\zeta)=\delta^{+W}=\delta^{+M}$ and we lift $(0,r,0)$-maximal trees on $\phU$ to essentially-$(0,0,0)$-maximal trees on the phalanx
\[ \phW'_{\pi(\zeta)}=((M,{<\delta}),(M,\delta),W,\delta^{+M}),\]
using $\id:M\to M$, inclusion $R\to R\pins M$ and $\pi:U\to W$ as initial copy maps.
 But this suffices,
since $\phW'_{\pi(\zeta)}$ is clearly equivalent to $\phW_{\pi(\zeta)}$,
so is $(0,0,0)$-maximally iterable,
and just like in the proofs of Lemmas \ref{lem:essentially_to_normal} and \ref{lem:ess_iter},
this implies it is essentially-$(0,0,0)$-maximally iterable. (Note that $\phW'_{\pi(\zeta)}$
has three models, where $\phW_{\pi(\zeta)}$ has only two.)

The reason we can only expect essentially-$(0,0,0)$-maximality instead of $(0,0,0)$-maximality for the tree on $\phW'_{\pi(\zeta)}$ is as follows. Let $\Uu$ on $\mathfrak{U}$ is $(0,r,0)$-maximal,
formed by lifting to $\Ww$ on $\mathfrak{W}'_{\pi(\zeta)}$.
 Suppose $\crit(E^\Uu_\alpha)=\delta$ and $E^\Uu_\alpha,E^\Ww_\alpha$ have superstrong type. Then
 $M^\Uu_{\alpha+1}=\Ult_r(R,E^\Uu_\alpha)$ and
  $\crit(E^\Ww_\alpha)=\delta$ and $M^\Ww_{\alpha+1}=\Ult_0(M,E^\Ww_\alpha)$.
 We get the copy map
  \[\pi_{\alpha+1}:M^\Uu_{\alpha+1}\to S_{\alpha+1}\pins M^\Ww_{\alpha+1},\]
  where $S_{\alpha+1}=i^{\Ww}_{-1,\alpha+1}(R)$, defined via the Shift Lemma (from earlier copy maps) in a straightforward manner. But $i^\Ww_{-1,\alpha+1}(\delta)=\lambda(E^\Ww_\alpha)=\nu(E^\Ww_\alpha)$,
  since $E^\Ww_\alpha$ is superstrong,
  and since $\rho_\om^R=\delta$,
  we have $\rho_\om^{S_{\alpha+1}}=\nu(E^\Ww_\alpha)$. Therefore $S_{\alpha+1}\pins M^\Ww_{\alpha+1}|\nu(E^\Ww_\alpha)^{+M^\Ww_{\alpha+1}}$,
  so $\OR(S^\Ww_{\alpha+1})<\nu(E^\Ww_\alpha)^{+M^\Ww_{\alpha+1}}=\lh(E^\Ww_{\alpha})$.
  But we will have $E^\Uu_{\alpha+1}\in\es_+(M^\Uu_{\alpha+1})$,
 so $E^\Ww_{\alpha+1}\in\es_+(S^\Ww_{\alpha+1})$  is the copy of $E^\Uu_{\alpha+1}$ under $\pi_{\alpha+1}$. Therefore $\lh(E^\Ww_{\alpha+1})<\lh(E^\Ww_\alpha)$,
 so $\Ww$ fails to be $(0,0,0)$-maximal.
 On the other hand, we have $\lh(E^\Uu_\alpha)\leq\lh(E^\Uu_{\alpha+1})$,
 and because of this,
 it is straightforward to see that $\nu(E^\Ww_\alpha)=\lambda(E^\Ww_\alpha)\leq\nu(E^\Ww_{\alpha+1})$. Moreover, it is only in this situation that we can have $\lh(E^\Ww_{\alpha+1})<\lh(E^\Ww_\alpha)$.
So $\Ww$ will be essentially-$(0,0,0)$-maximal.

The remaining details are quite standard,
so we leave them to the reader.
(There are also more details in the related proof of Claim \ref{clm:B_is_almost_iterable}
of the proof of Theorem \ref{thm:solidity}.)
\end{proof}
This completes the proof of the super-Dodd-soundness lemma \ref{lem:super-Dodd-soundness}.
\end{proof}

We can now complete the proof of the theorem on measures in mice:
\begin{proof}[Proof of Theorem \ref{thm:measures_in_mice}]
By Lemma \ref{lem:super-Dodd-soundness}, all proper segments of $M$ are Dodd-sound,
so Lemma \ref{lem:measures_in_mice} applies, so we are done.
\end{proof}

And we finish this section with the proof of the ISC for  normally iterable pseudo-premice (cf.~\cite[\S10]{fsit}):
\begin{tm}[Initial segment condition]\label{thm:ISC}
 Let $M$ be a $(0,\om_1+1)$-iterable pseudo-premouse.
 Then $M$ is a premouse.
\end{tm}

\begin{proof}
We use a trick similar to that used at the
beginning of
the proof of the super-Dodd-soundness lemma \ref{lem:super-Dodd-soundness},
to reduce to the classical case.
Let $\gamma$ be the largest generator of $F^M$, let $\delta=\lgcd(M)$ and let
$G=F^M\rest\delta$;
so $G\in M$. Let
\[ \Mbar=\cHull_1^M(\{\gamma,G\}), \]
let $\pi:\Mbar\to M$ be the uncollapse and let
$\pi(\gammabar,\deltabar,\Gbar)=(\gamma,\delta,G)$.
It is straightforward to see that $\gammabar$ is the largest generator of
$F^\Mbar$,
$\deltabar=\lgcd(\Mbar)$, $\Gbar=F^\Mbar\rest\deltabar$, and
if
$\Mbar$ is
a premouse then so
is $M$. So, resetting notation, we may assume that $\Mbar=M$ and $\pi=\id$. In
particular,
$\rho_1^M=\om$. It is then easy to see that if $M'$ is a pseudo-premouse and
$\sigma:M\to M'$ and
$\tau:M\to M'$ are cofinal and $\Sigma_1$-elementary, then
$\sigma(\gamma,G)=\tau(\gamma,G)$, and
therefore that $\sigma=\tau$. Using these facts in place of weak Dodd-Jensen,
the standard proof of
the ISC for $(0,\om_1,\om_1+1)^*$-iterable pseudo-premice goes through.
\end{proof}

\section{Projectum-finitely generated mice}\label{sec:finite_gen_hull}

\subsection{When mice are (or are not) iterates of their cores}\label{sec:iterates_of_cores}

In this section we will give the main argument for the projectum-finite generation theorem \ref{thm:finite_gen_hull},
which gives a simple criterion
guaranteeing that a mouse $M$ is an iterate of its core $C$. If we have an $m$-sound, $(m,\om_1+1)$-iterable premouse $M$, the criterion, \emph{projectum-finite generation},
is just that \[ M=\Hull_{m+1}^M(\rho_{m+1}^M\cup\{x\})
\]
for some $x\in\core_0(M)$.

But before we begin, we discuss some already known results which are related, and slight variants thereof. These results   will demonstrate that one cannot simply drop the criterion completely:
we can have $m<\om$ and an $(m,\om_1+1)$-iterable, $m$-sound premouse $M$
which is not an $m$-maximal iterate of $\core_{m+1}(M)$. This phenomenon first occurs roughly at the point of a cardinal which is ``strong past a measurable''. We will make this more precise in what follows; its interpretation is the deciding factor here.

\subsubsection{Mice which are iterates of their cores}

Recall $0^\pistol$ ($0$-pistol) is the least active mouse $M$ such that $M|\crit(F^M)\sats$ ``There is a strong cardinal''. If
$0^\pistol$ does not exist,
then by
%conf
\cite[Theorem 8.13]{cmip},
every universal weasel $W$ is a normal iterate of the core model $K$. (This result is also credited to Ronald Jensen in
%conf
\cite[between 8.20 and 8.21]{cmfali};
cf.~also \cite[Theorem 6.1]{hsstm}.) It is pointed out in
\cite[between 8.20 and 8.21]{cmfali} that  if $0^{\handgrenade}$ ($0$-hand-grenade,
%conf
\cite[Definition 2.3]{cmfali}) does not exist, then a similar result holds, assuming that for the model $K^c$ as defined in \cite{cmfali}, $K^c\sats$ ``there are no $\mu,\kappa$ such that $\kappa$ is measurable, $\mu<\kappa$ and $\mu$ is ${<\kappa}$-strong''.  In fact, one can go a little further, as shown in Theorem \ref{tm:if_U_not_iterate_of_core} below.

\begin{dfn}\label{dfn:limit_space_type}
Let $M$ be an active premouse
and let $\mu=\crit(F^M)$.
Say that $F^M$ is of \dfnemph{limit  space type} iff there is $\kappa>\mu$ such that $\kappa$ is measurable in $\Ult(M,F^M)$ and $\nu(F^M)=\kappa^{+M}$.
So $\OR^M=\kappa^{++\Ult(M,F^M)}$
and by coherence, the order $0$ measure on $\kappa$
in $\es^{\Ult(M,F^M)}$ is also in $\es^M$.\end{dfn}

\begin{tm}\label{tm:if_U_not_iterate_of_core}
Let $k<\om$ and let $K$ be a $(k+1)$-sound $(k,\om_1+1)$-iterable premouse
such that $\rho_{k+1}^K=\om$,
and let $\Sigma_K$ be the unique $(k,\om_1+1)$-iteration strategy for $K$.
Let $U$ be a $(k,\om_1+1)$-iterable premouse such that $\core_{k+1}(U)=K$.
Suppose that $U$ is not a $k$-maximal $\Sigma_K$-iterate of $K$.
Then there is a $K$-total extender $E\in\es^K_+$ which is of limit  space type.\end{tm}

\begin{proof}
Consider the $k$-maximal comparison $(\Uu,\Tt)$
of $(U,K)$, using some $(k,\om_1+1)$-strategy for $\Uu$ and $\Sigma_K$ for $\Tt$. Then
$M^\Uu_\infty=M^\Tt_\infty$  and $b^\Uu,b^\Tt$ do not drop in model or degree. Let $Q=M^\Tt_\infty$. Since $U$ is not a normal iterate of $K$
via $\Sigma_K$,
$\Uu$ is non-trivial. Note that by Lemma \ref{lem:basic_fs_pres},
since $Q$ is $(k+1)$-solid,
so is $U$, and $i^\Uu(p_{k+1}^U)=p_{k+1}^Q$. Letting $\pi:K\to U$ be the core map, we have $i^\Uu\com\pi=i^\Tt$.

Standard arguments with the hull and definability properties give that $\crit(i^\Tt)<\crit(i^\Uu)$,
and moreover, that there is $\alpha+1\in b^\Tt$ such that $\crit(E^\Tt_\alpha)<\crit(i^\Uu)<\nu(E^\Tt_\alpha)$. This is by calculations like in
%conf
\cite[Example 4.3 and Remark, p.~29]{cmip}. Let $\kappa=\crit(i^\Uu)$
and $\mu=\crit(E^\Tt_\alpha)$.

\begin{clm*} $E^\Tt_\alpha$
has an initial segment which is of
limit space type.\end{clm*}
\begin{proof}Suppose otherwise.
Let $\beta=\pred^\Tt(\alpha+1)$, $j=i^\Tt_{\beta\infty}$ and $k=i^\Uu$.
Noting that $\rg(j)\sub\rg(k)$, let
 $\pi':M^\Tt_\beta\to U$
be such that $k\com\pi'=j$.

By the ISC,  $E^\Tt_\alpha\rest\kappa\in\es^{M^\Tt_{\alpha+1}}$, and so $E^\Tt_\alpha\rest\kappa\in\es^U$.
Since $k\com\pi'=j$
and $\crit(k)=\kappa$,
$k(E^\Tt_\alpha\rest\kappa)=E_j\rest k(\kappa)$, where $E_j$ is the short extender derived from $j$.
But then $E_j\rest k(\kappa)\in\es^Q$, so $k(\kappa)<\nu(E^\Tt_\alpha)$.
Since $\kappa$ is measurable  in $U$, $k(\kappa)$ is measurable in $Q$, and hence in $M^\Tt_{\alpha+1}$. But since $E^\Tt_\alpha$ is not of limit space type, it follows that $k(\kappa)<\nu(E^\Tt_\alpha)<k(\kappa)^{+Q}=k(\kappa)^{+M^\Tt_{\alpha+1}}=\lh(E^\Tt_\alpha)$, and $E^\Tt_\alpha$ is type 2. Also since $E^\Tt_\alpha$ has no segment of limit space type, $\exit^\Tt_\alpha$
has no measurable cardinal $\theta$
such that $\crit(E^\Tt_\alpha)<\theta<k(\kappa)$,
and there is a unique total measure on $k(\kappa)$ in $\es(M^\Tt_{\alpha+1})$.

Like in \S\ref{sec:mim}
and \S\ref{sec:Dodd_proof},
we have $s_{E^\Tt_\alpha}=s_j=k(s_{\pi'})$ and $\sigma_{E^\Tt_\alpha}=\sigma_j=\sup k``\sigma_{\pi'}$.
Since $E^\Tt_\alpha$ is type 2,
$s_{E^\Tt_\alpha}\neq\emptyset$
and $k(\kappa)\leq\max(s_{E^\Tt_\alpha})<\lh(E^\Tt_\alpha)$. So $\kappa\leq\max(s_{\pi'})<\kappa^{+U}$.
Therefore $\sigma_{\pi'}\leq\kappa$, and therefore $\sigma_{\pi'}=\sigma_{j}=\sigma_{E^\Tt_\alpha}$.

Let $F=\core_{\ds}(M^\Tt_\alpha)$
and $\tau:\Ult(M^\Tt_\beta,F)\to Q$ be the standard factor map. If $E^\Tt_\alpha$ is Dodd-sound
then $F=E^\Tt_\alpha$ and $\tau=i^\Tt_{\alpha+1,\infty}$
and $\nu(E^\Tt_\alpha)<\crit(i^\Tt_{\alpha+1,\infty})$.
But by the previous paragraph, we can define $\pi'':\Ult(M^\Tt_\beta,F)\to U$ such that $k\com\pi''=\tau$.
Since $\crit(k)=\kappa$,
$\crit(\tau)\leq\kappa$.
So $E^\Tt_\alpha$ is non-Dodd-sound.

Let $\eta<^\Tt\alpha$ be such that $\core_{\ds}(M^\Tt_\alpha)\ins M^\Tt_\eta$, let $\gamma+1=\succ^\Tt(\eta,\alpha)$,
 and let  $\somevarpi:\core_{\ds}(M^\Tt_\alpha)\to M^\Tt_\alpha$ be the iteration map $\somevarpi=i^{*\Tt}_{\gamma+1,\alpha}$.
By the remarks above,
 the extenders used along the branch  $(\eta,\alpha]^\Tt$ are just the order $0$ measures on the largest cardinal of the current model (starting with $\core_{\ds}(M^\Tt_\alpha)$).
 If $\crit(\somevarpi)\geq\kappa$
 then let $F'=F^{\core_{\ds}(M^\Tt_\alpha)}$.
 If $\crit(\somevarpi)<\kappa$ then let
  $\xi\leq^\Tt\alpha$ be least such that $\xi=\alpha$ or $\lgcd(M^\Tt_\xi)\geq\kappa$, and let $F'=F^{M^\Tt_\xi}$.
  Now let
 $\tau':\Ult(M^\Tt_\beta,F')\to Q$ be the factor map
 (this is given by the tail end of the iteration of the order $0$ measures just mentioned, followed by $i^\Tt_{\alpha+1,\infty}$).
 Let $\pi''':\Ult(M^\Tt_\beta,F')\to U$ be such that $k\com\pi'''=\tau'$. We have $\crit(\tau')\geq\kappa$,
 so $\crit(\pi''')\geq\kappa$.
  But $\kappa\in\rg(\pi''')$, since $\kappa$ is a successor measurable of $U$. So $\pi'''(\kappa)=\kappa$ and
  $U||\kappa^{+U}=\Ult(M^\Tt_\beta,F')||\kappa^{+\Ult(M^\Tt_\beta,F')}$ and $k(\kappa)=\tau'(\kappa)$ and by commutativity, the short $(\kappa,k(\kappa))$ extenders derived from $k$ and $\tau'$ are identical. But these arise from comparison, a contradiction.
\end{proof}

Now note that if $\crit(E^\Tt_\alpha)=\lgcd(M^\Tt_\beta)$, so  $M^\Tt_\beta$ is active type 2,
then since $\crit(E^\Tt_\alpha)$
is a limit of measurables,
$F(M^\Tt_\beta)$ has a proper segment which is of limit space type, and since $(0,\beta]^\Tt$ does not drop, the same holds of $K$, as desired.  So suppose $\crit(E^\Tt_\alpha)^{+M^\Tt_\beta}<\OR^{M^\Tt_\beta}$. So $\crit(E^\Tt_\alpha)^{+M^\Tt_\beta}<\lh(E^\Tt_\beta)$.
Now arguing like in the proof of closeness \cite[6.1.5]{fsit},
it follows again that there is an $M^\Tt_\beta$-total extender $E\in\es_+(M^\Tt_\beta)$ which is of limit space type, and it easily follows that the same holds of $\es_+^K$.
\end{proof}

The result above is close to optimal.
One cannot obtain from the hypotheses of the theorem
an active mouse
$N$ and  $\kappa\in(\crit(F^N),\nu(F^N))$
such that $\kappa$ is measurable in $\Ult(N,F^N)$
and $\kappa^{+\Ult(N,F^N)}<\nu(F^N)$.
For letting $\gamma$ be the least generator of $F^N$ which is $>\xi=\kappa^{+\Ult(N,F^N)}$,
then $\gamma=\kappa^{++\Ult(N,F^N\rest\xi)}$
and $F^N\rest\xi=F^{N|\gamma}$.
Since $\kappa$ is measurable in $\Ult(N,F^N)$, letting $D$ be the order $0$ measure on $\kappa$ there,
we have $\lh(D)<\OR^N$,
and since we have the factor embedding $\pi:\Ult(N,F^N\rest\xi)\to\Ult(N,F^N)$,
and $\pi\rest(\xi+1)=\id$,
in fact $\lh(D)<\gamma$.
This is significantly beyond
the least mouse $P$ which models ZFC
and such that there are $D,E\in\es^P$
which are $P$-total
and such that $\crit(E)<\crit(D)<\lh(D)<\lh(E)$.
But such a mouse $P$ is enough to arrange a mouse which is not an iterate of its core, as we will see.

\subsubsection{A mouse which is not an iterate of its core}

%conf
In \cite[pp.~85--87]{cmip}, Steel
outlines a situation in which the core model $K$ exists and there is a universal weasel which is not an iterate of $K$ (and $K$ is ``below two strong cardinals'').
In Schindler \cite[Lemma 8.21]{cmfali} (this result is credited there to Steel), more details of such a construction are given, with a more precise large cardinal restriction.
But the hypothesis of \cite[8.21]{cmfali} was not actually quite enough for its stated purpose.
For the hypothesis of \cite[8.21]{cmfali}
is that $0^{\handgrenade}$ does not exist ($0^{\handgrenade}$, $0$-hand-grenade, is defined in \cite[Definition 2.3]{cmfali}), and letting $K$ be the core model (as constructed in \cite{cmfali} under this hypothesis) \footnote{The precise details of the core model under this hypothesis are not particularly important.
But if $M$ is a premouse which is below $0^{\handgrenade}$
and $\Tt$ is a limit length normal tree on $M$, then $M(\Tt)$ is the Q-structure for itself, and in fact the class of ordinals $\mu<\delta(\Tt)$ such that $M(\Tt)\sats$ ``$\mu$ is a strong cardinal'' gives a counterexample to Woodinness. This ensures that if a comparison of proper class models runs through $\OR$ stages,
then there are (proper class sized) cofinal branches at the end.},
there are $\mu<\kappa\in\OR$  such that $\kappa$ is measurable in $K$
and $K|\kappa\sats$ ``$\mu$ is strong'',
and the conclusion is that there is a generic extension of $K$ in which there is a universal weasel $U$ which is not an iterate of $K$. Because $0^{\handgrenade}$ does not exist, it follows that there is a (possibly proper class) successful comparison of $K$ and $U$. But a slight modification of the proof of
Theorem \ref{tm:if_U_not_iterate_of_core}
and standard core model arguments
show that if there is such a universal weasel $U$, then there is a $K$-total $E\in\es^K$ such that $E$ is of limit space type. But this is strictly beyond the hypothesis of \cite[8.21]{cmfali}.

We next give a  slight modification of that of \cite[8.21]{cmfali},
which achieves its goal from a slightly stronger hypothesis, and derive from this an example of a mouse which is not an iterate of its core.

\begin{tm}\label{tm:universal_weasel_not_iterate_of_K}
Let $K$ be a fully iterable premouse which models ZFC + ``there is no proper class premouse with a Woodin cardinal'',
and suppose there is no $K'\ins K$
such that $K'$ models ZFC + ``$\OR$ is Woodin''.
 Suppose there are $\mu,\kappa,D,E$
such that:
\begin{enumerate}[label=--]
\item  $D\in\es^K$ is a $K$-total order $0$ measure on $\kappa$,
\item
 $E\in\es^K$, $\crit(E)=\mu<\kappa$ and $\lh(D)<\lh(E)$.
 \end{enumerate}
 Then there are $\theta<\xi<\kappa$
 such that,
 letting $G$ be $(K,\Coll(\om,\xi^{+K}))$-generic,
 $K[G]\sats$
``There is a proper class premouse $W$
and a \tu{(}possibly proper class\tu{)} comparison $(\Tt,\Uu)$ of $(K,W)$
which terminates in a common
proper class premouse $Q=M^\Tt_\infty=M^\Uu_\infty$,
with $b^\Tt,b^\Uu$ non-dropping, and $\Uu$ is non-trivial, and in fact,
 $W||\theta^{++K}=K|\theta^{++K}$,
but $W|\theta^{++K}$ is active'', and therefore, $W$ is not an iterate of $K$.\end{tm}

\begin{proof}
We follow in outline the argument for \cite[8.21]{cmfali}, but making use of the stronger hypothesis and giving more details.

Let $E,D\in\es^K$ witness the hypothesis of the theorem,
with $(\lh(E),\lh(D))$ lexicographically least possible.
Let $\mu=\crit(E)$ and $\kappa=\crit(D)$. Let $U=\Ult(K,E)$.
Let
$g_0$ be  $(K,\Coll(\om,\mu^{++K}))$-generic.  Of course, $\mu^{++U}=\mu^{++K}$ and $g_0$ is also $(U,\Coll(\om,\mu^{++U}))$-generic.

 \begin{clmseven}\label{clm:T_for_Pi^1_3_properties} There is a tree $T\in K[g_0]$
  such that $T\sub\om\cross(\om\cross\OR^K)^{<\om}$
and in $K[g_0]$,
\begin{enumerate}[label=\arabic*.,ref=\arabic*]\item\label{item:forcing_below_kappa_implies_p[T]_sub_Pi^1_3}  for all $\alpha<\kappa$, $\Coll(\om,\alpha)$ forces that for all $\Pi^1_3$ formulas $\varphi$ and all reals $x$, if $(\varphi,x)\in p[T]$ then $\varphi(x)$, and
\item\label{item:forcing_at_kappa^+_implies_for_reals_in_U,g_0,g_1,_Pi^1_3_sub_p[T]} $\Coll(\om,\kappa^{+K})$ forces
that, letting $\dot{g}_1$ be the name for the generic filter, for all reals $x\in U[g_0][\dot{g}_1]$, if $\varphi(x)$ then $(\varphi,x)\in p[T]$.
\end{enumerate}
\end{clmseven}
\begin{proof}
This will be via Woodin's arguments as in
%conf
 \cite[\S3]{steel_dmt} and as in \cite[Theorem 1.5.12]{stattower}.
First fix the standard tree $S\in K$
such that $S\sub\om\cross({^{<\om}}\om\cross{^{<\om}}\mu)$
and in $K$, for each $\alpha<\mu$,
$\Coll(\om,\alpha)$
 forces that $p[S]$ is a universal $\Pi^1_2$ set. We get this as usual
from the closure of $K|\mu$  under sharps. Let $S'$ be the Shoenfield tree for $\Sigma^1_2$ on $\om\cross\mu$. So in $K$,
$S,S'$ are ${<\mu}$-complementing.
Let $R$ be the tree for a universal $\Sigma^1_3$ set derived from $S$
in the usual manner.

For $a\in i^K_E(\mu)^{<\om}$,
let $\nu_a$ be the measure over $\mu^{\lh(a)}$ derived from $i^K_E$ with seed $a$; so for $X\sub\mu^{\lh(a)}$, we have
\[ X\in \nu_a\iff a\in i^K_E(X).\]
So for each $a$,  $\nu_a\in K|\mu^{++K}$. Let $\sigma_a=i^K_E(\nu_a)$.
So $U\sats$ ``$\sigma_a$ is $\mu'$-complete measure over $(\mu')^{\lh(a)}$'',
where $\mu'=i^K_E(\mu)$.
Write $\dim(\sigma_a)=\lh(a)$.

Now work in $K[g_0]$, where $\mu^{++K}$ is countable. Let $\left<\tau_n,\theta_n\right>_{n<\om}$ enumerate all pairs $(\tau,\theta)$ such that $\tau\in{^{<\om}\om}$, $\theta=\sigma_a$ for some $a$, $\lh(\tau)=\dim(\theta)$,
and such that for each $n<\om$
and each $i<\lh(\tau_n)=\dim(\theta_n)$,
there is $m<n$ such that $(\tau_m,\theta_m)=(\tau_n,\theta_n)\rest m$,
where the latter restriction denotes the pair $(\tau_n\rest m,\theta_n\rest m)$,
where $\theta_n\rest m$ denotes the projection of $\theta_n$ to the first $m$ coordinates.
 For $m<\om$,
let $U_m=\Ult(U,\theta_m)$.
For $\ell<m<\om$ such that $(\tau_\ell,\theta_\ell)=(\tau_m,\theta_m)\rest\ell$,
let $i_{\ell m}:U_\ell\to U_m$ be the canonical factor map.
Let $T$ be the tree consisting of tuples
\[ (\varphi,\vec{x},\vec{\alpha})\in\om\cross{^{<\om}}\om\cross  {^{<\om}}(i^K_E(\mu^{++K}))\]
such that $\varphi(u)$ is a $\Pi^1_3$ formula of form $\all x\ \psi(x,u)$, where $\psi$ is $\Sigma^1_2$, and for some $n<\om$,
we have $n=\lh(\vec{x})=\lh(\vec{\alpha})$, and for  all $\ell,m$ with $\ell<m<n$ and $(\tau_\ell,\theta_\ell)=(\tau_m,\theta_m)\rest\ell$,
\[\text{ if }i^K_E(S_{(\neg\psi),\vec{x}\rest\lh(\tau_m),\tau_m})\in\theta_m\text{ then
 }i_{\ell m}(\alpha_\ell)>\alpha_m,\]
 where $\vec{\alpha}=(\alpha_0,\ldots,\alpha_{n-1})$.
 (Note here that if $\ell<m<n$ and $(\tau_\ell,\theta_\ell)=(\tau_m,\theta_m)\rest\ell$ and $i^K_E(S_{(\neg\psi),\vec{x}\rest\lh(\tau_m),\tau_m})\in\theta_m$
 then also $i^K_E(S_{(\neg\psi),\vec{x}\rest\lh(\tau_\ell),\tau_\ell})\in\theta_\ell$.)

\begin{sclmseven}
  $T$ is as desired.
  \end{sclmseven}
  \begin{proof}
 First let us verify clause \ref{item:forcing_below_kappa_implies_p[T]_sub_Pi^1_3}. So let $\alpha<\kappa$
 and $g_1$ be $(K[g_0],\Coll(\om,\alpha))$-generic. Let $x\in\RR^{K[g_0][g_1]}$
 and $\varphi$ be $\Pi^1_3$,
 and suppose that $(\varphi,x)\in p[T]$.
 Let $\vec{\alpha}\in{^{\om}}\OR$
 be such that $(\varphi,x,\vec{\alpha})\in[T]$.
 We have to see that $K[g_0][g_1]\sats\varphi(x)$, so suppose otherwise.
 We have $V_\kappa^K=V_\kappa^U$,
 so $\RR^{K[g_0][g_1]}=\RR^{U[g_0][g_1]}$,
 so $U[g_0][g_1]\sats\neg\varphi(x)$.
 But $i^K_E(R)$ projects to the universal $\Sigma^1_3$ set in $U[g_0][g_1]$,
 since $\alpha<\kappa<i^K_E(\mu)$.
 So letting $\varphi(x)$ be the formula $\all y\ \psi(x,y)$ where $\psi$ is $\Sigma^1_2$,
 we can fix $y,f\in U[g_0][g_1]$
 such that $y\in{^\om}\om$ and
  $((\neg\psi),x,y,f)\in[i^K_E(S)]$.
But then letting $a_n=f\rest n$,
$\left<\nu_{a_n}\right>_{n<\om}\in K[g_0][g_1]$
 is a tower of measures over $K$ and
  $\left<\sigma_{a_n}\right>_{n<\om}$
 is a tower of measures over $U$.
Each $\sigma_{a_n}$ canonically extends to a measure $\sigma^*_{a_n}$
 over $U[g_0][g_1]$,
 and $i^K_E(E)$ extends likewise to an extender $i^K_E(E)^*$ over $U[g_0][g_1]$.
The fact that $(\varphi,x)\in p[T]$
 establishes that $\Ult(U[g_0][g_1],\left<\sigma^*_{a_n}\right>_{n<\om})$
 is illfounded, but
 note that this ultrapower factors into the ultrapower $\Ult(U[g_0][g_1],(i^K_E(E))^*)$,
 which is wellfounded, a contradiction.

Now let us verify clause \ref{item:forcing_at_kappa^+_implies_for_reals_in_U,g_0,g_1,_Pi^1_3_sub_p[T]}:
Let $g_1$ be $(K[g_0],\Coll(\om,\kappa^{+K}))$-generic,
let $x\in\RR\cap U[g_0][g_1]$,
let $\varphi$ be $\Pi^1_3$,
and suppose $K[g_0][g_1]\sats\varphi(x)$.
We want to see that $(\varphi,x)\in p[T]$.  Since $x\in U[g_0][g_1]\sub K[g_0][g_1]$,
we have $U[g_0][g_1]\sats\varphi(x)$ also. Since $\kappa^{+K}<i^K_E(\mu)$, $i^K_E(R)$ projects to a universal $\Sigma^1_3$ set in $U[g_0][g_1]$.
So $((\neg\varphi),x)\notin p[i^K_E(R)]$.
So $i^K_E(R)_{(\neg\varphi),x}$ is wellfounded,
so has a rank function in $U[g_0][g_1]$. By using restrictions of this rank function to various sub-trees
(which are all  in $U[g_0][g_1]$
and hence relevant to forming the ultrapowers $U^*_m=\Ult(U[g_0][g_1],\theta_m^*)$, where $\theta_m^*$ is the extension of $\theta_m$ to $U[g_0][g_1]$),
we get a sequence $\vec{\alpha}\in{^\om}\OR\cap K[g_0][g_1]$ such that $(\varphi,x,\vec{\alpha})\in [T]$,
which suffices.
\end{proof}
Since $T$ works, we have established the claim.
\end{proof}

We will apply the claim
to the $\Pi^1_3$ notion of a \emph{good premouse}, which is as follows.
Working in  any set-generic extension $K[G]$ of $K$, we say that a countable premouse
$N$ is \emph{good} if it is sound,
and for every countable $N$-premouse $M$,
if $M$ is $\Pi^1_2$-iterable (above $N$, this means), then $M$ is equivalent to a  premouse $M'$ (so all of its proper segments are sound, as a premouse, as opposed to an $N$-premouse), and $M'$ is $\Pi^1_2$-iterable  (not just above $N$).

Let $g_1$ be $\Coll(\om,\kappa^{+K})$-generic and $g=(g_0,g_1)$.
\begin{clmseven}
$K[g]\sats$ ``$K|\lh(D)$ is good''.\end{clmseven}
\begin{proof}
Let $N=K|\lh(D)$. Let $M\in K[g]$ be such that in $K[g]$, $M$ is a $\Pi^1_2$-iterable countable $N$-premouse.
Then  working in $K[g]$, we can compare $U$ with $M$ (the resulting tree on $M$ only uses extenders with critical point $>\OR^N$).  By the smallness assumption on $K$,
standard absoluteness arguments ensure that this comparison terminates inside $K[g]$. That is, since there is no $K'\ins K$ satisfying ZFC + ``$\OR$ is Woodin'', the same holds for $U$,
and this means that segments of iterates of $U$ provide Q-structures which guide the formation of the tree on $M$ during the comparison.
And since $U$ computes many successor cardinals correctly,
$U$ out-iterates $M$. It follows that $M$ is in fact a standard premouse $M'$, and $M'$ is fully iterable in $V$, hence $\Pi^1_2$-iterable in $K[g]$.
\end{proof}

The following claim will yield the ordinal $\xi$ witnessing the theorem:

\begin{clmseven}\label{clm:there_is_a_generic_good_active_premouse_at_a_card} There are $\xi,\theta$ such that $\mu^{++K}<\theta<\xi<\kappa$
and $K[g_0]\sats$ ``it is forced by $\Coll(\om,\xi^{+})$
that there is an active good premouse $N$ such that $N^{\passive}=K||\theta^{++K}$''.
\end{clmseven}
\begin{proof}
Let $M=\Ult(K,D)$, so $g_0$ is also $(M,\Coll(\om,\mu^{++K}))$-generic,
and $D$ extends canonically to measure $D^*$ over $K[g_0]$, and $M[g_0]=\Ult(K[g_0],D^*)$.
 Let $\gamma=i^K_D(\kappa^{+K})$, so $\gamma$ is where $D$ would be indexed in Jensen indexing. Let $g_1'$ be $(K[g_0],\Coll(\om,\gamma))$-generic
 (of course, $\Coll(\om,\gamma)$
 is equivalent to $\Coll(\om,\kappa^{+K})$ in $K[g_0]$, as $\card^{K[g_0]}(\gamma)=\kappa^{+K}$). Then $g_1'$
is also $(M[g_0],\Coll(\om,\gamma))$-generic (but $\gamma$ is an $M[g_0]$-cardinal).
 Let $g'=(g_0,g_1')$.
Let $z\in\RR\cap M[g']$ be the canonical code for $M|\gamma$ determined by $(M|\gamma,g_1')$.

Working in $M[g']$,
let $T'_z$ be the tree of attempts
to build a real $y$ and a sequence $\vec{\alpha}$ such that $(\varphi_0,z\oplus y,\vec{\alpha})\in[i^{K[g_0]}_{D^*}(T)]$,
where $\varphi_0$
asserts that $z$ codes a passive premouse $N$
and $y$ codes an  extender $F$
such that for some other active
type 1 premouse $(N',F')$ with $N'\pins N$,
letting $\theta=\crit(F')$,
then $N=\Ult(N'|\theta^{+N'},F')$
and $F$ is the extender derived from $i^{N'|\theta^{+N'}}_{F'}$,
and $(N',F')$ is good.

We claim that $T'_z$ is illfounded.
For working in $K[g']$,
similarly define the tree $T_z$
of attempts to build a real $y$
and sequence $\vec{\alpha}$
such that $(\varphi_0,z\oplus y,\vec{\alpha})\in[T]$ (where $\varphi_0$ is as above). Then $T_z$ is illfounded, since $z\oplus y\in U[g']$ where $y$ codes the extender of $i^{K|\kappa^{+K}}_D$,
and by clause \ref{item:forcing_at_kappa^+_implies_for_reals_in_U,g_0,g_1,_Pi^1_3_sub_p[T]}.
But noting that $i^{K[g_0]}_{D^*}``T_z\sub T'_z$, it follows that $T'_z$ is illfounded.

So it is forced over $M[g_0]$ by $\Coll(\om,\gamma)$ that letting $z$ be the canonical code for $M|\gamma$,  $T'_z$ is illfounded (which is defined from  $i^{K[g_0]}_{D^*}(T)$). This fact pulls back under $i^{K[g_0]}_{D^*}$ to $K[g_0]$
and $K|\kappa^{+K}$ and $T$.

To establish the claim,
we need to reflect this to some $\xi<\kappa$ replacing $\kappa$,
so that we can apply clause \ref{item:forcing_below_kappa_implies_p[T]_sub_Pi^1_3}
of Claim \ref{clm:T_for_Pi^1_3_properties}.
Well, instead of $g_1'$ being $(M[g_0],\Coll(\om,\gamma))$-generic,
let $g_1''$ be $(K[g_0],\Coll(\om,\kappa^{+K}))$-generic, let $g''=(g_0,g_1'')$, and let $z\in M[g'']$ be the canonical code for $K|\kappa^{+K}$ determined by $(K|\kappa^{+K},g_1'')$.
Working in $M[g'']$,
define $T'_z$  as before (hence, still using $i^{K[g_0]}_{D^*}(T)$). By what was established above,  $T_z$  is illfounded.
But note  $i^{K[g_0]}_{D^*}``T_z\sub T'_z$, so $T'_z$ is illfounded.
Since $\kappa<i^{K[g_0]}_{D^*}(\kappa)$,
it follows by elementarity that there is $\xi<\kappa$ such that in $K[g_0]$,
$\Coll(\om,\xi^{+K})$ forces that
if $z'$ is the canonical real coding $K|\xi^{+K}$,
then $\bar{T}'_{z'}$ (defined over $K[g_0][h]$ from $T$, when $h$ is $(K[g_0],\Coll(\om,\xi^{+K}))$-generic) is illfounded. But since $\xi<\kappa$,
the claim therefore follows from
clause \ref{item:forcing_below_kappa_implies_p[T]_sub_Pi^1_3} of Claim \ref{clm:T_for_Pi^1_3_properties}.
\end{proof}

Fix $\theta<\xi<\kappa$
and $N$
witnessing Claim \ref{clm:there_is_a_generic_good_active_premouse_at_a_card},
with $N\in K[g]$ where
$g_1$ is generic over $K[g_0]$ for $\Coll(\om,\xi^{+K})$ and $g=(g_0,g_1)$.
So $N^{\passive}=K|\theta^{++K}$
and $N$ is active, and good in $K[g]$.

Now in $K[g]$,
there is no proper class inner model satisfying ZFC + ``there is a  a Woodin cardinal''.\footnote{This is a standard argument.
Suppose otherwise. Then by the methods of \cite{fsit}
and \cite{itertrees},
$K[g]$ satisfies ``there is a proper class $1$-small premouse $P$ with a Woodin cardinal such that for every $\alpha\in\OR$,
every countable elementary substructure of $P|\alpha$ is $\Pi^1_2$-iterable''.
But then working in $K$,
we can form a Boolean-valued comparison of possible interpretations of a name for
$P$. This leads to a proper class $1$-small premouse $P'$ with a Woodin cardinal $\delta$, with $P'|\delta\in K$, contradicting our smallness assumption on $K$.}
So using \cite{Kwithoutmeas},
we can build $K(N)$
(the core model over $N$),
which is a proper class $N$-premouse
which is iterable for set-sized trees. Since $N$ is good,
$K(N)$ is equivalent to a proper class premouse $W$ such that $N\pins W$ and $W$ is iterable for set-sized trees  (not just above $N$). (The latter is a straightforward reflection and consequence of goodness; just take an countable elementary substructure of a putative failure of iterability, containing the relevant Q-structures.)

\begin{clmseven}Work in $K[g]$.
Then $W$ has the properties claimed by the theorem.\end{clmseven}
\begin{proof}
Because $K,W$ are both fully iterable in $K[g]$, we can compare them through length $\OR$, or until the comparison $(\Tt,\Uu)$ terminates, if earlier.
If $\lh(\Tt,\Uu)<\OR$ then clearly $M^\Tt_\infty=M^\Uu_\infty$ and there are no drops on $b^\Tt,b^\Uu$.
Suppose $\lh(\Tt,\Uu)=\OR$.
Since $K$ is externally iterable
and $\Tt$ is correct,
we can externally fix a $\Tt$-cofinal wellfounded branch $b$
(or $\Tt$ uses only set-many extenders in $K[g]$).
By the smallness hypothesis,
$M(\Tt,\Uu)\sats$ ``$\OR$ is not Woodin'';
in other words, $M(\Tt,\Uu)$ is the Q-structure for itself.
Now working in $\J(K[g])$,
where $\OR^K$ is inaccessible,
we can continue the comparison,
using this Q-structure to determine the $\Tt$- and $\Uu$-cofinal branches $b,c$ as usual (see
%***check \cite{odle_v2} reference
the proof of \cite[Lemma 2.1]{odle_v2} for such a calculation). We claim that $i^\Tt_b``\OR^K\sub\OR^K$
and $i^\Uu_c``\OR^K\sub\OR^K$.
For by the usual weasel argument
in $\J(K[g])$,
 either $i^\Tt_b``\OR^K\sub\OR^K$
or $i^\Uu_c``\OR^K\sub\OR^K$.
But  then by weak covering for $K(N)$
in $K[g]$,
and since we are working in $K[g]$ (or $\J(K[g])$),
it follows that there is  $C\in\J(K[g])$ such that
$C\sub\OR^K$ is club,
and for every $\eta\in C$
which is a singular strong limit cardinal of $K$, we have $\eta^{+M(\Tt,\Uu)}=\eta^{+K}$.
So the usual calculations
therefore give what was claimed.

So $(\Tt\conc b,\Uu\conc c)$
forms a successful comparison with common last model, with $\Tt\conc b$ via $K$'s correct strategy, which suffices.
\end{proof}

This completes the proof of the theorem.
\end{proof}

As a corollary, we get a mouse which is not an iterate of its core:
\begin{cor}\label{cor:universal_weasel_not_iterate_of_K}Let $K$
be as in Theorem \ref{tm:universal_weasel_not_iterate_of_K} and countable, and suppose also that $K$ is pointwise definable. Let $W$ witness the theorem. Then $\J(K),\J(W)$ are premice, $\J(K)$ is sound with $\rho_1^{\J(K)}=\om$,
$\J(K),\J(W)$ are both $(0,\om_1+1)$-iterable, $\core_1(\J(W))=\J(K)$
but $\J(W)\neq\J(K)$, and letting $\gamma$ be the index of least disagreement between $\J(W),\J(K)$,
then $\J(W)|\gamma$ is active and $\gamma=\theta^{++\J(K)}$ for some $\theta$. Therefore $\J(W)$ is not a $0$-maximal iterate of $\J(K)$.
\end{cor}
\begin{proof} $\J(K)$  is a sound premouse with $\rho_1^K=\om$ and $p_1^K=\{\OR^K\}$
 because $K\sats$ ZFC and is pointwise definable.
And $\J(W)$ is a premouse
because $W\sats\ZFC$. Moreover,
$0$-maximal trees on $\J(K)$
are equivalent to those on $K$, and hence $\J(K)$ is $(0,\om_1+1)$-iterable. So by
%conf
\cite[Corollary 9.4]{iter_for_stacks}, $\J(K)$ is in fact $(0,\om_1,\om_1+1)^*$-iterable. And $0$-maximal trees on $\J(W)$ are also equivalent to those on $W$. So by the theorem, there is a successful comparison $(\Tt',\Uu')$ of $\J(K),\J(W)$,
and we get $i^{\Uu'}:\J(W)\to M^{\Uu'}_\infty=M^{\Tt'}_\infty$.
So $\J(W)$ is $(0,\om_1+1)$-iterable, and $\core_1(\J(W))=\core_1(M^{\Tt'}_\infty)=\J(K)$. But since there is $\theta$ as in Theorem \ref{tm:universal_weasel_not_iterate_of_K}, we are done.\end{proof}

\subsection{Projectum-finite generation}\label{sec:proj_fin_gen_theorem}

In this section we prove a key lemma toward the proof of the following theorem; recall
that \emph{almost-above} was defined in Definition
\ref{dfn:almost-above}.  We will rely on
%conf
\cite[\S2]{premouse_inheriting}, and also the methods used in \S\ref{sec:mim} of the present paper,
with both of which the reader should be familiar.

\begin{tm}[Projectum-finite-generation]\label{thm:finite_gen_hull}
Let $m<\om$ and  $M$ be an $m$-sound, $(m,\om_1+1)$-iterable premouse.
Suppose that
\[ M=\Hull_{m+1}^M(\rho_{m+1}^M\un \{x\})\]
for some $x\in M$.
Then there is a
successor length
$m$-maximal tree $\Tt$
on $\core_{m+1}(M)$,
such that $b^\Tt$ does not drop in model or degree,
$\Tt$ is strongly
finite and
almost-above $\rho_{m+1}^M$,
and $M=M^\Tt_\infty$. Moreover,
$i^\Tt_{0\infty}$ is just the core map.\end{tm}

The lemma just says that the theorem holds assuming also that $M$ is $(m+1)$-universal and $\core_{m+1}(M)$ is $(m+1)$-solid.
 So together with Theorem \ref{thm:solidity}, it immediately yields the theorem.

\begin{lem}\label{lem:finite_gen_hull}
Adopt the hypotheses of Theorem \ref{thm:finite_gen_hull}.
Suppose further that $M$ is $(m+1)$-universal
and $\core_{m+1}(M)$ is $(m+1)$-solid.
 Then the conclusion of Theorem \ref{thm:finite_gen_hull}
holds.
\end{lem}

The lemma has the following immediate corollary, using standard calculations:
\begin{cor}
\label{cor:def_over_core_from_extra_hypos} Adopt the hypotheses of Corollary \ref{cor:def_over_core}. Suppose further that $M$ is $(m+1)$-universal and $\core_{m+1}(M)$ is $(m+1)$-solid. Then $A$ is $\bfrSigma_{m+1}^{\core_{m+1}(M)}$.
\end{cor}
\begin{rem}\label{rem:simpler_corollary}
Our proof of solidity (Theorem \ref{thm:solidity}) will use Corollary \ref{cor:def_over_core_from_extra_hypos}, which follows immediately from Lemma \ref{lem:finite_gen_hull}.
Given iterability for stacks, one can directly prove the conclusion of Corollary \ref{cor:def_over_core_from_extra_hypos}
(in fact even Corollary \ref{cor:def_over_core},
which is just like Corollary \ref{cor:def_over_core_from_extra_hypos} but without the extra $(m+1)$-universality and $(m+1)$-solidity hypotheses)
by very standard arguments.
It seems tempting to try to prove Corollary \ref{cor:def_over_core_from_extra_hypos}
(from only normal iterability) via a bicephalus argument,
but otherwise using standard calculations, and hence
avoiding the extra effort  involved
in our proof of Lemma \ref{lem:finite_gen_hull}.
However, the author has not been successful in this,
as a difficulty arises in one special case.
We will describe this more carefully later,
in Remark \ref{rem:more_standard_attempt}.
\end{rem}

The plan for the proof of Lemma \ref{lem:finite_gen_hull} is as follows;
it is
motivated by the classical use of bicephali such as in \cite[\S9]{fsit}, and also the use of bicephali in \cite{premouse_inheriting}.
We may assume $M$ is countable.
Let $C=\core_{m+1}(M)$ and $\rho=\rho_{m+1}^M$.  Then $B=(C,M,\rho)$ is a degree $(m,m)$ bicephalus.
We will establish that $B$ is iterable,
and compare it with itself, producing trees $\Tt,\Uu$ on $B$. The comparison will proceed roughly via selecting extenders for least disagreement,
but at some stages of comparison,
for example when $B^\Uu_\alpha$  is a bicephalus, this can be ambiguous,
since we might consider either
 $M^{0\Uu}_\alpha$ (the model above $C$)
or $M^{1\Uu}_\alpha$ (that above $M$) for least disagreement with the  model(s)
$C^\Tt_\alpha=M^{0\Tt}_\alpha$ and/or $M^\Tt_\alpha=M^{1\Tt}_\alpha$ from $\Tt$ (whichever are defined). Likewise
with $\Tt,\Uu$ exchanged.
We will make the rules for deciding from which models we choose the least disagreeing extenders  precise, and arrange them so that the comparison can only terminate in a useful fashion. We want to see that one of the comparison trees,
say $\Tt$,
uses no extenders,
and $\Uu$ can be translated to an essentially equivalent tree $\Uu'$ on $C$
witnessing the statement of the lemma (so $\Uu'$ is strongly finite, $M^{\Uu'}_\infty=M$, etc).
Now the comparison will terminate for essentially the usual reasons. By the material in \S\ref{sec:bicephali},
if $B^\Uu_\alpha$ is a bicephalus then it has fine structural properties analogous
to those of $B$. Using the rules of the comparison and an variant of the proof of Lemma \ref{lem:measures_in_mice}, will show
that the comparison reaches a stage $\alpha$ at which in one of the trees, say $\Uu$,
we have  a bicephalus $B^\Uu_\alpha=(C^\Uu_\alpha,M^\Uu_\alpha,\rho^\Uu_\alpha)$, with
$C^\Uu_\alpha\ins M^{e\Tt}_\alpha$  for some $e\in\{0,1\}$, the exit extenders $E^\Uu_\alpha$ and $E^\Tt_\alpha$
are chosen by least disagreement
between $M^\Uu_\alpha$ and $C^\Uu_\alpha\ins M^{e\Tt}_\alpha$,
and the tail end $(\Tt,\Uu)\rest[\alpha,\infty]$ of $(\Tt,\Uu)$
after stage $\alpha$ has essentially the form
with respect to $B^\Uu_\alpha$
that we want to show that $(\Tt,\Uu)$
has with respect to $B$. (In particular, $\Uu\rest[\alpha,\infty]$ is trivial, and   $\Tt\rest[\alpha,\infty]$  translates to a ``tree'' $\Tt'$ on $C^\Uu_\alpha$ (recall $C^\Uu_\alpha\ins M^{e\Tt}_\alpha$) which is strongly finite,
$M^{\Tt'}_\infty=M^\Uu_\alpha$, etc.)\footnote{\label{ftn:wrinkles}The reason for the qualifier ``essentially''
and the quotation marks around ````tree'''' is that it seems we might not know the resulting ``tree'' actually has fully wellfounded models, at least in a certain special case,
in which that ``tree'' might use extenders overlapping $\rho^\Uu_\alpha$. However,
this illfoundedness could only occur up strictly above where the least disagreement occurs,
so it wouldn't cause a problem. In the end, we will instead describe this ``tree'' as a (real) tree on a certain phalanx, avoiding any such possibility of illfoundedness.}
In particular, $\lh(\Tt,\Uu)=\alpha+n+1$ for some $n<\om$.
So if $\alpha=0$, we would be done.
But if instead $\alpha>0$,
we will show that we can pull back the form of $(\Tt,\Uu)\rest[\alpha,\alpha+n]$
to $(\Tt,\Uu)\rest[0,n]$.
 It proceeds roughly as follows.
Let $G$ be the $(\rho,\rho^\Uu_\alpha)$-extender derived from $i^{0\Uu}_{0\alpha}$.
Let $\gamma$ be least such that $M|\gamma\neq C|\gamma$.
The calculation will show that
\begin{equation}\label{eqn:Ult_0(M|gamma,G)_ins} \Ult_0(M|\gamma,G)\ins M^\Uu_\alpha
\end{equation}
and
\begin{equation}\label{eqn:Ult_0(C|gamma,G)_ins}\Ult_0(C|\gamma,G)\ins C^\Uu_\alpha,
\end{equation}
and that these two structures have the same ordinal height and are distinct,
and therefore they form the least disagreement between $M^\Uu_\alpha$ and $C^\Uu_\alpha$.
But  as mentioned earlier, we had $C^\Uu_\alpha\ins M^{e\Tt}_\alpha$,
and $(E^\Uu_\alpha,E^\Tt_\alpha)$
constitute the least disagreement between $(M^\Uu_\alpha,M^{e\Tt}_\alpha)$,
and so $\exit^\Uu_\alpha=\Ult_0(M|\gamma,G)$ and $\exit^\Tt_\alpha=\Ult_0(C|\gamma,G)$.
But since we already know
$\Uu\rest[\alpha,\infty]$ uses no extenders and $E^\Tt_\alpha\neq\emptyset$,
it follows that $F^{M|\gamma}=\emptyset$ and $F^{C|\gamma}\neq\emptyset$,
and that the fine structural properties
of $C|\gamma$ are closely related to those of $\exit^\Tt_\alpha$.
This all generalizes to a similar situation at stages $i\in[0,n)$ (with slight wrinkles foreshadowed in Footnote \ref{ftn:wrinkles}).
From this correspondence between  $(\Tt,\Uu)\rest[0,n]$ and $(\Tt,\Uu)\rest[\alpha,\alpha+n]$,
we will conclude that
actually $(\Tt,\Uu)$ is as desired, and $\alpha=0$, a contradiction.

Before beginning the proof of the lemma,
we will describe the methods we will
use to establish lines (\ref{eqn:Ult_0(M|gamma,G)_ins}) and (\ref{eqn:Ult_0(C|gamma,G)_ins})  above.
The calculation is a generalization of the methods used in the analysis of comparisons
in \cite[\S4]{premouse_inheriting},
and in particular of how stages of comparison lift under ultrapower maps; this kind of analysis of preservation of comparison (under ultrapower maps) is also developed further  in
%***check where
\cite[\S8]{fullnorm}.
A key component of the calculation (Lemma \ref{lem:dropdown_lifting} below) involves arguments with condensation
much like those in \cite[\S4]{premouse_inheriting}, and which were
worked out more generally by Steel, and then sharpened independently by Steel and the author (see \cite{fullnorm}). These calculations are integral within full normalization \cite{fullnorm}.
We first recall a basic
definition from \cite{fsit},
though slightly modified:

\begin{dfn}\label{dfn:dropdown}
Let $M$ be an $m$-sound premouse and
 $\gamma\leq\OR^M$.
 The \dfnemph{extended dropdown sequence of $((M,m),\gamma)$}
is the sequence $\left<(M_i,m_i)\right>_{i\leq n}$ where
$(M_0,m_0)=(M|\gamma,0)$ and
for $i<n$,
 $(M_{i+1},m_{i+1})$ is the lexicographically
least $(M',m')$ such that $(M_i,m_i)\ins(M',m')\ins (M,m)$ and
either $(M',m')=(M,m)$ or $\rho_{m'+1}^{M'}<\rho_{m_i+1}^{M_i}$,
and moreover, $n<\om$ is as large as possible under these conditions.
\end{dfn}

\begin{rem}\label{rem:dropdown}
Note $(M,m)=(M_n,m_n)$  by definition.
Write $\rhotilde_{m+1}^M=0$, and $\rhotilde_k^K=\rho_k^K$ for other $K,k$.
Let $i<n$.
Note that if $M_i\pins M_{i+1}$ then
$\theta=\rho_{m_i+1}^{M_i}=\rho_\om^{M_i}$
is an $M_{i+1}$-cardinal and
$\rhotilde_{m_{i+1}+1}^{M_{i+1}}<\theta\leq\rho_{m_{i+1}}^{M_{i+1}}$.
And if $M_i=M_{i+1}$ then again
$\rhotilde_{m_{i+1}+1}^{M_{i+1}}<\rho_{m_{i+1}}^{M_{i+1}}=\rho_{m_i+1}^{M_i}$.
It can however be that $\rhotilde_{m_i+1}^{M_i}=\rho_{m_i}^{M_i}$,
but by the preceding remarks, this implies $i=0<n$ (so $m_i=m_0=0$)
and  $\rho_0^{M_0}=\rho_1^{M_0}$
(so $M_0$ is passive or type 3).
\end{rem}

For $m<\om$ and  $m$-sound premouse $M$,
\emph{suitable condensation
%conf
 below $(M,m)$} is defined in \cite[Definition 2.10]{premouse_inheriting}.
 %conf
By \cite[Lemma 2.11]{premouse_inheriting},
this notion is preserved by iteration maps at
the degrees we need it,
%conf
and by \cite[Lemma 2.14]{premouse_inheriting} (or alternatively,
the finer \cite[Theorem 5.2]{premouse_inheriting}),
it is a consequence of $(m,\om_1+1)$-iterability
(in fact $(m-1,\om_1+1)$-iterability if $m>0$).

Some form of the following lemma is due to Steel.\footnote{Components of the argument were shown in  \cite{premouse_inheriting},
in particular \cite[Lemma
3.17(11)]{premouse_inheriting}.
Steel showed that $U_0\ins N$ under the  circumstances of the lemma. The full lemma as stated here was then observed by the author and Steel independently, but it comes out quite readily from  Steel's proof that $U_0\ins N$.}
The proof uses an iterated version of the proof
%conf
of \cite[Lemma
3.17(11)]{premouse_inheriting}:

\begin{lem}[Dropdown lifting]\label{lem:dropdown_lifting}
 Let $M$ be an $m$-sound premouse.
 Suppose suitable condensation holds below $(M,m)$.
 Let $\rho<\OR^M$ be an $M$-cardinal
 with $\rho\leq\rho_m^M$.
 Let
$\left(\left<M_\alpha\right>_{\alpha\leq\lambda},
\left<E_\alpha\right>_{\alpha<\lambda}\right)$ be a wellfounded degree
$m$ abstract iteration of $M$
with abstract iteration maps $j_{\alpha\beta}:M_\alpha\to M_\beta$,
where  $\crit(E_\alpha)<\sup j_{0\alpha}``\rho$ for each $\alpha<\lambda$.
Let $\pi=j_{0\lambda}$ and $N=M_\lambda$.
Let $G$ be the $(\rho,\sup\pi``\rho)$-extender
over $M$ derived from $\pi$.

Let $\rho\leq\gamma\leq\OR^M$ and $\left<(M_i,m_i)\right>_{i\leq n}$
be the extended dropdown sequence of $((M,m),\gamma)$.
Let
$U_i=\Ult_{m_i}(M_i,G)$.
Then $U_i\ins N$ for each $i\leq n$,
and  $\left<(U_i,m_i)\right>_{i\leq n}$ is the
extended dropdown sequence of $((N,m),\OR^{U_0})$.
\end{lem}
\begin{proof}
By definition, $(M_n,m_n)=(M,m)$, and $U_n=N=\Ult_m(M,G)$,
so $(U_n,m)$ is the terminal element of the extended dropdown
sequence of $((N,m),\OR^{U_0})$.

Now suppose $U_{i+1}\ins
N$ where $i<n$, and we will show the following,
writing $\widetilde{\rho}^M_{m+1}=\widetilde{\rho}_{m+1}^N=0$
and $\widetilde{\rho}_{k+1}^K=\rho_{k+1}^K$ for other $K,k$:

\begin{clm*}We have:
\begin{enumerate}[label=\arabic*.,ref=\arabic*]
 \item\label{item:U_i_ins_U_i+1} $U_i\ins U_{i+1}$,
 \item\label{item:if_U_i=U_i+1} if $U_i=U_{i+1}$ then
$\widetilde{\rho}_{m_{i+1}+1}^{U_{i+1}}<
\rho_{m_{i+1}}^{U_{i+1}}=\rho_{m_{i+1}}^{U_i} =
\rho_{m_i+1}^{U_i}$,
 \item\label{item:if_U_i_pins_U_i+1} if $U_i\pins U_{i+1}$ then
$\widetilde{\rho}_{m_{i+1}+1}^{U_{i+1}}<
\rho_\om^{U_i}=\rho_{m_i+1}^{U_i}=\theta\leq \rho_{m_{i+1}}^{U_{i+1}}$ and
$\theta$ is a cardinal of $U_{i+1}$,
\item\label{item:if_rho_m_i+1<rho_m_i+1} if
$\widetilde{\rho}_{m_i+1}^{M_i}<\rho_{m_i}^{M_i}$ then
$\widetilde{\rho}_{m_i+1}^{U_i}<\rho_{m_i}^{U_i}$,
\item\label{item:if_rho_m_i+1=rho_m_i+1} if
$\widetilde{\rho}_{m_i+1}^{M_i}=\rho_{m_i}^{M_i}$
then $\widetilde{\rho}_{m_i+1}^{U_i}=\rho_{m_i}^{U_i}$
(see Remark \ref{rem:dropdown}).
\end{enumerate}
\end{clm*}
\begin{proof}

\begin{casetwo} $M_{i+1}=M_i$.

Let $A=M_{i+1}=M_i$. We have then $m_{i+1}>m_i$ and
$\widetilde{\rho}_{m_{i+1}+1}^A<\rho_{m_{i+1}}^A=\rho_{m_i+1}^A$.

For $m_i\leq k\leq m_{i+1}$, let
$B_k=\Ult_k(A,G)$
and $i_k=i^{A,k}_G:A\to B_k$.
For $m_i\leq k< m_{i+1}$ let
$\sigma_{k,k+1}:B_k\to B_{k+1}$
be the natural factor map,
so $\sigma_{k,k+1}\com i_k=i_{k+1}$
and $\sigma_{k,k+1}$ is $\pvec_{k+1}$-preserving
$k$-lifting.
Almost
as in the proof of \cite[Lemma
3.17(11)]{premouse_inheriting},
we have $B_k\ins B_{k+1}$. For self-containment, here is the reason: Using
Corollary \ref{cor:basic_fs_pres},
both $B_k,B_{k+1}$ are $(k+1)$-sound
with
\[ \rho'=\rho_{k+1}^{B_{k}}=\sup i_k``\rho_{k+1}^A=\sup
i_{k+1}``\rho_{k+1}^A=\rho_{k+1}^{B_{k+1}}, \]
and  $\rho_k^{B_k}=\sup i_k``\rho_k^A$.
Let $\sigma=\sigma_{k,k+1}$. Then part \ref{item:U_i_ins_U_i+1} holds in this case as:
\begin{enumerate}[label=\arabic*.,ref=\arabic*]
\item If $\sigma``\rho_k^{B_k}$ is unbounded in $\rho_k^{B_{k+1}}$,
%conf
then by \cite[Lemma 2.4]{premouse_inheriting}, $\sigma$
is a $k$-embedding,
which implies $\sigma=\id$ and $B_k=B_{k+1}$.
 \item If $\sigma``\rho_k^{B_k}$ is bounded in $\rho_k^{B_{k+1}}$,
then by \cite[Lemma 2.4]{premouse_inheriting},
 suitable condensation applies,
and as $\rho'$ is a cardinal in $B_{k+1}$,
we get $B_k\pins B_{k+1}$.
 \end{enumerate}

 Now $\widetilde{\rho}_{m_{i+1}+1}^A<\rho_{m_{i+1}}^A$, so
 \[ \widetilde{\rho}_{m_{i+1}+1}^{U_{i+1}}=\sup
i_{m_{i+1}}``\widetilde{\rho}_{m_{i+1}+1}^A<\sup
i_{m_{i+1}}``\rho_{m_{i+1}}^A=\rho_{m_{i+1}}^{U_{i+1}},\]
and since $\theta=\rho_{m_{i+1}}^A=\ldots=\rho_{m_{i}+1}^A$,
 $i_k\rest\theta$ is independent of $k\in[m_i,m_{i+1}]$,
so
\[
\widetilde{\rho}_{m_{i+1}+1}^{U_{i+1}}
<
\rho_{m_{i+1}}^{U_{i+1}}=
\rho_{m_{i+1}}^{B_{m_{i+1}}}=
\rho_{m_{i+1}-1}^{B_{m_{i+1}-1}}=\ldots=\rho_{m_i+1}^{B_{m_i+1}}=
\rho_{m_i+1}^{B_{m_i}}=\rho_{m_i+1}^{U_i}\]
and \[ \rho_{m_i+1}^A<\rho_{m_i}^A\iff
\rho_{m_i+1}^{U_i}<\rho_{m_i}^{U_i}.\]
If $\rho_{m_i+1}^A=\rho_{m_i}^A$ then also $i=0=m_0$, by Remark \ref{rem:dropdown}.

So if $U_i=U_{i+1}$ then we get part \ref{item:if_U_i=U_i+1}
as needed, and
if $U_i\pins U_{i+1}$ then note that
 $\rho_{m_{i+1}}^{U_{i+1}}=\rho_{m_i+1}^{U_i}$ is a cardinal of $U_{i+1}$,
  so $\rho_\om^{U_i}=\rho_{m_i+1}^{U_i}$,
  giving part \ref{item:if_U_i_pins_U_i+1}.
  Parts \ref{item:if_rho_m_i+1<rho_m_i+1} and  \ref{item:if_rho_m_i+1=rho_m_i+1}
also follow easily.
  \end{casetwo}

  \begin{casetwo} $M_{i}\pins M_{i+1}$.

  So $\theta=\rho_\om^{M_i}$ is an $M_{i+1}$-cardinal,
$\widetilde{\rho}_{m_{i+1}+1}^{M_{i+1}}<\theta\leq\rho_{m_{i+1}}^{M_{i+1}}$,
and either:
\begin{enumerate}[label=--]
\item $\rho_{m_i+1}^{M_i}=\theta<\rho_{m_i}^M$, or
\item
 $i=0=m_0$, $\rho_0^{M_0}=\rho_1^{M_0}=\rho_\om^{M_0}=\theta$ and $M_0$ is passive or type 3.
\end{enumerate}

If $M_{i+1}$ is active
let
$\psi:\Ult_0(M_{i+1},F^{M_{i+1}})\to\Ult_0(U_{i+1},F^{U_{i+1}})$
be given by the Shift Lemma with $j=i^{M_{i+1},m_{i+1}}_{G}$;
otherwise let $\psi=j$.
Note  $M_i\in\dom(\psi)$
and $\psi\rest\theta=j\rest\theta$.
Let
\[ \sigma:U_i=\Ult_{m_i}(M_i,G)\to\psi(M_i) \]
be the natural factor map.
Note that $\rho_{m_i+1}^{U_i}=\sup\psi``\theta$
is a cardinal in $U_{i+1}$, because if $\mu^{+M_{i+1}}\leq\theta\leq\rho_{m_{i+1}}^{M_{i+1}}$ then $\mu^{+M_{i+1}}$ is $\bfrSigma_{m_{i+1}}^{M_{i+1}}$-regular, and since
and $\rho\leq\mu^{+M_{i+1}}$ (recall $G$ is the $(\rho,\sup\pi``\rho)$-extender derived from $\pi$), therefore $j$ is continuous at $\mu^{+M_{i+1}}$. Also $\sigma\rest\rho_{m_i+1}^{U_i}=\id$,
and much as before, by elementarity,  suitable condensation applies
to $\sigma$, giving $U_i\pins U_{i+1}$. (In the case
that $M_{i+1}$ is type 3 and $\psi(M_i)\npins U_{i+1}$, we have
$\psi(M_i)|\OR^{U_{i+1}}=U_{i+1}||\OR^{U_{i+1}}$
by coherence, and we get $U_i\pins\psi(M_i)$ by condensation,
but $\rho_\om^{U_i}<\OR^{U_{i+1}}$, so $U_i\pins U_{i+1}$.)
The desired properties now easily follow.\qedhere
\end{casetwo}\end{proof}

This completes the proof of the claim, and the theorem
is an easy consequence.
\end{proof}

We can now begin the main argument of this section.

\begin{proof}[Proof of Lemma \ref{lem:finite_gen_hull}]
 We may assume that $M$ is not $(m+1)$-sound.
 Let  $\rho=\rho_{m+1}^C=\rho_{m+1}^M$
 and $\pi:C\to M$ the core map. So $C\neq M$ but
  $\rho^{+C}=\rho^{+M}$.
  Note  $\rho<\rho_0^C\leq\rho_0^M$.
 Define the bicephalus $B=(\rho,C,M)$, which is
degree-maximally iterable,
as there is an easy copying argument
to lift trees on $B$ to trees on $M$
(there are such iterability proofs in \cite{premouse_inheriting}).
For a degree-maximal tree $\Tt$ on $B$, recalling notation from \S\ref{sec:bicephali},
we write $C^\Tt_\alpha=M^{0\Tt}_\alpha$
and $M^\Tt_\alpha=M^{1\Tt}_\alpha$
and $i^\Tt_{\alpha\beta}=i^{0\Tt}_{\alpha\beta}$
and $j^\Tt_{\alpha\beta}=i^{1\Tt}_{\alpha\beta}$ etc,
and $\rho^{+\Tt}_\alpha$ for
$(\rho^\Tt_\alpha)^{+C^\Tt_\alpha}=(\rho^\Tt_\alpha)^{+M^\Tt_\alpha}$.

We compare $B$ with itself, producing padded degree-maximal trees $\Tt,\Uu$
respectively, much as in the bicephalus and cephalanx
comparisons in
\cite{premouse_inheriting}.\footnote{One might also consider comparing
$B$ with $C$, or $B$ with $M$. Say we produce
trees $\Tt$ on $B$ and $\Uu$ on $C$ or on $M$.
If we are comparing $B$ with $C$, it seems
we might reach a stage $\alpha$
such that
$(0,\alpha]^\Uu$ does not drop and $C^\Uu_\alpha\pins B^\Tt_\alpha$,
and so the comparison terminates at that point, without a useful outcome.
If instead we are comparing $B$ with $M$,
it seems we might reach a stage $\alpha$ with
$\alpha\in\curlyB^\Tt$
and $(0,\alpha]^\Uu$ does not drop,
$\rho^\Tt_\alpha=\rho_{m+1}(M^\Uu_\alpha)$
and
$M^\Tt_\alpha||(\rho^\Tt_\alpha)^{+M^\Tt_\alpha}
=M^\Uu_\alpha||(\rho^\Tt_\alpha)^{+M^\Uu_\alpha}$,
but $M^\Tt_\alpha\neq M^\Uu_\alpha$. If we then move into $C^\Tt_\alpha$
in $\Tt$, the comparison might terminate above $C^\Tt_\alpha$
and above $M^\Uu_\alpha$,
but this does not seem to yield relevant information about $B$.
If instead we move into $M^\Tt_\alpha$ in $\Tt$,
the comparison might terminate above $M^\Tt_\alpha$ and $M^\Uu_\alpha$,
but this also does not seem to yield useful information,
at least not directly.}
The comparison proceeds basically via least disagreement.
However, there are two caveats, just like in  \cite{premouse_inheriting}.
The first is that when selecting extenders $E$ by least disagreement,
we must also minimize on $\nu(E)$ before proceeding,
as described in Remark \ref{rem:superstrong_diff}.
Secondly, when there are multiple models available
on one (or both) side(s), we have to take some care
in determining exactly how to proceed,
also as in \cite{premouse_inheriting}.
There can be stages
$\alpha$ at which we have a bicephalus $B^\Tt_\alpha$ in $\Tt$ say,
and we decide that we want to make sure that the next
extender used in $\Tt$ comes from some particular model
of $B^\Tt_\alpha$, say $C^\Tt_\alpha$. In this case,
we pad in both trees at this stage, setting
$E^\Tt_\alpha=\emptyset=E^\Uu_\alpha$,
but restrict the \emph{available models} in $\Tt$ to
include only $C^\Tt_{\alpha+1}=C^\Tt_{\alpha}$ at stage $\alpha+1$,
and we say that  in $\Tt$ we \emph{move into} $C^\Tt_\alpha$
at stage $\alpha$.
We can similarly \emph{move into}  $M^\Tt_\alpha$ in $\Tt$.
Likewise in
$\Uu$.
To facilitate
this notationally,
we will keep track of sets
$S^\Tt_\alpha,S^\Uu_\alpha,\modelset^\Tt_\alpha,\modelset^\Uu_\alpha$,
specifying the models available at stage $\alpha$.
Here $\emptyset\neq S^\Tt_\alpha\sub\sides^\Tt_\alpha$ and
 $\modelset^\Tt_\alpha=\{M^{e\Tt}_\alpha\bigm|e\in S^\Tt_\alpha\}$,
and likewise for $S^\Uu_\alpha,\modelset^\Uu_\alpha$.
We will only ever take exit extenders from available models.
Formally, this means that if $S^\Tt_\alpha=\{n\}$ and
$E^\Tt_\alpha\neq\emptyset$
then $\exitside^\Tt_\alpha=n$.
Moreover, once we have restricted
the set of available models,
we may not remove this restriction,
except by using an exit extender from that model;
formally, this means that if $E^\Tt_\beta=\emptyset$
for all $\beta\in[\alpha,\gamma)$ then $S^\Tt_\gamma\sub S^\Tt_\alpha$.
If $E^\Tt_\alpha\neq\emptyset$ and $\alpha+1\in\curlyB^\Tt$, however,
then both models are automatically available at stage $\alpha+1$ in $\Tt$,
i.e.~$S^\Tt_{\alpha+1}=\{0,1\}$. Likewise
for a limit $\eta\in\curlyB^\Tt$
such that $\Tt$ uses extenders cofinally ${<\eta}$.
(As with standard comparison, there can
 also be stages $\alpha$ where $E^\Tt_\alpha\neq\emptyset=E^\Uu_\alpha$,
or vice versa.) If $E^\Tt_\alpha\neq\emptyset$
and $S^\Tt_\alpha=\{0,1\}$, then it will actually be that
 $E^\Tt_\alpha\in\es(C^\Tt_\alpha)\cap\es(M^\Tt_\alpha)$,
 and in fact $\lh(E^\Tt_\alpha)<\rho_\alpha^{+\Tt}$; in this case
 we always set $\exitside^\Tt_\alpha=0$.
 Once we have specified how to select $\exitside^\Tt_\alpha$ and $E^\Tt_\alpha$,
 the remaining features of $\Tt\rest(\alpha+2)$ are determined by the rules for degree-maximality (Definition \ref{dfn:bicephalus_tree}) and Remark \ref{rem:padding}.
  Likewise for $\Uu$.

We now explain the
rules for the comparison.\label{page:comparison_rules} The basic point is that, at stage $\alpha$,
if we have $S^\Tt_\alpha=\{0,1\}$,
and we  decide whether
to move into $C^\Tt_\alpha$ or $M^\Tt_\alpha$
at stage $\alpha$, we choose $C^\Tt_\alpha$
unless this would prematurely end the comparison.
Similarly in $\Uu$, we prefer $C^\Uu_\alpha$ over $M^\Uu_\alpha$.
If $S^\Tt_\alpha=\{0,1\}=S^\Uu_\alpha$,
we will only move into a model on both sides at stage $\alpha$
if $B^\Tt_\alpha||\rho_\alpha^{+\Tt}=B^\Uu_\alpha||\rho_\alpha^{+\Uu}$.
In this case, if $C^\Tt_\alpha=C^\Uu_\alpha$, then
we either (i) move into $C^\Tt_\alpha=C^\Uu_\alpha$ in $\Tt$
and $M^\Uu_\alpha$ in $\Uu$, or (ii) move into $C^\Tt_\alpha=C^\Uu_\alpha$
in $\Uu$ and $M^\Tt_\alpha$ in $\Tt$
(randomly choosing one option, for symmetry).
Formally, this is executed as follows.
We will construct degree-maximal trees $\Tt,\Uu$.
We start with $B^\Tt_0=B=B^\Uu_0$ and $S^\Tt_0=\{0,1\}=S^\Uu_0$.
Suppose we have $(\Tt,\Uu)\rest(\alpha+1)$ and $S^\Tt_\alpha,S^\Uu_\alpha$.

 \begin{casethree}
 $S^\Tt_\alpha=\{0,1\}=S^\Uu_\alpha$ and
 $B^\Tt_\alpha||\rho_+= B^\Uu_\alpha||\rho_+$ where $\rho_+=\min(\rho^{+\Tt}_\alpha,\rho^{+\Uu}_\alpha)$.

We set $E^\Tt_\alpha=\emptyset=E^\Uu_\alpha$, and will move into some model(s).

\begin{scasethree}
$\rho_+=\rho^{+\Tt}_\alpha=\rho^{+\Uu}
_\alpha$ and $C^\Tt_\alpha=C'=C^\Uu_\alpha$.

Either:
In $\Tt$  move into $C'$, and in $\Uu$
move into $M^\Uu_\alpha$;
or, in $\Tt$  move into $M^\Tt_\alpha$,
and in $\Uu$  move into $C'$ (irrespective of whether
$M^\Tt_\alpha=M^\Uu_\alpha$).
\end{scasethree}

\begin{scasethree}$\rho_+=\rho^{+\Tt}_\alpha=\rho^{+\Uu}_\alpha$ and
$C^\Tt_\alpha\neq C^\Uu_\alpha$.

In $\Tt$ move into $C^\Tt_\alpha$, and in $\Uu$  move into $C^\Uu_\alpha$.

(Note here that $C^\Tt_\alpha\nins C^\Uu_\alpha\nins C^\Tt_\alpha$.)
\end{scasethree}

\begin{scasethree}\label{scase:C^Tt_alpha_pins_C^Uu_alpha}
$\rho_+=\rho^{+\Tt}_\alpha<\rho^{+\Uu}_\alpha$ and $C^\Tt_\alpha\pins
B^\Uu_\alpha||\rho^{+\Uu}_\alpha$.

In $\Tt$ move into $M^\Tt_\alpha$. There is no change in $\Uu$.
\end{scasethree}

\begin{scasethree}\label{scase:C^Tt_alpha_npins_C^Uu_alpha}
$\rho_+=\rho^{+\Tt}_\alpha<\rho^{+\Uu}_\alpha$ and $C^\Tt_\alpha\npins
B^\Uu_\alpha||\rho_\alpha^{+\Uu}$.

In $\Tt$ move into $C^\Tt_\alpha$. There is no change in $\Uu$.
\end{scasethree}

\begin{scasethree}
$\rho_+=\rho^{+\Uu}_\alpha<\rho^{+\Tt}_\alpha$.

By symmetry with Subcases \ref{scase:C^Tt_alpha_pins_C^Uu_alpha} and
\ref{scase:C^Tt_alpha_npins_C^Uu_alpha}.
\end{scasethree}

\end{casethree}

\begin{casethree}\label{case:pfg_comparison_rules_both_S_singletons} $S^\Tt_\alpha\neq\{0,1\}\neq S^\Uu_\alpha$.

Select extenders by least disagreement, if possible.
(That is, letting $S^\Tt_\alpha=\{d\}$
and $S^\Uu_\alpha=\{e\}$,
if $M^{d\Tt}_\alpha\ins M^{e\Uu}_\alpha$ or vice versa, then the comparison
terminates.
Otherwise proceed by least disagreement between these models, in the manner described in Remark \ref{rem:superstrong_diff}.
If $E^\Tt_\alpha\neq\emptyset$ set $\exitside^\Tt_\alpha=d$,
and likewise for $\Uu,e$.)\end{casethree}

\begin{casethree}\label{case:B||rho_not_agrees} $S^\Tt_\alpha=\{0,1\}\neq
S^\Uu_\alpha$ and $B^\Tt_\alpha||\rho^{+\Tt}_\alpha\neq
B^\Uu_\alpha||\rho^{+\Tt}_\alpha$.

Select extenders by least disagreement, if possible.
(But maybe $B^\Uu_\alpha\pins
B^\Tt_\alpha||\rho_\alpha^{+\Tt}$,
in which case the comparison terminates.)\end{casethree}

\begin{casethree}\label{case:B||rho_agrees}$S^\Tt_\alpha=\{0,1\}\neq
S^\Uu_\alpha$ and
 $B^\Tt_\alpha||\rho^{+\Tt}_\alpha=B^\Uu_\alpha||\rho^{+\Tt}_\alpha$.

We set $E^\Tt_\alpha=E^\Uu_\alpha=\emptyset$ and move into some model,
according to  subcases.

\begin{scasethree}
 $C^\Tt_\alpha\ins
B^\Uu_\alpha$.

In $\Tt$,  move into $M^\Tt_\alpha$
 (so set $S^\Tt_{\alpha+1}=\{1\}$).
\end{scasethree}

\begin{scasethree}  $C^\Tt_\alpha\nins
B^\Uu_\alpha$.

In $\Tt$,  move into $C^\Tt_\alpha$
 (so set $S^\Tt_{\alpha+1}=\{0\}$).
 \end{scasethree}
\end{casethree}

\begin{casethree}
 $S^\Tt_\alpha\neq\{0,1\}=S^\Uu_\alpha$.

 By symmetry with Cases \ref{case:B||rho_not_agrees} and
\ref{case:B||rho_agrees}.
 \end{casethree}

\begin{casethree}\label{case:finite_gen_hull_proof_case_6} $S^\Tt_\alpha=\{0,1\}=S^\Uu_\alpha$
 and $B^\Tt_\alpha||\rho_+\neq B^\Uu_\alpha||\rho_+$ where $\rho_+=\min(\rho^{+\Tt}_\alpha,\rho^{+\Uu}_\alpha)$.

 Select extenders by least disagreement.
\end{casethree}

This completes all cases.\label{page:successor_detd} The rules for degree-maximality (Definition \ref{dfn:bicephalus_tree}),
along with Remark \ref{rem:padding} regarding padding, now determine $(\Tt,\Uu)\rest(\alpha+2)$. (In particular, if $E^\Tt_\alpha\neq\emptyset$ then
$\pred^\Tt(\alpha+1)$ is the least $\beta$ such that $E^\Tt_\beta\neq\emptyset$ and $\crit(E^\Tt_\alpha)<\nu(E^\Tt_\beta)$.) If $E^\Tt_\alpha=\emptyset$
and we do not move into a model at stage $\alpha$
in $\Tt$, then $S^\Tt_{\alpha+1}=S^\Tt_\alpha$.
If $E^\Tt_\alpha\neq\emptyset$
then $S^\Tt_{\alpha+1}=\sides^\Tt_{\alpha+1}$.
Given $(\Tt,\Uu)\rest\eta$ where $\eta$ is a limit, we extend to $(\Tt,\Uu)\rest(\eta+1)$ using our iteration strategy for $B$.
If $\Tt\rest\eta$ is eventually only padding then $S^\Tt_\eta=\lim_{\alpha<\eta}S^\Tt_\alpha$.
If  $\Tt\rest\eta$
is not eventually only padding
 then $S^\Tt_\eta=\sides^\Tt_\eta$. Likewise
for $\Uu$.

If $\alpha\in\curlyB^\Tt$ but $S^\Tt_\alpha=\{n\}$,
let $\movin^\Tt(\alpha)$ be the largest $\beta<\alpha$
such that $S^\Tt_\beta=\{0,1\}$
(so in $\Tt$
we move into $M^{n\Tt}_\beta=M^{n\Tt}_\alpha$
at stage $\beta$). Likewise for $\Uu$.

\begin{clmthree}\label{clm:proj-finitely_gend_basic_biceph_iter_pres}The following properties hold of $\Tt$, and likewise for $\Uu$:
\begin{enumerate}[label=\arabic*.,ref=\arabic*]\item\label{item:alpha+1_in_B_weakly amenable} If
$\alpha+1\in\curlyB^\Tt$ then $E^\Tt_\alpha$ is weakly amenable to
$B^{*\Tt}_{\alpha+1}$.

\item\label{item:props_B_alpha_pi_alpha_etc} Let $\alpha\in\curlyB^\Tt$. Then:
\begin{enumerate}
\item $\rho_\alpha^\Tt=\rho_{m+1}(C^\Tt_\alpha)=\rho_{m+1}(M^\Tt_\alpha)$
\item $i^\Tt_{0\alpha},j^\Tt_{0\alpha}$ are $p_{m+1}$-preserving
$m$-embeddings,
\item $M^\Tt_\alpha$ is $(m+1)$-solid and $(m+1)$-universal
but not $(m+1)$-sound,
\item $C^\Tt_\alpha=\core_{m+1}(M^\Tt_\alpha)$ is $(m+1)$-sound,
\item letting $\pi_\alpha:C^\Tt_\alpha\to M^\Tt_\alpha$
 be the core map (so $\pi_0=\pi$) and $\beta\leq^\Tt\alpha$,
 \[ \pi_\alpha\com
i^\Tt_{\beta\alpha}=j^\Tt_{\beta\alpha}\com\pi_\beta. \]
\end{enumerate}
\item\label{item:alpha+1_not_in_B} If
$\alpha+1\notin\curlyB^\Tt$ then $E^\Tt_\alpha$
is close to $B^{*\Tt}_{\alpha+1}$.\footnote{If $M$
is type 2 and $k=0$ and $\kappa=\crit(F^M)$
and $\kappa^{+M}=\rho=\rho_1^M$, then we do not yet know
that $F^M$ is close to $C$;
in particular this case causes some obstacles
to be dealt with.} Therefore (even if
extenders applied to bicephali $B^\Tt_\alpha$ are  not close
to $C^\Tt_\alpha$) all iteration maps preserve fine structure
as usual.
\end{enumerate}
\end{clmthree}
\begin{proof}
Part \ref{item:props_B_alpha_pi_alpha_etc} follows
from part \ref{item:alpha+1_in_B_weakly amenable}
and Corollary \ref{cor:basic_fs_pres}.
Part \ref{item:alpha+1_not_in_B} uses
a  simple variant of the
argument of \cite[6.1.5]{fsit}.
\end{proof}

\begin{clmthree}
 We have:
\begin{enumerate}[label=--]
 \item $E^\Tt_\alpha=\emptyset=E^\Uu_\alpha$
iff we move into a model at stage $\alpha$ in some tree.
\item  If $S^\Tt_\alpha=\{0,1\}$
and $E=E^\Tt_\alpha\neq\emptyset$
or $E=E^\Uu_\alpha\neq\emptyset$,
then $\lh(E)<\rho^{+\Tt}_\alpha$.
Likewise for $\Uu$.
\end{enumerate}
Let $\alpha+1,\beta+1<\lh(\Tt,\Uu)$
with
 $E^\Tt_\alpha\neq\emptyset\neq E^\Uu_\beta$.
Let
$\nu=\min(\nu^\Tt_\alpha,\nu^\Uu_\beta)$.
Then:
\begin{enumerate}[label=--]
\item If $\alpha=\beta$ then
$\lh(E^\Tt_\alpha)=\lh(E^\Uu_\alpha)$
and $\nu(E^\Tt_\alpha)=\nu=\nu(E^\Uu_\alpha)$.
\item In general, $E^\Tt_\alpha\rest\nu\neq
E^\Uu_\beta\rest\nu$.
\end{enumerate}
Hence, the comparison
terminates.\end{clmthree}

\begin{proof}This easily follows from the rules of comparison
and as usual.\end{proof}

\begin{clmthree}\label{clm:move_into_C_alpha_not_later_pseg} We have:
\begin{enumerate}[label=\arabic*.,ref=\arabic*]\item\label{item:when_move_into_M^Tt_alpha}
If we move into $M^\Tt_\alpha$ at stage $\alpha$
then $C^\Tt_\alpha\ins B^\Uu_{\alpha+1}$,
so at stage $\alpha+1$ we select the least disagreement
between $M^\Tt_\alpha$ (in $\Tt$)
and $C^\Tt_\alpha$ (in $\Uu$) (which exists).
\item\label{item:when_move_into_C^Tt_alpha} If we move into $C^\Tt_\alpha$ at stage $\alpha$
then for each $\beta>\alpha$ and $n\in S^\Uu_\beta$,
we have
 $C^\Tt_\alpha\nins M^{n\Uu}_\beta$.
 \end{enumerate}
Likewise symmetrically.
\end{clmthree}
\begin{proof}Part \ref{item:when_move_into_M^Tt_alpha} is just a matter of checking the definitions. Regarding part \ref{item:when_move_into_C^Tt_alpha},
consider, for example, the case that
we move into $C^\Tt_\alpha$ in $\Tt$ in Subcase
\ref{scase:C^Tt_alpha_npins_C^Uu_alpha}.
Then because
$\rho^{+\Tt}_\alpha<\rho^{+\Uu}_\alpha$ and either $\rho^{+\Uu}_\alpha$
is a cardinal of $C^\Uu_\alpha$
or $\rho^{+\Uu}_\alpha=\OR(C^\Uu_\alpha)$, and $\rho^{+\Uu}_\alpha$
is a cardinal of $M^\Uu_\alpha$, we have $C^\Tt_\alpha\npins
C^\Uu_\alpha$ and $C^\Tt_\alpha\npins M^\Uu_\alpha$. And
$C^\Tt_\alpha\neq M^\Uu_\alpha$ as $C^\Tt_\alpha$ is $(m+1)$-sound
but $M^\Uu_\alpha$ is not, and $C^\Tt_\alpha\neq C^\Uu_\alpha$
because $\rho_{m+1}^{+C^\Tt_\alpha}\neq\rho_{m+1}^{+C^\Uu_\alpha}$.
So suppose
$\beta>\alpha$
is least such that $C^\Tt_\alpha\ins C^\Uu_\beta$ or $C^\Tt_\alpha\ins
M^\Uu_\beta$.
Then there is $\gamma\in(\alpha,\beta)$ with
$E^\Uu_\gamma\neq\emptyset$,
and note then that $\lh(E^\Uu_\gamma)\geq\rho^{+\Tt}_\alpha$.
But then since $\rho_{m+1}(C^\Tt_\alpha)=\rho^\Tt_\alpha<\lh(E^\Uu_\gamma)$,
we have $C^\Tt_\alpha\npins C^\Uu_\beta$ and $C^\Tt_\alpha\npins M^\Uu_\beta$.
So $C^\Tt_\alpha=C^\Uu_\beta$ or $C^\Tt_\alpha=M^\Uu_\beta$.
But if $\beta\in\curlyB^\Uu$ then
\[\rho_{m+1}(M^\Uu_\beta)=\rho_{m+1}(C^\Uu_\beta)
\geq\lh(E^\Uu_\gamma)>\rho_{m+1}(C^\Tt_\alpha), \]
so $C^\Tt_\alpha\notin\{C^\Uu_\beta,M^\Uu_\beta\}$.
So $\beta\notin\curlyB^\Uu$;
let $\sides^\Uu_\beta=\{n\}$, so $C^\Tt_\alpha=M^{n\Uu}_\beta$.
As $C^\Tt_\alpha$ is $(m+1)$-sound,
 $\deg^{n\Uu}_\beta\geq m+1$ and
$\lh(E^\Uu_\gamma)\leq\rho_{m+1}(M^{n\Uu}_{\beta})$,
a contradiction. The remaining cases are similar.
\end{proof}

Let $\xi+1=\lh(\Tt,\Uu)$.
We say that the comparison \emph{terminates early}
if $\xi=\beta+1$ and $E^\Tt_\beta=\emptyset=E^\Uu_\beta$.

\begin{clmthree}$S^\Tt_\xi\neq\{0,1\}\neq S^\Uu_\xi$.\end{clmthree}
\begin{proof}By the comparison rules,
$S^\Tt_\xi\neq\{0,1\}$ or
$S^\Uu_\xi\neq\{0,1\}$.
Suppose $S^\Tt_\xi=\{0,1\}$ and $S^\Uu_\xi=\{n\}$,
so Case \ref{case:B||rho_not_agrees}
attains at stage $\xi$
and $M^{n\Uu}_\xi\pins
B^\Tt_\xi||\rho^{+\Tt}_\xi$,
so $M^{n\Uu}_\xi$ is fully sound.
Therefore $\xi\in\curlyB^\Uu$
and $n=0$,
so $C^\Uu_\xi\pins B^\Tt_\xi$.
But $C^\Uu_\xi=C^\Uu_\beta$
where $\beta=\movin^\Uu(\xi)$.
This contradicts Claim
\ref{clm:move_into_C_alpha_not_later_pseg}.
\end{proof}

So let $\modelset^\Tt_\xi=\{N^\Tt_\xi\}$
and $\modelset^\Uu_\xi=\{N^\Uu_\xi\}$,
so $N^\Tt_\xi\ins N^\Uu_\xi$ or vice versa.
By Claim \ref{clm:move_into_C_alpha_not_later_pseg},
it follows that
$N^\Tt_\xi,N^\Uu_\xi$
are non-sound and (let) $N=N^\Tt_\xi=N^\Uu_\xi$.

\begin{clmthree}
 Either:
 \begin{enumerate}[label=--]\item $S^\Tt_\xi=\{1\}$ and $b^\Tt$ does not drop in model
 or degree, or
 \item $S^\Uu_\xi=\{1\}$ and $b^\Uu$ does not drop in model or degree.
 \end{enumerate}
\end{clmthree}
\begin{proof}
Standard fine structure shows that
either $b^\Tt$ or $b^\Uu$
doesn't drop in model or
degree. We can assume $b^\Tt$ doesn't.
So suppose $S^\Tt_\xi=\{0\}$.
Then by Claim \ref{clm:move_into_C_alpha_not_later_pseg},
$\xi\notin\curlyB^\Tt$, and hence $N=C^\Tt_\xi$ is $m$-sound,
but not $(m+1)$-sound.
But then if $S^\Uu_\xi=\{0\}$ or $b^\Uu$
drops in model or degree (hence in model), then the core map $\core_{m+1}(N)\to
N$
is an iteration map of both $\Tt$ and $\Uu$, which contradicts
comparison.
\end{proof}

So  we can assume that $S^\treeUu_\xi=\{1\}$ and
$b^\treeUu$ does not drop in model or degree,
so in $\treeUu$ we moved into $M^\treeUu_\lambda$
at stage $\lambda=\movin^\treeUu(\xi)$, and $\lambda\leq^\treeUu\xi$,
and  (let)
\[ k=i^{1\treeUu}_{\lambda\xi}:M^\treeUu_\lambda\to M^\treeUu_\xi=N.\]
By Claim \ref{clm:move_into_C_alpha_not_later_pseg} then,
$\core_{m+1}(N)=C^\treeUu_\lambda\ins B^\treeTt_{\lambda+1}$.
But $C^\treeUu_\lambda\neq N$, so
letting $\sides^\Tt_\xi=\{e\}$ and $\eta+1=\succ^\Tt(\lambda+1,\xi)$, then
\[ j=i^{e*\treeTt}_{\eta+1,\xi}:M^{e*\Tt}_{\eta+1}=C^\treeUu_\lambda\to N=M^{e\Tt}_\xi \]
is the core map, and
$j\neq\id$.
Note $k\com\pi^\treeUu_\lambda=j$
(recall $\pi^\treeUu_\lambda$ is the core map).
We will show that
$\Uu\rest[\lambda,\infty]$ is trivial and $\Tt\rest[\lambda,\infty]$ is essentially a strongly finite tree $\Tt'$ on $C^\Uu_\lambda$ which is almost-above $\rho^\Uu_\lambda$,
with $M^{\Tt'}_\infty=N=M^\Uu_\lambda$ (so $k=\id$ and $j=\pi^\Uu_\lambda$). This will take an argument like the proof of Lemma \ref{lem:measures_in_mice} (on measures in mice). For this we first need to establish the analogue of Lemma \ref{lem:non-dropping-strongly_finite} (on strongly finite trees).

\begin{dfn}\label{dfn:captures_biceph}
Let  $\Ss$ be degree-maximal on $B$,
of successor length $\zeta+1$.
Let $x\in\core_0(M^{e\Tt}_\xi)$ and $\rho'=\rho_{m+1}(M^{e\Tt}_\xi)=\rho_\lambda^\Uu$.
Then
we say that $\Ss$ \dfnemph{captures $(\Tt,x,\rho'+1)$} iff
there is $\sigma$ such that:
\begin{enumerate}[label=\arabic*.,ref=\arabic*]
\item $\lambda\in b^\Ss\cap b^\Tt$,
\item $\Ss\rest(\lambda+1)=\Tt\rest(\lambda+1)$,
 \item either:
 \begin{enumerate}
 \item $\lambda=\zeta=\xi$ (so $\Ss=\Tt$),
 or
 \item we have:
 \begin{enumerate}[label=--]
 \item $\lambda<\xi$ (note then that since $\lambda<^\Tt\xi$,  letting $\eta+1=\succ^\Tt(\lambda,\xi)$, we have
$(\eta+1,\xi]^\Tt\cap\dropset^\Tt_{\deg}=\emptyset$),
\item  $\lambda<\zeta$ and letting
$\bar{\eta}+1=\succ^\Ss(\lambda,\zeta)$,
then $\sides^\Ss_{\bar{\eta}+1}=\{e\}$ and $(\bar{\eta}+1,\zeta]^\Ss\cap\dropset^\Ss_{\deg}=\emptyset$, and
\item $M^{e*\Ss}_{\bar{\eta}+1}=M^{e*{\Tt}}_{\eta+1}$
and $\deg^{e\Ss}_{\eta+1}=m=\deg^{e\Tt}_{\eta+1}$,
\end{enumerate}
\end{enumerate}
\item
$\sigma:M^{e\Ss}_{\zeta}\to M^{e\Tt}_\xi$ is an $m$-embedding with
$\{x\}\un(\rho'+1)\sub\rg(\sigma)$  and $\sigma\com
i^{e*\Ss}_{\bar{\eta}+1,\zeta}=i^{e*\Tt}_{\eta+1,\xi}$ (so $\crit(\sigma)>\rho'$).\qedhere
\end{enumerate}
\end{dfn}

\begin{dfn}[Core-strongly-finite]\label{dfn:core-strongly-finite}
 Let $\Ss$ be degree-maximal on $B$,
 of successor length $\zeta+1$, with $\zeta\notin\mathscr{B}^\Ss$. Let $\sides^\Ss_\zeta=\{e'\}$ and $m'=\deg^\Ss_\zeta$.
We say $\Ss$ is \dfnemph{core-strongly finite} iff for some $\lambda',\zeta\in\OR$ and $n<\om$, we have:
\begin{enumerate}[label=--]
 \item $\lh(\Tt)=\zeta+1$ and $\zeta=\lambda'+n$,
\item $\lambda'<^\Ss\zeta$,
\item either $b^\Ss$ drops in model or $e'=0$,
\item  $\core_{m'+1}(M^{e'\Ss}_{\zeta})\ins M^{e'\Ss}_{\lambda'}$,
\item $\core_{\ds}(E^\Ss_\gamma)$ is finitely
generated  for all $\gamma\leq^\Ss_{\Da}\alpha$,
for all $\alpha$ with
$\lambda'<^\Ss\alpha+1\leq^\Ss\zeta$.\footnote{\label{ftn:b^S_geq_lambda'_Dodd-nice}Note
that by the other conditions, $b^\Ss$ is $\geq\lambda'$-Dodd-nice,
so $\alpha$ is Dodd-nice (see Definition \ref{dfn:Dodd-nice}).}
 \qedhere
\end{enumerate}
\end{dfn}

\begin{dfn}
Let $\Ss$ be degree-maximal on $B$.
Let $\chi<\chi'<\lh(\Ss)$
and $d\in\sides^\Ss_{\chi'}$.
We say that $\chi$ is \dfnemph{$(d,\chi')$-transient \tu{(}with respect to $\Ss$\tu{)}}
iff for some $\eta$, we have:
\begin{enumerate}[label=--]\item $M^{d\Ss}_{\chi'}$ is active type 2,
\item  $\chi=\pred^\Ss(\eta+1)<^\Ss\eta+1\leq^\Ss\chi'$,
\item $(\eta+1,\chi']^\Ss$ does not drop (so $M^{d*\Ss}_{\eta+1}$ is active type 2 and $i^{d*\Ss}_{\eta+1,\chi'}:M^{d*\Ss}_{\eta+1}\to M^{d\Ss}_{\chi'}$) and
 \item $\crit(i^{d*\Ss}_{\eta+1,\chi'})=\lgcd(M^{d*\Ss}_{\eta+1})$.
 \end{enumerate}

 Note that if  $\chi$ is $(d,\chi')$-transient then, with notation as above,
$\sides^\Ss_{\chi'}=\sides^\Ss_{\eta+1}=\{d\}$
and $E^\Ss_\chi=F(M^{d*\Ss}_{\eta+1})$. In particular,
$d$ is uniquely determined by $\chi'$. We say \dfnemph{$\chi$ is $\chi'$-transient \tu{(}for $\Ss$\tu{)}} iff there is $d$ such that $\chi$ is $(d,\chi')$-transient.
\end{dfn}

We now state and prove the existence of core-strongly-finite trees on bicephali, capturing the kind of information which we need for the present argument:

\begin{clmthree}\label{clm:core-strongly-finite}
Let $\rho'=\rho_\lambda^\Uu$ and
$x\in\core_0(M^{e\Tt}_\xi)$.
Then
there is a degree-maximal tree $\Ss$ on $B$  such that:
\begin{enumerate}[label=\tu{(}\roman*\tu{)}]

\item\label{item:Tt_captures_Uu_x_rho} $\Ss$ captures $(\Tt,x,\rho'+1)$,

\item\label{item:b_geq_lambda_Dodd-nice}  $b^\Ss$ is
$\geq\lambda$-Dodd-nice, and
 \item\label{item:Tt_strongly_finite} $\Ss$ is core-strongly finite.
\end{enumerate}
\end{clmthree}
\begin{proof}
We will first observe that there
is $\Ss$ satisfying conditions \ref{item:Tt_captures_Uu_x_rho}
and \ref{item:b_geq_lambda_Dodd-nice}. By then minimizing as in the proof of Lemma \ref{lem:non-dropping-strongly_finite}, we will see that condition \ref{item:Tt_strongly_finite} is also satisfied.
 \begin{sclmthree}\label{sclm:tree_satisfying_all_but_core-strongly_finite}
There is $\Ss$  satisfying requirements
\ref{item:Tt_captures_Uu_x_rho} and \ref{item:b_geq_lambda_Dodd-nice}.\end{sclmthree}
\begin{proof}
 By Lemma \ref{lem:simple_embedding_exists},
 we can fix $\Ss,\zeta,n,\Phi$
 such that
 $n<\om$, $\lh(\Ss)=\zeta+1=\lambda+n+1$,
 $\Phi:\Ss\hookrightarrow_{\lambda\text{-}\simple}\Tt$ (see \ref{dfn:lambda-simple_embedding}), and
 $\Ss$ captures $(\Tt,x,\rho'+1)$.

So it suffices to see that $b^\Ss$ is  ${\geq\lambda}$-Dodd-nice, and since  $\Phi:\Ss\hookrightarrow_{\lambda\text{-}\simple}\Tt$,
 for this it is easily enough to see that $b^\Tt$ is ${\geq\lambda}$-Dodd-nice. But since
$C$ is Dodd-sound, the only extenders $E^\Tt_\alpha$ which might violate this are those which are an image of $F^M$ (that is, $\exitside^\Tt_\alpha=1$, $(0,\alpha]^\Tt\cap\dropset^\Tt=\emptyset$ and $E^\Tt_\alpha=F(M^{1\Tt}_\alpha)$). So suppose  $\alpha+1\in(\lambda,\infty]^\Tt$
and $\alpha$ has this form. Letting $\beta\leq^\Tt\alpha$ be least such that either $\beta=\alpha$ or $\crit(i^{1\Tt}_{\beta\alpha})>\crit(F(M^{1\Tt}_\beta))$,
note that $\beta=\pred^\Tt(\alpha+1)$
and that $(0,\alpha+1]^\Tt\cap\dropset^\Tt=\emptyset$. But since  $\lambda<^\Tt\alpha+1$
and recalling properties of $\Tt$, it follows that
 $b^\Tt\cap\dropset^\Tt=\emptyset$ and $\sides^\Tt_\infty=\{0\}$, so $e=0$. But $1\in\sides^\Tt_\alpha\sub\sides^\Tt_\beta$,
and since $\sides^\Tt_{\alpha+1}=\{0\}$
and $\pred^\Tt(\alpha+1)=\beta$, therefore
$\beta\in\mathscr{B}^\Tt$ and $\exitside^\Tt_\beta=0$
and $\rho(B^\Tt_\beta)\leq\crit(E^\Tt_\alpha)<\nu(E^\Tt_\beta)$.
Since $\exitside^\Tt_\alpha=1$, therefore $\beta<\alpha$.
Let $\vareps$ be least such that $\vareps+1\leq^\Tt\alpha$ and $\vareps\geq\beta$.
Then $\crit(E^\Tt_\vareps)>\crit(E^\Tt_\alpha)=\crit(F(M^{1\Tt}_\beta))$, by choice of $\beta$.  So $\pred^\Tt(\vareps+1)\geq\pred^\Tt(\alpha+1)=\beta$.
But by choice of $\vareps$,
$\pred^\Tt(\vareps+1)\leq\beta$.
So $\pred^\Tt(\vareps+1)=\beta$.
But since $\crit(E^\Tt_\vareps)>\crit(E^\Tt_\alpha)\geq\rho(B^\Tt_\beta)$ and $\exitside^\Tt_\beta=0$,
therefore $\sides^\Tt_{\vareps+1}=\{0\}$,
so since $1\in\sides^\Tt_{\alpha}$,
we can't have $\vareps+1\leq^\Tt\alpha$, a contradiction.
\end{proof}

Call a tree $\Ss$ as in the statement of Subclaim \ref{sclm:tree_satisfying_all_but_core-strongly_finite}
a \emph{candidate}. Given a candidate $\Ss$, adopting the notation above,  define the \emph{index} of $\Ss$ as in the proof of Lemma \ref{lem:non-dropping-strongly_finite},
except that we use
the set $A$  of ordinals
$\beta\leq^\Ss_{\Da}\alpha$ for some $\alpha+1\in(\lambda,\zeta]^\Ss$
(hence $\beta\geq\lambda$; note
$\crit(E^\Ss_\alpha)\geq\rho'$).

\begin{sclmthree}\label{sclm:Tt_strongly_finite}
 $\Ss$ is core-strongly finite.
\end{sclmthree}
\begin{proof}
Suppose not. We construct a candidate $\Ssbar$ with smaller index than $\Ss$, a
contradiction.  Let $\left<\gamma_i\right>_{i<\ell}$ be the index of $\Ss$ and $\kappa_i,\beta_i$ be as usual, so $A=\{\beta_i\bigm|i<\ell\}$
and $\kappa_i=\crit(E^\Ss_{\beta_i})$ with $\kappa_i>\kappa_{i+1}$ for $i+1<\ell$. Let
$a<\ell$ be least such that $\core_{\ds}(E^\Ss_{\beta_a})$
is not finitely
generated. Let $\kappa=\kappa_a$, $\beta=\beta_a$,
$Q=\core_\ds(\exit^\Ss_{\beta})$ and
$F=F^Q$. So
$\kappa^{+Q}<\sigma=\sigma_F=\tau_F$.
Let $\theta$ be as in the proof of Claim \ref{clm:Tt_strongly_finite_mim} of Lemma \ref{lem:non-dropping-strongly_finite},
so $\theta$ is least such that
$\sigma<\lh(E^\Ss_\theta)$;
also
$Q\ins M^\Ss_\theta$.
Now $\lambda\leq\theta$. For since $\sigma$ is a cardinal of $\exit^\Ss_\theta$,
we have $\sigma\leq\nu(E^\Ss_\theta)<\lh(E^\Ss_\theta)$.
And $\kappa_{\ell-1}=\crit(E^\Ss_{\beta_{\ell-1}})$
and $\pred^\Ss(\beta_{\ell-1})=\lambda$.
But $\kappa_{\ell-1}\leq\kappa<\sigma\leq\nu(E^\Ss_\theta)$,
so by normality, $\lambda=\pred^\Ss(\beta_{\ell-1}+1)\leq\theta$.

Let $\thetabar$ be least such that
$\lh(E^\Ss_\thetabar)>\kappa^{+Q}$; much as above,
$\lambda\leq\thetabar\leq\theta$.
Let $\thetabar+h=\theta$ (so $h<\om$) and
$\zetabar+h=\zeta$.
Let $\chi$ be least such that $\chi\geq\theta$ and $\chi+1\leq^\Tt\zeta$.
Let $\bar{\chi}+h=\chi$.
Our
plan
is to select some
$g\in\nu_F^{<\om}$
and $R\pins M^\Tt_\thetabar$ with $F^R\approx F\rest
g$, and such that we can define a degree-maximal
 tree $\Ssbar$ on $B$ such that:
\begin{enumerate}[label=--]
 \item   $\Ssbar\rest\thetabar+1=\Ss\rest\thetabar+1$ and
  $\lh(\Ssbar)=\zetabar+1$.
  \item
  $\lambda<^{\Ssbar}\zetabar$
  and letting  $\bar{\gamma}+1=\succ^{\bar{\Ss}}(\lambda,\zetabar)$, we have:
  \begin{enumerate}[label=--]\item
 $\sides^{\Xxbar}_{\gammabar+1}=\sides^{\Ssbar}_{\zetabar}=\{e\}$,\item
 $(\bar{\gamma}+1,\bar{\zeta}]^{\bar{\Ss}}\cap\dropset^{\bar{\Ss}}_{\deg}=\emptyset$ and
 $\deg^{\Ssbar}_{\gammabar+1}=\deg^{\Ssbar}_{\zetabar}=m$, and
 \item $M^{e*\bar{\Ss}}_{\bar{\gamma}+1}=\core_{m+1}(M^{e\Ssbar}_{\zetabar})=\core_{m+1}(M^{e\Ss}_{\zeta})=M^{e*\Ss}_{\gamma+1}$.
 \end{enumerate}
\item $b^{\Ssbar}$
is ${\geq\lambda}$-Dodd-nice.
 \item Replacing the role of $Q$ in
$\Ss$ with $R$
in $\Ssbar$, we perform a kind of reverse copy construction,
 much like in the proof of Lemma \ref{lem:non-dropping-strongly_finite},
so that
$\Ss\rest[\theta,\zeta]$ will be a ``copy'' of
$\Ssbar\rest[\thetabar,\zetabar]$.
Moreover, $\bar{\chi}$ is least such that $\bar{\chi}\geq\thetabar$ and $\bar{\chi}+1\leq^{\Ssbar}\zetabar$,
and for $\alpha\in[\bar{\chi}+1,\zetabar]^\Ssbar$, the copying process will yield a copy map
\[ \pi^e_{\alpha}:M^{e\bar{\Ss}}_{\alpha}\to M^{e\Ss}_{\alpha+h} \]
with $\crit(\pi^e_\alpha)>\kappa^{+R}=\kappa^{+Q}$.
\item
The final copy map
$\pi^e_{\bar{\zeta}}:M^{e\Ssbar}_{\zetabar}\to M^{e\Ss}_\zeta$ is an $m$-embedding
with
$\pi^e_{\zetabar}\com j^\Ssbar=j^\Ss$
where $j^{\Ssbar}:\core_{m+1}(M^{e\Ssbar}_{\zetabar})\to
M^{e\Ssbar}_{\zetabar}$ is the iteration map and likewise $j^\Ss$
(so $\dom(j^{\bar{\Ss}})=\dom(j^\Ss)$,
by  earlier points).
\item
$\tau^{-1}(x)\in\rg(\pi^e_{\bar{\zeta}})$,
where $\tau:M^{e\Ss}_{\zeta}\to M^{e\Tt}_\xi$
witnesses that $\Ss$ captures $(\Tt,x,\rho'+1)$;
it will follow that
$\tau\com\pi^e_{\bar{\zeta}}:M^{e\Ssbar}_{\zetabar}\to M^{e\Tt}_\xi$ witnesses that $\Ssbar$ captures $(\Tt,x,\rho'+1)$.
\end{enumerate}
Therefore $\Ssbar$ will be a candidate.

We will select $g,R,Q'$ and build $\Ssbar$  as in the proof of Claim \ref{clm:Tt_strongly_finite_mim} of Lemma \ref{lem:non-dropping-strongly_finite}. The following claim helps us see that the copying can be executed without problems, and is the analogue of Subclaim \ref{sclm:Tt_structure_mim} from there:

\begin{ssclmthree}\label{ssclm:Tt_structure} Let $\chi\in[\theta,\zeta)$.
Then $\crit(E^\Ss_\chi)\notin(\kappa,\sigma)$, so
$\pred^\Ss(\chi+1)\notin(\thetabar,\theta)$. In fact,  one of the
following options
holds:
\begin{enumerate}[label=\tu{(}\roman*\tu{)}]
\item\label{item:feeds_in_or_feeds_in_extra} there are
$\vareps,\chi'$ such that  $\lambda<^\Ss \eps+1\leq^\Ss\zeta$ and
$\chi'\leq^\Ss_{\Da}\eps$, and either:
\begin{enumerate}
\item\label{item:feeds_in} $\chi=\chi'$, or
\item\label{item:feeds_in_extra} $\chi<^\Ss\chi'$
and $\chi$ is $\chi'$-transient
and $E^\Ss_{\chi'}=F(M^{d\Ss}_{\chi'})$, where $\sides^\Ss_{\chi'}=\{d\}$,
\end{enumerate}
or
 \item\label{item:branch_extra}
 $\lambda\leq^\Ss\chi<^\Ss\xi$ and  $\chi$ is $\xi$-transient.
\end{enumerate}
Moreover, the three options \ref{item:feeds_in},
\ref{item:feeds_in_extra}
and \ref{item:branch_extra}
are mutually exclusive.
\end{ssclmthree}
\begin{proof}
This is just a slight variant of the proof of Subclaim \ref{sclm:Tt_structure_mim} within the proof of Lemma \ref{lem:non-dropping-strongly_finite}.
\end{proof}
Recall here that if \ref{item:feeds_in_extra}
holds and $\eta+1=\succ^\Ss(\chi,\chi')$, then $E^\Ss_{\chi}=F(M^{d*\Ss}_{\eta+1})$,
so $\crit(E^\Ss_\chi)=\crit(E^\Ss_{\chi'})$. Likewise, if \ref{item:branch_extra}
holds and $\eta+1=\succ^\Ss(\chi,\xi)$, then $E^\Ss_\chi=F(M^{e*\Ss}_{\eta+1})$, so $\crit(E^\Ss_\chi)=\crit(F(M^{e\Ss}_\xi))$.

 By Subsubclaim \ref{ssclm:Tt_structure},
$\Ss\rest[\theta,\zeta]$
can be viewed as a tree $\Ss'$ on the ``phalanx'' $\Phi(\Ss\rest(\bar{\theta}+1))\conc\left<(Q,0)\right>$, where the nodes of $\Phi(\Ss\rest(\bar{\theta}+1))$ index either a single model or a bicephalus, corresponding to $\Ss\rest(\bar{\theta}+1)$, but $Q$ is just a single model, irrespective of what is indexed at $\theta$ in $\Ss$,
and in $\Ss'$, if $\crit(E^{\Ss'}_\alpha)\leq\kappa$ then $\pred^{\Ss'}(\alpha+1)\leq\bar{\theta}$
(so $E^{\Ss'}_\alpha$ applies to a model or bicephalus of $\Ss\rest(\bar{\theta}+1)$),
whereas if $\crit(E^{\Ss'}_\alpha)>\kappa$ then $\pred^{\Ss'}(\alpha+1)\geq\bar{\theta}+1$,
and if $\pred^{\Ss'}(\alpha+1)=\bar{\theta}+1$ then
we take $(M^{*\Ss'}_{\alpha+1},\deg^{\Ss'}_{\alpha+1})\ins (Q,0)$
as large as possible.
In fact, by the subclaim,
if $\crit(E^{\Ss'}_\alpha)>\kappa$ then $\crit(E^{\Ss'}_\alpha)\geq\sigma$, and recall that $\exit^\Ss_\theta\ins Q$ and $\rho_1^Q\leq\sigma$,
so having $(M^{*\Ss'}_{\alpha+1},\deg^{\Ss'}_{\alpha+1})\ins (Q,0)$
in case $\pred^{\Ss'}(\alpha+1)=\bar{\theta}$ corresponds with what happens in $\Ss$. In particular, if $\alpha\geq\theta$ and $\crit(E^\Ss_\alpha)>\kappa$
and $\pred^\Ss(\alpha+1)=\theta$
then $(M^{e_\theta*\Ss}_{\alpha+1},\deg^{e_\theta\Ss}_{\alpha+1})\ins (Q,0)$ where $e_\theta=\exitside^\Ss_\theta$.

With these observations,
the construction of $Q',g,R,\Ssbar$
is just a slight variant of that in the proof of Claim \ref{clm:Tt_strongly_finite_mim} of Lemma \ref{lem:non-dropping-strongly_finite}.

The index of $\Ssbar$ is less than that of $\Ss$,
unreasonable,
establishing the subclaim.
\end{proof}
And hence that of the claim.
\end{proof}

\begin{clmthree}\label{clm:core-strongly-finite_transient_exts}
 Let $\Ss$ be core-strongly finite,
 $\lh(\Ss)=\zeta+1$,
 $\sides^\Ss_\zeta=\{e'\}$ and $m'=\deg^\Ss_\zeta$. Let $\lambda'<^\Ss\zeta$
 be such that $\core_{m'+1}(M^{e'\Ss}_\zeta)\ins M^{e'\Ss}_{\lambda'}$. Then for every $\chi\in[\lambda',\zeta)$,
one of the following options holds:

\begin{enumerate}[label=\tu{(}\roman*\tu{)}]
\item\label{item:feeds_in_or_feeds_in_extra_Tt_structure} there are
$\vareps,\chi'$ such that  $\lambda<^\Ss \eps+1\leq^\Ss\zeta$ and
$\chi'\leq^\Ss_{\Da}\eps$, and either:
\begin{enumerate}
\item\label{item:feeds_in_Tt_structure} $\chi=\chi'$, or
\item\label{item:feeds_in_extra_Tt_structure} $\chi<^\Ss\chi'$
and $\chi$ is $\chi'$-transient
and $E^\Ss_{\chi'}=F(M^{d\Ss}_{\chi'})$, where $\sides^\Ss_{\chi'}=\{d\}$,
\end{enumerate}
or
 \item\label{item:branch_extra_Tt_structure}
 $\lambda'\leq^\Ss\chi<^\Ss\zeta$ and  $\chi$ is $\zeta$-transient.
\end{enumerate}
Moreover, the three options \ref{item:feeds_in_Tt_structure},
\ref{item:feeds_in_extra_Tt_structure}
and \ref{item:branch_extra_Tt_structure}
are mutually exclusive.
\end{clmthree}
\begin{proof}See the proof of Subsubclaim \ref{ssclm:Tt_structure} of Subclaim \ref{sclm:Tt_strongly_finite} of Claim \ref{clm:core-strongly-finite}.
\end{proof}

\begin{clmthree}\label{clm:Uu_core-sf}
$\treeTt$ is core-strongly-finite
and
$E^\treeUu_\gamma=\emptyset$ for all $\gamma\geq\lambda$,
so
$M^\treeUu_\alpha=M^\treeUu_\lambda$
and $k=\id$.
\end{clmthree}
\begin{proof}The proof is like that of
 the analysis of measures in mice, Lemma \ref{lem:measures_in_mice}. Fix $x\in M$ with
$M=\Hull_{m+1}^M(\rho_{m+1}^M\cup\{x\})$.
By Claim \ref{clm:core-strongly-finite},
we can fix  a degree-maximal tree $\bar{\treeTt}$ on $B$
such that $\Ttbar$ captures $(\treeTt,i^{1\treeUu}_{0\alpha}(x),\rho^\treeUu_\lambda+1)$,  $b^{\Ttbar}$
 is $\geq\lambda$-Dodd-nice, and
  $\bar{\treeTt}$ is core-strongly-finite. Then
$\bar{\treeTt}\rest(\lambda+1)=\treeTt\rest(\lambda+1)$,
 $\sides^{\bar{\treeTt}}_\infty=\sides^\treeTt_\infty=\{e\}$,
and by the proof of Subclaim \ref{sclm:tree_satisfying_all_but_core-strongly_finite}
of Claim \ref{clm:core-strongly-finite},
$b^\treeTt$ is also $\geq\lambda$-Dodd-nice.

Let
$\bar{j}=C^\treeUu_\lambda\to\bar{N}=M^{e\bar{\treeTt}}
_\infty$
and $j:C^{\treeUu}_\lambda\to N=M^{e\Tt}_\infty$ be the $(m+1)$-core maps, which are also the tail iteration maps given by $\Ttbar$ and $\Tt$ respectively.
Let $\somevarpi:\bar{N}\to N$  witness the capturing, so $\somevarpi$ is an $m$-embedding with
$\somevarpi\com \jbar=j$,
 $i^{1\treeUu}_{0\alpha}(x)\in\rg(\somevarpi)$
and $\crit(\somevarpi)>\rho^\treeUu_\lambda$.
Define $\bar{k}=\somevarpi^{-1}\com k$.
Then the usual diagram commutes
(Figure \ref{fgr:proj_fin_comm_diag}).

\begin{figure}
\centering
\begin{tikzpicture}
 [mymatrix/.style={
    matrix of math nodes,
    row sep=0.35cm,
    column sep=0.4cm}]
   \matrix(m)[mymatrix]{
  M^\treeUu_\lambda&          {} &       {}& {} & N& \\
   {} & {} \\
   {} & {} & \bar{N}\\
   \\
 C^\treeUu_\lambda\\};
 \path[->,font=\scriptsize]
%maps from left $M$
(m-5-1) edge[bend right] node[below] {$\ j$} (m-1-5)
(m-5-1) edge node[left] {$\pi^\treeUu_\lambda$} (m-1-1)
(m-1-1)edge node[above] {$\ \ \bar{k}$} (m-3-3)
(m-5-1)edge node[above] {$\bar{j}$} (m-3-3)
(m-3-3)edge node[above] {$\somevarpi$} (m-1-5)
(m-1-1) edge node[above] {$k=i^{1\treeUu}_{\lambda+1,\alpha}$} (m-1-5);
\end{tikzpicture}
\caption{The diagram commutes.}
\label{fgr:proj_fin_comm_diag}
\end{figure}

It is now routine to execute the argument from
the proof of Theorem \ref{thm:measures_in_mice}
(which also appeared in the proof of super-Dodd-solidity).
The only formal difference is that here the trees are only
being analyzed above $\lambda$, and also that the core of interest,
$C^\treeUu_\lambda$, might be a proper segment of $M^{e\treeTt}_\lambda=M^{e\bar{\Tt}}_\lambda$,
in which case the first extenders used along $(\lambda,\infty]^{\bar{\treeTt}}$
and $(\lambda,\infty]^\treeTt$ cause a drop in model to $C^\treeUu_\lambda$;
similarly, if $(0,\lambda]^\Tt\cap\dropset^\Tt\neq\emptyset$
and $\deg^\Tt_\lambda>m$,
then they could cause a drop in degree to $m$, without a drop in model.
But these details have no significant impact. We leave the execution to the
reader.
\end{proof}

The remainder of the proof involves a slight split into two cases, as follows:
\begin{enumerate}[label=(\alph*)]\item\label{item:need_phalanxes}
 $C,M$ are active type 2, $m=0$ and $\kappa=\crit(F^{C})=\crit(F^M)<\rho$; and
 \item\label{item:no_phalanxes} otherwise.
 \end{enumerate}

 Let $B'=(C',M',\rho')=B^\treeUu_\lambda$.
 Suppose case \ref{item:need_phalanxes} holds. Then we define $\bar{C}=C|\kappa^{+C}$ and the phalanx \[ \ph=((\bar{C},{\kappa}),C,\rho),\]
 and likewise,
 letting $\kappa'=\crit(F^{C'})=\crit(F^{M'})<\rho'$,
 let $\bar{C}'=C'|(\kappa')^{+C'}$ and
 $\ph'=((\bar{C}',{\kappa'}),C',\rho')$.
 Say a $(0,0)$-maximal tree $\Vv$ on $\ph$ is \emph{relevant} if
 for all $\alpha+1<\lh(\Vv)$,
 either $\crit(E^\Vv_\alpha)\geq\rho$ or $\crit(E^\Vv_\alpha)=\kappa$ (and as usual,
 $\rho\leq\lh(E^\Vv_0)$, so in fact $\rho<\lh(E^\Vv_0)$).
 Likewise for $\Vv',\ph',\kappa',\rho'$.
 Note that $\ph$ is $(0,\om_1+1)$-iterable with respect to relevant trees, as is $\ph'$;
 for the case of $\ph'$,
 this is because relevant trees on $\ph'$
 can be lifted to normal continuations of $\Uu\rest(\lambda+1)$.
In case \ref{item:no_phalanxes}, we don't need to define $\ph,\ph'$.

Let $(\Vv',\Ww')$
be the initial segment of the $m$-maximal comparison of
\begin{enumerate}[label=(\roman*)]\item $(\ph',M')$, in case \ref{item:need_phalanxes} above holds,
\item $(C',M')$ otherwise,
\end{enumerate}
through length $\theta'$,
where $\theta'\in[1,\om]$ is largest such that $E^{\Vv'}_\gamma\neq\emptyset=E^{\Ww'}_\gamma$ for all $\gamma+1<\theta'$
(so iterability of $M'$ is irrelevant, by choice of $\theta'$).

\begin{clmthree}\label{clm:(V',W')_structure}
$(\Vv',\Ww')$ is finite,
$\Ww'$ (on $M'$) is trivial,
$\Vv'$ (on $\ph'$ or $C'$)
is finite and
 uses the same extenders as does $\Tt\rest[\lambda,\xi]$, and  $M^{\Vv'}_{\theta'-1}=M^{e\Tt}_\infty=M'$.
\end{clmthree}
\begin{proof}
The tree $\Vv'$ is just a simple reorganization of $\Tt\rest[\lambda,\xi]$, and the claim follows easily from Claims \ref{clm:core-strongly-finite_transient_exts} and \ref{clm:Uu_core-sf}. That is, recall that $E^\Tt_{\lambda}=\emptyset$ (as in $\Uu$, we moved into $M^{1\Uu}_\lambda$ at stage $\lambda$),
$\lh(\Tt)=\xi+1=\lambda+n+1$ and
$n>1$ (as $E^{\Tt}_\lambda=\emptyset$
but $E^\Tt_{\lambda+1}\neq\emptyset$),
and $E^\Tt_{\lambda+i+1}\neq\emptyset$ for all $i+1<n$. So we want to see that  $E^{\Vv'}_i=E^{\Tt}_{\lambda+1+i}$
for all $i+1<n$,
and $M^{\Vv'}_{n-1}=M^{e\Tt}_{\lambda+n}=M'$.
Clearly $E^{\Vv'}_0=E^\Tt_{\lambda+1}$,
since $M^{\Vv'}_0=C^\Uu_\lambda\ins M^{d\Tt}_{\lambda+1}$
 where $\exitside^\Tt_{\lambda+1}=d$.
So the point is that  given $i$ with $1< i+1\leq n$, if it is not the case that
\[ \sides^{\Tt}_{\lambda+i+1}=\{e\}\text{ and }M^{\Vv'}_{i}=M^{e\Tt}_{\lambda+i+1}\]
then $\crit(E^\Tt_{\lambda+i})<\rho'$, so by
 Claim \ref{clm:core-strongly-finite_transient_exts},  $\lambda+i$ is $\xi$-transient for $\Tt$. Since $M'=M^{e\Tt}_\xi$, it follows that case \ref{item:need_phalanxes} holds,
 $\lambda+1\leq^\Tt\lambda+i<^\Tt\xi$
and
$E=E^{\Vv'}_{i-1}=E^\Tt_{\lambda+i}$ is an image of $F^{C'}$
with $\kappa'=\crit(E)=\crit(F^{C'})<\rho'$. If $i>1$ then let $\mu$ be \[ \mu=\crit(i^{e\Tt}_{\lambda+i,\xi})=\lgcd(M^{e\Tt}_{\lambda+i}),\]
and if $i=1$ then let $\mu$ be \[ \mu=\crit(j)=\lgcd(C').
\]
Since  \ref{item:need_phalanxes} holds
and $\kappa'=\crit(E)$, we have $M^{\Vv'}_i=\Ult(\bar{C}',E)$.
So note for each $d\in\sides^\Tt_{\lambda+i+1}$, we have \begin{equation}\label{eqn:model_through_mu^++}
 M^{d\Tt}_{\lambda+i+1}||\mu^{++M^{d\Tt}_{\lambda+i+1}}= M^{\Vv'}_i||\mu^{++M^{\Vv'}_i}. \end{equation}

But letting $\gamma+1=\succ^\Tt(\lambda+i,\xi)$,
$E^\Tt_\gamma$
(which is finitely generated and Dodd-nice, with critical point $\mu$)
is produced by a finite iteration of
the model in line (\ref{eqn:model_through_mu^++}),
and either
 $E^{\Vv'}_i=E^{\Tt}_{\lambda+i+1}$ is
 the first extender in that iteration, or $E^\Tt_\gamma$ is Dodd-sound and
$E^{\Vv'}_i=E^{\Tt}_{\lambda+i+1}=E^\Tt_\gamma$. The claim easily follows from these considerations.
\end{proof}

Let $(\Vv,\Ww)$
be the initial segment of the $m$-maximal comparison of
\begin{enumerate}[label=(\roman*)]\item $(\ph,M)$, in case \ref{item:need_phalanxes} above holds,
\item $(C,M)$ otherwise,
\end{enumerate}
through length $\theta$,
where $\theta\in[1,\om]$ is largest such that $E^{\Vv}_\gamma\neq\emptyset=E^{\Ww}_\gamma$ for all $\gamma+1<\theta$.

Let $G$ be the $(\rho,\rho')$-extender over $C$ derived from
$i_{CC'}=i^{0\treeUu}_{0\lambda}$, or equivalently,
from $i_{MM'}=i^{1\treeUu}_{0\lambda}$. Let $\bar{G}$ be the restriction of $G$ to a $(\kappa,\kappa')$-extender.
So $M^{\Vv'}_{-1}=\bar{C}'=\Ult(\bar{C},\bar{G})$, $M^{\Vv'}_0=C'=\Ult_m(C,G)$ and
$M^{\Ww'}_0=M'=\Ult_m(M,G)$,
and the iteration maps $i^\Uu_{CC'}:C\to C'$ and $i^\Uu_{MM'}:M\to M'$
are just the associated ultrapower maps.

The following claim describes the relationship between $(\Vv,\Ww)$ and $(\Vv',\Ww')$:
\begin{clmthree}\label{clm:comparison_preservation}
 For all $\gamma<\min(\theta,\theta')$, we have:
 \begin{enumerate}[label=\arabic*.,ref=\arabic*]
  \item \label{item:degrees_match}
$\deg^{\Vv'}_\gamma=\deg^\Vv_\gamma$; let $d=\deg^\Vv_\gamma$,
\item\label{item:models_match} $M^{\Vv'}_{\gamma}=\Ult_d(M^\Vv_\gamma,G)$,
 \item\label{item:exit_extenders_match} if $\gamma+1<\theta'$ then
$\exit^{\Vv'}_\gamma=\Ult_0(\exit^\Vv_\gamma,G)\neq\emptyset$ is active,
 \item\label{item:no_exit_in_Ww-side}
if $\gamma+1<\theta'$ then $E^{\Ww'}_\gamma=E^{\Ww}_\gamma=\emptyset$,
\item\label{item:if_gamma+1=theta'_then_M'=M^Vv'_gamma} if $\gamma+1=\theta'$ then
$M=M^\Vv_\gamma$ and $M'=M^{\Vv'}_\gamma$.
 \end{enumerate}
\end{clmthree}

\begin{proof}
For this, we use calculations with commutativity of ultrapowers, and with  condensation,
like those used in the analyses of comparisons
in \cite[\S4]{premouse_inheriting}.
However, we need a more general
form, like that used in full normalization.
A key component of this is Lemma \ref{lem:dropdown_lifting},
and the main idea for that was due to Steel.
 The combination of methods
also relates to the  analysis of the interaction
between comparison and full normalization, which
%***check
is due to the author; see \cite[\S8]{fullnorm}.
\footnote{The proof
being
presented for Lemma \ref{lem:finite_gen_hull}
differs in approach with what I had in mind
during the
presentation of (most of) the results of this paper at the 2015 M\"unster
conference
in inner model theory; that approach did not depend on Lemma \ref{lem:dropdown_lifting} or any  full normalization calculations. Eventually I noticed  that the present argument is
simpler and more direct than the earlier approach, so I adopted it.}

First consider $\gamma=0$.
In this case, part \ref{item:degrees_match} is by definition (and $d=m$),
and part \ref{item:models_match}
just says that $C'=\Ult_d(C,G)$, which is also by definition.
We also know that $M'=\Ult_d(M,G)$
and that, in case \ref{item:need_phalanxes}, $\bar{C}'=\Ult_0(\bar{C},\bar{G})$.

Now fix $\gamma<\min(\theta,\theta')$ and suppose that all parts hold at all $\gamma'<\gamma$,
 and parts \ref{item:degrees_match}
and \ref{item:models_match}
hold at $\gamma$.

Since $E^\Ww_{\gamma'}=\emptyset$
for all $\gamma'<\gamma$,
we have $M^{\Ww}_\gamma=M$.
Let $\eta$ be least such that
either $M^\Vv_\gamma|\eta\neq
M|\eta$
or $\eta=\min(\OR^{M^\Vv_\gamma},\OR^M)$.
Note here that
if $M^\Vv_\gamma|\eta=M|\eta$
then in fact $M^\Vv_\gamma\ins M$, since $M$ is not sound.

Note that Lemma
\ref{lem:dropdown_lifting}
applies to $(M,M',G,m)$ and also to $(M^\Vv_\gamma,M^{\Vv'}_\gamma,G,d)$,
where $d=\deg^\Vv_\gamma=\deg^{\Vv'}_\gamma$.
Therefore,
\[ U_0=\Ult_0(M^\Vv_\gamma|\eta,G)\ins M^{\Vv'}_\gamma, \]
\[ U_1=\Ult_0(M|\eta,G)\ins M'.\]
Since $M^\Vv_\gamma||\eta=M||\eta$, we get
$U_0^\passive=U_1^\passive$,
and the ultrapower maps agree.

Suppose
 $M^\Vv_\gamma|\eta\neq M|\eta$.
Since the ultrapower maps agree, the distinction
between the active extenders of these models
lifts to give $F^{U_0}\neq F^{U_1}$.
So $U_0,U_1$ constitute the least disagreement between
$M^{\Vv'}_\gamma$ and $M^{\Ww'}_\gamma$.
Therefore $\gamma+1<\theta'$.
But $\Ww'$ is trivial,
so either
\begin{enumerate}[label=--]\item $F^{U_0}\neq\emptyset=F^{U_1}$,
 so $F^{M^\Vv_\gamma|\eta}\neq\emptyset=F^{M|\eta}$, or
\item $F^{U_0}\neq\emptyset\neq F^{U_1}$ and $\nu(F^{U_0})<\nu(F^{U_1})$,
 so
$F^{M^\Vv_\gamma|\eta}\neq\emptyset\neq F^{M|\eta}$
and $\nu(F^{M^\Vv_\gamma|\eta})<\nu(F^{M|\eta})$.
\end{enumerate}
It follows that $\gamma+1<\theta$,
with $E^\Vv_\gamma=F(M^\Vv_\gamma|\eta)$
and $E^\Ww_\gamma=\emptyset$.

Now suppose instead that $M^\Vv_\gamma=M|\eta$, so
$M^\Vv_\gamma\ins M$.
Let $d=\deg^\Vv_\gamma$.
If $M^\Vv_\gamma\pins M^\Ww_\gamma$ then
again by \ref{lem:dropdown_lifting},
we have $M^{\Vv'}_\gamma=\Ult_d(M^\Vv_\gamma,G)\pins M^{\Ww'}_\gamma=M'$,
so $\gamma+1=\theta'$
and $M^{\Vv'}_{\theta-1}\pins M'$, contradicting Claim \ref{clm:(V',W')_structure}. So $M^\Vv_\gamma=M$.
It follows that $d=m$ and $\root^\Vv(\gamma)=0$.
For if $\root^\Vv(\gamma)=-1$
then $M^\Vv_\gamma\sats\ZFC^-$, whereas $M\not\sats\ZFC^-$.
So $\root^\Vv(\gamma)=0$,
but then $M^\Vv_\gamma$ is $d$-sound but not $(d+1)$-sound, so $d=m$.
So again by \ref{lem:dropdown_lifting},
$M^{\Vv'}_\gamma=\Ult_m(M^\Vv_\gamma,G)=\Ult_m(M,G)=M'$,
so $\gamma+1=\theta'=\theta$.

This establishes parts \ref{item:exit_extenders_match}--\ref{item:if_gamma+1=theta'_then_M'=M^Vv'_gamma}
for $\gamma$. We next consider
parts \ref{item:degrees_match},\ref{item:models_match}
for $\gamma+1$,
assuming that $\gamma+1<\min(\theta,\theta')$, and so $M^\Vv_\gamma|\eta\neq M|\eta$.

Suppose first that $\kappa=\crit(E^\Vv_\gamma)\geq\rho$.
Then
$\kappa'=\crit(E^{\Vv'}_\gamma)\geq\rho'$.
Let $\delta=\pred^\Vv(\gamma+1)$.
Using part \ref{item:exit_extenders_match}, it easily follows
that $\delta=\pred^{\Vv'}(\gamma+1)$.
Let $M^*=M^{*\Vv}_{\gamma+1}$
and $d=\deg^\Vv(\gamma+1)$.
Then $(M^*,d)$ is an element of the extended dropdown
sequence of $((M^\Vv_\delta,\deg^\Vv_\delta),\OR(\exit^\Vv_\delta))$,
and so by dropdown lifting \ref{lem:dropdown_lifting},
$(\Ult_d(M^*,G),d)$ is likewise
an element of the extended dropdown sequence
of $((M^{\Vv'}_\delta,\deg^\Vv_\delta),\OR(\exit^{\Vv'}_\delta))$.
Considering the positions of the various projecta
and how they are shifted up by the ultrapower maps,
we moreover get
\begin{equation}\label{eqn:ult_*-model}M^{*\Vv'}_{\gamma+1}=\Ult_d(M^*,G)
\end{equation}
and $d=\deg^{\Vv'}_{\gamma+1}$.
It remains to see that
\[ \Ult_d(M^\Vv_{\gamma+1},G)=M^{\Vv'}_{\gamma+1}.\]
This is a calculation very much like
%conf
\cite[Lemma 3.9]{premouse_inheriting}.
Note that $\Ult_d(M^\Vv_{\gamma+1},G)$
results from the two-step abstract $d$-maximal iteration
of $M^{*\Vv}_{\gamma+1}$ given by first
using $E^\Vv_\gamma$, then using $G$.
On the other hand, $M^{\Vv'}_{\gamma+1}$
results from first using $G$ (producing $M^{*\Vv'}_{\gamma+1}$),
and then $E^{\Vv'}_{\gamma}$. We want to see
that these result in the same model; that is,
\[
\Ult_d(\Ult_d(M^{*\Vv}_{\gamma+1},E^\Vv_\gamma),
G)=\Ult_d(\Ult_d(M^{*\Vv}_{\gamma+1},G),E^{\Vv'}_\gamma).\]
So
letting $j,k$ be the two resulting ultrapower maps
(with domain $M^{*\Vv}_{\gamma+1}$),
it suffices
to see that the models are the $\rSigma_d$-hull
of $\rg(j)$ (or $\rg(k)$, respectively) and ordinals below
$\lh(E^{\Vv'}_\gamma)$,
and that the $(\kappa,\lh(E^{\Vv'}_\gamma))$-extenders
derived from $j,k$
are identical (as then both models
are just the degree $d$ ultrapower
of $M^{*\Vv}_{\gamma+1}$ by that common extender, and hence
they are equal).

But note that $(\exit^\Vv_\gamma)^\passive$ is a cardinal
proper segment
of $M^\Vv_{\gamma+1}$ of ordinal
height $\lh(E^\Vv_\gamma)\leq\rho_d(M^\Vv_{\gamma+1})$,
so
\[ i^{M^\Vv_{\gamma+1},d}_G\rest\exit^\Vv_\gamma=i^{\exit^\Vv_\gamma,0}_G, \]
and these maps map $\exit^\Vv_\gamma$ cofinally into $\exit^{\Vv'}_\gamma$.
But then since $i^{\exit^\Vv_\gamma,0}_G$
maps fragments of $E^\Vv_\gamma$ to fragments
of $E^{\Vv'}_\gamma$, it follows
that the two derived extenders are identical, as desired.

The second case is that $\crit(E^\Vv_\gamma)=\kappa$,
so $\crit(E^{\Vv'}_\gamma)=\kappa'$,
and $M^\Vv_{\gamma+1}=\Ult_0(\bar{C},E^\Vv_\gamma)$,
and $M^{\Vv'}_{\gamma+1}=\Ult_0(\bar{C}',E^{\Vv'}_\gamma)$
This case works almost the same, except that
$\bar{C}'=\Ult_0(\bar{C},\bar{G})$ instead of $\Ult_0(\bar{C},G)$. However, we still have $\exit^{\Vv'}_\gamma=\Ult_0(\exit^\Vv_\gamma,G)$
and $M^{\Vv'}_{\gamma+1}=\Ult_0(M^\Vv_{\gamma+1},G)$, both using the full $G$.
We leave the details to the reader.
This completes the proof the claim.
\footnote{Note that it appears important here that
we reach a contradiction
without having to continue the copying process
past stage $\om$,
since we are not assuming any
condensation
properties of the iteration strategies being used.}
\end{proof}

Now by Claim \ref{clm:(V',W')_structure},
$\theta'<\om$,
and so using Claim \ref{clm:comparison_preservation}, $\theta=\theta'$ and
$(\Vv,\Ww)$
is a successful comparison,
of length $\theta$,
with $\Ww$ trivial and $M^{\Vv}_\infty=M$. It can't be that \ref{item:need_phalanxes} holds and  $\root^\Vv(\theta-1)=-1$,
because in that case,
$M^\Vv_{\theta-1}\sats\ZFC^-$,
although $M\not\sats\ZFC^-$.
It follows that $\Tt$ is equivalent to $\Vv$, and $\Uu$ is trivial,
so $\lambda=0$, $B=B'$
and $M^\Tt_\xi=M$.
But $\Tt$ is core-strongly-finite by Claim \ref{clm:Uu_core-sf}.
Note then that
at stage $0$,
in $\Tt$ we moved into $C$
(and in $\Uu$ we moved into $M$)
and $b^\Tt\cap\dropset_{\deg}^\Tt=\emptyset$.
If \ref{item:no_phalanxes} holds, it now follows that $\Vv$ is a tree on $C$ of the right form (that is,
it witnesses the lemma with respect to $C,M$;
it is almost-above $\rho$
by Claim \ref{clm:core-strongly-finite_transient_exts}).  Otherwise, letting $\Vv'$ be the $m$-maximal tree on $C$ which is equivalent to $\Vv$ (using the same extenders), then $\Vv'$ has the right form. This completes the proof.
\end{proof}

\begin{rem}\label{rem:more_standard_attempt}
As mentioned in Remark \ref{rem:simpler_corollary},
it seems one might try to prove just Corollary \ref{cor:def_over_core_from_extra_hypos}
 by comparing bicephali,
but using more standard calculations
instead of proving Claim \ref{clm:comparison_preservation} above (which also requires Lemma \ref{lem:dropdown_lifting}).
But there seems to be a difficulty,
which we now describe.

Start by forming the same bicephalus comparison as before.
Say it reaches a bicephalus
$B'=(C',M',\rho')=B^\Tt_\lambda$ in $\Tt$, and that $C'\ins
M^{e\Uu}_\lambda$,
and   $(\Tt,\Uu)$ has length $\alpha+1$,
where $\lambda<^\Tt\alpha$, and $(0,\alpha]^\Tt$ does not
drop in model or degree, but $\sides^\Tt_\alpha=\{1\}$,
so $N'=M^\Tt_\alpha$ is defined, but $M'\neq N'$.
Suppose that $\lambda\in b^\Uu$ and the core map $k:C'\to N'$
 is in fact an iteration map along
the last part of the final branch $b^\Uu$.
Then standard arguments already show that
\[ t'=\Th_{\rSigma_{m+1}}^{M'}(\rho_{m+1}^{M'}\cup\{i^\Tt(x)\}) \]
is $\bfrSigma_{m+1}^{C'}$, where $x\in M$.
Fix $p'\in C'$ such that $t'$ is $\rSigma_{m+1}^{C'}(\{p'\})$.

One would like to be able to pull this back to $C,M$.
But we can construct a finite tree $\bar{\Tt}$ on $\Tt$,
whose last model is a bicephalus $\bar{B}=(\bar{C},\bar{M},\bar{\rho})$,
which captures the parameter $p'$. Let $\sigma^0:\bar{C}\to C'$
and $\sigma^1:\bar{M}\to M'$
be the capturing maps, which
are $m$-embeddings, and $\sigma^0\com i^{0\bar{\Tt}}=i^{0\Tt}$
and $\sigma^1\com i^{1\bar{\Tt}}=i^{1\Tt}$ and
 $\sigma(\pbar)=p'$, and $\sigma^0\rest\bar{\rho}=\sigma^1\rest\bar{\rho}$.
 By the elementarity etc, it follows that
\[ \bar{t}=\Th_{\rSigma_{m+1}}^{\bar{M}}(\bar{\rho}\cup\{i^{1\bar{\Tt}}(x)\}) \]
is $\rSigma_{m+1}^{\bar{C}}(\{\pbar\})$.
And $\bar{\Tt}$ is finite, and
the iteration maps $i^{0\bar{\Tt}}_{0\infty}$
and $i^{1\bar{\Tt}}_{0\infty}$ agree over $\rho$.
So
the desired theory $t$ is just
$(i^{0\bar{\Tt}}_{0\infty})^{-1}``\bar{t}$.
Since $\bar{\Tt}$ is finite, one might expect
that $t$ should therefore be definable
from parameters over $C$. But the problem is,
that $C,M$ might be active and there might
be a use of $F^M$ (or some iterate of $F^M$)
along $b^{\bar{\Tt}}$, hence with $\kappa=\crit(F^M)<\rho$.
Note then that $\kappa^{+M}\leq\rho=\rho_{m+1}^M$,
and so if $m>0$ or $M$ is type 3 or $\kappa^{+M}<\rho$,
then the component measures of $F^M$ would
be in $M$, and in $C$, and so we could replace
the use of $F^M$ with some measure in $C$, which would suffice.
So the real problem case is that $C,M$ are type
2, $m=0$, and $\kappa^{+M}=(\kappa^+)^C=\rho_1^C=\rho_1^M$.
In this case, the author does not see how to avoid the proof given above.
Moreover, in this case, it does seem that using $F^M$ along $b^{\bar{\Tt}}$
might ostensibly provide
the parameter $\pbar$.
\end{rem}

\section{Interlude: Grounds of mice via strategically $\sigma$-closed
forcings}\label{sec:sigma_grounds}

In this brief section we use Theorem \ref{thm:finite_gen_hull}
(finitely generated hulls)
to prove Theorem \ref{thm:sigma_grounds}.
However, the rest of the paper does not depend on this section.
 Recall that if $N\sats\ZFC$ then a \emph{ground} of $N$
 is a $W\sub N$ such that there is $\PP\in W$
 and $G\in N$ such that $G$ is $(W,\PP)$-generic
 and $N=W[G]$.

\begin{tm}\label{thm:sigma_grounds}
 Let $M$ be a $(0,\om_1+1)$-iterable mouse whose universe $\univ{M}$ models $\ZFC$.
Let $W$ be a ground of $\univ{M}$ via  a forcing $\PP\in W$
 such that $W\sats$ ``$\PP$ is $\sigma$-strategically-closed''.
 Suppose
 $M|\aleph_1^M\in W$.\footnote{\label{ftn:sigma-closed_forcing_tame}If $M$ is also tame,
 then the hypothesis that $M|\aleph_1^M\in W$ is superfluous,
 %***check
 as shown in \cite[Theorem 4.7]{odle_v2}.
 Of course, as $\PP$ is
$\sigma$-strategically-closed in
$W$,
 automatically $\RR^M\sub W$. But note that $M|\aleph_1^M$
 gives the extender sequence of length $\aleph_1^M$,
 not just $\HC^M$.}
Then the forcing is trivial; that is, $W=\univ{M}$.
\end{tm}

For the proof we use the following two facts. The first is well-known; see
\cite{stgeol}:

\begin{fact}[Ground definability, Laver/Woodin]\label{fact:ground_definability}
There is a formula $\varphi_{\mathrm{gd}}$ in the language of set theory,
such that the following holds: Let $N$ be a model of $\ZFC$ and $W\sub N$ be
 a ground of $N$. Then there is $p\in W$ such that $W=\{x\in
N\bigm|N\sats\varphi_{\mathrm{gd}}(p,x)\}$.
\end{fact}

The second
is by \cite{V=HODX_pub}:

\begin{fact}\label{fact:ext_seq_def} There is a formula $\varphi_{\mathrm{pm}}$
in the language of set theory such that for
any mouse $M$ with no largest cardinal,
given any $N\in M$,
we have $N\pins M$ iff $M\sats\varphi_{\mathrm{pm}}(M|\om_1^M,N)$.
\end{fact}

\begin{dfn}
Let $H$ be a transitive class
 which satisfies ``there is no largest
cardinal''. Let $X\in H$.
If there is a premouse $P$ such that $\univ{P}=H$ and for all $N\in H$, we have
\[ N\pins P\iff H\sats\varphi_{\mathrm{pm}}(X,N),\]
then let $S^H(X)$ denote $P$;
otherwise $S^H(X)$ is undefined.
Say that $X$ is \dfnemph{good for $H$} if $S^H(X)$ is defined.
Note that $\{X\in H\bigm|X\text{ is good for }H\}$ is a definable class of $H$, uniformly in such $H$.
\end{dfn}

\begin{proof}[Proof of Theorem \ref{thm:sigma_grounds}]
Suppose the theorem is false,
as witnessed by $M,W,\PP$.
Let $\mathfrak{m}=M|\om_1^M$.
Then we may assume that for all sufficiently large limit cardinals $\lambda$
of $W$, $\PP$ forces ``$\check{\mathfrak{m}}$
is good for $\her_\lambda$
and $\check{\mathfrak{m}}=S(\check{\mathfrak{m}})|\om_1^{S(\check{\mathfrak{m}})}$''. Later we will also add some further statements which we assume are forced by $\PP$, as they become relevant.
We force with $\PP\cross\PP$. Write $\dot{S}^\lambda_{0}$
and $\dot{S}^\lambda_{1}$ for the names for
$S^{\her_\lambda^{W[G_0]}}(\mathfrak{m})$ and
$S^{\her_\lambda^{W[G_1]}}(\mathfrak{m})$
respectively, where $G_0\cross G_1$ is the $\PP\cross\PP$-generic.

Fix some
sufficiently
 large limit cardinal $\lambda$ of $W$, with $\PP\in\her_\lambda^W$,
 and such that ($*$) there is $\eta<\lambda$ such that
there is $(p,q)\in\PP\cross\PP$ forcing
``$\dot{S}^{\lambda}_0|\eta\neq \dot{S}^{\lambda}_1|\eta$''.
Let $\eta$ be the least witness to ($*$) (relative to $\lambda$).
Note then that $M||\eta\in W$ and
$\PP\cross\PP$ forces
``$\dot{S}^\lambda_0||\eta=\dot{S}^\lambda_1||\eta=M||\eta$''.
There are two
possibilities:
\begin{enumerate}[label=\tu{(}\roman*\tu{)}]
 \item\label{item:two_exts} There is $(p,q)\in\PP\cross\PP$
 forcing ``$\dot{S}^\lambda_0|\eta$ and $\dot{S}^\lambda_1|\eta$
 are both active, and $\dot{S}^\lambda_0|\eta\neq \dot{S}^\lambda_1|\eta$'', or
 \item\label{item:only_one_ext} otherwise.
\end{enumerate}

\begin{clmfour}\label{clm:passive_means_non-card}
If \ref{item:only_one_ext} holds, then $\eta$ is not a cardinal
of $W$.\end{clmfour}
\begin{proof}
In case \ref{item:only_one_ext},
any condition $(p,q)$ which forces that $\dot{S}^\lambda_0|\eta$
is active decides  all elements of the  active extender $E=F^{\dot{S}^\lambda_0|\eta}$,
because otherwise we get case \ref{item:two_exts}.
So $E\in W$, and the active premouse $(M||\eta,E)\in W$.
But the active extender
of a premouse
definably collapses its index, so $\eta$ is not a cardinal in
$W$.\end{proof}

Fix $(p_0,q_0)\in\PP$ witnessing ($*$) and the choice of $\eta$, and witnessing \ref{item:two_exts} if it holds.

Using a local version of Fact \ref{fact:ground_definability} (for example
%conf
in \cite[\S2]{choice_principles_in_local_mantles}; or by choosing $\lambda$ closed
enough, one
could use more well known versions),
we can fix a formula $\psi$ in the language
of set theory and $a_0,a_1\in\her_\lambda^W$ such that $\PP\cross\PP$ forces
``for all $x\in\her_\lambda^{W[G_i]}$,
we have $x\in\her_\lambda^W$ iff
$\her_\lambda^{W[G_i]}\sats\psi(x,a_i)$'',
for each $i\in\{0,1\}$.\footnote{Actually one can take $a_0=a_1=\her_\alpha^W$ for a sufficiently large $\alpha<\lambda$.}

Fix a strategy $\Sigma$ witnessing the strategic-$\sigma$-closure of $\PP$ in
$W$,
and let ${<^*}\in W$ be a wellorder of
$\PP$.  Let
\[ z=(M|\om_1^M,a_0,a_1,\PP,{<^*},\Sigma,p_0,q_0).\]
Fix a recursive enumeration $\left<\somevarpi_n\right>_{n<\om}$
of all formulas in the language of passive premice.

We now define a run of the ``$\PP$-game'' $\left<p_n,p'_n\right>_{n<\om}$,
according to $\Sigma$, with conditions $p_n$ played by Player I, and $p'_n$ by
player II.
We chose $p_0$ earlier. Let $p'_0=\Sigma(p_0)$.
Given $p'_n$, let $p_{n+1}$
be the ${<^*}$-least $p\leq p'_n$
such that $(p,q_0)$ either forces
\begin{equation}\label{eqn:forced_about_rho_n}\dot{S}
^\lambda_0\sats\text{``there is a unique
ordinal }\alpha\text{ such that
}\somevarpi_n(\check{z},\alpha)\text{''}\end{equation}
or forces its negation, and such that if $(p,q_0)$ forces
(\ref{eqn:forced_about_rho_n})
then there is some $\alpha<\lambda$
such that $(p,q_0)$ forces
$\dot{S}^\lambda_0\sats\somevarpi_n(\check{z},\check{\alpha})$''.
Set $p'_{n+1}=\Sigma(p_0,p_1,\ldots,p_{n+1})$.
This determines the sequence. Choose $p_\om\leq p_n,p'_n$ for all
$n<\om$.

Define $\left<q_n,q'_n\right>_{n<\om}$
completely symmetrically.
Let $q_\om\leq q_n,q'_n$ for all $n<\om$.

For those $n<\om$ such that $p_{n+1}$ forces (\ref{eqn:forced_about_rho_n}),
let $\alpha_n$ be the witnessing ordinal, and otherwise let $\alpha_n=0$.
Likewise define
$\beta_n$ for $q_{n+1}$.

Note that the forcing statements used in determining
the sequences above are all definable
over $\her_\lambda^W$ from the parameter $z$,
uniformly in each level $\Sigma_n$ (regarding the complexity
of $\somevarpi_i$); this uses Fact \ref{fact:ext_seq_def}.
So letting
\[ H^W=\Hull_\om^{\her_\lambda^W}(\{z\}) \]
(the uncollapsed hull of all elements of $\her_\lambda^W$
definable over $\her_\lambda^W$ (in the language of set theory) from the
parameter $z$),
 then for each $n<\om$,
\begin{equation}\label{eqn:things_in_H^W}\{z,p_n,p_n',q_n,q_n',\alpha_n,\beta_n\}\sub H^W.\end{equation}
(But we do not claim that the sequence
$\left<p_n,q_n\right>_{n<\om}\in H^W$, for example.)

Now let $G_0\cross G_1$ be $(W,\PP)$-generic
with $(p_\om,q_\om)\in G_0\cross G_1$.
Let $M_i=(\dot{S}^\lambda_i)_{G_0\cross G_1}$.
Let\[ H_i=\Hull_{\rSigma_\om}^{M_i}(\{z\}) \]
(note this hull is computed
using the language of premice).

Let $A=\{\alpha_n\bigm|n<\om\}$ and
$B=\{\beta_n\bigm|n<\om\}$.

\begin{clmfour}\label{clm:hulls_same_ordinals} $A=H_0\cap\OR=H^W\cap\OR=H_1\cap\OR=B$.
\end{clmfour}
\begin{proof}
We have $A\sub H_0\cap\OR$  by choice
of $p_\om$.
To see $H_0\cap\OR\sub A$, let $\alpha\in H_0\cap\OR$. Then there is some formula
$\somevarpi$ such that $M_0\sats$ ``$\alpha$ is the unique ordinal
such that $\somevarpi(z,\alpha)$''.
But then $\somevarpi=\somevarpi_n$ for some $n$,
so $p_{n+1}$ forces this, and $\alpha=\alpha_n\in A$.

So $A=H_0\cap\OR$, and by symmetry, $B=H_1\cap\OR$.

By line (\ref{eqn:things_in_H^W}),
$A\cup B\sub H^W$.

It just remains to see that $H^W\cap\OR\sub A\cap B$.
Let $\alpha\in H^W\cap\OR$. Note that $\her_\lambda^W\sub M_0$
and by the (local) ground
definability
fact and our choice of $z$, $\her_\lambda^W$ is a class
of $M_0$ definable from the parameter $z$.
Therefore $\alpha\in H_0$, so $\alpha\in A$. By symmetry, $\alpha\in H_1$, so $\alpha\in B$, completing the proof.
\end{proof}

Now let $H'_i=\Hull_{\Sigma_1}^{\J(M_i)}(\{z,\lambda\})$
(here $\J(M_i)$ denotes one step in the $\J$-hierarchy
above $M_i$). A standard computation shows
that $H'_i\cap\lambda=H_i\cap\OR$.

Let $K_i$ be the transitive collapse of $H'_i$.
Let $\pi_i:K_i\to H_i'$ be the uncollapse map.
Let $\pi_i(\bar{z}_i,\bar{\lambda}_i)=(z,\lambda)$.
Then $K_i=\Hull_{\Sigma_1}^{K_i}(\{\bar{z}_i,\bar{\lambda}_i\})$.
So $\rho_1^{K_i}=\om$ and $K_i$ is finitely $\bfrSigma_1$-generated.
Since  $\om_1^W=\om_1^{M_i}$, we have
$\crit(\pi_i)=\om_1^{K_i}=\om_1^W\cap H'_i$,
so by Claim \ref{clm:hulls_same_ordinals},
$\crit(\pi_i)$ is independent of $i\in\{0,1\}$.

Let $C_i=\core_1(K_i)$. We may assume here that $K_i$
is $1$-solid and $1$-universal, and $C_i$ is likewise
and $1$-sound, and that $C_i\pins M_i|\om_1^{M_i}$,
since these are first-order facts about $M$.
But $M_0|\om_1^{M_0}=\mathfrak{m}=M_1|\om_1^{M_1}$,
by assumption. Therefore $C_0,C_1\pins \mathfrak{m}$.
But also, $\om_1^{C_0}=\om_1^{K_0}=\om_1^{K_1}=\om_1^{C_1}$ and $C_0,C_1$
project to $\om$.
It follows that $C_0=C_1$.

We can also assume that Theorem \ref{thm:finite_gen_hull}
holds with respect to $K_0,K_1$, since this also holds for $M$.
Therefore $K_0,K_1$ are both finite normal iterates of $C_0=C_1$.

Now note that $\OR^{K_i}$ is the ordertype
of $H'_i\cap\OR$, which is independent of $i$.
We have $\pi_i:K_i\to H'_i$ and $\rg(\pi_i)=H'_i$,
and therefore $\pi_0\rest\OR=\pi_1\rest\OR$.

By our choice of $(p_0,q_0)$ and $(p_\om,q_\om)$,
we have $\eta\in\rg(\pi_0),\rg(\pi_1)$,
 $M_0||\eta=M_1||\eta$
and $M_0|\eta\neq M_1|\eta$.
Let $\pi_i(\bar{\eta})=\eta$.
Clearly then $K_0||\etabar=K_1||\etabar$.
If $M_0|\eta$ is passive and $M_1|\eta$ active,
this reflects and $K_0|\bar{\eta}$ is passive
and $K_1|\bar{\eta}$ is active.
If both $M_0|\eta$ and $M_1|\eta$ are active, then the active extenders
differ, and because this was forced by $(p_0,q_0)$ over $H^W$,
and it was also forced that $M_0||\eta=M_1||\eta$,
it easily follows that there are ordinals in $H'_0,H'_1$
witnessing the distinction between the extenders. (That is, let
$F_i=F^{M_i|\eta}$. Clearly
$\kappa_i=\crit(F_i)\in H_i'$ (so $\kappa_0,\kappa_1\in H_i'$). If
$\kappa_0\neq\kappa_1$
then we are done. If $\kappa_0=\kappa_1$,
then there is some $\alpha<\kappa_i^{+(M_i||\eta)}\cap H_i'$
and some $\vec{\beta}\in[\eta]^{<\om}\cap H_i'$
such that $\vec{\beta}\in i^{F_0}(A_\alpha)$
iff $\vec{\beta}\notin i^{F_1}(A_\alpha)$,
where $A_\alpha$ is the $\alpha$th subset of $[\kappa_i]^{<\om}$
in the order of construction of $M_i||\eta$.)
Note then that $K_0|\bar{\eta}\neq K_1|\bar{\eta}$,
and the kind of disagreement corresponds to the case above (that is,
whether one is passive and the other active, or both are active but distinct).
So $\bar{\eta}$ indexes the least disagreement between $K_0,K_1$.

Now let $\Tt_i$ be the finite normal tree on $C_0=C_1$
whose last model is $K_i$,
given by Theorem \ref{thm:finite_gen_hull}.
Note that $\Tt_i$ is determined by $K_i$
(it is the result of comparing $C_i$ with $K_i$,
but no extenders
get used on the $K_i$ side, and also the trees
are finite length, so there are no branches to choose).
Note also that if $K_i|\xi$
is active then $\Tt_i$ uses no extender with index $\xi$,
so $K_i|\xi=M^{\Tt_i}_n|\xi$ where $n$ is least such that either $\lh(\Tt_i)=n+1$ or $\lh(E^{\Tt_i}_n)\geq\xi$.
(This holds even if there is $m+1<\lh(\Tt_i)$ such that $M^{\Tt_i}_{m+1}|\lh(E^{\Tt_i}_m)$
is active. For in this case, $E^{\Tt_i}_m$ is superstrong and $M^{\Tt_i}_{m+1}$ is active type 2
with $\OR(M^{\Tt_i}_{m+1})=\lh(E^{\Tt_i}_m)$.
But then $(0,m+1]^{\Tt_i}$ drops in model, since $C_i$ is passive.
Therefore $m+2<\lh(\Tt_i)$
and $E^{\Tt_i}_{m+1}=F^{M^{\Tt_i}_{m+1}}$,
so $\OR(M^{\Tt_i}_{m+1})$ is passive in $K_i$.)
Since $K_0||\bar{\eta}=K_1||\bar{\eta}$,
we get $n<\om$ with $\Tt_0\rest(n+1)=\Tt_1\rest(n+1)$,
and $M^{\Tt_i}_n||\bar{\eta}=K_0||\bar{\eta}=K_1||\bar{\eta}$,
and since at least one side is active at $\bar{\eta}$, say $K_0|\bar{\eta}$ is
active,
we must have $M^{\Tt_i}_n|\bar{\eta}=K_0|\bar{\eta}$ (as $K_0$ doesn't move in
the comparison). But likewise for $K_1$, and as $K_0|\bar{\eta}\neq
K_1|\bar{\eta}$, and $K_1$ doesn't move in comparison, $K_1|\bar{\eta}$ is
passive.
So $\Tt_1$ uses an extender indexed at $\bar{\eta}$,
so $\bar{\eta}$ is a cardinal of $K_1$.
Therefore, as $\lambda$ was a limit cardinal of all models
and $\pi_i$ is sufficiently elementary,
$\eta$ is a cardinal in $M_1$, hence in $W[G_1]$,
hence in $W$. But as $K_1|\bar{\eta}$ is passive,
so is $M_1|\eta$, so we are in case \ref{item:only_one_ext} above,
so by Claim \ref{clm:passive_means_non-card}, $\eta$ is not a cardinal in $W$,
contradiction.
\end{proof}

\section{The mouse order ${<^{\para}_k}$}\label{sec:mouse_order}

Our proof of solidity and universality will be by induction on a certain
mouse (or mouse-parameter) order. We introduce this now.

\begin{dfn}
 $\Mm_k$ denotes the class of all $k$-sound premice, and $\Mm_k^\iter$ denotes
class
of all $(k,\om_1+1)$-iterable premice in $\Mm_k$.
\end{dfn}
\begin{rem}
Let $M\in\Mm_k$.
A \emph{$k$-maximal stack}
on $M$ was defined in \S\ref{sec:notation}.
The iteration game $\Gg_{\mathrm{fin}}(M,k,\om_1+1)$
was introduced in
%conf
\cite[Definition 1.1]{iter_for_stacks}, and proceeds
as follows: Player I first plays a  $k$-maximal
stack $\vec{\Tt}=\left<\Tt_i\right>_{i\leq m}$, where $m<\om$
and each $\Tt_i$ has finite length, and then the players proceed
to play a round of the $(n,\om_1+1)$-iteration game on $M^{\Tt_m}_\infty$,
where $n=\deg^{\Tt_m}_\infty$.
\end{rem}

The following fact is one case of
%conf
\cite[Theorem 9.6]{iter_for_stacks}:

\begin{fact}\label{fact:Gg}
Let $M\in\Mm^\iter_k$.
Then there is a winning strategy $\Sigma^*_M$ for player
$\Two$ in $\Gg_{\mathrm{fin}}(M,k,\om_1+1)$. Moreover, given any
$(k,\om_1+1)$-strategy $\Sigma$ for
$M$,
there is a canonical winning strategy $\Sigma^*$ for $\Two$ in
this game
determined by $\Sigma$.

Let $\vec{\Tt}=\left<\Tt_n\right>_{n<\om}$ be any $k$-maximal stack
on $M$ of length $\om$, consisting of finite length trees $\Tt_n$. Then
$(0,\infty]^{\Tt_n}$ drops for only finitely many $n<\om$
and $M^{\vec{\Tt}}_\infty$ is wellfounded.
\end{fact}

\begin{dfn}[Parameter order]\label{dfn:<^*}
Let $\Mm_k^{*}$ denote the class of pairs $(M,p)$
such that $M\in\Mm_k$ and $p\in[\rho_0^M]^{<\om}$.
We order $\Mm_k^{*}$ as follows. For $(M,p),(N,q)\in\Mm_k^{\para}$, let
$(M,p)<^{\para*}_k(N,q)$ iff
there is a $k$-maximal stack $\Ttvec=\left<\Tt_i\right>_{i<n}$
on $N$
where $n<\om$ and each $\Tt_i$ is terminally-non-dropping
with $\lh(\Tt_i)<\om$,
and there is a $k$-embedding $\pi:M\to M^{\Ttvec}_\infty$
such that
$\pi(p)<i^\Ttvec(q)$.
\end{dfn}

We are actually only interested in the following sub-order of $<^{\para*}_k$:
\begin{dfn}[Standard parameter order]
 We order $\Mm_k$ as follows.
 Given $M,N\in\Mm_k$, let $M<^{\para}_kN$ iff
$(M,p_{k+1}^M)<^{\para*}_k(N,p_{k+1}^N)$.
\end{dfn}

\begin{lem}
If $(M,p)<^{\para*}_k(N,q)$ and $N\in\Mm^\iter_k$ then $M\in\Mm^\iter_k$.
\end{lem}
\begin{proof}
By Fact \ref{fact:Gg} and using the witnessing map $\pi$ to copy trees.
\end{proof}

By the following lemma, we can arrange our proof of solidity
and universality by induction on $<^{\para}_k$:

\begin{lem}\label{lem:param_order_wfd}
${<^{\para*}_k}\rest\Mm^\iter_k$ is transitive and wellfounded.
So is ${<^{\para}_k}\rest\Mm^\iter_k$.
\end{lem}
\begin{figure}
\centering
\begin{tikzpicture}
 [mymatrix/.style={
    matrix of math nodes,
    row sep=0.8cm,
    column sep=0.8cm}]
   \matrix(m)[mymatrix]{
  M_{0}&  M_{01} & M_{02} &{}  \\
   {} & M_{1} & M_{12} \\
   {} & {} & M_{2} \\};
 \path[->,font=\scriptsize]
%maps from left $M$
(m-1-1) edge node[above] {$i_{01}$} (m-1-2)
(m-1-2) edge node[above] {$i'$} (m-1-3)
(m-2-2)edge node[above] {$i_{12}$} (m-2-3)
(m-2-2)edge node[left] {$\pi_{10}$} (m-1-2)
(m-2-3)edge node[left] {$\sigma$} (m-1-3)
(m-3-3) edge node[left] {$\pi_{21}$} (m-2-3);
\end{tikzpicture}
\caption{The diagram commutes,
where $M_{01}$ denotes $M^{\vec{\Tt}_{01}}_\infty$,
$i_{01}$ denotes $i^{\vec{\Tt}_{01}}$, etc.}
\label{fgr:DJ}
\end{figure}

\begin{proof}
The transitivity of ${<^{\para*}_k}\rest\Mm_k^\iter$ follows from Fact
\ref{fact:Gg}
and the following copying argument (cf.~Figure \ref{fgr:DJ}). Let
\[ (M_2,q_2)<^{\para*}_k(M_1,q_1)<^{\para*}_k(M_0,q_0),\] as
witnessed by trees
 $\Ttvec_{12}$ on $M_1$ and $\Ttvec_{01}$ on $M_0$, and maps
\[ \pi_{21}:M_2\to
M^{\Ttvec_{12}}_\infty\text{ and }
\pi_{10}:M_1\to M^{\Ttvec_{01}}_\infty.\]
Let $\Ttvec'$ be the $\pi_{10}$-copy of $\Ttvec_{12}$ to a tree on
$M^{\Ttvec_{01}}_\infty$, let $\Ttvec_{02}=\Ttvec_{01}\conc\Ttvec'$
and
\[ \sigma:M^{\Ttvec_{12}}_\infty\to M^{\Ttvec_{02}}_\infty \]
be the final copy map and $\pi_{20}=\sigma\com\pi_{21}$.
Then  $\Ttvec_{02}$ and $\pi_{20}$
witness the fact that $(M_2,q_2)<^{\para*}_k(M_0,q_0)$.
\footnote{$<^{\para*}_k$ might not be transitive on all of $\Mm_k$
(including non-iterable premice), because the
copying
construction might produce an illfounded model.
This is why Fact \ref{fact:Gg} is relevant to transitivity.}

We now show that $<^{\para*}_k$ is
wellfounded below every $(M,p)$ with $M\in\Mm_k^{\iter}$. So fix such an $(M,p)$.
The argument is essentially the usual one for Dodd-Jensen,
but recall that we do not assume $\AC$;  one needs
a little care to deal with this.
First define a more restrictive relation $<'$ on $\Mm_k$
by setting $(N_1,q_1)<'(N_0,q_0)$ iff
$(N_1,q_1)<^{\para*}_k(N_0,q_0)$
as witnessed by a stack $\Ttvec$ (on $N_0$) with $N_1=M^{\Ttvec}_\infty$
and the map $\id:N_1\to N_1$ (so $q_1<i^{\Ttvec}(q_0)$).
Note that ${<'}$ is also transitive.
The set of all $(N,q)$ such that $(N,q)<'(M,p)$ is wellordered,
as is the set of all witnessing trees $\Ttvec$,
since those trees are finite.
So if ${<'}$ is illfounded then there is a ${<'}$-descending
sequence
$\left<(N_m,q_m)\right>_{m<\om}$,
and a sequence $\left<\vec{\Tt}_m\right>_{m<\om}$ of witnessing
stacks.
But letting $\Ttvec$ be the concatenation of all $\Ttvec_m$,
then $\Ttvec$ is a length $\om$, $k$-maximal stack on $M$
such that $M^{\Ttvec}_\infty$ is illfounded, contradicting
Fact \ref{fact:Gg}.

Now suppose that $<^{\para*}_k$ is illfounded below
$(M,p)$. Say that $(N,q)$ is \emph{bad}
if
$(N,q)$ is in the illfounded part of ${<^{\para*}_k}$
and either $(N,q)=(M,p)$ or $(N,q)<'(M,p)$.
Since $<'$ is wellfounded below $(M,p)$, we can fix a bad $(N,q)$ such that there is no bad
$(N',q')<'(N,q)$.
Let $(P',r')<^{\para*}_k(N,q)$
with $(P',r')$ also in the illfounded part of ${<^{\para*}_k}$,
as witnessed by $\Ttvec$ and $\pi:P'\to M^{\Ttvec}_\infty$.
Let $P=M^{\Ttvec}_\infty$ and $r=\pi(r')$.
So $(P,r)<'(N,q)$, so by choice of $(N,q)$, $(P,r)$ is in the wellfounded part of
${<^{\para*}_k}$. But by copying using $\pi$,
\[ \text{for every }(P'',r'')<^{\para*}_k(P',r'),\text{ we have
}(P'',r'')<^{\para*}_k(P,r),\]
so $<^{\para*}_k$ is illfounded below $(P,r)$, a contradiction.
\end{proof}

\begin{lem}\label{lem:p-pres_for_premouse_from_para_order}
Let $M\in\Mm^\iter_k$ and suppose $H$ is $(k+1)$-solid
for all $H<^{\para}_k M$.
Let $\Tt$ be $k$-maximal
on $M$ of successor length, such that $b^\Tt$
does not drop in model or degree.
Then $p^{M^\Tt_\infty}_{k+1}=i^\Tt(p^M_{k+1})$ and
$\rho_{k+1}^{M^\Tt_\infty}=\sup i^\Tt``\rho_{k+1}^M$.
\end{lem}
\begin{proof}
Write $N=M^\Tt_\infty$.
By preservation of fine structure (Lemma
\ref{lem:basic_fs_pres} and Corollary \ref{cor:basic_fs_pres}), and without considering $<^{\para}_k$, we have:
\begin{enumerate}[label=\tu{(}\alph*\tu{)}]
 \item\label{item:rho_p_sub_pres}
$\rho^{N}_{k+1}=\sup i^\Tt``\rho^M_{k+1}$
and $p^{N}_{k+1}\leq i^\Tt(p^M_{k+1})$,
\item\label{item:solidity_equiv}
 $N$ is $(k+1)$-solid iff $M$ is $(k+1)$-solid,
\item\label{item:solid_implies_p-pres} if $M$ is $(k+1)$-solid then
$p^{N}_{k+1}=i^\Tt(p^M_{k+1})$.
\end{enumerate}
So suppose $p^{N}_{k+1}<i^\Tt(p^M_{k+1})$; we will reach a contradiction.

Suppose  $\lh(\Tt)<\om$.
Then $N<_k^{\para}M$, as witnessed by
$\Tt$ and $\id:N\to N$.
Therefore by the lemma's hypothesis,
$N$ is $(k+1)$-solid. So by \ref{item:solidity_equiv},
$M$ is $(k+1)$-solid, so by \ref{item:solid_implies_p-pres},
$p^{N}_{k+1}=i^\Tt(p^M_{k+1})$, a contradiction.

So $\lh(\Tt)\geq\om$.
Let $\bar{\Tt}$ be a  $k$-maximal tree of finite length
which captures $(\Tt,p^{N}_{k+1},0)$
(exists by \ref{lem:simple_embedding_exists}). Write
$\bar{N}=M^{\bar{\Tt}}_\infty$.
So  $b^{\bar{\Tt}}$ does not drop in model or degree,
and there is a $k$-embedding $\sigma:\bar{N}\to N$
with $\sigma\com i^{\bar{\Tt}}=i^\Tt$
and $\pbar\in \bar{N}$
with $p^{N}_{k+1}=\sigma(\pbar)$.
Using \ref{item:rho_p_sub_pres} and its version for $\bar{\Tt}$ and commutativity,
we have
$\rho^N_{k+1}=\sup\sigma``\rho^{\bar{N}}_{k+1}$,
so $\sigma(\rho^{\bar{N}}_{k+1})\geq\rho^{N}_{k+1}$.
Note that $\pbar<i^{\bar{\Tt}}(p^M_{k+1})$,
since $\sigma(\pbar)=p^N_{k+1}<i^\Tt(p^M_{k+1})=\sigma(i^{\bar{\Tt}}(p^M_{k+1}))$.
As in the previous paragraph, $i^{\bar{\Tt}}$ is $p_{k+1}$-preserving,
so $\pbar<p^{\bar{N}}_{k+1}=i^{\bar{\Tt}}(p^M_{k+1})$.
So
\[
t=\Th_{\rSigma_{k+1}}^{\bar{N}}(\rho^{\bar{N}}
_{k+1}\cup\{\pbar,\pvec_k^{\bar{N}}\})\in \bar{N}.\]
So $\sigma(t)\in N$,
and note that from $\sigma(t)$,
we get
\[\Th_{\rSigma_{k+1}}^{N}(\sigma(\rho^{\bar{N}}_{k+1})
\cup\{\sigma(\pbar),\pvec_k^{N}\})\in N.\]
(This is as in the computation
of solidity witnesses from generalized solidity witnesses,
using that $\sigma$ is $\rSigma_{k+1}$-elementary.)
But as $\sigma(\pbar)=p^{N}_{k+1}$
and $\sigma(\rho^{\bar{N}}_{k+1})\geq\rho^{N}_{k+1}$,
this is a contradiction.
\end{proof}

\section{Solidity and universality}\label{sec:solidity}

In this section we prove that normally iterable mice
are solid and universal.
For the entire section, we will fix some $k<\om$,
and deal with $k$-sound premice $M$,
and proving fine structure at the $(k+1)$th level, like $(k+1)$-solidity, etc. To reduce notation, we will usually drop the subscript ``$k+1$'' from the notation $p_{k+1}^M$, $\rho_{k+1}^M$, $\core_{k+1}(M)$, $\Hull_{k+1}^M$, $\cHull_{k+1}^M$,
writing instead $p^M$, etc. (But we still write $p_k^M$, etc,
for the objects at the $k$th level.)

In proving the $(k+1)$-solidity of a premouse $\widetilde{M}$, we will
want to show that
for certain $\widetilde{H}$ and near $k$-embeddings
\[ \widetilde{\pi}:\widetilde{H}\to\widetilde{M},\]
we
have
$\widetilde{H}\in\widetilde{M}$. In some cases we will deduce this from
Lemma \ref{lem:finite_gen_hull} (on projectum-finitely generated mice), in
others, we will quote
facts on iterable bicephali from
\cite{premouse_inheriting}, and in the remaining cases, where
the main work is, we will form and analyze a
comparison of a certain kind of bicephalus
$\widetilde{B}=(\widetilde{\delta},\widetilde{\gamma},\widetilde{\pi},\widetilde{H},\widetilde{M})$. The details of the relevant kind of bicephali are given in Definition \ref{dfn:relevant_B}, but its two models are $\widetilde{H}$ and $\widetilde{M}$, and $\widetilde{\gamma}=\crit(\widetilde{\pi})$ is an element of $p^{\widetilde{M}}$, and $\widetilde{\delta}=\card^{\widetilde{M}}(\widetilde{\gamma})$.

\subsection{The main argument}\label{sec:solidity_the_main_argument}

\begin{tm}[Solidity and universality]\label{thm:solidity}
Let $k<\om$ and let $M$ be a $k$-sound, $(k,\om_1+1)$-iterable premouse.
Then $M$ is $(k+1)$-solid and $(k+1)$-universal.\end{tm}

We gave an outline for the plan of the proof in \S\ref{sec:outline_solidity}. The reader might want to review that prior to beginning the proof below  to get the general idea in mind (though we won't actually rely on it).

\begin{proof}
It suffices to consider only premice in $\Mm_k^\iter\cap\HC$,
since given any $M\in\Mm_k^\iter$,
working in $L[M]$, we can find a countable elementary
substructure of $M$.

The proof is by induction on ${<^{\para}_k}\rest(\Mm_k^\iter\cap\HC)$ (see
\S\ref{sec:mouse_order}),
and this order is in large part a substitute for Dodd-Jensen.
So fix $M\in\Mm_k^\iter\cap\HC$
and suppose by induction that
\begin{equation}\label{eqn:solid_univ_induction}\all H \left[(H<^{\para}_k M)
\Longrightarrow
H\text{ is }(k+1)\text{-solid and }(k+1)\text{-universal}\right];\end{equation}
note here that  every $H<^{\para}_k M$ is in
$\HC$, since the trees $\Tt$ witnessing that $H<^{\para}_kM$ are finite.
We must prove that $M$ is $(k+1)$-solid and $(k+1)$-universal.

Before we begin the main argument,
sketched in \S\ref{sec:outline_solidity},
we want to  observe that we may assume that
$M=\Hull^M(\rho^M\un
\pvec^M_{k+1})$, and also reduce the problem to $(k+1)$-solidity (dispensing with $(k+1)$-universality).
Let $\rho=\rho^M$.
Let $\Mbar=\core(M)$ and let $\pi:\Mbar\to M$ be the uncollapse.
Then $\pi$ is a $k$-embedding and $\Mbar$ is $(k,\om_1+1)$-iterable.
Let $\pi(q)=p^M$. We have \begin{equation}\label{eqn:Mbar_is_Hull}\Mbar=\Hull^{\Mbar}(\rho\cup\{q,\pvec_k^{\Mbar}\}).\end{equation}

\begin{clmfive}\label{clm:p^Mbar=pbar} $\rho^\Mbar=\rho$ and
$p^\Mbar=q$.
\end{clmfive}
\begin{proof}
That $\rho^\Mbar=\rho$ is as usual:
By line (\ref{eqn:Mbar_is_Hull}),
 $\rho^\Mbar\leq\rho$.
Conversely, given  $\delta<\rho$
and $A\sub\delta$ which is
 $\bfrSigma_{k+1}^\Mbar$-definable, then since  $\pi$ is $\rSigma_{k+1}$-elementary
and $\crit(\pi)\geq\rho$, $A$ is also $\bfrSigma_{k+1}^M$-definable,  so $A\in M|\delta^{+M}\sub M|\rho\sub \Mbar$.
So $\rho^{\Mbar}=\rho$.

Since $\rho^{\Mbar}=\rho$ and by line (\ref{eqn:Mbar_is_Hull}),
 $p^\Mbar\leq q$.
Suppose $p^\Mbar<q$. Then
$\Mbar<^{\para}_k M$, as witnessed by $\Tt,\pi$, where $\Tt$ is trivial. So by
line
(\ref{eqn:solid_univ_induction}), $\Mbar$ is $(k+1)$-solid and
$(k+1)$-universal.
So letting $C=\core(\Mbar)$
and $\sigma:C\to\Mbar$ the core map,
we have $\rho^C=\rho$ and $\pi(p^C)=p^{\Mbar}$,
so  $\pi(\sigma(p^C))<p^M$. Therefore $C\in M$ (by the
minimality
of $p^M$) and $C<^{\para}_kM$, so $C$ is also $(k+1)$-solid,
so $C$ is $(k+1)$-sound.
So by line (\ref{eqn:Mbar_is_Hull}) and Corollary \ref{cor:def_over_core_from_extra_hypos}
applied to $\bar{M}$, the theory
\[ t=\Th_{\rSigma_{k+1}}^{\bar{M}}(\rho\cup\{q,\pvec_k^{\bar{M}}\}),\]
coded as a subset of $\rho$, is definable from parameters over $C$.
Since $C\in M$, therefore $t\in M$.
But $t$ is equivalent to
\[t'= \Th_{\rSigma_{k+1}}^M(\rho\cup\{p^M,\pvec_k^M\}),\]
so $t'\in M$, a contradiction.
\end{proof}

\begin{clmfive}\label{clm:Mbar_solid_implies_M_solid,univ} If $\Mbar$ is $(k+1)$-solid then $M$ is $(k+1)$-solid and
$(k+1)$-universal.\end{clmfive}
\begin{proof}
Suppose $\Mbar$ is $(k+1)$-solid. Then by Claim \ref{clm:p^Mbar=pbar} and line (\ref{eqn:Mbar_is_Hull}), $\Mbar$ is
$(k+1)$-sound, and
$\pi(p^\Mbar)=p^M$, so $M$ is $(k+1)$-solid. Consider $(k+1)$-universality.
We have $\Mbar=\core(M)\notin M$,
and the core map $\pi:\Mbar\to M$ satisfies
the requirements of
 condensation Fact \ref{fact:condensation}, with the
parameter
``$\rho$'' there being $\rho=\rho^{\Mbar}=\rho^M$.  If $\rho^{+\Mbar}<\rho^{+M}$,
then by Fact \ref{fact:condensation} part \ref{item:fact_condensation_H_notin_M}\ref{item:cond_1d},
$M|\rho$ is active, but $\rho$ is an $M$-cardinal.   So $\rho^{+\Mbar}=\rho^{+M}$.
But then $\Mbar||\rho^{+\Mbar}=M||\rho^{+M}$, since
$\crit(\pi)\geq\rho$ and if $\crit(\pi)=\rho$ then we can apply condensation for $\om$-sound mice (Fact \ref{fact:om_condensation}) to $\pi\rest N:N\to\pi(N)$ for each $N\pins\Mbar$ with $\rho_\om^N=\rho$, to see that $N\pins M$.
This gives in particular that $\pow(\rho)\cap \Mbar=\pow(\rho)\cap M$,
so $M$ is $(k+1)$-universal, as desired.
\end{proof}

So it suffices to prove that $\Mbar$ is $(k+1)$-solid.

\begin{clmfive} For every $H$, if $H<^{\para}_k\Mbar$ then
$H<^{\para}_kM$.\end{clmfive}

\begin{proof}
 By Claim \ref{clm:p^Mbar=pbar}, and
via copying as in the proof of transitivity of $<^{\para}_k$.
\end{proof}
So the
induction hypothesis, line (\ref{eqn:solid_univ_induction}),
still applies after replacing $\Mbar$ with $M$.
We reset notation, writing ``$M$''
instead of ``$\Mbar$''; this amounts to proving $M$ is $(k+1)$-solid under the added
assumption, which we now make, that
\begin{asstwo}\label{ass:M_generated_by_rho^+}$M=\Hull^M(\rho^M\un
\pvec^M_{k+1})$.
\end{asstwo}

For $q\ins p^M$, we will prove that $(M,q)$ is $(k+1)$-solid, by
induction on $\lh(q)$. So let $q\pins p^M$ be such that $(M,q)$
is $(k+1)$-solid and let
$n=\lh(q)$.
For $k$-sound premice $N$, write $q^N$ for $p^N\rest n$. Let
$\gamma=\max(p^M\cut q^M)$. Let
\[ H=\cHull^M(\{q^M,\pvec^M_k\}\un\gamma) \]
and let $\pi:H\to M$ be the uncollapse and $\pi(r)=q^M$ (we have since discarded
the core map $\pi$
defined earlier). We
have to see that $H\in M$.
So we now assume otherwise,
and will draw out a contradiction:

\begin{asstwo}\label{ass:H_notin_M}$H\notin M$.
\footnote{Once
we have proven that in fact, $H\in M$,
we will anyway be able to deduce a precise description
of $H$ in terms of proper segments of $M$
via Theorem \ref{thm:finite_gen_hull}, and thus
recover at that point information that might otherwise
have been obtained by arguing directly, without contradiction.}
\end{asstwo}

We now prepare for the main argument.
Let $B=(\delta,\gamma,\pi,H,M)$.
We will consider $B$ as a kind of bicephalus, for which $\delta$ is the primary exchange ordinal; so extenders $E$ with $\crit(E)<\delta$ will lift the entire bicephalus, whereas if $\crit(E)\geq\delta$ then $E$ will apply to  just one model.
The plan is to form and analyse a comparison of $B$ with $M$.
This bicephalus $B$ will be iterable because $M$ is, and as we will be able to lift trees
on $B$ to trees on $M$. The comparison between $B$ and $M$ will terminate via essentially the usual argument for comparison of mice. The fine structural properties of $B$ will be preserved nicely by iterations, so that iterates $B'$ will have similar properties. But using these properties, and the precise rules for forming the comparison, we will be able to argue that the comparison cannot terminate, giving the desired contradiction.

We will specify precisely how we form iteration trees on  $B$ in Definition \ref{dfn:tree_on_rel_B}.
But before that, it will be useful to establish some more of the fine structural properties of $B$, some of which
will be abstracted, in Definition  \ref{dfn:relevant_B}, into the kind of bicephali $B'$ which will arise as iterates of $B$. We will then establish the basic properties of iteration trees on $B$, adapting the picture for trees on premice and phalanxes. After that, immediately
following Claim \ref{clm:B_iterable_fs_pres}, we will specify the precise rules for the comparison,
and then proceed to the actual analysis of the comparison, which constitutes Claims \ref{clm:T,U_normality}--\ref{clm:solidity_M|gamma_not_active}.

\begin{clmfive}\label{clm:solidity_pi_is_k-embedding} $\pi$ is a $k$-embedding.
\end{clmfive}
\begin{proof}
Suppose not.
Then $\pi``\rho_k^H$ is bounded in $\rho_k^M$,
which  implies $H\in M$; see
%conf
\cite[Lemma 2.4]{premouse_inheriting}. This contradicts Assumption \ref{ass:H_notin_M}.
\end{proof}

Therefore $\crit(\pi)=\gamma<\rho_k^H$
and $\pi(\gamma)<\rho_k^M$. Let $\delta=\card^M(\gamma)$.
As usual, there is a significant break into two cases,
and using condensation for $\om$-sound mice (Fact \ref{fact:om_condensation}) in the usual manner like in the proof of Claim \ref{clm:Mbar_solid_implies_M_solid,univ}, we have
either:
\begin{enumerate}[label=\arabic*.,ref=\arabic*]
 \item $\delta=\gamma$ is a limit cardinal of $M$
 and inaccessible in $H$, $\gamma^{+H}\leq\gamma^{+M}$
and $H||\gamma^{+H}=M||\gamma^{+H}$, or
 \item $\gamma=\delta^{+H}$ and $\pi(\gamma)=\delta^{+M}$
 and $H||\gamma=M||\gamma$, and either:
 \begin{enumerate}[label=--]
 \item $M|\gamma$
 is passive and $H||\gamma^{+H}=M||\gamma^{+H}$,
 or
\item $M|\gamma$ is active with an $M$-total extender $F$
 and $H||\gamma^{+H}=\Ult(M,F)||\gamma^{+H}$.
 \end{enumerate}
\end{enumerate}

So in any case, we have $H||\gamma^{+H}\in M$.

\begin{clmfive}
 $H\nins M$, and if
 $M|\gamma$ is active then
$H\nins U=\Ult(M|\gamma,F^{M|\gamma})$.
\end{clmfive}
\begin{proof}
We have $H\neq M$, so if $H\ins M$ then $H\pins M$ and $H\in M$. And if
  $M|\gamma$ is active then $U$ (as above) is in $M$,
 so if $H\ins U$ then $H\in M$.
\end{proof}

\begin{clmfive}\label{clm:not_1_2_with_gamma_between}\
\begin{enumerate}[label=\arabic*.,ref=\arabic*]
\item If $\gamma=\lgcd(H)$ then
$H,M$ are active type 2 and $k=0$.
\item\label{item:not_1_2_with_gamma_between} If $H,M$ are active type 1 or 2
and $k=0$
then $\gamma\notin[\kappa,\kappa^{+M}]$,
where $\kappa=\crit(F^M)$.
\end{enumerate}
\end{clmfive}
\begin{proof}
 Suppose $\gamma=\lgcd(H)$. Then if
 $H$ is passive then $H=H||\gamma^{+H}\in M$. If $H$ is active
type 3
then $\rho_0^H=\gamma$,
contradicting that $\gamma<\rho_k^H$.
If $H$ is type 1
then
$\gamma>\delta$ and $\gamma=\nu(F^H)$,
so $k=0$, but then $\pi``\gamma$ is bounded in $\pi(\gamma)=\nu(F^M)$,
so by standard calculations,
$\pi$
is
bounded in $\OR^M=\rho_0^M=\rho_k^M$, a
contradiction. So $H$ is type 2, so $M$ is also, and also $\rho_1^H\leq\gamma$, so $k=0$.

Now suppose  $H,M$ are active type 1/2, $k=0$
and  $\gamma\in[\kappa,\kappa^{+M}]$.
Note that $\kappa<\gamma=\kappa^{+H}<\kappa^{+M}$
and
$\pi``\kappa^{+H}=\kappa^{+H}$ is bounded in $\kappa^{+M}$.
So again $\pi$ is bounded in
$\rho_k^M$.
\end{proof}

\begin{clmfive}\label{clm:pi(p_k+1^H)<p_k+1^M}
We have:
\begin{enumerate}[label=\arabic*.,ref=\arabic*]
 \item\label{item:pi(p^H)<p^M} $\pi(p^H)<p^M$, so $H<^{\para}_kM$,
 \item\label{item:H_solid_and_univ} $H$ is $(k+1)$-solid and $(k+1)$-universal
with $\rho^H\leq\gamma$,
\item\label{item:q^H=r_and_p^H=q^H_or_q^H_conc_s} $q^H=r$,
and either $p^H=q^H$ or $p^H=q^H\conc s$ for some $s\sub\gamma$,
\item\label{item:rho^M<rho^H} $\rho^M<\rho^H$.
\end{enumerate}
\end{clmfive}
\begin{proof}
Parts \ref{item:pi(p^H)<p^M}--\ref{item:q^H=r_and_p^H=q^H_or_q^H_conc_s}: Recall $\pi(r)=q^M$. So $H=\Hull^H(\gamma\cup\{r,\pvec_k^H\})$.
So $\rho^H\leq\gamma$ and either:
\begin{enumerate}[label=(\roman*)]
 \item\label{item:p^H_leq_r} $p^H\leq r$, so $\pi(p^H)\leq q^M<p^M$, or
 \item\label{item:p^H=r_conc_s} $p^H=r\conc s$ where $s\sub\gamma$,
 so $\pi(p^H)=q^M\conc\pi(s)=q^M\conc s<p^M$.
\end{enumerate}

So in either case $\pi(p^H)<p^M$,
so by line (\ref{eqn:solid_univ_induction}), $H$ is $(k+1)$-solid and $(k+1)$-universal.

We have verified parts \ref{item:pi(p^H)<p^M} and
\ref{item:H_solid_and_univ}.
Now consider part \ref{item:q^H=r_and_p^H=q^H_or_q^H_conc_s}. It is enough to see that either $p^H=r$ or $p^H=r\conc s$ for some $s\sub\gamma$, since this implies $q^H=r$, as we already know that $H$ is $(k+1)$-solid.
And since
\ref{item:p^H_leq_r} or \ref{item:p^H=r_conc_s} above holds,  it therefore suffices to see that $p^H\not<r$. So
suppose $p^H<r$. Let $C=\core(H)$. First note that since $H<^{\para}_kM$, also $C<^{\para}_kM$, $C$ is $(k+1)$-sound and $\rho^C=\rho^H$.

Now we will show that $C\in M$
and that $H$ is finitely generated above its projectum; that is, \begin{equation}\label{eqn:H_fin_proj_gend} H=\Hull^H(\rho^H\un\{x\})\text{ for some }x\in\core_0(H).\end{equation}
Given this, as in the proof of Claim
\ref{clm:p^Mbar=pbar},
we can use Lemma \ref{lem:finite_gen_hull} to deduce that
 $H\in M$.

So, since $p^H<r$, we have $\pi(p^H)<q^M$, and because $(M,q^M)$ is solid and
$\rho^H\leq\gamma=\crit(\pi)$,
therefore $C\in M$.

Now let us establish line (\ref{eqn:H_fin_proj_gend}).
Since $C\in M$
and $M|\delta\ins C$,
we have $\rho^H=\rho^C\geq\delta$. Also $M||\gamma\ins H$ and clearly $\rho^H\leq\gamma$.
So $\rho^C=\rho^H\in\{\delta,\gamma\}$
and
 if $\gamma$ is an $M$-cardinal then $\rho^H=\gamma=\delta$.
 But if $\rho^H=\gamma$ then line (\ref{eqn:H_fin_proj_gend}) holds, as witnessed by $x=(r,\pvec_k^H)$, so suppose
$\rho^H=\delta<\delta^{+H}=\gamma$.
Then since $H$ is $(k+1)$-universal, we have
$\gamma\sub\Hull^H((\delta+1)\un \{\pvec_{k+1}^H\})$ (note the hull is
uncollapsed), and
line (\ref{eqn:H_fin_proj_gend}) follows, as desired.

Part \ref{item:rho^M<rho^H}:
 Since $\crit(\pi)\geq\rho^M$,  easily $\rho^H\geq\rho^M$. So suppose $\rho=\rho^H=\rho^M$. Let
 \[ t=\Th_{\rSigma_{k+1}}^H(\rho\cup\{p^H,\pvec_k^H\})=
  \Th_{\rSigma_{k+1}}^M(\rho\cup\{\pi(p^H),\pvec_k^M\}).
 \]
Since $\pi(p^H)<p^M$
(by Claim \ref{clm:pi(p_k+1^H)<p_k+1^M}),
  $t\in M$, so $C=\core(H)\in M$.
 Since $H$ is $(k+1)$-universal,
 $C$ codes a surjection $\rho\to H\cap\pow(\rho)$, so it follows that $\rho^{+H}<\rho^{+M}$.
 Since $H|\delta=M|\delta$ and $\delta$ is an $M$-cardinal, therefore $\delta\leq\rho$, so $\gamma\in\{\rho,\rho^{+H}\}$.
 So if $\gamma=\rho$
 then then since $\pi(r)<p^M$,
 we have $H\in M$ by the minimality of $p^M$.
 And if $\gamma=\rho^{+H}$
 then since $H$ is $(k+1)$-universal, $H=\Hull^H(\rho\cup\{r,\alpha,\pvec_k^H\})$
 for some $\alpha<\gamma$,
 and since $\pi(r\conc\left<\alpha\right>)<p^M$, the minimality of $p^M$ again yields that $H\in M$.\footnote{An alternative to these last two sentences would be to argue via Lemma \ref{lem:finite_gen_hull} as in the proof of parts \ref{item:pi(p^H)<p^M}--\ref{item:q^H=r_and_p^H=q^H_or_q^H_conc_s},
 but  the argument provided is simpler.}
\end{proof}

 In the following definition we abstract out the key properties of $B$ that we have established so far, and which will also hold for the kinds of iterates $B'$ to be considered:

\begin{dfn}\label{dfn:relevant_B}
A structure
$B'=(\delta',\gamma',\pi',H',M')$
is a \dfnemph{pre-relevant bicephalus of degree $k$ and length $n$}
iff:
\begin{enumerate}[label=--]
 \item $M'$ is a $k$-sound pm,
$\lh(p^{M'})\geq n$,
 $(M',q^{M'})$ is
solid, $\delta'=\card^{M'}(\gamma')$ and
\[ M'=\Hull^{M'}((\gamma'+1)\cup\{q^{M'},
\pvec_k^{M'}\}),\]
 \item $H'$ is a $k$-sound pm,
$\lh(p^{H'})\geq n$,
 $(H',q^{H'})$ is solid, $\gamma'$ is an $H'$-cardinal,
 $\gamma'\leq\min(q^{H'})$ and
\[ H'=\Hull^{H'}(\gamma'\cup\{q^{H'},\pvec_k^{H'}\}),\]
\item $\pi':\core_0(H')\to\core_0(M')$
 is a $k$-embedding, $\crit(\pi')=\gamma'$ and  $\pi(q^{H'})=q^{M'}$.
\end{enumerate}

Let $B'$ be a pre-relevant bicephalus with notation as above.
We write $\delta^{B'}=\delta'$ and $\gamma^{B'}=\gamma'$ etc.
Note that
$H'=\cHull^{M'}(\gamma'\cup\{q^{M'},\pvec_k^{M'}\})$
and $\pi'$ is the uncollapse  map.
Note that $q^{M'}\ins p^{M'}$ and either
\begin{enumerate}[label=(\roman*)]
 \item\label{item:relevant_bicephalus}
$q^{M'}\conc\left<\gamma\right>\ins p^{M'}$,
or
\item$p^{M'}< q^{M'}\conc\left<\gamma'\right>$.
\end{enumerate}
We say that $B'$ is \dfnemph{relevant}
iff \ref{item:relevant_bicephalus}  holds.
\end{dfn}

\begin{dfn}\label{dfn:ult_of_relevant_bicephalus}
Let $B'$ be a pre-relevant bicephalus and $E$ be a short extender.
We say that $E$ is \dfnemph{weakly amenable to $B'$}
iff  $\crit(E)<\delta'$ and $E$ is weakly amenable to
$B'|\delta^{B'}=H'|\delta^{B'}=M'|\delta^{B'}$.
We (attempt to) define
\[
\Ult(B',E)=(\widetilde{\delta},\widetilde{\gamma},
\widetilde{\pi},\widetilde{H},\widetilde{M})\]
as follows. We set $\widetilde{M}=\Ult_k(M',E)$;
suppose this is wellfounded.
Letting $j=i^{M',k}_E$,
then $(\widetilde{\delta},\widetilde{\gamma})=j(\delta',\gamma')$,
\[
\widetilde{H}=\cHull^{\widetilde{M}}(\widetilde{\gamma}
\cup\{q^{\widetilde{M}},
\pvec_k^{\widetilde{M}}\}),\]
and $\widetilde{\pi}:\widetilde{H}\to\widetilde{M}$
is the uncollapse. Define $i^{B'}_E:H'\to\widetilde{H}$
as
$i^{B'}_E=\widetilde{\pi}^{-1}\com j\com\pi'$.
\end{dfn}
\begin{dfn}\label{dfn:abstract_relevant_B}
Let $B'$ be a relevant bicephalus of degree $k$.
An \dfnemph{abstract degree $k$ weakly amenable iteration}
of $B'$ is the obvious analogue of Definition \ref{dfn:abstract_iteration}:
a pair
$\left(\left<E_\alpha\right>_{\alpha<\lambda},\left<B_\alpha\right>_{
\alpha\leq\lambda}\right)$
where $B_0=B'$, $B_\alpha$ is a relevant bicephalus
for each $\alpha<\lambda$,
each $E_\alpha$ is a short extender weakly amenable to $B_\alpha$
(so $\crit(E_\alpha)<\delta^{B_\alpha}$),
$B_{\alpha+1}=\Ult(B_\alpha,E_\alpha)$,
and $B_\eta$ is the resulting direct limit when $\eta\leq\lambda$
is a limit. \dfnemph{Wellfoundedness} of the iteration requires
that $B_\lambda$ is wellfounded.\end{dfn}

\begin{lem}\label{lem:ult_relevant_B_props}
Continuing as in
Definition \ref{dfn:ult_of_relevant_bicephalus}, if $\widetilde{M}$ is
wellfounded then:
\begin{enumerate}[label=\arabic*.,ref=\arabic*]
\item\label{item:j_is_k-emb_etc} $j$ is a $k$-embedding with
$j(q^{M'})=q^{\widetilde{M}}$.
\item\label{item:H_in_M_iff_Ht_in_Mt} $H'\in M'$ iff
$\widetilde{H}\in\widetilde{M}$.
\item $\widetilde{B}$ is a pre-relevant bicephalus.
\item $\widetilde{\pi}$, $i^{B'}_E$ are $k$-embeddings,
$i^{B'}_E\rest(\gamma')^{+H'}\sub j$
and $\widetilde{\pi}\com i^{B'}_E=j\com\pi'$.
\item\label{item:iteration_pres_relevance_B} If $j$ is $p_{k+1}$-preserving and $B'$ is relevant
then $\widetilde{B}$ is relevant.
\end{enumerate}
Likewise for abstract degree $k$ weakly amenable
iterations of $B'$, with $B_\lambda$
replacing $\widetilde{B}$, $j_{0\lambda}$ replacing $j$, etc.\end{lem}
\begin{proof}
Part \ref{item:j_is_k-emb_etc} is completely routine.
Part \ref{item:H_in_M_iff_Ht_in_Mt} is a consequence
of this and $(z,\zeta)$-preservation
(Fact \ref{fact:z,zeta_pres}). That is, for the more subtle direction, suppose that $H'\notin M'$.
Then since $(M',q^{M'})$ is $(k+1)$-solid, we have $q^{M'}\ins z^{M'}$ and
$(z^{M'},\zeta^{M'})\leq(q^{M'},\gamma')$.
So $(z^{\widetilde{M}},\zeta^{\widetilde{M}})=(j(z^{M'}),\sup j``\zeta^{M'})\leq(j(q^{M'}),j(\gamma'))$,
but $j(q^{M'})=q^{\widetilde{M}}$ and $(\widetilde{M},q^{\widetilde{M}})$ is $(k+1)$-solid,
so
\[ \Th_{\rSigma_{k+1}}^{\widetilde{M}}(j(\gamma')\cup\{q^{\widetilde{M}},\pvec_k^{\widetilde{M}}\})\notin\widetilde{M},\]
so $\widetilde{H}\notin \widetilde{M}$.

For the remaining parts, one should first prove everything other than
the fact that $i\rest(\gamma')^{+H'}\sub j$, where $i=i^{B'}_E$,
and we leave those first parts to the reader and assume them.
Let us now deduce that $i\rest(\gamma')^{+H'}\sub
j$. We first show $i(\gamma')=j(\gamma')$.
Note that
\[ \pi'(\gamma')=\text{ the least
}\gamma^*\in\Hull^{M'}\big(\gamma'\cup\{q^{M'},\pvec_k^{M'}\}\big)\cut\gamma'.\]
We just need to see that
\[ j(\pi'(\gamma'))=\text{ the least
}\gamma^*\in\Hull^{\widetilde{M}}\big(j(\gamma')\cup\{q^{\widetilde{M}},
\pvec_k^{\widetilde{M}}\}\big)\cut j(\gamma').\]
But supposing that $\gamma^*\in[j(\gamma'),j(\pi'(\gamma')))$ is also in that hull,
then the existence of such a $\gamma^*$ is an
$\rSigma_{k+1}^{\widetilde{M}}$
assertion
of the parameter $j(q^{M'},\gamma',\pvec_k^{M'})$,
hence pulls back to $M'$, a contradiction.

So $i(\gamma')=j(\gamma')$. But then it is easy to deduce that
$i\rest\pow(\gamma')\cap H'\sub j$, using that
$\widetilde{\pi}\com i=j\com\pi'$ and $\crit(\pi')=\gamma'$
and $\crit(\widetilde{\pi})=j(\gamma')$. It follows
that $i\rest(\gamma')^{+H'}\sub j$.
\end{proof}

\begin{rem}\label{rem:abstract_iter_relevant_B_props}
Continuing as in Definition \ref{dfn:abstract_relevant_B},
suppose that $i^{M_\alpha,k}_{E_\alpha}$ is $p_{k+1}$-preserving
for each $\alpha<\lambda$,
and $M_\lambda$
is wellfounded.
Then note that $B_\lambda$ is also a relevant
bicephalus and
the maps
$j_{\alpha\lambda}:M_\alpha\to M_\lambda$
are $p_{k+1}$-preserving $k$-embeddings,
and the analogue of Lemma \ref{lem:ult_relevant_B_props} holds.\end{rem}

We now define the kind of iteration tree on $B$ which we will use for comparison: a \emph{degree-maximal iteration tree}.
These are analogous to those in \S\ref{sec:bicephali}, but there are some key
differences:
when forming an ultrapower  $\widetilde{B}=\Ult(B',E)$ of a bicephalus
and the associated maps $i:H'\to \widetilde{H}$ and $j:M'\to \widetilde{M}$,
we follow  \ref{dfn:ult_of_relevant_bicephalus},
and when $\delta^{B'}<\gamma^{B'}$ and $\crit(E)=\delta^{B'}$,
then we do not form a bicephalus, but we need to be careful about how to
proceed:
 in condition
in
\ref{dfn:tree_on_rel_B}(\ref{item:losing_bicephalus})\ref{item:crit=delta} below
we apply the extender $E_\alpha$ to some $Q\pins M_\beta$,
although one might have considered applying it to $H_\beta$.
One further difference, the analogue of the anomalous case
in phalanx iterations (see Footnote \ref{ftn:anomalous_footnote}), is that we can have non-premice
appearing in the tree (they arise in the situation
just mentioned, if $Q$ is type 3
with $\rho_0^Q=\delta_\beta$);
thus, we only say \emph{pre-ISC-}premouse
in condition \ref{dfn:tree_on_rel_B}(\ref{item:seg-premice}).

\begin{dfn}\label{dfn:tree_on_rel_B}
Let $B'$ be a degree $k$ relevant bicephalus.
A \dfnemph{degree-maximal iteration tree $\Tt$ on $B'$ of length $\lambda$}
is a system
\[
\Tt=\left({<^\Tt},\curlyB,
\dropset,\dropset_{\deg},\left<B_\alpha,\gamma_\alpha,\delta_\alpha,
\pi_\alpha,\sides_\alpha\right>_ {
\alpha<\lambda},
\left<M^{e}_{\alpha},\deg^e_\alpha,i^e_{\alpha\beta}\right>_{\alpha\leq
\beta<\lambda\text{ and }e<2}, \right. \]
\[ \left.\left<\exitside_\alpha,\exit_\alpha,E_\alpha,
B^*_{\alpha+1}\right>_{\alpha+1<\lambda},
\left<M^{e*}_{\alpha+1},
i^{e*}_{\alpha+1}\right>_{\alpha+1<\lambda\text{ and }e< 2}\right), \]
with the following properties
for all $\alpha<\lambda$,
where we write $H_\alpha=M^0_\alpha$,
$M_\alpha=M^1_\alpha$, $i_{\alpha\beta}=i^0_{\alpha\beta}$,
$j_{\alpha\beta}=i^1_{\alpha\beta}$, etc:
\begin{enumerate}[label=\arabic*.,ref=\arabic*]
  \item  ${<^\Tt}$ is an iteration tree order on $\lambda$,
 \item $\dropset\sub\lambda$ is the set of \dfnemph{dropping nodes},
 \item $\emptyset\neq\sides_\alpha\sub\{0,1\}$,
 \item $\alpha\in\mathscr{B}$
 iff $\sides_\alpha=\{0,1\}$,
\item\label{item:curly_B_closed_downward} $0\in\curlyB\sub\lambda$
and $\curlyB\cap[0,\alpha]^\Tt$
is a closed initial segment of $[0,\alpha]^\Tt$,
\item $B_0=B'$ and $(\deg^0_0,\deg^1_0)=(k,k)$,
\item If
$\alpha\in\curlyB$
then $B_\alpha=(\gamma_\alpha,\delta_\alpha,\pi_\alpha,
H_\alpha,M_\alpha)$
is a degree $k$ relevant bicephalus
and $(\deg^0_\alpha,\deg^1_\alpha)=(k,k)$.
\item\label{item:seg-premice} If $\alpha\notin\curlyB$
and $\sides_\alpha=\{e\}$
then $B_\alpha=M^e_\alpha$ is a $\deg^e_\alpha$-sound pre-ISC-premouse,
and $M^{1-e}_\alpha=\emptyset$,
\item if $\sides_\alpha=\{0\}$ then $B_\alpha=H_\alpha$ is a premouse,
\item If $\alpha+1<\lambda$
then  $e=\exitside_\alpha\in\sides_\alpha$, $\exit_\alpha\ins
M^e_\alpha$, and  $E_\alpha=F^{\exit_\alpha}\neq\emptyset$.
\item If $\alpha+1<\beta+1<\lh(\Tt)$ then $\lh(E_\alpha)\leq\lh(E_\beta)$.
\item\label{item:successor_step} Suppose $\alpha+1<\lh(\Tt)$
and let $\beta=\pred^\Tt(\alpha+1)$.
Then:
\begin{enumerate}
\item\label{item:solidity_biceph_tree_predecessor_value} $\beta$ is the least $\beta'$
such that $\crit(E_\alpha)<\exchnu(\exit_{\beta'})$.\footnote{Recall from Remark \ref{rem:superstrong_diff} that for an active pre-ISC-pm $S$,
$\exchnu(S)=\max(\nu(S),\lgcd(S))$.}
\item $\sides_{\alpha+1}\sub\sides_\beta$
\item $\alpha+1\in\curlyB$ iff
$\big[\beta\in\curlyB$ and $\crit(E_\alpha)<\delta_\beta$
and $E_\alpha$ is $B_\beta$-total$\big]$.
\item If $\alpha+1\in\curlyB$ then $B^{*}_{\alpha+1}=B_\beta$
and
$B_{\alpha+1}=\Ult(B_\beta,E_\alpha)$
and $i^{*}_{\alpha+1},j^{*}_{\alpha+1}$ are the associated
maps (all defined as in \ref{dfn:ult_of_relevant_bicephalus}).

\item\label{item:losing_bicephalus} if $\beta\in\curlyB^\Tt$ but
$\alpha+1\notin\curlyB^\Tt$
then:
\begin{enumerate}[label=(\roman*)]
\item\label{item:crit=delta} if $\delta^B<\gamma^B$
and $\exitside_\beta=0$ and $\gamma_\beta\leq\lh(E_\beta)$
and $\delta_\beta=\crit(E_\alpha)$, then
$\sides_{\alpha+1}=\{1\}$, and
\item otherwise,
$\sides_{\alpha+1}=\{\exitside_\beta\}$.
\end{enumerate}
\item If $\sides_{\alpha+1}=\{e\}$
then
 $M^{e*}_{\alpha+1}\ins M^e_\beta$ and $d=\deg^e_{\alpha+1}$
are determined as usual for degree-maximality
(with $d\leq k$ if
$(0,\alpha+1]^\Tt\cap\dropset=\emptyset$),
\[ M^e_{\alpha+1}=\Ult_{d}(M^{e*}_{\alpha+1},E_\alpha), \]
and $i^{e*}_{\alpha+1}$ is the ultrapower map.
Here if $M^{e*}_{\alpha+1}$ is type 3 with largest cardinal
$\crit(E_\alpha)$, then $d=-1$, so the ultrapower is just that formed using functions in $M^{e*}_{\alpha+1}$,  without squashing; see \S\ref{sec:notation_extenders_and_ultrapowers} and Definition \ref{dfn:anomalous} below. Also if $\deg^e_\beta=-1$
and $\alpha+1\notin\mathscr{D}$
then $d=-1$.
\end{enumerate}
\item The remaining objects are determined
as usual, with direct limits at limit $\eta$, so
$H_\eta,M_\eta$ are the direct limits under the iteration maps,
and for $\alpha<^\Tt\eta$, set
$\gamma_\eta=i_{\alpha\eta}(\gamma_\alpha)=j_{\alpha\eta}(\gamma_\alpha)$
and  likewise for $\delta_\eta$,
and $\pi_\eta\com
i_{\alpha\eta}=j_{\alpha\eta}\com\pi_\alpha$.
\end{enumerate}

Note that part of the definition is that for each $\alpha\in\mathscr{B}$,  $B_\alpha$
is a degree $k$ relevant bicephalus.
Also define
$\curlyH^\Tt=\{\alpha<\lambda\bigm|\sides_\alpha=\{0\}\}$
and $\curlyM^\Tt$ likewise but with $1$ instead of $0$.
And for $\alpha<\lh(\Tt)$, define
$\curlyB^\Tt_\downarrow(\alpha)=\max(\curlyB^\Tt\cap[0,\alpha]^\Tt)$.
\end{dfn}

\begin{dfn}\label{dfn:anomalous}
Continue with the notation from
\ref{dfn:tree_on_rel_B}.
As mentioned above, if $\delta<\gamma$
 and $\sides_\beta=\{0,1\}$
 and $\exitside_\beta=\{0\}$
 and $\gamma_\beta\leq\lh(E_\beta)$
 and $\beta=\pred(\alpha+1)$ and $\crit(E_\alpha)=\delta_\beta$,
 then we set $\sides_{\alpha+1}=\{1\}$ (which is important
 as $H_\beta$ need not be $\delta_\beta$-sound),
 and $M^{*}_{\alpha+1}=J_\beta=$ the least $J^*\pins M_\beta$
 such that $\rho_\om^{J^*}=\delta_\beta$ and $\gamma_\beta\leq\OR^{J^*}$.
 Note $J^*=j_{0\beta}(J_0)$. We
 say that $\alpha+1$ is a \dfnemph{mismatched dropping node}
 of $\Tt$, and all nodes $\xi$ such that $\alpha+1\leq^\Tt\xi$
 and $(\alpha+1,\xi]^\Tt\cap\dropset=\emptyset$
 we call \dfnemph{weakly anomalous nodes}.
 If $M'|\gamma'$ is also type 3, we call such nodes $\xi$
 \dfnemph{anomalous nodes}. In case $\alpha+1$ is an anomalous mismatched dropping node,
 $M^{*}_{\alpha+1}=M_\beta|\gamma_\beta$
 and $\deg_{\alpha+1}=-1$ and $\nu(M^*_{\alpha+1})=\delta_\beta=\crit(E_\alpha)$,
 so we  form $M_{\alpha+1}=\Ult_{-1}(M^*_{\alpha+1},E_\alpha)=\Ult(M^*_{\alpha+1},E_\alpha)$
 (see \S\ref{sec:notation_extenders_and_ultrapowers}). And in case  $\zeta=\pred^\Tt(\alpha'+1)$ is anomalous and $\alpha'+1\notin\mathscr{D}$,
 we have $\deg_{\alpha+1}=\deg_\zeta=-1$ and $M_{\alpha'+1}=\Ult_{-1}(M^e_\zeta,E_{\alpha'})$.
 If $\xi$ is anomalous then $M_{\xi}$ is not a premouse,
 as it fails the ISC.
 \end{dfn}
 \begin{rem}\label{rem:simple_ult}
 For anomalous $\xi$,
 if $E_\xi=F(M_\xi)$,
 then $\nu(E_\xi)=\sup_{\zeta<\xi}\nu(E_\zeta)$,
 so it is possible that $\nu(E_\xi)<\lgcd(M_\xi)$,
in which case $\nu(E_\xi)<\exchnu(\exit_\xi)=\lgcd(M_\xi)$.
We use $\lgcd(\exit_\xi)$ as the exchange ordinal in this situation
mainly because it is more convenient in the iterability
proof later.
\end{rem}

\begin{dfn}\label{dfn:relevance-putative}
Let $B'$ be a degree $k$ relevant bicephalus.
A \dfnemph{putative degree-maximal
tree $\Tt$ on $B'$}
is a system satisfying
all of the requirements of an iteration
tree on $B'$,
with all models formed  as in \ref{dfn:tree_on_rel_B},
except that if
$\lh(\Tt)=\alpha+1>1$
then we make no demands on the wellfoundedness of $B^\Tt_\alpha$,
nor its first order properties.
And $\Tt$ is \dfnemph{relevance-putative} iff it is putative
degree-maximal, and if
$\lh(\Tt)=\alpha+1$ then $B_\alpha$ is wellfounded,
and if $\alpha\in\curlyB^\Tt$ then $B_\alpha$ is pre-relevant.
\end{dfn}

Using Lemma \ref{lem:ult_relevant_B_props},
it is straightforward to verify:

\begin{clmfive}\label{clm:solidity_wellfounded_implies_relevance-putative}
If $\Tt$ is a putative degree-maximal tree on $B$ with wellfounded models, then $\Tt$ is relevance-putative.
\end{clmfive}

\begin{dfn}
A $((k,k),\om_1+1)$\dfnemph{-iteration strategy}
for a degree $k$ relevant bicephalus $B'$ is defined using the iteration game
defined with (putative) degree-maximal trees $\Tt$ on $B'$.
If a putative tree is
reached which is not a true
degree-maximal tree, then player I wins.

An \dfnemph{almost $((k,k),\om_1+1)$-iteration strategy} for a degree $k$ relevant bicephalus $B'$ is as above, except that if tree is reached which is relevance-putative
but not a true tree, then player II wins immediately.
\end{dfn}

\begin{rem}\label{rem:end_digression}
We have now completed our digression which began with Definition \ref{dfn:relevant_B} above, and return to the context of the proof of solidity.
Let $B=(\delta,\gamma,\pi,H,M)$
be from there.
Note that $B$  is a relevant bicephalus
(see \ref{dfn:relevant_B}).
As stated earlier, the plan is to form and analyse a comparison of $B$ with $M$, forming a degree-maximal tree on $B$ (see \ref{dfn:tree_on_rel_B}).\footnote{
In the original version of the argument presented
at the M\"unster conference 2015,
all comparison arguments  were
formed between bicephali and themselves
(see \cite{ralf_notes_solidity_talk}).
Afterward,  John Steel suggested that a comparison
between a bicephalus and a premouse might suffice,
and said that he had also been working on related arguments
toward \cite{ACPFMP}.
For the present proof, such a simplification did indeed work out.
The author did not see how to simplify
the proof of projectum-finite generation in this way,
though in that case, the bicephali
and rules for comparison are
simpler anyway.} Before we begin this,
we discuss in Claims \ref{clm:B_is_almost_iterable}--\ref{clm:B_iterable_fs_pres} below the iterability of $B$
and various preservation facts which are essential to the analysis of the
comparison. \end{rem}

\begin{clmfive}[Almost iterability]\label{clm:B_is_almost_iterable}
$B$ is almost $((k,k),\om_1+1)$-iterable.
\end{clmfive}

\begin{rem}\label{rem:copying_features}
The proof is a copying process
much like that used
in the proof of condensation
 from normal iterability,
\cite[Theorem 5.2]{premouse_inheriting}.
In order to first focus the more novel aspects of the solidity proof, we postpone the proof of Claim \ref{clm:B_is_almost_iterable} to \S\ref{sec:appendix}. In the  proof of Claim \ref{clm:p_k+1_pres} below, we will need the following details from the copying process:
Let $\Tt$ be a  putative degree-maximal
tree on $B$, of finite length $\alpha+1<\om$,
with $\alpha\in\curlyB^\Tt\cup\curlyM^\Tt$.
Then we will have a $k$-maximal
tree $\Uu$ on $M$ with $\lh(\Uu)=\alpha+1$,
such that $(0,\alpha]^\Uu\cap\dropset^\Uu_{\deg}=\emptyset$, and a $k$-embedding $\sigma:M^\Tt_\alpha\to M^\Uu_\alpha$
such that $\sigma\com j^\Tt_{0\alpha}=i^\Uu_{0\alpha}$.
The almost iterability proof is self-contained and can be read directly at this stage,
so the reader who prefers to proceed more linearly through the logic should make a detour there now.
\end{rem}

We next discuss closeness of extenders
to their target models in trees $\Tt$ on $B$. We restrict to  our particular $B$, as opposed to dealing with an arbitrary relevant bicephalus, because we will use Claim \ref{clm:not_1_2_with_gamma_between}
to rule out certain cases in which closeness would otherwise not obviously hold.

\begin{clmfive}[Closeness]\label{clm:solidity_closeness} Let $\Tt$ be any putative degree-maximal tree on $B$.
Let $\xi+1<\lh(\Tt)$ and
$\beta=\pred^\Tt(\xi+1)$. Then:
 \begin{enumerate}[label=\arabic*.,ref=\arabic*]
  \item If $1\in\sides^\Tt_{\xi+1}$ then $E^\Tt_\xi$
  is close to $M^{*\Tt}_{\xi+1}$.\footnote{However,
if $\sides^\Tt_{\xi+1}=\{0,1\}$,
then it is not relevant whether $E^\Tt_\xi$ is
close
to $H^\Tt_\beta$, as recall that $H^\Tt_{\xi+1}$
is not formed in general as
$\Ult_k(H^\Tt_\beta,E^\Tt_\xi)$, but as a certain hull of $M^\Tt_{\xi+1}$.}
  \item If $\sides^\Tt_{\xi+1}=\{0\}$
  then $E^\Tt_\xi$
  is close to $H^{*\Tt}_{\xi+1}$.
 \end{enumerate}
\end{clmfive}

The proof of closeness is very close to the usual one (see \cite[6.1.5]{fsit}), so it is also postponed to \S\ref{sec:appendix}.

\begin{clmfive}\label{clm:p_k+1_pres}
Let $\Tt$ be any relevance-putative
degree-maximal tree on $B$.
Let $\alpha\in\curlyB^\Tt\cup\curlyM^\Tt$
be such that $(0,\alpha]^\Tt\cap\dropset_{\deg}^\Tt=\emptyset$.
Then:
\begin{enumerate}[label=\arabic*.,ref=\arabic*]
 \item\label{item:Tt_is_degree-max} $\Tt$ is degree-maximal.
\item\label{item:j^Tt_p-pres} $\rho^{M^\Tt_\alpha}=\sup
i^\Tt_{0\alpha}``\rho^M\leq\delta^\Tt_\alpha$ and
 $j^\Tt_{0\alpha}$ is $p_{k+1}$-preserving.
 \item\label{item:rho^H_alpha>rho^M_alpha} If $\alpha\in\curlyB^\Tt$ then:
\begin{enumerate}
\item\label{item:H^Tt_alpha_is_gamma^Tt_alpha-sound}
$\rho^{H^\Tt_\alpha}\leq\gamma^\Tt_\alpha$,
 $H^\Tt_\alpha$ is $\gamma^\Tt_\alpha$-sound
and
$q^{H^\Tt_\alpha}=p^{H^\Tt_\alpha}\cut\gamma^\Tt_\alpha=i^\Tt_{0\alpha}
(p^H\cut\gamma)=i^\Tt_{0\alpha}(q^H)$, \footnote{If $\Tt$
is finite, then we also get that $H^\Tt_\alpha$
is $(k+1)$-solid and $(k+1)$-universal, like in the proof of the earlier parts.
But the author is not sure whether one can show that $H^\Tt_\alpha$
is $(k+1)$-solid and $(k+1)$-universal for infinite $\Tt$.}
\item\label{item:inner_rho^H_alpha>rho^M_alpha}
$\rho^{M^\Tt_\alpha}<\rho^{H^\Tt_\alpha}$,
\item\label{item:when_alpha_in_B^Tt_M^Tt_alpha_non-solid}  $H^\Tt_\alpha\notin M^\Tt_\alpha$ and $M^\Tt_\alpha$ is non-solid.
\end{enumerate}
 \end{enumerate}
\end{clmfive}

\begin{proof}
Part \ref{item:j^Tt_p-pres}: The proof is similar to that of
Lemma \ref{lem:p-pres_for_premouse_from_para_order},
but now using the Closeness Claim \ref{clm:solidity_closeness} for part of the
argument.
Let $N=M^\Tt_\alpha$
and $j=j^\Tt_{0\alpha}:M\to N$.
By Claim \ref{clm:solidity_closeness}, all extenders applied along the
branch
from $M$ to $N$ are close to their target model,
so together with Corollary \ref{cor:basic_fs_pres}
(and without considering
$<^{\para}_k$)
we have:
\begin{enumerate}[label=\tu{(}\alph*\tu{)}]
 \item\label{item:rho_p_sub_pres_biceph}
$\rho^{N}=\sup j``\rho^M$
and $p^{N}\leq j(p^M)$,
\item\label{item:solidity_equiv_biceph}
 $N$ is non-$(k+1)$-solid (as $M$ is non-$(k+1)$-solid as $H\notin M$).
\end{enumerate}
So it suffices to see that $j(p^M)=p^N$, so suppose $p^{N}<j(p^M)$.

Suppose that $\lh(\Tt)<\om$.
Then by Remark \ref{rem:copying_features} and as $\alpha\in\curlyB^\Tt\cup\curlyM^\Tt$,
we have a $k$-maximal tree $\Uu$  on $M$
such that $\lh(\Uu)<\om$ and $b^\Uu$ does not drop in model or degree, and
we have a
 $k$-embedding $\sigma:N\to M^\Uu_\alpha$ such that $\sigma\com j=i^\Uu_{0\alpha}$.
 Therefore $\sigma(j(p^M))=i^\Uu_{0\alpha}(p^M)$.

Since $p^N<j(p^M)$, we have $\sigma(p^N)<\sigma(j(p^M))=i^\Uu_{0\alpha}(p^M)$, so
 $N<_k^{\para}M$,
as witnessed by
$\Uu$
and $\sigma:N\to M^\Uu_\infty$.
So by our global inductive hypothesis
(line (\ref{eqn:solid_univ_induction})),
$N$ is $(k+1)$-solid,
contradicting \ref{item:solidity_equiv_biceph}.

So $\lh(\Tt)\geq\om$.
But like done in the proof of Lemma
\ref{lem:p-pres_for_premouse_from_para_order},
we can build a finite length tree $\bar{\Tt}$ capturing
$(\Tt,p^N)$, and this leads to a contradiction just as there.
(Define \emph{capturing} for
such trees like in
\ref{dfn:captures}, and construct $\bar{\Tt}$ via
a finite support argument like in \ref{lem:simple_embedding_exists}. We leave the details to the reader.)

Part \ref{item:Tt_is_degree-max}:
Since $\gamma\in p^M$, this is an immediate consequence of part \ref{item:j^Tt_p-pres}.

Part \ref{item:H^Tt_alpha_is_gamma^Tt_alpha-sound}:
This follows easily from the $\gamma$-soundness
of $H$ (use
preservation
of generalized solidity witnesses under  (near) $k$-embeddings).

Part \ref{item:inner_rho^H_alpha>rho^M_alpha}:
Write $B'=(\delta',\gamma',\pi',H',M')=B^\Tt_\alpha$,
and $i=i^\Tt_{0\alpha}$ and $j=j^\Tt_{0\alpha}$.
Using properties of $\pi'$ and that $j$ is
$p_{k+1}$-preserving, it is easy to reduce
to the case that
\begin{equation}\label{eqn:rho^H'=rho^M'=delta'}
\rho^{H'}=\rho^{M'}=\delta'<
\gamma'<(\rho^{M'})^{+M'},\end{equation}
so assume this. We will show that $H\in M$,
a contradiction.

Note that either $p^{H'}=q^{H'}$ or $p^{H'}=q^{H'}\conc\left<\beta\right>$ for some $\beta\in[\delta',\gamma')$.
And line (\ref{eqn:rho^H'=rho^M'=delta'})
and part \ref{item:j^Tt_p-pres} imply together
that $\rho^M=\delta<\gamma$
and $i,j$ are continuous at $\delta$.

Let $\bar{\Tt}$ be a finite tree on $B$
capturing
$(\Tt,\{\delta',\gamma',p^{H'}\},\{\delta',\gamma',p^{M'}\})$,
meaning here in particular that $\lh(\bar{\Tt})=\bar{\alpha}+1$
and $\bar{\alpha}\in\curlyB^{\bar{\Tt}}$
and  letting
$\bar{B}=B^{\bar{\Tt}}_{\bar{\alpha}}=(\bar{\delta},\bar{\gamma},\bar{\pi},\bar{
H},\bar{M})$,
we have $k$-embedding capturing maps
$\sigma:\bar{H}\to H'$
and $\tau:\bar{M}\to M'$ with
$\delta',\gamma',p^{H'}\in\rg(\sigma)$
and $\delta',\gamma',p^{M'}\in\rg(\tau)$
and $\pi'\com\sigma=\tau\com\bar{\pi}$,
and the capturing maps commute with the iteration maps. Note $\sigma(\bar{\delta})=\tau(\bar{\delta})=\delta'=\rho^{H'}=\rho^{M'}$.  Let $\bar{\rho}=\bar{\delta}$.
Let
$\sigma(\bar{p})=p^{H'}$. Let
 $\bar{i}=i^{\bar{\Tt}}_{0\bar{\alpha}}$
 and $\bar{j}=j^{\bar{\Tt}}_{0\bar{\alpha}}$.

Now  $\bar{H}<^{\para}_kM$. For
$\bar{\pi}(p^{\bar{H}})<p^{\bar{M}}$,
since $\bar{j}$ is $p_{k+1}$-preserving
and considering the hull of $\bar{M}$ that forms $\bar{H}$.
So we get $\bar{H}<^{\para}_kM$
by lifting $\bar{\Tt}$
to a tree $\bar{\Uu}$ on $M$ as in part \ref{item:j^Tt_p-pres} (again using Remark \ref{rem:copying_features}).

Since $\bar{H}<^{\para}_kM$, $\bar{H}$ is $(k+1)$-solid
and $(k+1)$-universal.

Recall $\sigma(\bar{p})=p^{H'}$
and $\sigma(\bar{\rho})=\delta'$. Let
\[
\bar{t}=\Th_{\rSigma_{k+1}}^{\bar{H}}(\bar{\rho}\cup\{\bar{p},\pvec_k^{\bar{H}}
\}).\]
Then $\bar{t}\notin\bar{H}$, because otherwise
$\sigma(\bar{t})\in H'$, and then from $\sigma(\bar{t})$ we can
recover the theory
\[ t' = \Th_{\rSigma_{k+1}}^{H'}(\rho'\cup\{p^{H'},\pvec_k^{H'}\}),\]
but $t'\notin H'$.
So $\rho^{\bar{H}}\leq\bar{\rho}$,
but by part \ref{item:j^Tt_p-pres} and since $j,\bar{j}$
are continuous at $\delta$,
we have $\rho^{\bar{M}}=\bar{\rho}=\bar{\delta}$,
which implies $\rho^{\bar{H}}\geq\bar{\rho}$.
So $\rho^{\bar{H}}=\bar{\rho}$,
and $p^{\bar{H}}\leq\bar{p}$.

By $(k+1)$-universality for $\bar{H}$,
it follows that there is $\beta<\bar{\gamma}$
such that
\[\bar{H}=\Hull^{\bar{H}}(\bar{\rho}\cup\{\beta,q^{\bar{H}},\pvec_k^{\bar{H}}\}
).\]
But $\bar{\pi}(q^{\bar{H}}\conc\left<\beta\right>)<p^{\bar{M}}$,
and it follows that $\bar{H}\in\bar{M}$,
so $H\in M$ by Lemma \ref{lem:ult_relevant_B_props}, as desired.

Part \ref{item:when_alpha_in_B^Tt_M^Tt_alpha_non-solid}
 follows from
  Lemma \ref{lem:ult_relevant_B_props},
 since $H\notin M$.
\end{proof}

\begin{clmfive}\label{clm:solidity_full_iterability}
There is a $((k,k),\om_1+1)$-iteration strategy for $B$.
Moreover, every almost $((k,k),\om_1+1)$-strategy for $B$ is an (actual) $((k,k),\om_1+1)$-strategy.
\end{clmfive}
\begin{proof}
This is an immediate consequence
of Claims \ref{clm:B_is_almost_iterable}
and \ref{clm:p_k+1_pres}.
\end{proof}

We now summarize the fine structural properties of iterates $H^\Tt_\alpha$ and $M^\Tt_\alpha$ when $\alpha\notin\curlyB^\Tt$. This complements Claim \ref{clm:p_k+1_pres}:

\begin{clmfive}\label{clm:B_iterable_fs_pres}
 Let $\Tt$
 be any degree-maximal tree on $B$. Then:
 \begin{enumerate}[label=\arabic*.,ref=\arabic*]
 \item\label{item:alpha_notin_curlyB_fs_pres} Suppose
$\alpha\notin\curlyB^\Tt$ and
$(0,\alpha]^\Tt\cap\dropset_{\deg}^\Tt=\emptyset$
 and let $\beta=\curlyB^\Tt_\downarrow(\alpha)$. Then:
 \begin{enumerate}
 \item Suppose $\alpha\in\curlyH^\Tt$. Then:
\begin{enumerate}[label=--]
\item $H^\Tt_\alpha$ is a $k$-sound premouse,
\item $\rho^{H^\Tt_\alpha}=\rho^{H^\Tt_\beta}\leq\gamma^\Tt_\beta
\leq\crit(i^\Tt_{\beta\alpha})$,
\item $H^\Tt_\beta$ is the
$\gamma^\Tt_\beta$-core
 of $H^\Tt_\alpha$ and
$i^\Tt_{\beta\alpha}$ is the $\gamma^\Tt_\beta$-core
map,
\item $i^\Tt_{\beta\alpha}$ is $p_{k+1}$-preserving.\footnote{But
the author does not know whether $i^\Tt_{0\alpha},i^\Tt_{0\beta}$
are $p_{k+1}$-preserving.}
\end{enumerate}
 \item Suppose $\alpha\in\curlyM^\Tt$. Then:
\begin{enumerate}[label=--]
\item $M^\Tt_\alpha$ is a $k$-sound premouse,
\item $\rho^{M^\Tt_\alpha}=\rho^{M^\Tt_\beta}
\leq\delta^\Tt_\beta\leq\crit(j^\Tt_{\beta\alpha})$,
\item $M^\Tt_\beta$ is the $\delta^\Tt_\beta$-core of $M^\Tt_\alpha$
and $j^\Tt_{\beta\alpha}$
 is the $\delta^\Tt_\beta$-core map,%\setcounter{footnote}{0}
 \footnote{Here
 the $\delta^\Tt_\beta$-core of $M^\Tt_\alpha$
 just means
$\cHull^{M^\Tt_\alpha}
(\delta^\Tt_\beta\cup\{p^{M^\Tt_\alpha}
\cut\delta^\Tt_\beta,\pvec_k^{M^\Tt_\alpha}\})$;
the terminology does not presuppose any solidity.}

 \item
 $j^\Tt_{0\beta}$, $j^\Tt_{\beta\alpha},j^\Tt_{0\alpha}$
 are $p_{k+1}$-preserving, and
 \item  $M^\Tt_\alpha$ is non-solid.
 \end{enumerate}
 \end{enumerate}
 \item\label{item:M_drop_fs_pres} If $\alpha\in\curlyM^\Tt$
 and $(0,\alpha]^\Tt\cap\dropset^\Tt_{\deg}\neq\emptyset$
 and $\alpha$ is
non-anomalous then letting
 $d=\deg^\Tt(\alpha)$
 and $\xi+1\leq^\Tt\alpha$ be largest such that
$\xi+1\in\dropset_{\deg}^\Tt$, and letting
$\beta=\pred^\Tt(\xi+1)$,
we have:
\begin{enumerate}[label=--]
\item $M^\Tt_\alpha$ is premouse,
and is $d$-sound, $(d+1)$-solid, $(d+1)$-universal, but fails to be
$(d+1)$-sound,
\item $M^{*\Tt}_{\xi+1}=\core_{d+1}(M^\Tt_\alpha)$ is $(d+1)$-sound
and $\core_{d+1}(M^\Tt_\alpha)\ins M^\Tt_\beta$,
\item $j^{*\Tt}_{\xi+1,\alpha}$ is the core map,
\item
$\rho_{d+1}(M^\Tt_\alpha)=\rho_{d+1}(\core_{d+1}(M^\Tt_\alpha))\leq\crit(j^{*\Tt}_{\xi+1,\alpha})$,
\item $j^{*\Tt}_{\xi+1,\alpha}$ is $p_{d+1}$-preserving.
\end{enumerate}

 \item\label{item:H_drop_fs_pres} If $\alpha\in\curlyH^\Tt$
 and $(0,\alpha]^\Tt\cap\dropset^\Tt_{\deg}\neq\emptyset$,
 it is like in part \ref{item:M_drop_fs_pres} (with $H^\Tt_\alpha$, not $M^\Tt_\alpha$).

 \item\label{item:anomalous_fs} Suppose $\alpha$ is anomalous. Then
 it is like in part \ref{item:M_drop_fs_pres}, except that $M^\Tt_\alpha$ is not a premouse.
 \end{enumerate}
\end{clmfive}

\begin{proof}
Let $\Tt$ be any relevance-putative tree of length $\alpha+1$.

Part \ref{item:alpha_notin_curlyB_fs_pres}:
Suppose $\alpha\in\curlyH^\Tt$. We have
$\gamma^\Tt_\beta\leq\crit(i^\Tt_{\beta\alpha})$
because otherwise  letting $\xi+1=\successor^\Tt(\beta,\alpha)$,
we would have
$\crit(E^\Tt_\xi)=\delta^\Tt_\beta$, but then
 $\xi+1$ should in fact be a mismatched dropping node,
so $\xi+1\notin\curlyH^\Tt$.
And by the closeness (Claim \ref{clm:solidity_closeness}),
all the extenders used along the branch
$(\beta,\alpha]^\Tt$
are close to the models to which they apply;
since $H^\Tt_\beta$ is $\gamma^\Tt_\beta$-sound
and $\rho(H^\Tt_\beta)\leq\gamma^\Tt_\beta$,
this suffices. If instead $\alpha\in\curlyM^\Tt$, use
Claims \ref{clm:solidity_closeness}
 and
\ref{clm:p_k+1_pres} (in particular for $p_{k+1}$-preservation),
together with Corollary \ref{cor:basic_fs_pres}.

The remaining parts
follow  in the usual manner from closeness (Claim \ref{clm:solidity_closeness}).
\end{proof}

We are now ready to proceed with the comparison.
We compare $B$ with $M$, defining padded
trees $\solTt,\solUu$ respectively, with $\solTt$ being degree-maximal on $B$
and $\solUu$ being $k$-maximal on $M$. We will also define $S^\solTt,\modelset^\solTt,\movin^\solTt$,
with $\emptyset\neq S^\solTt_\alpha\sub\sides^\solTt_\alpha$;
these bookkeeping devices  are defined completely analogously to those in the proof of Lemma \ref{lem:finite_gen_hull}
(see the paragraph immediately following the proof of Claim \ref{clm:proj-finitely_gend_basic_biceph_iter_pres} of that lemma's proof).
As before, if $S^\solTt_\alpha=\{0,1\}$,
we may \emph{move into} a model of
$\solTt$
at stage $\alpha$, setting $S^\solTt_{\alpha+1}=\{0\}$ or
$S^\solTt_{\alpha+1}=\{1\}$ and
$E^\solTt_\alpha=\emptyset=E^\solUu_\alpha$. And as usual, after selecting models
for potential exit extenders, we minimize
on $\exchnu(E)$ before actually selecting extenders $E$ (see Remark \ref{rem:superstrong_diff}).

Let us describe the rules for forming the comparison (how we move into models and select extenders). We start with $S^\solTt_0=\{0,1\}$ at stage $0$.
At stage $\alpha$,
if $S^\solTt_\alpha\neq\{0,1\}$,
we select extenders as usual (or terminate the comparison). Suppose
$S^\solTt_\alpha=\{0,1\}$.
If $B^\solTt_\alpha|\delta^\solTt_\alpha\neq M^\solUu_\alpha|\delta^\solTt_\alpha$
we select extenders as usual
(and we do not  move into any model in this case).
Suppose $B^\solTt_\alpha|\delta^\solTt_\alpha=
M^\solUu_\alpha|\delta^\solTt_\alpha$.
Then we move into a model
in $\solTt$, setting $E^\solTt_\alpha=\emptyset=E^\solUu_\alpha$.
If  $M^\solTt_\alpha=M^\solUu_\alpha$
then
we move into $H^\solTt_\alpha$ (set $S^\solTt_{\alpha+1}=\{0\}$),
and otherwise we move into $M^\solTt_\alpha$ (set $S^\solTt_{\alpha+1}=\{1\}$).
Of course, we use some $((k,k),\om_1+1)$-strategy
to form $\solTt$, and some $(k,\om_1+1)$-strategy
to form $\solUu$.
This completes the description of the comparison.

We now proceed to the analysis of the comparison. We want to see that
we reach a stage $\beta$ with $\beta\in\mathscr{B}^\solTt$,
and at that stage,
we move into $H^\solTt_\beta$ in $\solTt$,
 $\solTt$ does not use any extenders
after that point, and then the comparison terminates at a stage $\alpha$ shortly after $\beta$,  with $H^\solTt_\beta=H^\solTt_\alpha=M^\solUu_\alpha$.
 We will then analyse the fine structure of $H^\solTt_\alpha,M^\solTt_\alpha$,
 and using the fact that $H^\solTt_\alpha\notin M^\solTt_\alpha$, reach a contradiction.

The first thing to verify is that
the trees $\solTt,\solUu$ are normal.

\begin{clmfive}\label{clm:T,U_normality}
  Let $\zeta\leq\xi\leq\beta<\lh(\solTt,\solUu)$.
  Then:
  \begin{enumerate}[label=\arabic*.,ref=\arabic*]
   \item\label{item:lhs_prior_to_S=0,1} Suppose $\xi<\beta$ and
$S^\solTt_\beta=\{0,1\}$.
Then $\lh(E)\leq\delta^\solTt_\beta$ whenever
$E=E^\solTt_\xi\neq\emptyset$ or $E=E^\solUu_\xi\neq\emptyset$.

\item Suppose $\beta+1<\lh(\solTt,\solUu)$ and  $P,Q,R\neq\emptyset$ with
$P\in\{\exit^\solTt_\zeta,\exit^\solUu_\zeta\}$,
  $Q\in\{\exit^\solTt_\xi,\exit^\solUu_\xi\}$ and
$R\in\{\exit^\solTt_\beta,\exit^\solUu_\beta\}$.
 Then:
 \begin{enumerate}
  \item\label{item:OR^P_leq_OR^Q} $\OR^P\leq\OR^Q$ and $\exchnu^P\leq\exchnu^Q$.
  \item If $\zeta=\xi$ then $\OR^P=\OR^Q$ and $\exchnu^P=\exchnu^Q$.
\item Suppose $\zeta<\xi$ and $\exchnu^P=\exchnu^Q$.
Then $\xi=\zeta+1$
is an anomalous mismatched dropping node,
 $J=M|\gamma$ is active type 3,
  $E^\solTt_\zeta\neq\emptyset$ is superstrong,
  $Q=\exit^\solTt_{\zeta+1}=M^\solTt_{\zeta+1}$, $E^\solUu_{\zeta+1}=\emptyset$,
and if $\zeta+1<\beta$ then $\exchnu^Q<\exchnu^R$.
\end{enumerate}
 \end{enumerate}
\end{clmfive}
\begin{proof}
The proof is by induction on $\beta$.
Limits are easy. Suppose $\beta=\eta+1$ for some $\eta$.

 Part \ref{item:lhs_prior_to_S=0,1}:
 Since $S^\solTt_{\eta+1}=\{0,1\}$, either $E^\solTt_\eta\neq\emptyset$
 or $E^\solUu_\eta\neq\emptyset$. Let $\lambda=\lh(E^\solTt_\eta)$
 or $\lambda=\lh(E^\solUu_\eta)$, whichever is defined.
 If $E^\solTt_\eta\neq\emptyset$ then letting $\chi=\pred^\solTt(\eta+1)$,
 we have $\crit(E^\solTt_\eta)<\delta^\solTt_\chi$,
 so $\lambda\leq\delta^\solTt_{\eta+1}$.
 Suppose $E^\solTt_\eta=\emptyset\neq E^\solUu_\eta$
 and $\delta_{\eta+1}^\solTt<\lambda$. Then
 $S^\solTt_\eta=\{0,1\}$ and $B^\solTt_{\eta}=B^\solTt_{\eta+1}$,
and  $B^\solTt_\eta|\delta^\solTt_\eta=M^\solUu_\eta|\delta^\solTt_\eta$,
so by the rules of comparison, we move into either $H^\solTt_\eta$
or $M^\solTt_\eta$ at stage $\eta$, so
$E^\solTt_\eta=\emptyset=E^\solUu_\eta$
(and $S^\solTt_{\eta+1}\neq\{0,1\}$), a contradiction.

We leave the rest to the reader.\end{proof}

\begin{clmfive}\label{clm:no_matching_exts}
 There is no pair $(\zeta,\xi)$
 such that $\zeta<\lh(\solTt)$ and $\xi<\lh(\solUu)$
 and $E^\solTt_\zeta\rest\nu(E^\solTt_\zeta)=E^\solUu_\xi\rest\nu(E^\solUu_\xi)\neq\emptyset$.
\end{clmfive}
\begin{proof}
 Suppose $(\zeta,\xi)$ is so. As $E^\solUu_\xi\rest\nu(E^\solUu_\xi)$ satisfies the ISC,
 so does $E^\solTt_\zeta\rest\nu(E^\solTt_\zeta)$. So $\zeta$ is non-anomalous, so $\exchnu(\exit^\solTt_\zeta)=\nu(E^\solTt_\zeta)=\nu(E^\solUu_\xi)=\exchnu(\exit^\solUu_\xi)$.
So $\zeta\neq\xi$ by the comparison rules, and using Claim \ref{clm:T,U_normality},
 it follows that $\zeta=\xi+1$
 and $\zeta$ is anomalous,
  contradiction.
\end{proof}

\begin{clmfive}\label{clm:solidity_comparison_termination}
 The comparison terminates at some stage $\alpha<\om_1$.
\end{clmfive}

The proof is just a slight variant of the usual one, and is relegated to \S\ref{sec:appendix}.

Now that we know the comparison terminates,
we will analyze the manner in which it does, and use the properties of $H,M$
(in particular that $H\notin M$)
and the  preservation properties of the iteration maps, to arrive at a contradiction.
Let  $\alpha+1$ be the length of the full comparison $(\solTt,\solUu)$.
So the comparison terminates at stage $\alpha$, and in fact, in the following fashion:

\begin{clmfive}\label{clm:S^Tt_alpha=0}
$S^\solTt_\alpha=\{0\}$ and $H^\solTt_\alpha\ins  M^\solUu_\alpha$.
\end{clmfive}

In the proof and later,
given $\beta\in\curlyB^\solTt$,  write
$\delta^{+\solTt}_\beta=(\delta_\beta^\solTt)^{+M^\solTt_\beta}$
(so
$(\delta_\beta^\solTt)^{+H^\solTt_\beta}\leq\delta^{+\solTt}_\beta$).

\begin{proof}
Suppose $S^\solTt_\alpha=\{0,1\}$. Then $M^\solUu_\alpha\pins
B^\solTt_\alpha|\delta^\solTt_\alpha$. But then $b^\solUu\cap\dropset^\solUu_{\deg}=\emptyset$
and $M^\solUu_\alpha$
is solid, so by Corollary \ref{cor:basic_fs_pres},
$M$ is solid, a contradiction.
If  $S^\solTt_\alpha=\{0\}$ then we can similarly
rule out $M^\solUu_\alpha\pins H^\solTt_\alpha$,
giving the claim.
So suppose
$S^\solTt_\alpha=\{1\}$; we will reach a contradiction.

\begin{sclmfive}
 $M^\solTt_\alpha= M^\solUu_\alpha$, $b^\solTt,b^\solUu$ do not drop in model or degree, and
$j^\solTt,i^\solUu$ are $p_{k+1}$-preserving.\end{sclmfive}
\begin{proof}
If $M^\solUu_\alpha\pins M^\solTt_\alpha$ it is again as before.
If $M^\solTt_\alpha\pins M^\solUu_\alpha$
then
$M^\solTt_\alpha$ is a sound premouse, and
 by Claim \ref{clm:B_iterable_fs_pres},
$b^\solTt\cap\dropset^\solTt_{\deg}=\emptyset$ and
$M$ is solid.
So $M^\solTt_\alpha=M^\solUu_\alpha$.

Suppose both $b^\solTt,b^\solUu$ drop in model or degree. Then  the usual arguments combined with Claim \ref{clm:B_iterable_fs_pres}
yield a contradiction.
(Anomalous extenders are not ``partial'', so they do not interfere with the incompatibility of extenders relevant to this argument. That is, let $\theta+1\in\dropset^\solTt_{\deg}$
and $\theta+1\leq^\solTt\xi+1\leq^\solTt\alpha$.
Then  $\exit^\solTt_\xi$ is a premouse; that is, it is not the case that $\xi$ is anomalous and $E^\solTt_\xi=F(M^\solTt_\xi)$. For suppose otherwise.
Then  $\xi+1\in\curlyB^\solTt$.
For letting  $\chi=\curlyB^\solTt_\downarrow(\xi)$ and $J_\chi=M^\solTt_\chi|\gamma^\solTt_\chi$ and $j:J_\chi\to M^{\solTt}_\xi$ be the iteration map, $J_\chi$ is active type 3 with
 $\nu(F^{J_\chi})=\delta^\solTt_\chi=\crit(j)$,
so $\crit(E^\solTt_\xi)=\crit(F^{J_\chi})<\delta^\solTt_\chi$
and $E^\solTt_\xi$ is $\curlyB^\solTt_\chi$-total, and this implies
that $\pred^\solTt(\xi+1)\leq^\solTt\chi$
and $\xi+1\in\curlyB^\solTt$.)

If one side drops but the other does not, then
$M^\solTt_\alpha=M^\solUu_\alpha$ is
a premouse and is solid,
which again implies $M$ is solid. So neither side drops.
The  $p_{k+1}$-preservation
now follows Lemma \ref{lem:p-pres_for_premouse_from_para_order}
and Claims \ref{clm:p_k+1_pres} and \ref{clm:B_iterable_fs_pres}.
\end{proof}

\begin{sclmfive}
$\alpha\in\curlyB^\solTt$.\end{sclmfive}
\begin{proof}Suppose not, so $\sides^\solTt_\alpha=S^\solTt_\alpha=\{1\}$.
Let $\beta=\movin^\solTt(\alpha)$. So $\beta\leq^\solTt\alpha$
and (by Claim \ref{clm:B_iterable_fs_pres}) $j^\solTt_{\beta\alpha}$ is
$p_{k+1}$-preserving, $M^\solTt_\beta$ is the $\delta^\solTt_\beta$-core
of $M^\solTt_\alpha$, $j^\solTt_{\beta\alpha}$ is the $\delta^\solTt_\beta$-core map
and $\delta^\solTt_\beta\leq\crit(j^\solTt_{\beta\alpha})$.
So
$M^\solTt_\beta||{\delta^{+\solTt}_\beta}
=M^\solTt_\alpha||{\delta^{+\solTt}_\beta}$
and
\[p\eqdef j^\solTt_{\beta\alpha}(q^{M^\solTt_\beta}
\conc\left<\gamma^\solTt_\beta\right>)=p^{M^\solTt_\alpha}
\cut\delta^\solTt_\beta=i^\solUu_{0\alpha}(p^M\cut\delta).\]

\begin{ssclm*} There is $\beta'\leq^\solUu\alpha$
with $M^\solUu_{\beta'}=M^\solTt_\beta$.\end{ssclm*}
\begin{proof}We first show that there is no
$\xi+1\leq^\solUu\alpha$ with $E^\solUu_\xi\neq\emptyset$
and $\crit(E^\solUu_\xi)<\delta^\solUu_\beta<\nu(E^\solUu_\xi)$.
This is just via the (relevant version of)
the argument with the ISC and hull property
(in the sense of \ref{dfn:hull_prop})
from \cite[Example 4.3]{cmip},
combined with the $p_{k+1}$-preservation that we have.
That is, suppose $\xi$ is a counterexample.
Then $M^\solUu_\alpha$ does not have
the $(k+1,p)$-hull property at $\delta^\solTt_\beta$,
as $p\in\rg(i^\solUu_{0\alpha})$ (in fact, the ISC gives
$E=E^\solUu_\xi\rest\delta^\solTt_\beta\in M^\solUu_\alpha$,
but
\[
E\notin\Ult_k(M^{*\solUu}_{\xi+1},E)=\cHull^{M^\solUu_\alpha}(\delta^\solTt_\beta\cup\{
\pvec_k^{M^\solUu_\alpha},p\});\]
cf.~\cite{cmip} for more details).
But
$M^\solTt_\beta=\cHull^{M^\solTt_\alpha}(\delta^\solTt_\beta\cup\{\pvec_k^{
M^\solTt_\alpha},p\})$, so $M^\solTt_\alpha$ does have the $(k+1,p)$-hull
property at $\delta^\solTt_\beta$, contradicting
that $M^\solTt_\alpha=M^\solUu_\alpha$.

So let $\beta'\leq^\solUu\alpha$ be least such that
either $\beta'=\alpha$
or $\delta^\solTt_\beta\leq\crit(i^\solUu_{\beta'\alpha})$.
Noting that
$i^\solUu_{\beta'\alpha}(i^\solUu_{0\beta'}(p^M\cut\delta))=p$,
it now easily follows
that $M^\solUu_{\beta'}=M^\solTt_\beta$, as
desired.
\end{proof}

Now at stage $\beta$, in $\solTt$, we moved into $M^\solTt_\beta$,
so by the rules of comparison, $M^\solTt_\beta\neq M^\solUu_{\beta}$,
so $\beta\neq\beta'$.
But if $\beta'<\beta$, then note that $E^\solUu_{\xi}=\emptyset$
for all $\xi\in[\beta',\beta)$,
but then $M^\solTt_\beta=M^\solUu_\beta$, a contradiction.
So $\beta<\beta'$, and note that $E^\solTt_\xi=\emptyset$
for all $\xi\in[\beta,\beta')$. But then $M^\solTt_{\beta'}=M^\solUu_{\beta'}$,
and $S^\solTt_{\beta'}=\{1\}$ although $\sides^\solTt_{\beta'}=\{0,1\}$
(so $\beta'\in\curlyB^\solTt$).
Hence, the comparison terminates at stage $\beta'$, so $\beta'=\alpha$,
contradicting the assumption that $\sides^\solTt_\alpha=\{1\}$,
proving the subclaim.\end{proof}

So $S^\solTt_\alpha=\{1\}$ but $\alpha\in\curlyB^\solTt$.
Let $\beta=\movin^\solTt(\alpha)$. So $\beta<\alpha$ and at stage $\beta$, in $\solTt$
we move into $M^\solTt_\beta=M^\solTt_\alpha$ in $\solTt$; also,
$E^\solTt_\xi=\emptyset$ for all $\xi\in[\beta,\alpha)$,
and $E^\solUu_\beta=\emptyset$, but $E^\solUu_{\xi}\neq\emptyset$
for all $\xi\in(\beta,\alpha)$.
Since we moved into $M^\solTt_\beta$, we have $M^\solTt_\beta\neq M^\solUu_\beta$,
so $E^\solUu_{\beta+1}\neq\emptyset$.
Also, $M^\solTt_\beta|\delta^\solTt_\beta=M^\solUu_\beta|\delta^\solTt_\beta$,
which is passive, so $\delta^\solTt_\beta<\lh(E^\solUu_{\beta+1})$.
As $M^\solTt_\beta=M^\solUu_\alpha$
and $M^\solUu_\beta=M^\solUu_{\beta+1}$,
therefore
\[ M^\solTt_\beta|\delta^{+\solTt}_\beta=M^\solUu_\beta||\delta^{+\solTt}_\beta=M^\solUu_{\beta+1}||\delta_\beta^{+\solTt}\]
and
${\delta^{+\solTt}_\beta}=
(\delta^\solTt_\beta)^{+\exit^\solUu_{\beta+1}}$.
Let $\zeta\geq\beta+1$ be least such that $\zeta+1\leq^\solUu\alpha$.
Let $\kappa=\crit(E^\solUu_\zeta)$.
Since $M^\solTt_\beta=M^\solTt_\alpha$
is $\delta^\solTt_\beta$-sound and by $p_{k+1}$-preservation,
we have
$\kappa\leq\gamma^\solTt_\beta$,
hence $\kappa\leq\delta^\solTt_\beta$.
But $\delta^\solTt_\beta\in\rg(i^\solUu_{0\alpha})$, so  $i^{*\solUu}_{\zeta+1}(\kappa)=\delta^\solTt_\beta$.
Therefore $E^\solUu_\zeta$ is superstrong and $\zeta=\beta+1$
and $\lh(E^\solUu_\zeta)=\delta_\beta^{+\solTt}$.
So $\beta+2=\alpha$. Let $\varepsilon=\pred^\solUu(\beta+2)$.
Then  computing
with the hull property
like before, we get that there is $\chi<^\solTt\beta$
such that $M^\solTt_\chi=M^\solUu_\varepsilon$,
so $\chi\in\curlyB^\solTt$, and also get that $E^\solUu_{\beta+1}$
was also used in $\solTt$, a contradiction.
\end{proof}

By Claim \ref{clm:S^Tt_alpha=0}, $S^\solTt_\alpha=\{0\}$.
However, by the next claim,
$\sides^\solTt_\alpha=\{0,1\}$:

\begin{clmfive}
 $\alpha\in\curlyB^\solTt$.
\end{clmfive}

\begin{proof}
Suppose not.
By Claim \ref{clm:S^Tt_alpha=0}, then $\alpha\in\curlyH^\solTt$
and $H^\solTt_\alpha\ins M^\solUu_\alpha$,
but
$H^\solTt_\alpha$ is not sound.
So $ H^\solTt_\alpha=M^\solUu_\alpha$.
So $(0,\alpha]^\solTt$ does not drop in model or degree,
as otherwise $M$ is
solid.
Let $\beta=\movin^\solTt(\alpha)$.
By Claim \ref{clm:B_iterable_fs_pres},
$i^\solTt_{\beta\alpha}:H^\solTt_\beta\to H^\solTt_\alpha$
is the $\gamma^\solTt_\beta$-core map,
$\rho(H^\solTt_\alpha)=\rho(H^\solTt_\beta)\leq\gamma^\solTt_\beta
\leq\crit(i^\solTt_{\beta\alpha})$, etc.
But then we get a pair of identical extenders used in $\solTt$ and $\solUu$,
much like before; this is a contradiction.
\end{proof}

So  $\alpha\in\curlyB^\solTt$
but $S^\solTt_\alpha=\{0\}$. Let $\beta=\movin^\solTt(\alpha)$.
Let
$B^\solTt_\beta=B'=(\delta',\gamma',\pi',H',M')$.
So $H'=H^\solTt_\alpha\ins M^\solUu_\alpha$,
and in $\solTt$ we move into $H'$ at stage $\beta$,
and therefore \begin{equation}\label{eqn:M^Tt_beta=M^Uu_beta}M'=M^\solTt_\beta=M^\solUu_\beta,\end{equation}
 $E^\solTt_\beta=\emptyset=E^\solUu_\beta$,
and  note  that
$E^\solTt_\xi=\emptyset\neq E^\solUu_\xi$ for all $\xi>\beta$
with $\xi+1<\lh(\solTt,\solUu)$ (and by the next claim
and since $H'\neq M'=M^\solUu_\beta$, such a $\xi$ exists).

 \begin{clmfive}\label{clm:if_H'_pins_M^U_alpha} We have:
 \begin{enumerate}[label=(\roman*)]
 \item\label{item:(0,beta]^Uu_no_drop}  $(0,\beta]^\solUu$ does not drop in model or degree,
 \item\label{item:j^Tt_0,beta(gamma)=i^Uu_0,beta(gamma)}   $\gamma'=j^\solTt_{0\beta}(\gamma)=i^\solUu_{0\beta}(\gamma)$,
  \item\label{item:H'=M^Uu_alpha} $H'=M^\solUu_\alpha$,
  \item\label{item:E^Uu_beta+1_exists}  $E^\solUu_{\beta+1}\neq\emptyset$ exists, so $\alpha\geq\beta+2$.
  \end{enumerate}
 \end{clmfive}
\begin{proof}
Part \ref{item:(0,beta]^Uu_no_drop}:
By line (\ref{eqn:M^Tt_beta=M^Uu_beta}), $M^\solTt_\beta=M^\solUu_\beta$.
So if $(0,\beta]^\solUu$ drops
then $M^\solTt_\beta$ is $(k+1)$-solid, so $M$ is too, contradiction.

Part \ref{item:j^Tt_0,beta(gamma)=i^Uu_0,beta(gamma)}:
Using the previous part,
 $j^\solTt_{0\beta}$ and $i^\solUu_{0\beta}$ are $p_{k+1}$-preserving, and since $\gamma\in p^M$ and $M^\solTt_\beta=M^\solUu_\beta$, therefore
 $\gamma'=j^\solTt_{0\beta}(\gamma)=i^\solUu_{0\beta}(\gamma)$.

Part \ref{item:H'=M^Uu_alpha}:
By Claim \ref{clm:S^Tt_alpha=0}, $H'\ins M^\solUu_\alpha$.
So suppose $H'\pins M^\solUu_\alpha$.
We claim that either
\begin{enumerate}[label=--]
 \item $H'\pins M'$, or
 \item $\gamma>\delta$ and $M'|\gamma'$ is active with
$F'=E^\solUu_{\beta+1}$,
 and $H'\pins U\eqdef\Ult(M'|\gamma',F')$,
 \end{enumerate}
which suffices, since then $H'\in M'$,
so $H\in M$, contradicting
Assumption \ref{ass:H_notin_M}.

So suppose $H'\npins M'$.
Then $E^\solUu_{\beta+1}\neq\emptyset$.
But as $\rho^{H'}\leq\gamma'$,
it follows that $\lh(E^\solUu_{\beta+1})=\gamma'>\delta'$,
and so $E^\solUu_{\beta+1}=F'=F^{M'|\gamma'}$.
If $F'$ is non-superstrong, then  immediately
 $H'\pins U$.
Suppose $F'$ is superstrong.
Let $\zeta=\pred^\solUu(\beta+2)$.
Since $i^\solUu_{0\beta}(\gamma)=\gamma'$,
we have $F'\in\rg(i^\solUu_{0\beta})$, which gives
$\zeta\leq^\solUu\beta$. If $\zeta=\beta$
then we are done, so suppose $\zeta<^\solUu\beta$.
Let $\mu=\crit(F')$
and note $\mu^{++M^\solUu_\zeta}<\crit(i^\solUu_{\zeta\beta})$,
so
\[ M^\solUu_\zeta|\mu^{++M^\solUu_\zeta}=
M^\solUu_\beta|\mu^{++M^\solUu_\beta}, \]
\[U|(\delta')^{++U}= M^\solUu_{\zeta+1}|(\delta')^{++M^\solUu_{\zeta+1}},\]
so $H'\pins U$, as desired.

Part \ref{item:E^Uu_beta+1_exists}:
This follows immediately from the previous parts.
\end{proof}
Note at this point that it might not be obvious
that $H'$ itself is solid,
so we cannot immediately rule out having $H'=M^\solUu_\alpha$
with $(0,\alpha]^\solUu$ non-dropping.

\begin{clmfive}\label{clm:gamma^+^H<gamma^+^M}
 $\gamma^{+H}<\gamma^{+M}$.\footnote{As usual, if $\gamma=\lgcd(H)$ then $\gamma^{+H}$ denotes $\OR^H$.}
\end{clmfive}
\begin{proof}
Suppose not. Then note $\delta=\gamma$ and
$(\delta^\solTt_\alpha)^{+H^\solTt_\alpha}=\delta^{+\solTt}_\alpha$.
We have $E^\solUu_{\beta+1}\neq\emptyset$ by Claim \ref{clm:if_H'_pins_M^U_alpha}.
But $M^\solUu_\beta=M^\solTt_\beta$, so
 ${\delta^{+\solTt}_\alpha}<\lh(E^\solUu_{\beta+1})$,
which contradicts the fact that $H^\solTt_\alpha$ is $\delta^\solTt_\alpha$-sound
and projects $\leq\delta^\solTt_\alpha$.
\end{proof}

Recall that $E^\solUu_\beta=\emptyset$,
and after moving into $H'$ in $\solTt$, the only extenders
used in the comparison  are the $E^\solUu_{\beta+n}$ with $n>0$,
and there is at least one such extender.
Call these extenders \emph{terminal}. There
is either 1 or 2 of these:

\begin{clmfive}\label{clm:terminal_extenders}
 Either:
 \begin{enumerate}[label=\arabic*.,ref=\arabic*]
  \item \label{item:alpha=beta+2}
   $E^\solUu_{\beta+1}$ is the only
   terminal extender and either:
   \begin{enumerate}[label=(\roman*)]
   \item $\delta'<\gamma'=\lh(E^\solUu_{\beta+1})$, or
   \item $\gamma'<(\gamma')^{+H'}=\lh(E^\solUu_{\beta+1})$
   and $E^\solUu_{\beta+1}$ is type 1/3,
   \end{enumerate}
    or
   \item\label{item:alpha=beta+3}
   $E^\solUu_{\beta+1},E^\solUu_{\beta+2}$ are the only terminal extenders,
   and:
   \begin{enumerate}[label=--]
   \item $\delta'<\gamma'=\lh(E^\solUu_{\beta+1})$, and
 \item $\gamma'<(\gamma')^{+H'}=\lh(E^\solUu_{\beta+2})$,
 and  $E^\solUu_{\beta+2}$  is type 1/3.
 \end{enumerate}
 \end{enumerate}
\end{clmfive}
\begin{proof}
 This is because  $\solUu$ cannot use any extenders
 with generators $\geq\gamma'$,
 which follows from routine fine structure
 together with the facts that  $H'$ is $\gamma'$-sound
 with $p^{H'}\cut\gamma'=q^{H'}$,
 and if
 $b^\solUu\cap\dropset^\solUu_{\deg}=\emptyset$ then $i^\solUu$ is $p_{k+1}$-preserving.
\end{proof}

\begin{clmfive}\label{clm:gamma^+^H_movement}
We have:
\begin{enumerate}[label=\arabic*.,ref=\arabic*]
 \item\label{item:i,j_map_agmt}
$i^\solTt_{0\beta}\rest(H||\gamma^{+H})\sub j^\solTt_{0\beta}$, and
\item\label{item:gamma^+^H_movement}  $(\gamma')^{+H'}=\sup
i^\solTt_{0\beta}``\gamma^{+H}=\sup
j^\solTt_{0\beta}``\gamma^{+H}=\sup i^\solUu_{0\beta}``\gamma^{+H}$.
\end{enumerate}
\end{clmfive}
\begin{proof}
Part \ref{item:i,j_map_agmt}: By
Lemma \ref{lem:ult_relevant_B_props}.

Part \ref{item:gamma^+^H_movement}:
We first prove that $(\gamma')^{+H'}=\sup
i^\solTt_{0\beta}``\gamma^{+H}$.
This holds because
$i^\solTt_{0\beta}$
is a $k$-embedding, $H=\Hull^H(\gamma\cup\{q^H,\pvec_k^H\})$
and likewise for $H'$. That is, we can write $\core_0(H)$ as an increasing
union of
hulls
\[ \core_0(H)=\bigcup_{\xi<\rho_k^H}\widetilde{H}_\xi, \]
where for $\xi<\rho_k^H$, $\widetilde{H}_\xi$
is the set of all $x\in\core_0(H)$ with
$x\in\Hull^H(\gamma\cup\{q^H,\pvec_k^H\})$
as witnessed by some segment/theory below ``level $\xi$''
(using the stratification of $\rSigma_{k+1}$ as in \cite[Appendix to \S2]{fsit}).
Let $\eta_\xi=\widetilde{H}_\xi\cap\gamma^{+H}$.
Then each $\eta_\xi<\gamma^{+H}$,
and $\sup_{\xi<\rho_k^H}\eta_\xi=\gamma^{+H}$.
But $\widetilde{H}_\xi$ is determined by
the corresponding $\rSigma_k$ theory (or ``initial segment'')
$t_\xi\in H$, and $i^\solTt_{0\beta}(t_\xi)=t'_{i^\solTt_{0\beta}(\xi)}$,
where $t'_{\xi'}$ is defined analogously with respect to $H'$.
Thus, $i^\solTt_{0\beta}(\eta_\xi)=\eta'_{i^\solTt_{0\beta}(\xi)}$,
where $\eta'_{\xi'}$ is defined analogously to $\eta_\xi$.
We also have $\sup_{\xi'<\rho_k^{H'}}\eta'_{\xi'}=(\gamma')^{+H'}$,
and $i^\solTt_{0\beta}$ is a $k$-embedding,
so $\rho_k^{H'}=\sup i^\solTt_{0\beta}``\rho_k^H$,
so we are done.

So $(\gamma')^{+H'}=\sup i^\solTt_{0\beta}``\gamma^{+H}
=\sup j^\solTt_{0\beta}``\gamma^{+H}$,
using  part \ref{item:i,j_map_agmt}.
But now note that although $\solUu$ is a tree on $M$, not $B$,
we can make  all the
analogous definitions with $\solUu$,
defining $H^\solUu_\beta$ as the relevant collapsed hull of $M^\solUu_\beta$,
and defining $\pi^\solUu_\beta:H^\solUu_\beta\to M^\solUu_\beta$ etc.
Everything we have just done also applies there,
but since $M^\solTt_\beta=M^\solUu_\beta$ and $j^\solTt_{0\beta},i^\solUu_{0\beta}$
are $p_{k+1}$-preserving, and these parameters together
with the model determine the hull $H'$, it follows
that $\sup i^\solUu_{0\beta}``\gamma^{+H}=(\gamma')^{+H'}$ also.
(But note we don't claim $j^\solTt_{0\beta}\rest\gamma^{+H}\sub
i^\solUu_{0\beta}$.)
\end{proof}

The remaining two claims rule out all possibilities, giving the desired contradiction:
\begin{clmfive}\label{clm:solidity_M|gamma_passive} $M|\gamma$ is not passive.
\end{clmfive}
\begin{proof}
Suppose $M|\gamma$ is passive.
By condensation for $\om$-sound mice (Fact \ref{fact:om_condensation}), $M||\gamma^{+H}=H||\gamma^{+H}$.

 \begin{casefive}\label{scase:overlapping_G}
  $M|\gamma^{+H}$ is active with extender $G$
with $\kappa=\crit(G)<\delta$.

  We have $j^\solTt_{0\beta}(\gamma)=\gamma'$,
  so $M'|\gamma'$ is passive.
  And because $\kappa<\delta<\gamma^{+H}=\lh(G)$, $G$ is $M$-total
  and $\cof^M(\lh(G))=\kappa^{+M}$. It follows
  that $j^\solTt_{0\beta}$
  is continuous at $\gamma^{+H}$,
  so by Claim \ref{clm:gamma^+^H_movement},
  $j^\solTt_{0\beta}(\gamma^{+H})=(\gamma')^{+H'}$.
  Therefore $M'|(\gamma')^{+H'}$
  is active with $G'=j^\solTt_{0\beta}(G)$.
 So $E^\solUu_{\beta+1}=G'$, and by Claim \ref{clm:terminal_extenders},
 this is the only terminal extender.
 Also $i^\solUu_{0\beta}(G)=G'=j^\solTt_{0\beta}(G)$, by Claim
\ref{clm:gamma^+^H_movement} and  for
 the same continuity reasons as for $j^\solTt_{0\beta}$.
 (Note that we don't know that $j^\solTt_{0\beta}=i^\solUu_{0\beta}$
 though;  in fact if the embeddings are non-trivial, then they must
 be distinct.)

Let $\zeta=\pred^\solTt(\beta+2)$. Let
$\kappa'=\crit(G')=i^\solUu_{0\beta}(\kappa)<\delta'$.
Note that because
$G'=i^\solUu_{0\beta}(G)$,
we have $\zeta\leq^\solUu\beta$.\footnote{In fact,
$\zeta$ is the least $\zeta'\leq^\solUu\beta$
such that $\zeta'=\beta$
or $i^\solUu_{0\zeta'}(\kappa)<\crit(i^\solUu_{\zeta'\beta})$,
which is the least $\zeta'\leq\beta$ (as opposed to $\leq^\solUu$)
such that $\zeta'=\beta$ or $\kappa'<\nu(E^\solUu_{\zeta'})$.}
So $(0,\zeta]^\solUu$ does not drop in model or degree.
Moreover,  $G'$ is
$M^\solUu_\zeta$-total
and $\kappa'<i^\solUu_{0\zeta}(\delta)<\rho_k(M^\solUu_\zeta)$,\footnote{Note we refer here to $\rho_k$, not $\rho_{k+1}$.}
so $(0,\beta+2]^\solUu$ also does not drop in model or degree.
And either
$M^\solUu_\zeta=M^\solUu_\beta$ or $\kappa'<\crit(i^\solUu_{\zeta\beta})$.

Now if $\kappa'<\rho(M^\solUu_\zeta)$ then
by Lemma \ref{lem:basic_fs_pres} and as $H'=M^\solUu_{\beta+2}$, we have
\begin{equation}\label{eqn:rho^H'>=gamma'^+H'} \rho^{H'}=\rho(M^\solUu_{\beta+2})=\sup
i^\solUu_{\zeta,\beta+2}``\rho(M^\solUu_{\zeta})\geq\lh(G')=(\gamma')^{+H'}, \end{equation}
contradicting the fact that $\rho^{H'}\leq\gamma'$.

So $\rho(M^\solUu_{\zeta})\leq\kappa'<\crit(i^\solUu_{\zeta\beta})$.
Therefore, recalling line (\ref{eqn:M^Tt_beta=M^Uu_beta}), we have
\[
\rho^{M'}=\rho^{M^\solUu_\beta}=\rho^{M^\solUu_\zeta}=\rho^{M^\solUu_{\beta+2}}=\rho^{H'},\]
but by Claim \ref{clm:p_k+1_pres}(\ref{item:rho^H_alpha>rho^M_alpha}),
we have $\rho^{M'}<\rho^{H'}$, contradiction.

 \end{casefive}
\begin{casefive}\label{scase:solidity_M|gamma_passive_and_otherwise}
Otherwise.

So either $M|\gamma^{+H}$ is passive,
or active with an extender $G$
with $\crit(G)\geq\delta$.
Note that if $G$ exists then in fact
$G$ is type 1, $\crit(G)=\delta<\gamma<\gamma^{+H}$,
and $G$ is $M$-partial.

Suppose
$j^\solTt_{0\beta}(\gamma^{+H})=(\gamma')^{+H'}$.
Then $M'|(\gamma')^{+H'}$ is passive or active type 1.
Passivity
contradicts Claim \ref{clm:terminal_extenders},
so it is active type 1, but then note
\[ H'=M^\solUu_{\beta+2}=\Ult_k(M^{*\solUu}_{\beta+2},G') \]
with $M^{*\solUu}_{\beta+2}\pins M'$, and therefore
$H'\in M'$, contradiction.

So $j^\solTt_{0\beta}(\gamma^{+H})>(\gamma')^{+H'}$.
Let $\theta=\cof^M(\gamma^{+H})$.
Since $\gamma^{+H}<\rho_k^M$,
it follows that
$\theta$ is measurable in $M$
and there is some $\zeta<^\solTt\beta$
such that $\crit(j^\solTt_{0\beta})=j^\solTt_{0\zeta}(\theta)$.
Letting $\zeta$ be least such, then $j^\solTt_{0\zeta}$
is continuous at $\theta$ and hence at $\gamma^{+H}$.
Letting $f\in M$ be such that $f:\theta\to\gamma^{+H}$
is normal,
it follows that
\[ j^\solTt_{0\beta}(f)\rest j^\solTt_{0\zeta}(\theta)=j^\solTt_{\zeta\beta}\com
(j^\solTt_{0\zeta}(f)):j^\solTt_{0\zeta}(\theta)\to\sup j^\solTt_{0\beta}``\gamma^{+H} \]
is also a normal function, which  is in $M'$.
So (together with Claim \ref{clm:gamma^+^H_movement}),
we have $\cof^{M'}((\gamma')^{+H'})=j^\solTt_{0\zeta}(\theta)$,
which is inaccessible in $M'$. Therefore
$M'|(\gamma')^{+H'}$ is not active with an extender $F$ having $\crit(F)<\delta'$
(such an $F$ would be $M'$-total,
which contradicts the cofinality just computed).
But it is active by Claim \ref{clm:terminal_extenders}, hence with a non-$M'$-total extender $F$, which implies that $F$ is type 1
with $\crit(F)=\delta'<\gamma'<(\gamma')^{+H'}$, which
gives $H'\in M'$ as before.\qedhere
 \end{casefive}
\end{proof}

\begin{clmfive}\label{clm:solidity_M|gamma_not_active}
$M|\gamma$ is not active.\end{clmfive}
\begin{proof}
Suppose $M|\gamma$ is active. So $\gamma>\delta$ and $\crit(F)<\delta$ and $F$ is $M$-total.
Let $U=\Ult(M,F)$. By condensation for $\om$-sound mice (Fact \ref{fact:om_condensation}), $H||\gamma^{+H}=U||\gamma^{+H}$.
By Claim \ref{clm:if_H'_pins_M^U_alpha}, $\gamma'=j^\solTt_{0\beta}(\gamma)=i^\solUu_{0\beta}(\gamma)$,
so $E^\solUu_{\beta+1}=F'=j^\solTt_{0\beta}(F)=i^\solUu_{0\beta}(F)$.

\begin{casesix}\label{scase:solidity_M|gamma_active_U|gamma^+_active} $U|\gamma^{+H}$ is active with an extender $G$.

Then $j^\solTt_{0\beta}$ and $i^\solUu_{0\beta}$ are continuous at $\gamma^{+H}$.
For if $\crit(G)<\delta$,
this is as before,
and if $\crit(G)=\delta$
then $\cof^M(\gamma^{+H})=\cof^M(\delta^{+H})$,
definably over $M|\lh(G)=M|\gamma^{+H}$,
but since $\delta^{+H}=\gamma$ and
 $M|\gamma$ is active with $F$, the embeddings are continuous at $\gamma$,
 and hence at $\gamma^{+H}$.
 So in any case,  $E^\solUu_{\beta+2}$ exists
 and $E^\solUu_{\beta+2}=G'=j^\solTt_{0\beta}(G)=i^\solUu_{0\beta}(G)$.
 And since $H'$ is $\gamma'$-sound
 with $\rho^{H'}\leq\gamma'$
 and
  by $p_{k+1}$-preservation,
 $G'$ is type 1/3.

 We claim that $G$ is type 1, so $\crit(G)=\delta$.
For otherwise,  $G$ is type 3, so $\crit(G)<\delta$ and $\crit(G')<\delta'$ and
$G'\in\rg(j^\solUu_{0\beta})$,
so we can reach a contradiction as in Subcase
\ref{scase:overlapping_G}.

And we claim that $F$ is type 3. For otherwise $F$ is  type
2, so note that $H'=\Ult_0(M'|\gamma',G')\in M'$, a
contradiction.

So $F$ is type 3 and $G$ type 1, so  $\nu(F)=\delta=\crit(G)$. So $E^\solUu_{\beta+1}=F'=j^\solTt_{0\beta}(F)$ and $E^\solUu_{\beta+2}=G'=j^\solTt_{0\beta}(G)$.
But then we again reach a contradiction like in Subcase \ref{scase:overlapping_G}
(but just with $G'\com F'$
replacing the single extender $G'$ considered there).

\end{casesix}

\begin{casesix} $U|\gamma^{+H}$ is passive.

We already know $E^\solUu_{\beta+1}=F'$.
Suppose $E^\solUu_{\beta+2}$ exists.
Then as before and since $U|\gamma^{+H}$ is passive,
$i^\solUu_{0\beta}$ is discontinuous at $\gamma^{+H}$,
which, as in Case \ref{scase:solidity_M|gamma_passive_and_otherwise},
implies that  $\crit(E^\solUu_{\beta+2})=\delta'$. But then if $F'$ is type 2,
this gives $H'\in M'$,
and if $F'$ is type 3,
gives a contradiction like in Case \ref{scase:solidity_M|gamma_active_U|gamma^+_active}.

So $E^\solUu_{\beta+2}$ does not exist, so $H'=M^\solUu_{\beta+2}$.
Since $F'\in\rg(i^\solUu_{0\beta})$, as usual  $\zeta=\pred^\solUu(\beta+2)\leq^\solUu\beta$, so $(0,\beta+2]^\solUu$ does not drop in model or degree, and $H'=\Ult_k(M^\solUu_\zeta,F')$.
Let $\kappa'=\crit(F')$.

If $\kappa'<\rho(M^\solUu_\zeta)$
then much as in Case \ref{scase:overlapping_G}
and since $\rho^{H'}\leq\gamma'$, we have
\begin{equation}\label{eqn:rho^H'>=gamma'} \rho^{H'}=\rho(M^\solUu_{\beta+2})=\sup
i^\solUu_{\zeta,\beta+2}``\rho(M^\solUu_{\zeta})=\lh(F')=\gamma', \end{equation}
which implies that $F'$ is superstrong and $\rho(M^\solUu_\zeta)=(\kappa')^{+M^\solUu_\zeta}$. But $\lh(p^{M^\solUu_{\beta+2}})>n=\lh(q^M)=\lh(p^H\cut\gamma)$,
and since $i^\solUu_{0,\beta+2}$ is $p_{k+1}$-preserving,
therefore $\lh(p^{M^\solUu_{\beta+2}})>n$,
so $\lh(p^{H'})>n$,
which contradicts the fact that $H'$ is $\gamma'$-sound
with $\rho^{H'}=\gamma'$
and $\lh(p^{H'}\cut\gamma')=n$.

So $\rho^{M^\solUu_\zeta}\leq\kappa'$, which leads to a contradiction like at the end of Case \ref{scase:overlapping_G}.\qedhere
\end{casesix}
\end{proof}

The last two claims are incompatible,
yielding a contradiction which completes the
proof of solidity and universality.
\end{proof}

\subsection{Iterability, closeness and comparison termination}\label{sec:appendix}

In this section we give the proofs of Claims \ref{clm:B_is_almost_iterable},  \ref{clm:solidity_closeness} and \ref{clm:solidity_comparison_termination} from the proof of Theorem \ref{thm:solidity}.
\begin{proof}[Proof of Claim \ref{clm:B_is_almost_iterable}]

The proof is similar to
that in the proof of condensation
 from normal iterability,
 %conf
\cite[Theorem 5.2 proof, Claim 5]{premouse_inheriting}.
However, the details of the bicephali/cephalanxes there and trees formed on them are somewhat different to the bicephali and trees currently under consideration. Moreover,
there is an issue which arises
in connection with superstrong extenders,
which was unfortunately ignored in \cite{premouse_inheriting},
leading to a minor mistake in the proof
of \cite[Theorem 5.2 proof, Claim 5]{premouse_inheriting}.
So we will handle it correctly and in some detail here.
The issue is just the analogue of that discussed at the end of the proof of Claim \ref{clm:Ds_phU_iterable}, within  the proof of Lemma \ref{lem:super-Dodd-soundness},
which led  to the consideration of \emph{essentially}-$(0,0,0)$-maximal trees on the phalanx
$\mathfrak{W}'_{\pi(\zeta)}$, instead of standard $(0,0,0)$-maximal trees.
(If one is not familiar with  the setup in \cite[Theorem 5.2]{premouse_inheriting},
nor the proof of Claim \ref{clm:Ds_phU_iterable} of \ref{lem:super-Dodd-soundness},
then one should simply read ahead for the present,
as
the details regarding the issue are reasonably self-contained. If more details are desired after reading here, then \cite{premouse_inheriting} might be useful,
though the setup there is somewhat different.)

First let us describe the basic mechanism we will use for copying extenders.
Recall the extender copying
function $\extcopy(\somevarpi,E)$
from Definition \ref{dfn:extcopy},
which is defined under the assumption that $\somevarpi$ is
a non-$\nu$-high embedding between premice.
This will not suffice for our present purposes,
as we will have to work with copy maps which might be $\nu$-high.
So we now define a variant $\concopy(\somevarpi,E)$,
which also works
when $\somevarpi$ is $\nu$-high, assuming enough condensation
(which we will have).
The desire is that when
copying
an iteration tree $\Tt$ to a tree $\Uu$,
we can copy $E^\Tt_\alpha$
 to $E^\Uu_\alpha=\concopy(\somevarpi,E)$, and moreover,
 in such a manner that $\Tt,\Uu$ have the same tree order (and more).\footnote{The improved method was noticed by the author in about 2015.
Prior to that point, the standard method for copying
in this situation was to use an extra extender in the upper tree,
so as to ``reveal'' the extender $\Shift(\somevarpi)(E)$ (see \S\ref{sec:notation_embeddings}),
and then to copy $E$ to that extender. But this
creates notational hassles, in particular
because the trees are indexed differently; the method we describe here
takes a little thought at the outset, but works out simpler
in the end.}

\begin{dfn}\label{dfn:concopy}
Let $P,N$ be $k$-sound premice such that all proper segments of $N$ satisfy standard condensation facts
(which facts
become clear below).
Let $\somevarpi:P\to N$ be  $k$-lifting.

Let $E\in\es_+^P$. Set $\concopy(\somevarpi,E)=\extcopy(\somevarpi,E)$ unless $\somevarpi$ is $\nu$-high and $E\in\es^P\cut\core_0(P)$, so assume this is the case.
In particular, $P,N$ are active type 3,
 $\nu(F^P)<\lh(E)<\OR^P$,
 and so $\widehat{\somevarpi}(E)\notin\es_+^N$
where $\widehat{\somevarpi}=\Shift(\somevarpi)$ (see \S\ref{sec:notation_embeddings}).
 We want to replace $\widehat{\somevarpi}(E)$ with a certain hull  in $\es^N$.
We will also define some other associated objects,
$Q\pins P$, $Q'\pins N$ and $\somevarpi':Q\to Q'$.

Let $Q\pins P$ be least such that $\rho_\om^Q=\nu$
and $E\in\es_+^Q$. Let $Q^\uparrow=\widehat{\somevarpi}(Q)$,
so $Q^\uparrow\pins\Ult(N,F^N)$ and
$\rho_\om^{Q^{\uparrow}}=\widehat{\somevarpi}(\nu)>\nu(F^N)$.
Let $q<\om$ be least such that $\rho_{q+1}^Q=\nu$
(so
if $\rho_0^Q=\nu$ then  $q=0$).
Let
\[ H'=\Hull_{q+1}^{Q^\uparrow}(\nu(F^N)\cup\{\pvec_{q+1}^{Q^{\uparrow}}\})
\]
 $Q'$ be the transitive collapse of $H'$,
and $\sigma:Q'\to H'$ the uncollapse.

Note that $Q'$ is $(q+1)$-sound,
because (i) $\nu(F^N)$ is a cardinal in $N$ and in $\Ult(N,F^N)$,
so $\rho_{q+1}^{Q'}=\nu(F^N)$;
and (ii) letting $w\in Q$ be the set of $(q+1)$-solidity witnesses
for $Q$, then there is $\gamma<\nu$ such that
$w\in\Hull_{q+1}^Q(\{\gamma,\pvec_{q+1}^Q\}$
which  implies $\widehat{\somevarpi}(w)\in
H'$,
but $\widehat{\somevarpi}(w)$ is the set of $(q+1)$-solidity
witnesses for $Q^{\uparrow}$.

So by condensation
with $\sigma$ (in $\Ult(N,F^N)$), we get $Q'\pins\Ult(N,F^N)$,
so by coherence,  $Q'\pins N$.
Let $\somevarpi':Q\to Q'$ be
$\somevarpi'=\sigma^{-1}\com\widehat{\somevarpi}\rest Q$.
Then $\somevarpi'$ is a $p_{q+1}$-preserving
near $q$-embedding and
$\somevarpi\sub\somevarpi'$.
And $\somevarpi'$ is $\nu$-preserving,
because $\widehat{\somevarpi}\rest Q:Q\to Q^\uparrow$
is $\nu$-preserving and by commutativity.
We now set $F=\extcopy(\somevarpi',E)$,
so $F\in\es_+^{Q'}\sub\es^N$.
Moreover, because $\somevarpi\sub\somevarpi'$,
$\somevarpi'$ has appropriate properties
to use as a copy  map at this stage.
\end{dfn}
\begin{dfn}\label{dfn:copy_for_pre-ISC-premice}
Let $P,N$ be   active pre-ISC-premice.
Let $\somevarpi:P\to N$ be $(-1)$-lifting (that is, $\Sigma_0$-elementary).\footnote{
See \S\ref{sec:notation_embeddings}. Recall the $\Sigma_0$-elementarity is with respect to  $\es^P,F^P$, where $F^P$ is encoded as
specified in Definition \ref{dfn:pre-ISC-premouse}. Also recall that $\dom(\somevarpi)=P$; there is no squashing being considered here.} Let $E\in\es_+^P$. Then $\concopy(\somevarpi,E)$
denotes $F^N$ if $E=F^P$,
and denotes $\somevarpi(E)$ otherwise.
\end{dfn}

Note that in all cases above,
we have $\Shift(\somevarpi)(\nu(E))\geq\nu(\concopy(\somevarpi,E))$.

Now fix a $(k,\om_1+1)$-strategy $\Sigma$ for $M$. By Lemma \ref{lem:ess_iter}, $\Sigma$ induces a canonical strategy $\Sigma'$ for essentially-$k$-maximal trees. (We will only need to make use of the ``essentially'' aspect of $\Sigma'$ when $\alpha+1$ in a special circumstance, explained below.)

We define an almost $((k,k),\om_1+1)$-strategy $\Gamma$ for $B$,
such that  trees $\Tt$ on $B$ via $\Gamma$ lift to (essentially-$k$-maximal) trees
$\Uu$ on $M$ via $\Sigma'$. This will suffice, since by Claim \ref{clm:solidity_wellfounded_implies_relevance-putative}, we just need to ensure wellfoundedness of the models of $\Tt$. Much as in Definition \ref{dfn:anomalous}, for $\alpha\in\mathscr{B}^\Tt$
let $J^\Tt_\alpha=j^\Tt_{0\alpha}(J)$
where $J\pins M$ is least such that $\gamma\leq\OR^J$ and $\rho_\om^J=\delta$.
We  define copy maps $\sigma_\alpha,\somevarpi_\alpha$, and structures $K^\Uu_\alpha\ins M^\Uu_\alpha$ (for some values of $\alpha$, some of these are irrelevant), with the following properties:
\begin{enumerate}[label=\arabic*.,ref=\arabic*]\item $\lh(\Uu)=\lh(\Tt)$
and ${<^\Uu}={<^\Tt}$.
\item
Suppose $\alpha\in\curlyB^\Tt$, so we have $B^\Tt_\alpha=(\delta_\alpha,\gamma_\alpha,\pi_\alpha,H^\Tt_\alpha,M^\Tt_\alpha)$. Then:
\begin{enumerate}[label=--]
\item $[0,\alpha]^\Uu\cap\dropset_{\deg}^\Uu=\emptyset$, so $\deg^\Uu_\alpha=k$,
\item $\sigma_\alpha:M_\alpha^\Tt\to M^\Uu_\alpha$ is a $\nu$-preserving $k$-embedding
and $\sigma_\alpha\com j^\Tt_{0\alpha}=i^\Uu_{0\alpha}$.
\item
$ \somevarpi_\alpha=\sigma_\alpha\com\pi_\alpha:H^\Tt_\alpha\to M^\Uu_\alpha$
is a $k$-embedding, and note
$\somevarpi_\alpha\rest\gamma_\alpha=\sigma_\alpha\rest\gamma_\alpha$.
\end{enumerate}
(However, $\pi_\alpha$, and therefore also $\somevarpi_\alpha$,
might be $\nu$-high.
In fact $\pi_0$ might be $\nu$-high. But $\pi_\alpha,\somevarpi_\alpha$ are non-$\nu$-low, since they are $k$-embeddings.)

\item
Suppose $\alpha\in\curlyH^\Tt$. Then:
\begin{enumerate}[label=--]\item
$ \somevarpi_\alpha:H^\Tt_\alpha\to M^\Uu_\alpha$ is
  c-preserving $\deg^{0\Tt}_\alpha$-lifting.
\item If $[0,\alpha]^\Tt\cap\dropset^\Tt_{\deg}=\emptyset$
then $[0,\alpha]^\Uu\cap\dropset^\Uu_{\deg}=\emptyset$ and $\somevarpi_\alpha\com i^\Tt_{0\alpha}=i^\Uu_{0\alpha}\com\pi_0$
 and
$\somevarpi_\alpha$ is a $k$-embedding (note $\pi_0=\somevarpi_0=\pi$).\end{enumerate}

\item Suppose $\alpha\in\curlyM^\Tt$ and $\alpha$ is not weakly anomalous. Then:
\begin{enumerate}[label=--]\item $\sigma_\alpha:M^\Tt_\alpha\to M^\Uu_\alpha$
is c-preserving $\deg^{1\Tt}_\alpha$-lifting.
\item If  $[0,\alpha]^\Tt\cap\dropset^\Tt_{\deg}=\emptyset$ then $[0,\alpha]^\Uu\cap\dropset^\Uu_{\deg}=\emptyset$ and $\sigma_\alpha\com j^\Tt_{0\alpha}=i^\Uu_{0\alpha}$
and $\sigma_\alpha$ is   a $k$-embedding.\end{enumerate}
\item\label{item:when_alpha_is_weakly_anomalous}
Suppose $\alpha\in\curlyM^\Tt$ and $\alpha$ is weakly anomalous.
Then:
\begin{enumerate}[label=--]\item $K^\Uu_\alpha\ins M^\Uu_\alpha$.
\item If $\alpha$ is non-anomalous then
$\sigma_\alpha:M^\Tt_\alpha\to K^\Uu_\alpha$
is c-preserving $\deg^{1\Tt}_\alpha$-lifting.
\item If $\alpha$ is anomalous (so $\deg^{1\Tt}_\alpha=-1$)
then $K^\Uu_\alpha$ is active type 3 and $\sigma_\alpha:M^\Tt_\alpha\to K^\Uu_\alpha$ is
$(-1)$-lifting.
Note that in this case, $M^\Tt_\alpha$ is a pre-ISC-premouse which fails the ISC, with $\nu(F(M^\Tt_\alpha))\leq\lgcd(M^\Tt_\alpha)$,
and $\sigma_\alpha$ is also c-preserving.
There is no squashing involved,
so $\dom(\sigma_\alpha)=M^\Tt_\alpha$ (even though $K^\Uu_\alpha$ is actually a premouse). See \S\ref{sec:notation_embeddings}.
\end{enumerate}

Moreover, let $\beta=\max(\mathscr{B}^\Tt\cap[0,\alpha]^\Tt)$ and  $\varepsilon+1=\succ^\Tt(\beta,\alpha)$. Then:
\begin{enumerate}[label=--]
\item $J^\Tt_\beta=M^{*\Tt}_{\varepsilon+1}\pins M^\Tt_\beta$
and $(\varepsilon+1,\alpha]^\Tt\cap\dropset^\Tt=\emptyset$ and $j^{*\Tt}_{\varepsilon+1,\alpha}:J^\Tt_\beta\to M^\Tt_\alpha$
and $\crit(j^{*\Tt}_{\varepsilon+1,\alpha})=\delta_\beta$.
\item Suppose there is no $\chi\in[\varepsilon+1,\alpha)^\Tt$
such that $\crit(j^\Tt_{\chi\alpha})\geq j^{*\Tt}_{\varepsilon+1,\chi}(\delta_\beta)$. Then $[0,\alpha]^\Uu\cap\dropset^\Uu_{\deg}=\emptyset$
and $K^\Uu_\alpha=i^\Uu_{0\alpha}(J)\pins M^\Uu_\alpha$
and \[ \sigma_\alpha\com j^{*\Tt}_{\varepsilon+1,\alpha}\com j^\Tt_{0\beta}\rest J=i^\Uu_{0\alpha}\rest J.\]
\item Suppose there is $\chi$ as above. Then
$\alpha$ is non-anomalous and $K^\Uu_\alpha=M^\Uu_\alpha$. Moreover,  taking $\chi$ least such, and $\xi+1=\succ^\Tt(\chi,\alpha)$, then $\dropset^\Uu\cap[0,\alpha]^\Uu=\{\xi+1\}$
and $M^{*\Uu}_{\xi+1}=K^\Uu_\chi=i^\Uu_{0\chi}(J)$ and $i^{*\Uu}_{\xi+1}:K^\Uu_\chi\to K^\Uu_\alpha=M^\Uu_\alpha$,
and
\[ \sigma_\alpha\com j^{*\Tt}_{\varepsilon+1,\alpha}\com j^\Tt_{0\beta}\rest J=i^{*\Uu}_{\xi+1,\alpha}\com i^\Uu_{0\chi}\rest J.\]

\end{enumerate}

\item\label{item:exitside=0_copy} Let $\alpha+1<\lh(\Tt)$.
If $\exitside^\Tt_\alpha=0$ then
$E^\Uu_\alpha=\concopy(\somevarpi_\alpha,E^\Tt_\alpha)$, whereas if $\exitside^\Tt_\alpha=1$
then  $E^\Uu_\alpha=\concopy(\sigma_\alpha,E^\Tt_\alpha)$.

\item Let $\alpha+1<\lh(\Tt)$. If $\alpha+1\in\mathscr{B}^\Tt\cup\mathscr{M}^\Tt$ then  $\sigma_{\alpha+1}$ is just the map given by the natural application of the Shift Lemma.
(If $\alpha+1$ is a mismatched dropping node and $\beta=\pred^\Tt(\alpha+1)$, then
apply the Shift Lemma to
$\sigma_\beta\rest J^\Tt_\beta:J^\Tt_\beta\to i^\Uu_{0\beta}(J)$
and the map $\psi:\exit^\Tt_\alpha\to\exit^\Uu_\alpha$ extracted from the definition of $\concopy$ and either $\somevarpi_\alpha$ (if $\exitside^\Tt_\alpha=0$)
or $\sigma_\alpha$ (if $\exitside^\Tt_\alpha=1$).)
If $\alpha+1\in\mathscr{H}^\Tt$
then $\somevarpi_\alpha$ is likewise produced by the Shift Lemma.
\item If $\eta<\lh(\Tt)$ is a limit
(so $[0,\eta)^\Tt=[0,\eta)^\Uu$
and $[0,\eta)^\Uu=\Sigma'(\Uu\rest\eta)$),
then $\sigma_\eta,\somevarpi_\eta$
are just formed  via commutativity as usual.
\end{enumerate}

These conditions determine how $\Tt,\Uu$ are formed, and also $\sigma_\alpha,\somevarpi_\alpha,K^\Uu_\alpha$ in the cases that they are relevant. We now make a couple of remarks on the less standard considerations here.

First, in the context of clause \ref{item:exitside=0_copy}, suppose $\somevarpi_\alpha$ is $\nu$-high
and $\nu(H^\Tt_\alpha)<\lh(E^\Tt_\alpha)<\OR(H^\Tt_\alpha)$.
Then defining $\somevarpi'_\alpha,Q,Q'$ as in \ref{dfn:concopy},
note that if $\beta+1<\lh(\Tt)$
and $\pred^\Tt(\beta+1)=\alpha$
and $\nu(H^\Tt_\alpha)\leq\crit(E^\Tt_\beta)$,
then $H^{*\Tt}_{\beta+1}\ins Q$
and $M^{*\Uu}_{\alpha+1}\ins Q'$,
and $\somevarpi'_\alpha\rest H^{*\Tt}_{\beta+1}\to M^{*\Uu}_{\alpha+1}$
is the appropriate map with which to apply the Shift Lemma
when defining $\somevarpi_{\alpha+1}$.

Second, let us consider the manner in which $\Uu$ can fail to be $k$-maximal, and what happens in that situation.
Suppose we have determined $\Tt\rest(\alpha+1)$ and $\Uu\rest(\alpha+1)$
and $\left<\somevarpi_\beta,\sigma_\beta,J^\Uu_\beta\right>_{\beta\leq\alpha}$,
and  $\exitside^\Tt_\alpha$ and $E^\Tt_\alpha$ and are
chosen appropriately.
In particular,
$E^\Tt_\alpha$ is $\Tt\rest(\alpha+1)$-normal; that is, $\lh(E^\Tt_\beta)\leq\lh(E^\Tt_\alpha)$ for all $\beta<\alpha$.
Then $E^\Uu_\alpha$ is determined by clause \ref{item:exitside=0_copy}.
We get $\lh(E^\Uu_\beta)\leq\lh(E^\Uu_\alpha)$
for all $\alpha\leq\beta$
much as usual,
except in the case that $\alpha=\varepsilon+1$ for some $\varepsilon$, $\varepsilon+1$ is a mismatched dropping node
and $E^\Uu_\varepsilon$ is of superstrong type (hence $E^\Tt_\varepsilon$ is also of superstrong type). Suppose this is the case. Let $\beta=\pred^\Tt(\varepsilon+1)=\max(\mathscr{B}^\Tt\cap[0,\varepsilon+1]^\Tt)$. So $\beta\leq^\Tt\varepsilon$ and $\beta=\max(\mathscr{B}^\Tt\cap[0,\varepsilon]^\Tt)$
and $\exitside^\Tt_\beta=0$
and $\gamma_\beta=\delta_\beta^{+H^\Tt_\beta}<\lh(E^\Tt_\beta)$.

Now $K^\Uu_{\varepsilon+1}=i^\Uu_{0,\varepsilon+1}(J)=i^\Uu_{\beta,\varepsilon+1}(J')$ where $J'=i^\Uu_{0\beta}(J)$,
and $J'\pins M^\Uu_\beta$ with $\rho_\om^{J'}=i^\Uu_{0\beta}(\delta)=\somevarpi_\varepsilon(\delta_\beta)=\somevarpi_\varepsilon(\crit(E^\Uu_\varepsilon))=\crit(E^\Uu_\varepsilon)$,
so
$K^\Uu_{\varepsilon+1}\pins M^\Uu_{\varepsilon+1}$ with $\rho_\om(K^\Uu_{\varepsilon+1})=i_{E^\Uu_\varepsilon}(\crit(E^\Uu_\varepsilon))=\lambda(E^\Uu_\varepsilon)=\nu(E^\Uu_\varepsilon)$ (since $E^\Uu_\varepsilon$ is superstrong).
But $\lh(E^\Uu_\varepsilon)=\nu(E^\Uu_\varepsilon)^{+M^\Uu_{\varepsilon+1}}$,
so $\OR(K^\Uu_{\varepsilon+1})<\lh(E^\Uu_\varepsilon)$.
Since $\sigma_{\varepsilon+1}:M^\Tt_{\varepsilon+1}\to K^\Uu_{\varepsilon+1}$
and $\sides^\Tt_{\varepsilon+1}=\{1\}$,
we must have $\exitside^\Tt_{\varepsilon+1}=1$ and $E^\Tt_{\varepsilon+1}\in\es_+(M^\Tt_{\varepsilon+1})$,
and as $E^\Uu_{\varepsilon+1}=\concopy(\sigma_{\varepsilon+1},E^\Tt_{\varepsilon+1})\in\es_+(K^\Uu_{\varepsilon+1})$,
so $\lh(E^\Uu_{\varepsilon+1})<\lh(E^\Uu_\varepsilon)$.
However,  $\lh(E^\Tt_\varepsilon)\leq\lh(E^\Tt_{\varepsilon+1})$
(with equality iff $M|\varepsilon$ is active iff $J=M|\varepsilon$), and as $E^\Tt_\alpha$ also has superstrong type,
we have $\sigma_{\varepsilon+1}(\nu(E^\Tt_\varepsilon))=\nu(E^\Uu_\varepsilon)$,
so either $\lh(E^\Tt_{\varepsilon})=\lh(E^\Tt_{\varepsilon+1})$
and $\exit^\Uu_{\varepsilon+1}=K^\Uu_{\varepsilon+1}$,
or $\lh(E^\Tt_\varepsilon)<\lh(E^\Tt_{\varepsilon+1})$
and \[ \Shift(\sigma_{\varepsilon+1})(\lh(E^\Tt_\varepsilon))=\nu(E^\Uu_\varepsilon)^{+K^\Uu_{\varepsilon+1}}<\lh(E^\Uu_{\varepsilon+1}).\]
In either case,
$\nu(E^\Uu_\varepsilon)\leq\nu(E^\Uu_{\varepsilon+1})$,
as required for the essential-$m$-maximality of $\Uu$.

Note also that if $\varepsilon+1$ is anomalous then
$J=M|\gamma$ and $J^\Tt_\beta=M^\Tt_\beta|\gamma_\beta=M^{*\Tt}_{\varepsilon+1}$ are type 3
and $\nu(F(J^\Tt_\beta))=\delta_\beta=\crit(E^\Tt_\varepsilon)$,
and
\[ M^\Tt_{\varepsilon+1}=\Ult_{-1}(J^\Tt_\beta,E^\Tt_\varepsilon)=\Ult(J^\Tt_\beta,E^\Tt_\varepsilon),\]
which is a pre-ISC-premouse
with largest cardinal
\[ i_{E^\Tt_\varepsilon}(\delta_\beta)=\lambda(E^\Tt_\varepsilon)=\nu(E^\Tt_\varepsilon)=\nu(F(M^\Tt_{\varepsilon+1})),\]
$K^\Uu_{\varepsilon+1}$ is a type 3 premouse
with largest cardinal $\nu(E^\Uu_\varepsilon)=\lambda(E^\Uu_\varepsilon)$, and $\sigma_{\varepsilon+1}:M^\Tt_{\varepsilon+1}\to K^\Uu_{\varepsilon+1}$ is $\Sigma_0$-elementary. In this case,
$E^\Tt_{\varepsilon+1}=F(M^\Tt_{\varepsilon+1})$,
since $\lh(E^\Tt_\varepsilon)=\OR(M^\Tt_{\varepsilon+1})$, and since $\lh(E^\Tt_\varepsilon)\leq\lh(E^\Tt_{\varepsilon+1})$,
$F(M^\Tt_{\varepsilon+1})$ is the only valid option for $E^\Tt_{\varepsilon+1}$. So $\nu(E^\Tt_\varepsilon)=\nu(E^\Tt_{\varepsilon+1})$, so the rules of normality ensure that there is no $\alpha$ with $\pred^\Tt(\alpha+1)=\varepsilon+1$,
so $\varepsilon+1$ is an end-node of $\Tt$,
and also of $\Uu$. On the other hand, if $\varepsilon+1$ is non-anomalous (but it is weakly anomalous in the present discussion)
then there is in general nothing preventing there
being $\alpha<\lh(\Tt)$ with $\varepsilon+1<^\Tt\alpha$,
hence  with $\nu(E^\Tt_\varepsilon)\leq\crit(i^\Tt_{\varepsilon+1,\alpha})$,
even with $(\varepsilon+1,\alpha]^\Tt\cap\dropset^\Tt=\emptyset$;
in this case we get $(0,\alpha]^\Uu\cap\dropset^\Uu=\emptyset$
and $K^\Uu_\alpha=M^\Uu_\alpha$. (It can be that $\deg^\Tt_{\alpha}<\deg^\Tt_{\varepsilon+1}$ here; it seems it might also be that $\deg^\Tt_\alpha<\deg^\Uu_\alpha$.)

The remaining details of the copying process are left to the reader; they are basically routine. The overall picture is also similar
to that in the proof of \cite[Theorem 5.2]{premouse_inheriting}, though there, the possibility that $\Uu$ fails to be $k$-maximal
was overlooked. That proof should be corrected by allowing $\Uu$ (as there) to be essentially-$k$-maximal. The proof there then adapts much as above.
\end{proof}

\begin{proof}[Proof of Claim \ref{clm:solidity_closeness}]
This is basically
as in the proof of Closeness \cite[Lemma 6.1.5]{fsit},
proceeding by induction on $\xi$.
However, there are some small, but important, differences.
They are mostly because we are iterating bicephali, as opposed
to premice. Similar issues arise in the classical proof
of solidity etc, when one iterates phalanxes, but some of these
things were not covered in detail in \cite{fsit}, for example.
So we discuss enough details to describe the differences.
One should also keep in mind that
we allow superstrong extenders on
$\es$, whereas this was not the case in \cite{fsit};
however, this does not have a significant impact here.

Let $\kappa=\crit(E^\Tt_\xi)$. So $\kappa<\exchnu(\exit^\Tt_\beta)$
and $\kappa^{+\exit^\Tt_\xi}\leq\OR(\exit^\Tt_\beta)$.
Let $e=\exitside^\Tt_\xi$ and $E=E^\Tt_\xi$,  so $E\in\es_+(M^{e\Tt}_\xi)$.

\begin{casefour}\label{case:solidity_closeness_xi+1_in_B^Tt} $\xi+1\in\curlyB^\Tt$.

So we must see that $E$ is close to $M^\Tt_\beta$.
We have $\beta\in\curlyB^\Tt$
and $\kappa<\delta^\Tt_\beta$ and
$\kappa^{+\exit^\Tt_\beta}=\kappa^{+B^\Tt_\beta}\leq\delta^\Tt_\beta$,
so $\kappa^{+B^\Tt_\beta}<\lh(E^\Tt_\beta)$.

\begin{scasefour} $\beta=\xi$.

 We may assume $E\in\es_+(H^\Tt_\beta)$.
We have
$H^\Tt_\beta||\kappa^{++H^\Tt_\beta}\sub M^\Tt_\beta$,
so if $E\in H^\Tt_\beta$ then
$E$ is close to $M^\Tt_\beta$.
Otherwise $E=F^{H^\Tt_\beta}$,
which is close to $M^\Tt_\beta$ because
$\pi_\beta:H^\Tt_\beta\to M^\Tt_\beta$
is $\rSigma_1$-elementary with $\kappa^{+H^\Tt_\beta}<\crit(\pi_\beta)$,
so $F^{H^\Tt_\beta}$ is a sub-extender of $F^{M^\Tt_\beta}$.
\end{scasefour}

\begin{scasefour}\label{scase:solidity_closeness_xi+1_in_B,beta<xi} $\beta<\xi$.

By the proof of \cite[6.1.5]{fsit}, we may assume
$E=F(M^{e\Tt}_\xi)$.
Likewise, we may assume either $\xi$ is anomalous,
or $M^{e\Tt}_\xi$ is active type 2 with
 $\rho_1(M^{e\Tt}_\xi)\leq\kappa^{+M^\Tt_\beta}$.

Let $\zeta$ be least with $\zeta\geq\beta$ and $\zeta+1\leq^\Tt\xi$.
It follows that $(\zeta+1,\xi]^\Tt\cap\dropset_{\deg}^\Tt=\emptyset$ and either (i) $\zeta+1\in\curlyB^\Tt$ and $k=0$,  or  (ii)
$\zeta+1,\xi\notin\curlyB^\Tt$ and $\deg^\Tt_{\zeta+1}=\deg^\Tt_\xi\in\{-1,0\}$
and $\rho_1(M^{e*\Tt}_{\zeta+1})=\rho_1(M^{e\Tt}_\xi)\leq\kappa^{+M^{e*\Tt}_{\zeta+1}}=\kappa^{+M^{e\Tt}_\xi}<\crit(E^\Tt_\zeta)$, and therefore $\pred^\Tt(\zeta+1)=\beta$
(as in the proof of \cite[6.1.5]{fsit}).

\begin{sscasefour} $\zeta+1\in\curlyB^\Tt$ and $k=0$.

First suppose $e=1$.
By induction, the extenders used along $(\beta,\xi]^\Tt$
are close to the $M^\Tt_\alpha$s. As
$\kappa<\crit(E^\Tt_\zeta)$, it follows
as usual that $E$ is close to $M^\Tt_\beta$, as desired.

Now suppose  $e=0$.
Let $\beta'=\curlyB^\Tt_\downarrow(\xi)$. Then the extenders used along $(\beta',\xi]^\Tt$ are close to the $H^\Tt_\alpha$s, and so $E$ is close to $H^\Tt_{\beta'}$.
But $\pi_{\beta'}:H^\Tt_{\beta'}\to M^\Tt_{\beta'}$ is $\rSigma_1$-elementary and $\crit(E)<\crit(E^\Tt_\zeta)<\delta_{\beta'}$,
so $E$ is close to $M^\Tt_{\beta'}$.
 But the extenders used along $(\beta,\beta']^\Tt$ are close to the $M^\Tt_\alpha$s,
 and it follows that $E$ is close to $M^\Tt_\beta$, as desired.
\end{sscasefour}

\begin{sscasefour} $\zeta+1\notin\curlyB^\Tt$
and $\deg^\Tt_{\zeta+1}=\deg^\Tt_{\xi}\in\{-1,0\}$.

By induction, the extenders used along $(\beta,\xi]^\Tt$
are close to their target models,
so essentially as usual,
$E$
is close to $M^{e*\Tt}_{\zeta+1}$,
and since $M^{e*\Tt}_{\zeta+1}\ins H^\Tt_\beta$
or $M^{e*\Tt}_{\zeta+1}\ins M^\Tt_\beta$, this suffices as before,
using $\pi_\beta$ to lift the $\bfrSigma_1$ definitions
of measures if
$M^{e*\Tt}_{\zeta+1}\ins H^\Tt_\beta$.
(Maybe
$\zeta+1,\xi$ are anomalous, in which case
$e=1$ and $M^\Tt_{\zeta+1},M^\Tt_\xi$ are not premice; in this case
\emph{close} is defined as usual, but
using (unsquashed) $\bfSigma_1$-definability,
and the usual calculations work.)
\end{sscasefour}
\end{scasefour}
\end{casefour}

\begin{casefour}\label{case:solidity_closeness_xi+1_notin_B^Tt} $\xi+1\notin\curlyB^\Tt$.

If $\xi+1$ is not a mismatched dropping node, this is
just like in \cite{fsit}. So suppose $\xi+1$ is mismatched dropping.
So $\beta\in\curlyB^\Tt$, $\xi+1\in\curlyM^\Tt$, $M^{*\Tt}_{\xi+1}=J^\Tt_\beta=i^\Tt_{0\beta}(J)$,
 $\exitside^\Tt_\beta=0$ and
$\kappa=\delta^\Tt_\beta<
\kappa^{+H^\Tt_\beta}=\gamma^\Tt_\beta<\lh(E^\Tt_\beta)$
(recall $\gamma<\OR^H$).

\begin{sclmfive}The component measures of $E$
are all in
$H^\Tt_\beta$.
\end{sclmfive}

\begin{proof} Suppose not. Suppose $\xi=\beta$.
Since $\exitside^\Tt_\beta=0$,
we must have
$E^\Tt_\beta=F^{H^\Tt_\beta}$ and $H^\Tt_\beta$ is type 1/2
and $\rho_1^{H^\Tt_\beta}\leq\kappa^{+H^\Tt_\beta}=\gamma^\Tt_\beta$.
But this contradicts Claim
\ref{clm:not_1_2_with_gamma_between}
part \ref{item:not_1_2_with_gamma_between}.

So $\xi>\beta$.
Note that
 $E_a\notin M^{e\Tt}_\xi$ for some $a$,
so
 $M^{e\Tt}_\xi$ is active type 2, $E=F(M^{e\Tt}_\xi)$  and
$\rho_1(M^{e\Tt}_\xi)\leq\kappa^{+H^\Tt_\beta}=\gamma^\Tt_\beta$.
Let $\zeta\geq\beta$ be least such that $\zeta+1\leq^\Tt\xi$.
Arguing much as usual, we get that $\pred^\Tt(\zeta+1)=\beta$
and $(\zeta+1,\xi]^\Tt\cap\dropset^\Tt_{\deg}=\emptyset$,
 $\kappa<\crit(E^\Tt_\zeta)$, etc, so $\zeta+1\in\curlyH^\Tt$
 and $H^{*\Tt}_{\zeta+1}\ins H^\Tt_\beta$.
 So by induction, $F^{H^\Tt_\xi}$
 is close to $H^{*\Tt}_{\zeta+1}$. So if $H^{*\Tt}_{\zeta+1}\pins
H^\Tt_\beta$ we are done.
 Otherwise $H^{*\Tt}_{\zeta+1}=H^\Tt_\beta$ is active type 2
 and $\rho_1(H^\Tt_\beta)\leq\gamma^\Tt_\beta$ and
  $k=0$,
  again contradicting Claim \ref{clm:not_1_2_with_gamma_between}.
\end{proof}

So the component measures are all in $H^\Tt_\beta||\chi$,
where $\chi=\kappa^{++H^\Tt_\beta}$.
But by condensation using $\pi_\beta$,\footnote{Should point out that this is is internal to $M^\Tt_\beta$, preserved by $j_{0\beta}$ from $M$.} either
\begin{enumerate}[label=(\roman*)]
 \item\label{item:agreement_passive} $H^\Tt_\beta||\chi=M^\Tt_\beta||\chi$, or
 \item\label{item:agreement_active} $M|\gamma^\Tt_\beta$ is active
 with extender $F$ and $H^\Tt_\beta||\chi=U||\chi$
 where $U=\Ult(M^\Tt_\beta,F)$.
\end{enumerate}
If \ref{item:agreement_passive} holds,
note that since $J\pins M$ and $\rho^J_\om=\kappa$,
we have $\chi\leq\OR^J$, so $H^\Tt_\beta||\chi=J||\chi$,
so the measures are in $J$, and so $E$ is close to $J$.
And if \ref{item:agreement_active} holds,
then $J=M|\gamma^\Tt_\beta$,
and then since the measures are in $U||\chi$,
it again follows that $E$ is close to $J$, as desired.
\end{casefour}
This completes the proof of closeness (or its approximation).
\end{proof}
\begin{proof}[Proof of Claim \ref{clm:solidity_comparison_termination}]
This is by the usual argument unless
the comparison produces a tree $\solTt$ uses
extenders $E^\solTt_\alpha$ which are not premouse extenders
(that is, not the active extender of a premouse),
so assume this is the case. So then $M|\gamma$
is active type 3, $\delta<\gamma$, and $\delta=\crit(E)$
for some $H$-total extenders on $\es_+^H$.

Running the usual argument, we take an elementary
$\varphi:A\to V_\Omega$
with some large enough $\Omega$ and $A$ countable transitive,
and $\varphi(\om_1^A)=\om_1$, and everything relevant in $\rg(\varphi)$.
Let $\kappa=\om_1^A=\crit(\varphi)$. As usual,
$\solTt\rest\om_1,\solUu\rest\om_1$ are both cofinally non-padded
and $\kappa<^\solTt\om_1$ and $\kappa<^\solUu\om_1$.
We may
assume that $\kappa\in\curlyB^\solTt$ (so $\om_1\in\curlyB^\solTt$ also)
as otherwise
the extenders used along $(\kappa,\om_1]^\solTt$
satisfy the ISC. We have that
 $M^\solTt_\kappa,M^\solTt_{\om_1},M^\solUu_\kappa,M^\solUu_{\om_1}$
all have the same $\pow(\kappa)$,
$\pow(\kappa)\cap H^\solTt_\kappa\sub\pow(\kappa)\cap M^\solTt_\kappa$,
and
\[
i^\solTt_{\kappa\om_1},j^\solTt_{\kappa\om_1},i^\solUu_{\kappa\om_1}
\sub\varphi.\]
Let $\alpha+1=\succ^\solTt((\kappa,\om_1]^\solTt)$
and $\beta+1=\succ^\solUu((\kappa,\om_1]^\solUu)$.
Then $E^\solTt_\alpha$ and $E^\solUu_\beta$ have critical point $\kappa$,
measure the same
$\pow(\kappa)$, and letting
$\iota=\min(\exchnu(\exit^\solTt_\alpha),\nu(E^\solUu_\beta))$,
we have
$E^\solTt_\alpha\rest\iota=E^\solUu_\beta\rest\iota$.

By the rules of comparison, $\alpha\neq\beta$.
If $\alpha<\beta$ then we would have $\exchnu(\exit^\solTt_\alpha)<\nu(E^\solUu_\beta)$,
but this contradicts the ISC as usual.
So $\beta<\alpha$.
So if $\exit^\solTt_\alpha$ is a premouse, then
by Claim \ref{clm:T,U_normality},
we get
$\nu(E^\solUu_\beta)=\exchnu(E^\solUu_\beta)<\exchnu(E^\solTt_\alpha)=\nu(E^\solTt_\alpha)$,
but then $\exit^\solTt_\alpha$ fails ISC, a contradiction.
So $\alpha$
is anomalous. Let $\zeta$ be least such that $\zeta+1\leq^\solTt\alpha$ and
$\zeta+1$ is anomalous. Let $Q^*=M^{*\solTt}_{\zeta+1}$,
so $Q^*$ is a type 3 premouse and $\crit(E^\solTt_\zeta)=\nu(Q^*)$.
We have $F^{Q^*}\rest\nu(Q^*)\sub E^\solTt_\alpha$,
and $F^{Q^*}\notin B^\solTt_{\alpha+1}$.
The ISC and compatibility then gives $F^{Q^*}=E^\solUu_\beta$.

But $\zeta$ is non-anomalous (anomalous extenders
apply to bicephali), so $E^\solTt_\zeta$ is a premouse extender.
Let $\xi+1=\succ^\solUu((\beta+1,\om_1]^\solUu)$.
Then standard extender factoring calculations give
 $E^\solUu_{\xi}=E^\solTt_\zeta$ (e.g.~\cite[\S5]{extmax}). This
contradicts Claim \ref{clm:no_matching_exts}.
\end{proof}

\section{Conclusion}\label{sec:conclusion}

Now that we know that solidity and universality holds, we can
deduce that the desired theorems on condensation,
super-Dodd-soundness and projectum-finite-generation hold,
as discussed in \S\ref{sec:plan}:

\begin{proof}[Proof of Theorems
\ref{thm:condensation}, \ref{thm:super-Dodd-soundness} and
\ref{thm:finite_gen_hull}]\

Theorem \ref{thm:condensation}:
Given a premouse $M$ which is $m$-sound and $(m,\om_1+1)$-iterable, by Theorem \ref{thm:solidity},
$M$ is $(m+1)$-solid,
which by Fact \ref{fact:condensation}
parts \ref{item:fact_condensation_H_notin_M}\ref{item:cond_fact_abbr_i} and \ref{item:fact_condensation_H_notin_M}\ref{item:cond_fact_abbr_ii} immediately  yields the result.

The other results hold because
by Theorem \ref{thm:solidity}, we have the necessary
solidity and universality to apply
and Lemmas \ref{lem:super-Dodd-soundness} and \ref{lem:finite_gen_hull}.
\end{proof}

Let us deduce a few simple corollaries to the main theorems.
The main one is  naturally:
\begin{cor}
 Let $R$ be countable and transitive,
 and suppose $R\sats$ ZFC + ``$\delta$ is Woodin''.
Suppose there is an $(\om_1+1)$-iteration strategy
$\Sigma$ for (coarse) normal iteration trees on $R$.
Let $\CC^R$ be the maximal $L[\es]$-construction
of $R$ (using extenders of $R$ to background
the construction in an appropriate manner).
Then $\CC^R$ does not break down;
it converges on a class $M$ of $R$
such that $M\sats$ ``$\delta$ is Woodin''.
\end{cor}
\begin{proof}
 One can run the standard proof from \cite[\S11]{fsit},
 using the results of the paper,
 since normal trees on creatures of $\CC^R$ can be lifted
 to normal trees on $R$.
\end{proof}

We can deduce  a canonical form
for all solidity witnesses:
\begin{cor}
 Let $M$ be a $k$-sound, $(k,\om_1+1)$-iterable premouse.
 Let $\gamma\in p_{k+1}^M$
 and $H=H_\gamma$ the solidity witness at $\gamma$ (the collapsed
 hull form).
 Let $\rho=\rho_{k+1}(H)$ and $C=\core_{k+1}(H)$. Then either
$C\pins M$ or
 $M|\rho$ is active with extender $F$
  and $C\pins U=\Ult(M,F)$.
Moreover, $H$ is an iterate of $C$
 via a $k$-maximal tree $\Tt$
 which is strongly finite,
 and almost-above $\rho$.
\end{cor}
\begin{proof}
This follows easily from the theorems,
together with Fact \ref{fact:condensation} part \ref{item:condensation_from_normal_H_in_M}, to get $C\pins M$ or $C\pins U$,
and using Theorem \ref{thm:finite_gen_hull} to get the rest.
(Since $H\in M$,
note that either $\rho=\gamma$ or $\rho=\card^M(\gamma)$ and $\gamma=\rho^{+H}$. Using this and $(k+1)$-universality, observe that $H$ is projectum-finitely generated.)
\end{proof}

\begin{cor}\label{cor:ult_is_it}
 Let $M$ be an $(m+1)$-sound, $(m,\om_1+1)$-iterable
 premouse, where $m<\om$.
 Let $E$ be a finitely generated short extender which is weakly amenable to $M$
with $\rho^M_{m+1}\leq\crit(E)<\rho_m^M$.
Let $U=\Ult_m(M,E)$. Then $U$  is $(m,\om_1+1)$-iterable
iff there is an $m$-maximal tree $\Tt$ on $M$
of finite length $n+1$ with $U=M^\Tt_n$
(moreover, if $\Tt$ exists, then it is unique,
$[0,n]^\Tt$ does not drop in model or degree, and $i^{M,m}_E$
is the iteration map).
\end{cor}
\begin{proof}
If $U=M^\Tt_n$ for such a tree $\Tt$, then $U$ is $(m,\om_1+1)$-iterable
by Fact \ref{fact:Gg}.

Now suppose $U$ is $(m,\om_1+1)$-iterable. By
 Theorem \ref{thm:solidity}, $U$ is $(m+1)$-solid.
Since $E$ is weakly amenable to $M$ and $M$ is $(m+1)$-solid,
 by Lemma \ref{lem:basic_fs_pres},
we have
 $\rho_{m+1}^U=\rho_{m+1}^M$ and $p_{m+1}^U=i^{M,m}_E(p_{m+1}^M)$.
Letting $E$ be generated by $x$, it follows that
\[ U=\Hull_{m+1}^M(\rho_{m+1}^U\cup\{x\}),\]
so  Theorem \ref{thm:finite_gen_hull} applies and proves the corollary.
\end{proof}

\begin{rem}\label{rem:referee_obs}
An anonymous referee noticed that we get an alternate proof of
Theorem \ref{thm:measures_in_mice},
by arguing first as there, until the point at which $\bar{M},\bar{U},\bar{x}$ have been defined
and $\bar{M},\bar{U}$ shown iterable. Then, since $\bar{M}=\core_{m+1}(\bar{U})$ and $\bar{U}=\Hull_{m+1}^{\bar{U}}(\{\bar{x}\})$,
Theorem \ref{thm:finite_gen_hull}
gives that $\bar{U}$ is an iterate of $\bar{M}$ via a strongly finite terminally non-dropping $m$-maximal tree, which is a contradiction as in the proof of Theorem \ref{thm:measures_in_mice}.
\end{rem}

The theorems in this paper leave the following questions open:
\begin{ques}\label{ques:om_1-it_suffices?}
Does $\om_1$-iterability suffice
to prove fine structure (either for normal trees,
or for stacks)? That is, let
$m<\om$.
Is it true
that every $(m,\om_1,\om_1)$-iterable premouse
is $(m+1)$-solid?
Does this follow
from $(m,\om_1)$-iterability?
Likewise for the other fine structural properties.
\end{ques}

It is not only in comparison termination
that $(\om_1+1)$-iterability is used;
%conf
it is also used in the proof in \cite[Theorem 9.6]{iter_for_stacks},
which was in turn used to verify that ${<^{\para}_k}$ is wellfounded.
Of course these consequences
are used very heavily in the argument.
 On the other hand,
 the non-solid premouse constructed in \S\ref{sec:non-solid} is not $\om_1$-iterable, and it does not seem clear how one  would produce an $\om_1$-iterable one. So although the author is not aware of any applications of an answer to the question, it seems that resolving it would require some significant new ideas.

\begin{ques}
 Does normal iterability imply iterability for stacks
 (without assuming any condensation for the normal strategy)?
\end{ques}

\begin{ques}\label{ques:fsfni_for_long_exts}
Does normal iterability suffice to prove the basic fine structural results for mice with long extenders, in the $\kappa^{+}$-supercompactness  region? Here for definiteness, we adopt the rules for normal trees  as in the preprint \cite{kappa-plus_v2}.

What about for $\kappa^{+n}$-supercompactness? (In this case, the iteration rules would need to be determined in order to make the question precise.)
\end{ques}

The author does not have any reason to expect particular difficulties combining the methods of this paper with the known fine structural arguments for long extenders,
but there are plenty of details to work through, and he has not tried to do so.

\begin{ques}\label{ques:least_mouse_not_iterate_of_core}
Let $M$ be the least active sound mouse such that $F^M$ is of limit space type (see Definition \ref{dfn:limit_space_type}). Is there a mouse $U$ whose core is $M$, but
such that $U$  is not an iterate of $M$?
\end{ques}

Corollary \ref{cor:universal_weasel_not_iterate_of_K} used a mouse $M$ satisfying a slightly stronger hypothesis to get such a $U$.
If the answer is  affirmative, then
by Theorem \ref{tm:if_U_not_iterate_of_core},
$M$ would be least for which  such such a $U$ exists.

\begin{ques}\label{ques:sigma-closed_forcing_ground_of_mouse}
Let $M$ be a non-tame mouse modelling ZFC,
and $W\sub M$ be a ground of the universe $\univ{M}$ of $M$, via a forcing $\PP\in W$ such that $W\sats$ ``$\PP$ is $\sigma$-closed''. Is $W=\univ{M}$?
\end{ques}

%***check
 Theorem \ref{thm:sigma_grounds}
 does not answer this question,
 because that result assumes $M|\om_1^M\in W$.
If $M$ is instead assumed to be tame,
the answer is ``yes'', by \cite[Theorem 4.7]{odle_v2}.

\section*{Acknowledgements}

This work
partially supported by Deutsche Forschungsgemeinschaft (DFG, German
Research
Foundation) under Germany's Excellence Strategy EXC 2044-390685587,
Mathematics M\"unster: Dynamics-Geometry-Structure.

Work in \S\ref{sec:iterates_of_cores} and editing in general funded by the Austrian Science Fund (FWF) [10.55776/Y1498].

\bibliographystyle{plain}
\bibliography{../bibliography/bibliography}

\end{document}